\documentclass[12pt]{amsart}
\usepackage[utf8]{inputenc}
\usepackage[margin=1in]{geometry}
\usepackage{color,amsmath,amssymb,mathtools,graphicx,bm,ulem,soul,stmaryrd,esvect,cancel}
\usepackage{amsthm}
\usepackage{bbold}
\usepackage{epsfig,fancyhdr}
\usepackage{dsfont} 
\usepackage{graphicx,float}
\usepackage{graphicx}
\usepackage{epsfig,color,fancyhdr,setspace}
\usepackage{epstopdf}
\usepackage{import}
\usepackage{lineno}
\usepackage{stackrel}
\usepackage[dvipsnames]{xcolor}
\usepackage[allcolors = blue,colorlinks]{hyperref}
\usepackage[abs]{overpic}
\usepackage[center]{caption}
\usepackage{subcaption}
\usepackage{pst-node}
\usepackage{tikz-cd}
\usepackage{enumitem}
\counterwithin{figure}{section}
\usepackage{wasysym}
\usepackage{marvosym}
\usepackage{extarrows}
\usepackage{contour}
\usepackage{tikz}
\usetikzlibrary{decorations.pathreplacing}
\newtheorem{theorem}{Theorem}[section]

\newtheorem{lemma}[theorem]{Lemma}

\newtheorem{remark}[theorem]{Remark}
\newtheorem{corollary}[theorem]{Corollary}
\newtheorem{example}[theorem]{Example}
\usetikzlibrary{shapes.geometric} 

\def\a{\alpha}

\def\l{\lambda}
\def\N{\mathbb{N}}

\def\edge{\,\textrm{\textbf{-}}\, }

\numberwithin{equation}{section}
\def\TT{\mathbb{T}}
\def\MM{\mathbb{M}}
\def\mob{\text{M{\" o}b}}
\def\tor{\text{Tor}}
\def\V{\mathcal{V}}
\def\R{\mathbb{R}}
\def\CC{\mathcal{C}}
\def\G{\mathcal{G}}
\def\b{\beta}
\def\del{\partial}
\def\Z{\mathbb{Z}}

\def\K{\mathbb{K}}

\def\I{\mathcal{I}}
\def\e{\epsilon}

\def\Th{\Theta}
\def\id{\text{id}}

\def\nth{\text{nth}}
\def\sth{\text{sth}}
\def\int{\text{int}}
\def\ext{\text{ext}}

\def\SQ{\mathcal{SQ}}
\def\iso{\text{iso.}}
\def\hs{\text{h.s.}}
\def\isosim{\overset{\text{iso.}}{\sim}}
\def\hssim{\overset{\text{h.s.}}{\sim}}

\def\SS{\mathfrak{S}}

\DeclareFontFamily{U}{mathb}{}
\DeclareFontShape{U}{mathb}{m}{n}{
     <-5.5> mathb5
  <5.5-6.5> mathb6
  <6.5-7.5> mathb7
  <7.5-8.5> mathb8
  <8.5-9.5> mathb9
  <9.5-11>  mathb10
  <11->     mathb12
}{}
\DeclareSymbolFont{mathb}{U}{mathb}{m}{n}
\DeclareMathSymbol{\lefttorightarrow}{3}{mathb}{"FC}
\DeclareMathSymbol{\righttoleftarrow}{3}{mathb}{"FD}

\makeatletter
\DeclareRobustCommand{\looparrow}[1]{%
  \mathrel{\mathpalette\looparrow@{#1}}%
}
\newcommand{\looparrow@}[2]{\reflectbox{$\m@th#1#2$}}
\makeatother

\newcommand{\beq}{\begin{equation}}
\newcommand{\eq}{\end{equation}}

\newtheorem{prop}[theorem]{Proposition}

\newcounter{minidef}[section]

\newcommand{\mdef}{\refstepcounter{theorem} 
\medskip \noindent ({\bf \thetheorem}) }

\def\xmapsfrom#1{ \reflectbox{$\xmapsto{\reflectbox{$#1$}}$}}
\def\xmapstoo#1{\reflectbox{$\xmapsfrom{\reflectbox{$#1$}}$}}
\def\del{\partial}

\newtheorem*{namedtheorem}{\theoremname}
\newcommand{\theoremname}{testing}
\newenvironment{named}[1]{\renewcommand{\theoremname}{#1}\begin{namedtheorem}}{\end{namedtheorem}}

\def\frone#1{\begin{tikzpicture}[every node/.style={shape=circle},inner sep=1,scale=#1,anchor=base,baseline=(current bounding box.center)]

\pgfdeclarelayer{1}
\pgfdeclarelayer{2}

\pgfsetlayers{main,
1,
2} 

\draw[] (-9,0)--(0,0);
\draw[] (-9,0)--(-9,9);
\draw[] (-9,9)--(0,9);
\draw[] (0,0) --(0,1);
\draw[] (0,2) -- (0,3); 
\node[fill=black,inner sep =0] (A1) at (0,1) {};
\node[fill=black,inner sep =0] (B1) at (0,2) {};
\draw[] (0,4) -- (0,5); 
\node[fill=black,inner sep =0] (A2) at (0,3) {};
\node[fill=black,inner sep =0] (B2) at (0,4) {};
\draw[] (0,6) -- (0,7); 
\node[fill=black,inner sep =0] (A3) at (0,5) {};
\node[fill=black,inner sep =0] (B3) at (0,6) {};
\draw[] (0,8) -- (0,9); 
\node[fill=black,inner sep =0] (A4) at (0,7) {};
\node[fill=black,inner sep =0] (B4) at (0,8) {};

\node[fill=black] (S1) at (-6.75,0) {};
\node[fill=black] (S2) at (-4.5,0) {};
\node[fill=black] (S3) at (-2.25,0) {};

\node[fill=black] (N1) at (-4.5,9) {};

\pgfonlayer{2}

\draw[white,line width=#1*5mm] (0,3.5)   arc (-90:90:2cm);
\draw[white,line width=#1*5mm] (0,3.75)   arc (-90:90:2cm);
\draw[white,line width=#1*5mm] (0,3.25)   arc (-90:90:2cm);
\draw[] (0,3) arc (-90:90:2cm);
\draw[] (0,4) arc (-90:90:2cm);
\draw[] (2,5)--(2,6);

\node[fill=black] (A21) at (0,3.5) {};
\node[fill=black] (A41) at (0,7.5) {};

\endpgfonlayer

\pgfonlayer{1}

\draw[white,line width=#1*10mm] (0,1.5)   arc (-90:90:2cm);
\draw[] (0,1) arc (-90:90:2.5cm);
\draw[] (0,2) arc (-90:90:1.5cm);

\node[fill=black] (A11) at (0,1.3333333333333333) {};
\node[fill=black] (A31) at (0,5.333333333333333) {};
\node[fill=black] (A12) at (0,1.6666666666666665) {};
\node[fill=black] (A32) at (0,5.666666666666667) {};

\endpgfonlayer

\node[anchor=north] () at (S1) {\footnotesize$(0,1)$};
\node[anchor=north] () at (S2) {\footnotesize$(0,2)$};
\node[anchor=north] () at (S3) {\footnotesize$(0,3)$};

\node[anchor=east] () at ([yshift=-0.1cm]A11) {\footnotesize$(1,1)$};
\node[anchor=east] () at ([yshift=0.1cm]A12) {\footnotesize$(1,2)$};
\node[anchor=east] () at (A21) {\footnotesize$(2,1)$};

\node[anchor=east] () at ([yshift=-0.1cm]A31) {\footnotesize$(3,1)$};
\node[anchor=east] () at ([yshift=0.1cm]A32) {\footnotesize$(3,2)$};

\node[anchor=east] () at ([yshift=0.1cm]A41) {\footnotesize$(4,1)$};

\node[anchor=south] () at (N1) {\footnotesize$(5,1)$};

\end{tikzpicture}
}

\def\frtwo#1{\begin{tikzpicture}[every node/.style={shape=circle},inner sep=1,scale=#1,anchor=base,baseline=(current bounding box.center)]

\pgfdeclarelayer{1}
\pgfdeclarelayer{2}
\pgfdeclarelayer{3}

\pgfsetlayers{main,
1,
2,3} 

\draw[] (-13,0)--(0,0);
\draw[] (-13,0)--(-13,13);
\draw[] (-13,13)--(0,13);
\draw[] (0,0) --(0,1);
\draw[] (0,2) -- (0,3); 
\node[fill=black,inner sep =0] (A1) at (0,1) {};
\node[fill=black,inner sep =0] (B1) at (0,2) {};
\draw[] (0,4) -- (0,5); 
\node[fill=black,inner sep =0] (A2) at (0,3) {};
\node[fill=black,inner sep =0] (B2) at (0,4) {};
\draw[] (0,6) -- (0,7); 
\node[fill=black,inner sep =0] (A3) at (0,5) {};
\node[fill=black,inner sep =0] (B3) at (0,6) {};
\draw[] (0,8) -- (0,9); 
\node[fill=black,inner sep =0] (A4) at (0,7) {};
\node[fill=black,inner sep =0] (B4) at (0,8) {};
\draw[] (0,10) -- (0,11); 
\node[fill=black,inner sep =0] (A5) at (0,9) {};
\node[fill=black,inner sep =0] (B5) at (0,10) {};
\draw[] (0,12) -- (0,13); 
\node[fill=black,inner sep =0] (A6) at (0,11) {};
\node[fill=black,inner sep =0] (B6) at (0,12) {};

\node[fill=black] (S1) at (-8.666666666666668,0) {};
\node[fill=black] (S2) at (-4.333333333333334,0) {};

\node[fill=black] (N1) at (-10.4,13) {};
\node[fill=black] (N2) at (-7.8,13) {};
\node[fill=black] (N3) at (-5.2,13) {};
\node[fill=black] (N4) at (-2.5999999999999996,13) {};

\pgfonlayer{3}

\draw[white,line width=#1*5mm] (0,3.5)   arc (-90:90:4cm);
\draw[white,line width=#1*5mm] (0,3.75)   arc (-90:90:4cm);
\draw[white,line width=#1*5mm] (0,3.25)   arc (-90:90:4cm);
\draw[] (0,3) arc (-90:90:4cm);
\draw[] (0,4) arc (-90:90:4cm);
\draw[] (4,7)--(4,8);

\node[fill=black] (A21) at (0,3.3333333333333335) {};
\node[fill=black] (A61) at (0,11.333333333333334) {};
\node[fill=black] (A22) at (0,3.6666666666666665) {};
\node[fill=black] (A62) at (0,11.666666666666666) {};

\draw[red] (A21) arc (-90:90:4.0);
\draw[blue] (A22) arc (-90:90:4.0);

\endpgfonlayer

\pgfonlayer{2}

\draw[white,line width=#1*10mm] (0,5.5)   arc (-90:90:2cm);
\draw[] (0,5) arc (-90:90:2.5cm);
\draw[] (0,6) arc (-90:90:1.5cm);

\node[fill=black] (A31) at (0,5.25) {};
\node[fill=black] (A51) at (0,9.25) {};
\node[fill=black] (A32) at (0,5.5) {};
\node[fill=black] (A52) at (0,9.5) {};
\node[fill=black] (A33) at (0,5.75) {};
\node[fill=black] (A53) at (0,9.75) {};

\draw[red] (A31) arc (-90:90:2.25);
\draw[blue] (A32) arc (-90:90:2.0);
\draw[green] (A33) arc (-90:90:1.75);

\endpgfonlayer

\pgfonlayer{1}

\draw[white,line width=#1*10mm] (0,1.5)   arc (-90:90:3cm);
\draw[] (0,1) arc (-90:90:3.5cm);
\draw[] (0,2) arc (-90:90:2.5cm);

\node[fill=black] (A11) at (0,1.5) {};
\node[fill=black] (A41) at (0,7.5) {};

\draw[red] (A11) arc (-90:90:3.0);
\endpgfonlayer

\end{tikzpicture}
}

\def\SWBone#1{\begin{tikzpicture}[every node/.style={shape=circle},inner sep=0,scale=#1,anchor=base,baseline=(current bounding box.center)]

\pgfdeclarelayer{1}
\pgfdeclarelayer{2}

\pgfsetlayers{main,
1,
2} 
\pgfdeclarelayer{1}
\pgfdeclarelayer{2}

\pgfsetlayers{main,
1,
2} 

\draw[] (-9,0)--(0,0);
\draw[] (-9,0)--(-9,9);
\draw[] (-9,9)--(0,9);
\draw[] (0,0) --(0,1);
\draw[] (0,2) -- (0,3); 
\node[fill=black,inner sep =0] (C1) at (0,1) {};
\node[fill=black,inner sep =0] (D1) at (0,2) {};
\draw[] (0,4) -- (0,5); 
\node[fill=black,inner sep =0] (C2) at (0,3) {};
\node[fill=black,inner sep =0] (D2) at (0,4) {};
\draw[] (0,6) -- (0,7); 
\node[fill=black,inner sep =0] (C3) at (0,5) {};
\node[fill=black,inner sep =0] (D3) at (0,6) {};
\draw[] (0,8) -- (0,9); 
\node[fill=black,inner sep =0] (C4) at (0,7) {};
\node[fill=black,inner sep =0] (D4) at (0,8) {};

\node[fill=black] (So1) at (-6.75,0) {};
\node[fill=black] (So2) at (-4.5,0) {};
\node[fill=black] (So3) at (-2.25,0) {};

\node[fill=black] (Nor1) at (-4.5,9) {};

\pgfonlayer{2}

\draw[white,line width=#1*5mm] (0,3.5)   arc (-90:90:2cm);
\draw[white,line width=#1*5mm] (0,3.75)   arc (-90:90:2cm);
\draw[white,line width=#1*5mm] (0,3.25)   arc (-90:90:2cm);
\draw[] (0,3) arc (-90:90:2cm);
\draw[] (0,4) arc (-90:90:2cm);
\draw[] (2,5)--(2,6);

\node[fill=black] (C21) at (0,3.5) {};
\node[fill=black] (C41) at (0,7.5) {};

\endpgfonlayer

\pgfonlayer{1}

\draw[white,line width=#1*10mm] (0,1.5)   arc (-90:90:2cm);
\draw[] (0,1) arc (-90:90:2.5cm);
\draw[] (0,2) arc (-90:90:1.5cm);

\node[fill=black] (C11) at (0,1.3333333333333333) {};
\node[fill=black] (C31) at (0,5.333333333333333) {};
\node[fill=black] (C12) at (0,1.6666666666666665) {};
\node[fill=black] (C32) at (0,5.666666666666667) {};

\endpgfonlayer

\draw[blue] (So3).. controls (-2.25,0.8888888888888888) and (-1.5,1.3333333333333333) .. (C11);
\draw[blue] (So1).. controls (-6.75,3.777777777777778) and (-4.5,5.666666666666667) .. (C32);

\pgfonlayer{1}
\draw[blue] (C11) arc (-90:90:2.1666666666666665);
\endpgfonlayer

\draw[red] (So2).. controls (-4.5,3.5555555555555554) and (-3.0,5.333333333333333) .. (C31);
\draw[red] (C12).. controls (-0.9166666666666667,1.6666666666666665) and (-0.9166666666666667,3.5) .. (C21);
\draw[red] (Nor1).. controls (-4.5,8.0) and (-3.0,7.5) .. (C41);

\pgfonlayer{1}
\draw[red] (C31) arc (90:-90:1.8333333333333333);
\endpgfonlayer

\pgfonlayer{2}
\draw[red] (C21) arc (-90:90:2.0);
\endpgfonlayer

\end{tikzpicture}
}

\def\SWBonef#1{\begin{tikzpicture}[every node/.style={shape=circle},inner sep=0,scale=#1,anchor=base,baseline=(current bounding box.center)]
\pgfdeclarelayer{1}
\pgfdeclarelayer{2}

\pgfsetlayers{main,
1,
2} 

\draw[] (-9,0)--(0,0);
\draw[] (-9,0)--(-9,9);
\draw[] (-9,9)--(0,9);
\draw[] (0,0) --(0,1);
\draw[] (0,2) -- (0,3); 
\node[fill=black,inner sep =0] (C1) at (0,1) {};
\node[fill=black,inner sep =0] (D1) at (0,2) {};
\draw[] (0,4) -- (0,5); 
\node[fill=black,inner sep =0] (C2) at (0,3) {};
\node[fill=black,inner sep =0] (D2) at (0,4) {};
\draw[] (0,6) -- (0,7); 
\node[fill=black,inner sep =0] (C3) at (0,5) {};
\node[fill=black,inner sep =0] (D3) at (0,6) {};
\draw[] (0,8) -- (0,9); 
\node[fill=black,inner sep =0] (C4) at (0,7) {};
\node[fill=black,inner sep =0] (D4) at (0,8) {};

\node[fill=black] (So1) at (-4.5,0) {};

\node[fill=black] (Nor1) at (-6.75,9) {};
\node[fill=black] (Nor2) at (-4.5,9) {};
\node[fill=black] (Nor3) at (-2.25,9) {};

\pgfonlayer{2}

\draw[white,line width=#1*5mm] (0,1.5)   arc (-90:90:2cm);
\draw[white,line width=#1*5mm] (0,1.75)   arc (-90:90:2cm);
\draw[white,line width=#1*5mm] (0,1.25)   arc (-90:90:2cm);
\draw[] (0,1) arc (-90:90:2cm);
\draw[] (0,2) arc (-90:90:2cm);
\draw[] (2,3)--(2,4);

\node[fill=black] (C11) at (0,1.5) {};
\node[fill=black] (C31) at (0,5.5) {};

\endpgfonlayer

\pgfonlayer{1}

\draw[white,line width=#1*10mm] (0,3.5)   arc (-90:90:2cm);
\draw[] (0,3) arc (-90:90:2.5cm);
\draw[] (0,4) arc (-90:90:1.5cm);

\node[fill=black] (C21) at (0,3.3333333333333335) {};
\node[fill=black] (C41) at (0,7.333333333333333) {};
\node[fill=black] (C22) at (0,3.6666666666666665) {};
\node[fill=black] (C42) at (0,7.666666666666667) {};

\endpgfonlayer

\draw[blue] (Nor3).. controls (-2.25,8.11111111111111) and (-1.5,7.666666666666667) .. (C42);
\draw[blue] (Nor1).. controls (-6.75,5.222222222222223) and (-4.5,3.3333333333333335) .. (C21);

\pgfonlayer{1}
\draw[blue] (C42) arc (90:-90:2.1666666666666665);
\endpgfonlayer

\draw[red] (Nor2).. controls (-4.5,5.444444444444445) and (-3.0,3.6666666666666665) .. (C22);
\draw[red] (C41).. controls (-0.9166666666666667,7.333333333333333) and (-0.9166666666666667,5.5) .. (C31);
\draw[red] (So1).. controls (-4.5,1.0) and (-3.0,1.5) .. (C11);

\pgfonlayer{1}
\draw[red] (C22) arc (-90:90:1.8333333333333333);
\endpgfonlayer

\pgfonlayer{2}
\draw[red] (C31) arc (90:-90:2.0);
\endpgfonlayer

\end{tikzpicture}}

\def\LocalU#1{\begin{tikzpicture}[every node/.style={shape=circle},inner sep=0,scale=#1,anchor=base,baseline=(current bounding box.center)]

\draw[dashed] (-3,0)--(2.5,0);
\draw[dashed] (-3,5)--(2.5,5);
\draw[dashed] (-3,0)--(-3,5);
\draw[dashed] (2.5,0)--(2.5,5);

\pgfdeclarelayer{1}
\pgfsetlayers{main,1}
\draw (0,0) -- (0,0.5);
 \draw (0,1.5) -- (0,2); 
 \draw (0,3) -- (0,3.5); 
 \draw (0,4.5) -- (0,5);
\draw (-1,0) -- (-1,5);
\node[] at (0,2.3) {$\vdots $};
\node at (-2,2.5) {$E$};
\pgfonlayer{1}

\draw[white,line width=#1*10mm] (0,1)   arc (-90:90:1.5cm);
\draw[] (0,0.5) arc (-90:90:2cm);
\draw[] (0,1.5) arc (-90:90:1cm);

\node[fill=black] (A11) at (0,1.1428571428571428-0.5) {};
\node[fill=black] (A31) at (0,5.142857142857143-1.5) {};
\node[fill=black] (A12) at (0,1.2857142857142856-0.5) {};
\node[fill=black] (A32) at (0,5.285714285714286-1.5) {};
\node[fill=black] (A13) at (0,1.4285714285714286-0.5) {};
\node[fill=black] (A33) at (0,5.428571428571429-1.5) {};
\node[fill=black] (A14) at (0,1.5714285714285714-0.5) {};
\node[fill=black] (A34) at (0,5.571428571428571-1.5) {};
\node[fill=black] (A15) at (0,1.7142857142857144-0.5) {};
\node[fill=black] (A35) at (0,5.714285714285714-1.5) {};
\node[fill=black] (A16) at (0,1.8571428571428572-0.5) {};
\node[fill=black] (A36) at (0,5.857142857142857-1.5) {};
\endpgfonlayer
\draw[purple] (A11) -- +(-1,0);
\draw[purple] (A36) -- +(-1,0);

\pgfonlayer{1}
\draw[purple] (A11) arc (-90:90:2.357142857142857-0.5);
\endpgfonlayer
\draw[blue] (A12) -- +(-1,0);
\draw[blue] (A35) -- +(-1,0);

\pgfonlayer{1}
\draw[blue] (A12) arc (-90:90:2.2142857142857144-0.5);
\endpgfonlayer
\draw[green] (A13) -- +(-1,0);
\draw[green] (A34) -- +(-1,0);

\pgfonlayer{1}
\draw[green] (A13) arc (-90:90:2.0714285714285716-0.5);
\endpgfonlayer
\draw[yellow] (A14) -- +(-1,0);
\draw[yellow] (A33) -- +(-1,0);

\pgfonlayer{1}
\draw[yellow] (A14) arc (-90:90:1.9285714285714286-0.5);
\endpgfonlayer
\draw[orange] (A15) -- +(-1,0);
\draw[orange] (A32) -- +(-1,0);

\pgfonlayer{1}
\draw[orange] (A15) arc (-90:90:1.7857142857142858-0.5);
\endpgfonlayer
\draw[red] (A16) -- +(-1,0);
\draw[red] (A31) -- +(-1,0);

\pgfonlayer{1}
\draw[red] (A16) arc (-90:90:1.6428571428571428-0.5);
\endpgfonlayer

\draw[blue,double distance=2pt] (0,2.5) -- +(-1,0);

\end{tikzpicture}}

\def\LocalT#1{\begin{tikzpicture}[every node/.style={shape=circle},inner sep=0, scale=#1,anchor=base,baseline=(current bounding box.center)]

\draw[dashed] (-3,0)--(2.5,0);
\draw[dashed] (-3,5)--(2.5,5);
\draw[dashed] (-3,0)--(-3,5);
\draw[dashed] (2.5,0)--(2.5,5);

\pgfdeclarelayer{1}
\pgfsetlayers{main,1}
\draw (0,0) -- (0,0.5);
 \draw (0,1.5) -- (0,2); 
 \draw (0,3) -- (0,3.5); 
 \draw (0,4.5) -- (0,5);
\draw (-1,0) -- (-1,5);
\node[] at (0,2.3) {$\vdots $};
\node at (-2,2.5) {$E$};
\pgfonlayer{1}

\draw[white,line width=#1*5mm] (0,1.5-0.5)   arc (-90:90:1.5cm);
\draw[white,line width=#1*5mm] (0,1.75-0.5)   arc (-90:90:1.5cm);
\draw[white,line width=#1*5mm] (0,1.25-0.5)   arc (-90:90:1.5cm);
\draw[] (0,0.5) arc (-90:90:1.5cm);
\draw[] (0,1.5) arc (-90:90:1.5cm);
\draw[] (1.5,2)--(1.5,3);

\node[fill=black] (A11) at (0,1.1428571428571428-0.5) {};
\node[fill=black] (A31) at (0,5.142857142857143-1.5) {};
\node[fill=black] (A12) at (0,1.2857142857142856-0.5) {};
\node[fill=black] (A32) at (0,5.285714285714286-1.5) {};
\node[fill=black] (A13) at (0,1.4285714285714286-0.5) {};
\node[fill=black] (A33) at (0,5.428571428571429-1.5) {};
\node[fill=black] (A14) at (0,1.5714285714285714-0.5) {};
\node[fill=black] (A34) at (0,5.571428571428571-1.5) {};
\node[fill=black] (A15) at (0,1.7142857142857144-0.5) {};
\node[fill=black] (A35) at (0,5.714285714285714-1.5) {};
\node[fill=black] (A16) at (0,1.8571428571428572-0.5) {};
\node[fill=black] (A36) at (0,5.857142857142857-1.5) {};
\endpgfonlayer
\draw[purple] (A11) -- +(-1,0);
\draw[purple] (A31) -- +(-1,0);

\pgfonlayer{1}
\draw[purple] (A11) arc (-90:90:1.5);
\endpgfonlayer
\draw[blue] (A12) -- +(-1,0);
\draw[blue] (A32) -- +(-1,0);

\pgfonlayer{1}
\draw[blue] (A12) arc (-90:90:1.5);
\endpgfonlayer
\draw[green] (A13) -- +(-1,0);
\draw[green] (A33) -- +(-1,0);

\pgfonlayer{1}
\draw[green] (A13) arc (-90:90:1.5);
\endpgfonlayer
\draw[yellow] (A14) -- +(-1,0);
\draw[yellow] (A34) -- +(-1,0);

\pgfonlayer{1}
\draw[yellow] (A14) arc (-90:90:1.5);
\endpgfonlayer
\draw[orange] (A15) -- +(-1,0);
\draw[orange] (A35) -- +(-1,0);

\pgfonlayer{1}
\draw[orange] (A15) arc (-90:90:1.5);
\endpgfonlayer
\draw[red] (A16) -- +(-1,0);
\draw[red] (A36) -- +(-1,0);

\pgfonlayer{1}
\draw[red] (A16) arc (-90:90:1.5);
\endpgfonlayer

\draw[blue,double distance=2pt] (0,2.5) -- +(-1,0);

\end{tikzpicture}}

\def\genSWB#1{\begin{tikzpicture}[every node/.style={shape=circle},inner sep=0,scale=#1,anchor=base,baseline=(current bounding box.center)]

\draw[] (0,2.25) -- (0,2.75);
\draw[] (0,0.25) -- (0,0.75);
\draw[] (0,0.25) -- (-2.5,0.25);
\draw[] (0,2.75) -- (-2.5,2.75);
\draw[] (-2.5,0.25) -- (-2.5,2.75);

\node[] (S1) at (-1.875,0.25) {};
\node[] (S3) at (-0.625,0.25) {};
\node[] (SS1) at (-5/3-1/2+0.375,0.75) {};
\node[] (SS3) at (-5/3-1/2+3*+0.375,0.75) {};

\node[] (N1) at (-1.875,2.75) {};
\node[] (N3) at (-0.625,2.75) {};
\node[] (NN1) at (-5/3-1/2+0.375,2.25) {};
\node[] (NN3) at (-5/3-1/2+3*+0.375,2.25) {};

\draw[] (-1/3,1) -- (2/3,1);
\draw[] (-1/3,2) -- (2/3,2);
\draw[] (-1/3,1) -- (-1/3,2);
\draw[] (2/3,1) -- (2/3,2);
\draw[] (2/3+1/4,1) -- (2/3+1/4,2);
\draw[] (2/3,0.75) -- (0,0.75);
\draw[] (2/3,2.25) -- (0,2.25);
\draw[] (2/3,2.25) arc (90:0:0.25cm);
\draw[] (2/3,0.75) arc (-90:0:0.25cm);
\node[] (box) at (1/6,1.5) { $(P,s)$};

\draw[] (-2/3,0.75) -- (-2/3,2.25);
\draw[] (-5/3-1/2,1-1/4) -- (-5/3-1/2,2+1/4);
\draw[] (-2/3,1-1/4) -- (-5/3-1/2,1-1/4);
\draw[] (-2/3,2+1/4) -- (-5/3-1/2,2+1/4);
\node[] (L1) at (-7/6-1/4,1.5) {\large $E$};

\draw[blue] (-2/3,1.25) -- (-1/3,1.25);
\draw[blue] (-2/3,1.75) -- (-1/3,1.75);
\node[] () at (-0.5,1.35) {$\vdots$};
\node[] () at (-0.5,1.875) {\footnotesize $2 C$};

\draw[blue] (S1).. controls ([yshift=0.5cm]S1) and ([yshift=-0.5cm]SS1)..(SS1);
\draw[blue] (S3).. controls ([yshift=0.5cm]S3) and ([yshift=-0.5cm]SS3)..(SS3);
\node[] () at (-5/3-1/2+2*0.375,0.625) {$\dots$};
\node[] () at (-1.25,0.35) {$n$};
\draw[blue] (N1).. controls ([yshift=-0.5cm]N1) and ([yshift=0.5cm]NN1)..(NN1);
\draw[blue] (N3).. controls ([yshift=-0.5cm]N3) and ([yshift=0.5cm]NN3)..(NN3);
\node[] () at (-5/3-1/2+2*0.375,2.375) {$\dots$};
\node[] () at (-1.25,2.55) {$m$};

\end{tikzpicture}}

\def\SWBtw#1{\begin{tikzpicture}[every node/.style={shape=circle},inner sep=0,scale=#1,anchor=base,baseline=(current bounding box.center)]

\pgfdeclarelayer{1}
\pgfdeclarelayer{2}
\pgfdeclarelayer{3}

\pgfsetlayers{main,
1,
2,
3} 

\draw[] (-13,0)--(0,0);
\draw[] (-13,0)--(-13,13);
\draw[] (-13,13)--(0,13);
\draw[] (0,0) --(0,1);
\draw[] (0,2) -- (0,3); 
\node[fill=black,inner sep =0] (A1) at (0,1) {};
\node[fill=black,inner sep =0] (B1) at (0,2) {};
\draw[] (0,4) -- (0,5); 
\node[fill=black,inner sep =0] (A2) at (0,3) {};
\node[fill=black,inner sep =0] (B2) at (0,4) {};
\draw[] (0,6) -- (0,7); 
\node[fill=black,inner sep =0] (A3) at (0,5) {};
\node[fill=black,inner sep =0] (B3) at (0,6) {};
\draw[] (0,8) -- (0,9); 
\node[fill=black,inner sep =0] (A4) at (0,7) {};
\node[fill=black,inner sep =0] (B4) at (0,8) {};
\draw[] (0,10) -- (0,11); 
\node[fill=black,inner sep =0] (A5) at (0,9) {};
\node[fill=black,inner sep =0] (B5) at (0,10) {};
\draw[] (0,12) -- (0,13); 
\node[fill=black,inner sep =0] (A6) at (0,11) {};
\node[fill=black,inner sep =0] (B6) at (0,12) {};

\node[fill=black] (S1) at (-8.666666666666668,0) {};
\node[fill=black] (S2) at (-4.333333333333334,0) {};

\node[fill=black] (N1) at (-10.4,13) {};
\node[fill=black] (N2) at (-7.8,13) {};
\node[fill=black] (N3) at (-5.2,13) {};
\node[fill=black] (N4) at (-2.5999999999999996,13) {};

\pgfonlayer{3}

\draw[white,line width=#1*10mm] (0,1.5)   arc (-90:90:3cm);
\draw[] (0,1) arc (-90:90:3.5cm);
\draw[] (0,2) arc (-90:90:2.5cm);

\node[fill=black] (A11) at (0,1.5) {};
\node[fill=black] (A41) at (0,7.5) {};

\endpgfonlayer

\pgfonlayer{2}

\draw[white,line width=#1*5mm] (0,3.5)   arc (-90:90:4cm);
\draw[white,line width=#1*5mm] (0,3.75)   arc (-90:90:4cm);
\draw[white,line width=#1*5mm] (0,3.25)   arc (-90:90:4cm);
\draw[] (0,3) arc (-90:90:4cm);
\draw[] (0,4) arc (-90:90:4cm);
\draw[] (4,7)--(4,8);

\node[fill=black] (A21) at (0,3.3333333333333335) {};
\node[fill=black] (A61) at (0,11.333333333333334) {};
\node[fill=black] (A22) at (0,3.6666666666666665) {};
\node[fill=black] (A62) at (0,11.666666666666666) {};

\endpgfonlayer

\pgfonlayer{1}

\draw[white,line width=#1*10mm] (0,5.5)   arc (-90:90:2cm);
\draw[] (0,5) arc (-90:90:2.5cm);
\draw[] (0,6) arc (-90:90:1.5cm);

\node[fill=black] (A31) at (0,5.25) {};
\node[fill=black] (A51) at (0,9.25) {};
\node[fill=black] (A32) at (0,5.5) {};
\node[fill=black] (A52) at (0,9.5) {};
\node[fill=black] (A33) at (0,5.75) {};
\node[fill=black] (A53) at (0,9.75) {};

\endpgfonlayer

\draw[blue] (S1) -- (-8.666666666666668,1);
\draw[blue] (S2) -- (-4.333333333333334,1);

\draw[blue] (N1) -- (-10.4,12);
\draw[blue] (N2) -- (-7.8,12);
\draw[blue] (N3) -- (-5.2,12);
\draw[blue] (N4) -- (-2.5999999999999996,12);

\pgfonlayer{3}
\draw[blue] (A11) arc (-90:90:3.0);
\endpgfonlayer
\draw[blue] (A11) -- (-1,1.5);
\draw[blue] (A41) -- (-1,7.5);

\pgfonlayer{2}
\draw[blue] (A21) arc (-90:90:4.0);
\endpgfonlayer
\draw[blue] (A21) -- (-1,3.3333333333333335);
\draw[blue] (A61) -- (-1,11.333333333333334);

\pgfonlayer{2}
\draw[blue] (A22) arc (-90:90:4.0);
\endpgfonlayer
\draw[blue] (A22) -- (-1,3.6666666666666665);
\draw[blue] (A62) -- (-1,11.666666666666666);

\pgfonlayer{1}
\draw[blue] (A31) arc (-90:90:2.25);
\endpgfonlayer
\draw[blue] (A31) -- (-1,5.25);
\draw[blue] (A53) -- (-1,9.75);

\pgfonlayer{1}
\draw[blue] (A32) arc (-90:90:2.0);
\endpgfonlayer
\draw[blue] (A32) -- (-1,5.5);
\draw[blue] (A52) -- (-1,9.5);

\pgfonlayer{1}
\draw[blue] (A33) arc (-90:90:1.75);
\endpgfonlayer
\draw[blue] (A33) -- (-1,5.75);
\draw[blue] (A51) -- (-1,9.25);

\draw (-12,1) -- (-1,1);
\draw (-12,1) -- (-12,12);
\draw (-1,1) -- (-1,12);
\draw (-12,12) -- (-1,12);
\node (boxlabel) at (-6.5,6.5) {\Huge{$E_2$}};

\end{tikzpicture}
}

\def\SWBtwo#1{\begin{tikzpicture}[every node/.style={shape=circle},inner sep=0,scale=#1,anchor=base,baseline=(current bounding box.center)]\pgfdeclarelayer{1}
\pgfdeclarelayer{2}
\pgfdeclarelayer{3}

\pgfsetlayers{main,
1,
2,
3} 

\draw[] (-13,0)--(0,0);
\draw[] (-13,0)--(-13,13);
\draw[] (-13,13)--(0,13);
\draw[] (0,0) --(0,1);
\draw[] (0,2) -- (0,3); 
\node[fill=black,inner sep =0] (C1) at (0,1) {};
\node[fill=black,inner sep =0] (D1) at (0,2) {};
\draw[] (0,4) -- (0,5); 
\node[fill=black,inner sep =0] (C2) at (0,3) {};
\node[fill=black,inner sep =0] (D2) at (0,4) {};
\draw[] (0,6) -- (0,7); 
\node[fill=black,inner sep =0] (C3) at (0,5) {};
\node[fill=black,inner sep =0] (D3) at (0,6) {};
\draw[] (0,8) -- (0,9); 
\node[fill=black,inner sep =0] (C4) at (0,7) {};
\node[fill=black,inner sep =0] (D4) at (0,8) {};
\draw[] (0,10) -- (0,11); 
\node[fill=black,inner sep =0] (C5) at (0,9) {};
\node[fill=black,inner sep =0] (D5) at (0,10) {};
\draw[] (0,12) -- (0,13); 
\node[fill=black,inner sep =0] (C6) at (0,11) {};
\node[fill=black,inner sep =0] (D6) at (0,12) {};

\node[fill=black] (So1) at (-8.666666666666668,0) {};
\node[fill=black] (So2) at (-4.333333333333334,0) {};

\node[fill=black] (Nor1) at (-10.4,13) {};
\node[fill=black] (Nor2) at (-7.8,13) {};
\node[fill=black] (Nor3) at (-5.2,13) {};
\node[fill=black] (Nor4) at (-2.5999999999999996,13) {};

\pgfonlayer{3}

\draw[white,line width=#1*5mm] (0,3.5)   arc (-90:90:4cm);
\draw[white,line width=#1*5mm] (0,3.75)   arc (-90:90:4cm);
\draw[white,line width=#1*5mm] (0,3.25)   arc (-90:90:4cm);
\draw[] (0,3) arc (-90:90:4cm);
\draw[] (0,4) arc (-90:90:4cm);
\draw[] (4,7)--(4,8);

\node[fill=black] (C21) at (0,3.3333333333333335) {};
\node[fill=black] (C61) at (0,11.333333333333334) {};
\node[fill=black] (C22) at (0,3.6666666666666665) {};
\node[fill=black] (C62) at (0,11.666666666666666) {};

\endpgfonlayer

\pgfonlayer{2}

\draw[white,line width=#1*10mm] (0,5.5)   arc (-90:90:2cm);
\draw[] (0,5) arc (-90:90:2.5cm);
\draw[] (0,6) arc (-90:90:1.5cm);

\node[fill=black] (C31) at (0,5.25) {};
\node[fill=black] (C51) at (0,9.25) {};
\node[fill=black] (C32) at (0,5.5) {};
\node[fill=black] (C52) at (0,9.5) {};
\node[fill=black] (C33) at (0,5.75) {};
\node[fill=black] (C53) at (0,9.75) {};

\endpgfonlayer

\pgfonlayer{1}

\draw[white,line width=#1*10mm] (0,1.5)   arc (-90:90:3cm);
\draw[] (0,1) arc (-90:90:3.5cm);
\draw[] (0,2) arc (-90:90:2.5cm);

\node[fill=black] (C11) at (0,1.5) {};
\node[fill=black] (C41) at (0,7.5) {};

\endpgfonlayer

\draw[blue] (So2).. controls (-4.333333333333334,2.444444444444444) and (-2.8888888888888893,3.6666666666666665) .. (C22);
\draw[blue] (Nor4).. controls (-2.5999999999999996,12.11111111111111) and (-1.7333333333333332,11.666666666666666) .. (C62);

\pgfonlayer{3}
\draw[blue] (C22) arc (-90:90:4.0);
\endpgfonlayer

\draw[red] (So1).. controls (-8.666666666666668,6.166666666666667) and (-5.777777777777779,9.25) .. (C51);
\draw[red] (C33).. controls (-0.875,5.75) and (-0.875,7.5) .. (C41);
\draw[red] (C11).. controls (-0.9166666666666666,1.5) and (-0.9166666666666666,3.3333333333333335) .. (C21);
\draw[red] (Nor3).. controls (-5.2,11.88888888888889) and (-3.466666666666667,11.333333333333334) .. (C61);

\pgfonlayer{2}
\draw[red] (C51) arc (90:-90:1.75);
\endpgfonlayer

\pgfonlayer{1}
\draw[red] (C41) arc (90:-90:3.0);
\endpgfonlayer

\pgfonlayer{3}
\draw[red] (C21) arc (-90:90:4.0);
\endpgfonlayer

\draw[green] (Nor1).. controls (-10.4,10.666666666666668) and (-6.933333333333334,9.5) .. (C52);
\draw[green] (C32).. controls (-0.125,5.5) and (-0.125,5.25) .. (C31);
\draw[green] (Nor2).. controls (-7.8,10.833333333333334) and (-5.2,9.75) .. (C53);

\pgfonlayer{2}
\draw[green] (C52) arc (90:-90:2.0);
\endpgfonlayer

\pgfonlayer{2}
\draw[green] (C31) arc (-90:90:2.25);
\endpgfonlayer

\end{tikzpicture}}

\def\SWBth#1{\begin{tikzpicture}[every node/.style={shape=circle},inner sep=0,scale=#1,anchor=base,baseline=(current bounding box.center)]
\pgfdeclarelayer{1}
\pgfdeclarelayer{2}

\pgfsetlayers{main,
1,
2} 

\draw[] (-9,0)--(0,0);
\draw[] (-9,0)--(-9,9);
\draw[] (-9,9)--(0,9);
\draw[] (0,0) --(0,1);
\draw[] (0,2) -- (0,3); 
\node[fill=black,inner sep =0] (A1) at (0,1) {};
\node[fill=black,inner sep =0] (B1) at (0,2) {};
\draw[] (0,4) -- (0,5); 
\node[fill=black,inner sep =0] (A2) at (0,3) {};
\node[fill=black,inner sep =0] (B2) at (0,4) {};
\draw[] (0,6) -- (0,7); 
\node[fill=black,inner sep =0] (A3) at (0,5) {};
\node[fill=black,inner sep =0] (B3) at (0,6) {};
\draw[] (0,8) -- (0,9); 
\node[fill=black,inner sep =0] (A4) at (0,7) {};
\node[fill=black,inner sep =0] (B4) at (0,8) {};

\node[fill=black] (S1) at (-6.75,0) {};
\node[fill=black] (S2) at (-4.5,0) {};
\node[fill=black] (S3) at (-2.25,0) {};

\node[fill=black] (N1) at (-4.5,9) {};

\pgfonlayer{2}

\draw[white,line width=#1*10mm] (0,1.5)   arc (-90:90:2cm);
\draw[] (0,1) arc (-90:90:2.5cm);
\draw[] (0,2) arc (-90:90:1.5cm);

\node[fill=black] (A11) at (0,1.3333333333333333) {};
\node[fill=black] (A31) at (0,5.333333333333333) {};
\node[fill=black] (A12) at (0,1.6666666666666665) {};
\node[fill=black] (A32) at (0,5.666666666666667) {};

\endpgfonlayer

\pgfonlayer{1}

\draw[white,line width=#1*10mm] (0,3.5)   arc (-90:90:2cm);
\draw[] (0,3) arc (-90:90:2.5cm);
\draw[] (0,4) arc (-90:90:1.5cm);

\node[fill=black] (A21) at (0,3.25) {};
\node[fill=black] (A41) at (0,7.25) {};
\node[fill=black] (A22) at (0,3.5) {};
\node[fill=black] (A42) at (0,7.5) {};
\node[fill=black] (A23) at (0,3.75) {};
\node[fill=black] (A43) at (0,7.75) {};

\endpgfonlayer

\draw[blue] (S2).. controls (-4.5,1.111111111111111) and (-3.0,1.6666666666666665) .. (A12);
\draw[blue] (A31).. controls (-0.7916666666666667,5.333333333333333) and (-0.7916666666666667,3.75) .. (A23);
\draw[blue] (A41).. controls (-0.7916666666666667,7.25) and (-0.7916666666666667,5.666666666666667) .. (A32);
\draw[blue] (S3).. controls (-2.25,0.8888888888888888) and (-1.5,1.3333333333333333) .. (A11);

\pgfonlayer{2}
\draw[blue] (A12) arc (-90:90:1.8333333333333333);
\endpgfonlayer

\pgfonlayer{1}
\draw[blue] (A23) arc (-90:90:1.75);
\endpgfonlayer

\pgfonlayer{2}
\draw[blue] (A32) arc (90:-90:2.1666666666666665);
\endpgfonlayer

\draw[green] (N1).. controls (-4.5,8.166666666666666) and (-3.0,7.75) .. (A43);
\draw[green] (S1).. controls (-6.75,2.1666666666666665) and (-4.5,3.25) .. (A21);

\pgfonlayer{1}
\draw[green] (A43) arc (90:-90:2.25);
\endpgfonlayer

\draw[red] (A22).. controls (-2.0,3.5) and (-2.0,7.5) .. (A42);

\pgfonlayer{1}
\draw[red] (A42) arc (90:-90:2.0);
\endpgfonlayer

\end{tikzpicture}}

\def\SWBfo#1{\begin{tikzpicture}[every node/.style={shape=circle},inner sep=0,scale=#1,anchor=base,baseline=(current bounding box.center)]
\pgfdeclarelayer{1}

\pgfsetlayers{main,
1} 

\draw[] (-5,0)--(0,0);
\draw[] (-5,0)--(-5,5);
\draw[] (-5,5)--(0,5);
\draw[] (0,0) --(0,1);
\draw[] (0,2) -- (0,3); 
\node[fill=black,inner sep =0] (A1) at (0,1) {};
\node[fill=black,inner sep =0] (B1) at (0,2) {};
\draw[] (0,4) -- (0,5); 
\node[fill=black,inner sep =0] (A2) at (0,3) {};
\node[fill=black,inner sep =0] (B2) at (0,4) {};

\node[fill=black] (S1) at (-2.5,0) {};

\node[fill=black] (N1) at (-3.75,5) {};
\node[fill=black] (N2) at (-2.5,5) {};
\node[fill=black] (N3) at (-1.25,5) {};

\pgfonlayer{1}

\draw[white,line width=#1*5mm] (0,1.5)   arc (-90:90:1cm);
\draw[white,line width=#1*5mm] (0,1.75)   arc (-90:90:1cm);
\draw[white,line width=#1*5mm] (0,1.25)   arc (-90:90:1cm);
\draw[] (0,1) arc (-90:90:1cm);
\draw[] (0,2) arc (-90:90:1cm);
\draw[] (1,2)--(1,3);

\node[fill=black] (A11) at (0,1.25) {};
\node[fill=black] (A21) at (0,3.25) {};
\node[fill=black] (A12) at (0,1.5) {};
\node[fill=black] (A22) at (0,3.5) {};
\node[fill=black] (A13) at (0,1.75) {};
\node[fill=black] (A23) at (0,3.75) {};

\endpgfonlayer

\draw[blue] (A23).. controls (-1.25,3.75) and (-1.25,1.25) .. (A11);
\draw[blue] (A21).. controls (-0.75,3.25) and (-0.75,1.75) .. (A13);

\pgfonlayer{1}
\draw[blue] (A11) arc (-90:90:1.0);
\endpgfonlayer

\pgfonlayer{1}
\draw[blue] (A13) arc (-90:90:1.0);
\endpgfonlayer

\draw[blue] (N2).. controls (-2.5,4.375) and (-1.25,4.375) .. (N3);

\draw[blue] (S1).. controls (-2.5,2.5) and (-3.75,2.5).. (N1);

\draw[red] (A22).. controls (-1.0,3.5) and (-1.0,1.5) .. (A12);

\pgfonlayer{1}
\draw[red] (A12) arc (-90:90:1.0);
\endpgfonlayer

\end{tikzpicture}}

\def\CPTone#1{\begin{tikzpicture}[every node/.style={shape=circle},inner sep=0,scale=#1,anchor=base,baseline=(current bounding box.center)]
\pgfdeclarelayer{1}
\pgfdeclarelayer{2}

\pgfsetlayers{main,
1,
2} 

\draw[] (-9,0)--(0,0);
\draw[] (-9,0)--(-9,9);
\draw[] (-9,9)--(0,9);
\draw[] (0,0) --(0,1);
\draw[] (0,2) -- (0,3); 
\node[fill=black,inner sep =0] (C1) at (0,1) {};
\node[fill=black,inner sep =0] (D1) at (0,2) {};
\draw[] (0,4) -- (0,5); 
\node[fill=black,inner sep =0] (C2) at (0,3) {};
\node[fill=black,inner sep =0] (D2) at (0,4) {};
\draw[] (0,6) -- (0,7); 
\node[fill=black,inner sep =0] (C3) at (0,5) {};
\node[fill=black,inner sep =0] (D3) at (0,6) {};
\draw[] (0,8) -- (0,9); 
\node[fill=black,inner sep =0] (C4) at (0,7) {};
\node[fill=black,inner sep =0] (D4) at (0,8) {};

\node[fill=black,inner sep =1] (So1) at (-7.875,0) {};
\node[fill=black] (So2) at (-6.75,0) {};
\node[fill=black,inner sep =1] (So3) at (-5.625,0) {};
\node[fill=black] (So4) at (-4.5,0) {};
\node[fill=black,inner sep =1] (So5) at (-3.375,0) {};
\node[fill=black] (So6) at (-2.25,0) {};
\node[fill=black,inner sep =1] (So7) at (-1.125,0) {};

\node[fill=black,inner sep =1] (Nor1) at (-6.75,9) {};
\node[fill=black] (Nor2) at (-4.5,9) {};
\node[fill=black,inner sep =1] (Nor3) at (-2.25,9) {};

\pgfonlayer{2}

\draw[white,line width=#1*5mm] (0,3.5)   arc (-90:90:2cm);
\draw[white,line width=#1*5mm] (0,3.75)   arc (-90:90:2cm);
\draw[white,line width=#1*5mm] (0,3.25)   arc (-90:90:2cm);
\draw[] (0,3) arc (-90:90:2cm);
\draw[] (0,4) arc (-90:90:2cm);
\draw[] (2,5)--(2,6);

\node[fill=black,inner sep =1] (C21) at (0,3.25) {};
\node[fill=black,inner sep =1] (C41) at (0,7.25) {};
\node[fill=black] (C22) at (0,3.5) {};
\node[fill=black] (C42) at (0,7.5) {};
\node[fill=black,inner sep =1] (C23) at (0,3.75) {};
\node[fill=black,inner sep =1] (C43) at (0,7.75) {};

\endpgfonlayer

\pgfonlayer{1}

\draw[white,line width=#1*10mm] (0,1.5)   arc (-90:90:2cm);
\draw[] (0,1) arc (-90:90:2.5cm);
\draw[] (0,2) arc (-90:90:1.5cm);

\node[fill=black,inner sep =1] (C11) at (0,1.1666666666666667) {};
\node[fill=black,inner sep =1] (C31) at (0,5.166666666666667) {};
\node[fill=black] (C12) at (0,1.3333333333333333) {};
\node[fill=black] (C32) at (0,5.333333333333333) {};
\node[fill=black,inner sep =1] (C13) at (0,1.5) {};
\node[fill=black,inner sep =1] (C33) at (0,5.5) {};
\node[fill=black] (C14) at (0,1.6666666666666665) {};
\node[fill=black] (C34) at (0,5.666666666666667) {};
\node[fill=black,inner sep =1] (C15) at (0,1.8333333333333335) {};
\node[fill=black,inner sep =1] (C35) at (0,5.833333333333333) {};

\endpgfonlayer

\draw[blue] (So6).. controls (-2.25,0.8888888888888888) and (-1.5,1.3333333333333333) .. (C12);
\draw[blue] (So2).. controls (-6.75,3.777777777777778) and (-4.5,5.666666666666667) .. (C34);

\pgfonlayer{1}
\draw[blue] (C12) arc (-90:90:2.1666666666666665);
\endpgfonlayer

\draw[red] (So4).. controls (-4.5,3.5555555555555554) and (-3.0,5.333333333333333) .. (C32);
\draw[red] (C14).. controls (-0.9166666666666667,1.6666666666666665) and (-0.9166666666666667,3.5) .. (C22);
\draw[red] (Nor2).. controls (-4.5,8.0) and (-3.0,7.5) .. (C42);

\pgfonlayer{1}
\draw[red] (C32) arc (90:-90:1.8333333333333333);
\endpgfonlayer

\pgfonlayer{2}
\draw[red] (C22) arc (-90:90:2.0);
\endpgfonlayer

\pgfonlayer{1}
\draw[dashed] (C11) arc (-90:90:2.3333333333333335);
\draw[dashed] (C13) arc (-90:90:2.0);
\draw[dashed] (C15) arc (-90:90:1.6666666666666667);
\endpgfonlayer

\pgfonlayer{2}
\draw[dashed] (C21) arc (-90:90:2.0);
\draw[dashed] (C23) arc (-90:90:2.0);
\endpgfonlayer

\draw[dashed] (So7).. controls (-1.125,0.7777777777777778) and (-0.75,1.1666666666666667) .. (C11);
\draw[dashed] (So5).. controls (-3.375,1.0) and (-2.25,1.5) .. (C13);
\draw[dashed] (So5).. controls (-3.375,2.5) and (-2.25,3.75) .. (C23);
\draw[dashed] (So5).. controls (-3.375,3.4444444444444446) and (-2.25,5.166666666666667) .. (C31);
\draw[dashed] (So3).. controls (-5.625,3.6666666666666665) and (-3.75,5.5) .. (C33);
\draw[dashed] (So1).. controls (-7.875,3.888888888888889) and (-5.25,5.833333333333333) .. (C35);
\draw[dashed] (So1).. controls (-7.875,4.833333333333333) and (-5.25,7.25) .. (C41);
\draw[dashed] (So1).. controls (-7.875,4.5) and (-6.75,4.5).. (Nor1);
\draw[dashed] (Nor1).. controls (-6.75,7.833333333333334) and (-4.5,7.25) .. (C41);
\draw[dashed] (Nor1).. controls (-6.75,6.888888888888889) and (-4.5,5.833333333333333) .. (C35);
\draw[dashed] (Nor3).. controls (-2.25,8.066666666666666) and (-1.5,7.6) .. (C43);
\draw[dashed] (C35).. controls (-0.7083333333333333,5.833333333333333) and (-0.7083333333333333,7.25) .. (C41);
\draw[dashed] (C23).. controls (-1.125,3.75) and (-1.125,1.5) .. (C13);
\draw[dashed] (C31).. controls (-0.7083333333333333,5.166666666666667) and (-0.7083333333333333,3.75) .. (C23);
\draw[dashed] (C13).. controls (-1.8333333333333333,1.5) and (-1.8333333333333333,5.166666666666667) .. (C31);
\draw[dashed] (C15).. controls (-0.7083333333333333,1.8333333333333335) and (-0.7083333333333333,3.25) .. (C21);

\end{tikzpicture}}

\def\inslone#1{\begin{tikzpicture}[every node/.style={shape=circle},inner sep=0,scale=#1,anchor=base,baseline=(current bounding box.center)]

\draw[dashed] (1,0.5) -- (1,6.5);
\draw[dashed] (-2.5,0.5) -- (-2.5,6.5);
\draw[dashed] (1,0.5) -- (-2.5,0.5);
\draw[dashed] (1,6.5) -- (-2.5,6.5);

\draw[] (0,0.5) -- (0,1);
\draw[] (0,1) -- (1,1);
\draw[] (0,2) -- (1,2);
\draw[] (0,2) -- (0,2.75);
\draw[] (0,2.75) -- (2/3,2.75);

\draw[blue] (-4/3,1.2) -- (1,1.2);
\draw[blue] (-4/3,1.8) -- (1,1.8);
\node[] (1) at (0,1.4) {$\vdots$};

\draw[] (-1/3,3) -- (2/3,3);
\draw[] (-1/3,4) -- (2/3,4);
\draw[] (-1/3,3) -- (-1/3,4);
\draw[] (2/3,3) -- (2/3,4);
\draw[] (2/3+1/4,3) -- (2/3+1/4,4);
\draw[] (2/3,4.25) arc (90:0:0.25cm);
\draw[] (2/3,2.75) arc (-90:0:0.25cm);

\draw[blue] (-4/3,3.2) -- (-1/3,3.2);
\draw[blue] (-4/3,3.8) -- (-1/3,3.8);

\node[] (box) at (1/6,3.5) {\scriptsize $(P_1,s_1)$};
\node[] (2) at (-0.5,3.4) {$\vdots$};

\draw[] (0,4.25) -- (2/3,4.25);
\draw[] (0,4.25) -- (0,5);
\draw[] (0,5) -- (1,5);
\draw[] (0,6) -- (1,6);
\draw[] (0,6) -- (0,6.5);

\draw[blue] (-4/3,5.2) -- (1,5.2);
\draw[blue] (-4/3,5.8) -- (1,5.8);
\node[] (3) at (0,5.4) {$\vdots$};

\draw[] (-4/3,0.5) -- (-4/3,6.5);
\node[] (L1) at (-2,3.5) { $E$};

\end{tikzpicture}}

\def\insltw#1{\begin{tikzpicture}[every node/.style={shape=circle},inner sep=0,scale=#1,anchor=base,baseline=(current bounding box.center)]

\draw[dashed] (1,0.5) -- (1,6.5);
\draw[dashed] (-3.5,0.5) -- (-3.5,6.5);
\draw[dashed] (1,0.5) -- (-3.5,0.5);
\draw[dashed] (1,6.5) -- (-3.5,6.5);

\draw[] (0,0.5) -- (0,1);
\draw[] (0,1) -- (1,1);
\draw[] (0,2) -- (1,2);
\draw[] (0,2) -- (0,2.75);
\draw[] (0,2.75) -- (2/3,2.75);

\draw[blue] (-2/3,1.2) -- (1,1.2);
\draw[blue] (-2/3,1.8) -- (1,1.8);
\node[] (1) at (0,1.4) {$\vdots$};

\draw[] (-1/3,3) -- (2/3,3);
\draw[] (-1/3,4) -- (2/3,4);
\draw[] (-1/3,3) -- (-1/3,4);
\draw[] (2/3,3) -- (2/3,4);
\draw[] (2/3+1/4,3) -- (2/3+1/4,4);
\draw[] (2/3,4.25) arc (90:0:0.25cm);
\draw[] (2/3,2.75) arc (-90:0:0.25cm);

\draw[blue] (-2/3,3.2) -- (-1/3,3.2);
\draw[blue] (-2/3,3.8) -- (-1/3,3.8);

\node[] (box) at (1/6,3.5) {\scriptsize $(P_1,s_1)$};
\node[] (2) at (-0.5,3.4) {$\vdots$};
\node[anchor=south] (f) at (-0.5,3.8) {\scriptsize $C_1$};

\draw[] (0,4.25) -- (2/3,4.25);
\draw[] (0,4.25) -- (0,5);
\draw[] (0,5) -- (1,5);
\draw[] (0,6) -- (1,6);
\draw[] (0,6) -- (0,6.5);

\draw[blue] (-2/3,5.2) -- (1,5.2);
\draw[blue] (-2/3,5.8) -- (1,5.8);
\node[] (3) at (0,5.4) {$\vdots$};

\draw[] (-2/3,3) -- (-2/3,4);
\draw[] (-5/3,3) -- (-5/3,4);
\draw[] (-2/3,3) -- (-5/3,3);
\draw[] (-2/3,4) -- (-5/3,4);
\node[] (L1) at (-7/6,3.5) { $E_1$};

\draw[] (-2/3,2) -- (-5/3,2);
\draw[] (-5/3,2) -- (-5/3,0.5);
\draw[] (-2/3,2) -- (-2/3,0.5);
\node[] (L2) at (-7/6,1.5) { $E'$};

\draw[] (-2/3,5) -- (-5/3,5);
\draw[] (-5/3,5) -- (-5/3,6.5);
\draw[] (-2/3,5) -- (-2/3,6.5);
\node[] (L3) at (-7/6,5.5) { $E'$};

\draw[] (-8/3,0.5) -- (-8/3,6.5);
\node[] (L4) at (-37/12,3.5) { $E'$};

\draw[blue] (-5/3+0.2,3) arc (0:-90:0.2 cm);
\draw[blue] (-5/3+0.8,3) arc (0:-90:0.8 cm);

\draw[blue] (-5/3,3-0.2) arc (90:180:0.8 cm);
\draw[blue] (-5/3,3-0.8) arc (90:180:0.2 cm);

\draw[blue] (-8/3+0.2,2) -- (-8/3+0.2,0.5);
\draw[blue] (-8/3+0.8,2) -- (-8/3+0.8,0.5);

\node[fill=black,inner sep=1pt] (u1) at (-8/3+0.2,0.5) {};
\node[] () at ([yshift=-0.2cm]u1) {\scriptsize $u_1$};
\node[fill=black,inner sep=1pt] (un) at (-8/3+0.8,0.5) {};
\node[] () at ([yshift=-0.2cm]un) {\scriptsize $u_n$};
\node[] () at (-5/3+0.5,2.9) {$\dots$};

\draw[blue] (-5/3+0.2,4) arc (0:90:0.2 cm);
\draw[blue] (-5/3+0.8,4) arc (0:90:0.8 cm);

\draw[blue] (-5/3,4+0.2) arc (-90:-180:0.8 cm);
\draw[blue] (-5/3,4+0.8) arc (-90:-180:0.2 cm);

\draw[blue] (-8/3+0.2,5) -- (-8/3+0.2,6.5);
\draw[blue] (-8/3+0.8,5) -- (-8/3+0.8,6.5);

\node[fill=black,inner sep=1pt] (w1) at (-8/3+0.2,6.5) {};
\node[] () at ([yshift=0.2cm]w1) {\scriptsize $w_1$};
\node[fill=black,inner sep=1pt] (wm) at (-8/3+0.8,6.5) {};
\node[] () at ([yshift=0.2cm]wm) {\scriptsize $w_m$};
\node[] () at (-5/3+0.5,4.1) {$\dots$};

\end{tikzpicture}}

\def\vjuxtega#1{\begin{tikzpicture}[every node/.style={shape=circle},inner sep=0,scale=#1,anchor=base,baseline=(current bounding box.center)]

\pgfdeclarelayer{1}

\pgfsetlayers{main,
1} 

\draw[] (-5,0)--(0,0);
\draw[] (-5,0)--(-5,5);
\draw[] (-5,5)--(0,5);
\draw[] (0,0) --(0,1);
\draw[] (0,2) -- (0,3); 
\node[fill=black,inner sep =0] (A1) at (0,1) {};
\node[fill=black,inner sep =0] (B1) at (0,2) {};
\draw[] (0,4) -- (0,5); 
\node[fill=black,inner sep =0] (A2) at (0,3) {};
\node[fill=black,inner sep =0] (B2) at (0,4) {};

\node[fill=black] (S1) at (-3.75,0) {};
\node[fill=black] (S2) at (-2.5,0) {};
\node[fill=black] (S3) at (-1.25,0) {};

\node[fill=black] (N1) at (-4.5,5) {};
\node[fill=black] (N2) at (-4.0,5) {};
\node[fill=black] (N3) at (-3.5,5) {};
\node[fill=black] (N4) at (-3.0,5) {};
\node[fill=black] (N5) at (-2.5,5) {};
\node[fill=black] (N6) at (-2.0,5) {};
\node[fill=black] (N7) at (-1.5000000000000002,5) {};
\node[fill=black] (N8) at (-0.9999999999999998,5) {};
\node[fill=black] (N9) at (-0.4999999999999999,5) {};

\pgfonlayer{1}

\draw[white,line width=#1*5mm] (0,1.5)   arc (-90:90:1cm);
\draw[white,line width=#1*5mm] (0,1.75)   arc (-90:90:1cm);
\draw[white,line width=#1*5mm] (0,1.25)   arc (-90:90:1cm);
\draw[] (0,1) arc (-90:90:1cm);
\draw[] (0,2) arc (-90:90:1cm);
\draw[] (1,2)--(1,3);

\node[fill=black] (A11) at (0,1.5) {};
\node[fill=black] (A21) at (0,3.5) {};

\endpgfonlayer

\draw[green] (S1).. controls (-3.75,0.625) and (-2.5,0.625) .. (S2);

\draw[blue] (S3).. controls (-1.25,2.5) and (-2.5,2.5).. (N5);

\draw[blue] (N9).. controls (-0.4999999999999999,4.0) and (-0.33333333333333326,3.5) .. (A21);
\draw[blue] (N6).. controls (-2.0,2.6666666666666665) and (-1.3333333333333333,1.5) .. (A11);

\pgfonlayer{1}
\draw[blue] (A21) arc (90:-90:1.0);
\endpgfonlayer

\draw[orange] (N1).. controls (-4.5,4.75) and (-4.0,4.75) .. (N2);

\draw[red] (N3).. controls (-3.5,4.75) and (-3.0,4.75) .. (N4);

\draw[red] (N7).. controls (-1.5000000000000002,4.75) and (-0.9999999999999998,4.75) .. (N8);

\end{tikzpicture}}

\def\vjuxtegb#1{\begin{tikzpicture}[every node/.style={shape=circle},inner sep=0,scale=#1,anchor=base,baseline=(current bounding box.center)]

\pgfdeclarelayer{1}
\pgfdeclarelayer{2}

\pgfsetlayers{main,
1,
2} 

\draw[] (-9,0)--(0,0);
\draw[] (-9,0)--(-9,9);
\draw[] (-9,9)--(0,9);
\draw[] (0,0) --(0,1);
\draw[] (0,2) -- (0,3); 
\node[fill=black,inner sep =0] (A1) at (0,1) {};
\node[fill=black,inner sep =0] (B1) at (0,2) {};
\draw[] (0,4) -- (0,5); 
\node[fill=black,inner sep =0] (A2) at (0,3) {};
\node[fill=black,inner sep =0] (B2) at (0,4) {};
\draw[] (0,6) -- (0,7); 
\node[fill=black,inner sep =0] (A3) at (0,5) {};
\node[fill=black,inner sep =0] (B3) at (0,6) {};
\draw[] (0,8) -- (0,9); 
\node[fill=black,inner sep =0] (A4) at (0,7) {};
\node[fill=black,inner sep =0] (B4) at (0,8) {};

\node[fill=black] (S1) at (-8.1,0) {};
\node[fill=black] (S2) at (-7.2,0) {};
\node[fill=black] (S3) at (-6.3,0) {};
\node[fill=black] (S4) at (-5.3999999999999995,0) {};
\node[fill=black] (S5) at (-4.5,0) {};
\node[fill=black] (S6) at (-3.6,0) {};
\node[fill=black] (S7) at (-2.7,0) {};
\node[fill=black] (S8) at (-1.7999999999999996,0) {};
\node[fill=black] (S9) at (-0.8999999999999998,0) {};

\node[fill=black] (N1) at (-7.5,9) {};
\node[fill=black] (N2) at (-6.000000000000001,9) {};
\node[fill=black] (N3) at (-4.5,9) {};
\node[fill=black] (N4) at (-3.0000000000000004,9) {};
\node[fill=black] (N5) at (-1.4999999999999996,9) {};

\pgfonlayer{2}

\draw[white,line width=#1*10mm] (0,1.5)   arc (-90:90:2cm);
\draw[] (0,1) arc (-90:90:2.5cm);
\draw[] (0,2) arc (-90:90:1.5cm);

\node[fill=black] (A11) at (0,1.3333333333333333) {};
\node[fill=black] (A31) at (0,5.333333333333333) {};
\node[fill=black] (A12) at (0,1.6666666666666665) {};
\node[fill=black] (A32) at (0,5.666666666666667) {};

\endpgfonlayer

\pgfonlayer{1}

\draw[white,line width=#1*10mm] (0,3.5)   arc (-90:90:2cm);
\draw[] (0,3) arc (-90:90:2.5cm);
\draw[] (0,4) arc (-90:90:1.5cm);

\node[fill=black] (A21) at (0,3.5) {};
\node[fill=black] (A41) at (0,7.5) {};

\endpgfonlayer

\draw[purple] (N2).. controls (-6.000000000000001,5.333333333333334) and (-4.000000000000001,3.5) .. (A21);
\draw[purple] (N5).. controls (-1.4999999999999996,8.0) and (-0.9999999999999997,7.5) .. (A41);

\pgfonlayer{1}
\draw[purple] (A21) arc (-90:90:2.0);
\endpgfonlayer

\draw[orange] (N1).. controls (-7.5,4.5) and (-8.1,4.5).. (S1);

\draw[orange] (S2).. controls (-7.2,1.111111111111111) and (-4.8,1.6666666666666665) .. (A12);
\draw[orange] (N3).. controls (-4.5,6.555555555555555) and (-3.0,5.333333333333333) .. (A31);

\pgfonlayer{2}
\draw[orange] (A12) arc (-90:90:1.8333333333333333);
\endpgfonlayer

\draw[blue] (N4).. controls (-3.0000000000000004,6.777777777777779) and (-2.0000000000000004,5.666666666666667) .. (A32);
\draw[blue] (S9).. controls (-0.8999999999999998,0.8888888888888888) and (-0.5999999999999999,1.3333333333333333) .. (A11);

\pgfonlayer{2}
\draw[blue] (A32) arc (90:-90:2.1666666666666665);
\endpgfonlayer

\draw[red] (S3).. controls (-6.3,1.5) and (-1.7999999999999996,1.5) .. (S8);

\draw[red] (S4).. controls (-5.3999999999999995,1) and (-2.7,1) .. (S7);

\draw[blue] (S5).. controls (-4.5,0.45) and (-3.6,0.45) .. (S6); 

\end{tikzpicture}}

\def\vjuxteg#1{\begin{tikzpicture}[every node/.style={shape=circle},inner sep=0,scale=#1,anchor=base,baseline=(current bounding box.center)]

\pgfdeclarelayer{1}
\pgfdeclarelayer{2}

\pgfsetlayers{main,
1,
2} 

\draw[] (-5,0)--(0,0);
\draw[] (-5,0)--(-5,5);
\draw[] (-5,5)--(0,5);
\draw[] (0,0) --(0,1);
\draw[] (0,2) -- (0,3); 
\node[fill=black,inner sep =0] (A1) at (0,1) {};
\node[fill=black,inner sep =0] (B1) at (0,2) {};
\draw[] (0,4) -- (0,5); 
\node[fill=black,inner sep =0] (A2) at (0,3) {};
\node[fill=black,inner sep =0] (B2) at (0,4) {};

\node[fill=black] (S1) at (-3.75,0) {};
\node[fill=black] (S2) at (-2.5,0) {};
\node[fill=black] (S3) at (-1.25,0) {};

\node[fill=black] (N1) at (-4.5,5) {};
\node[fill=black] (N2) at (-4.0,5) {};
\node[fill=black] (N3) at (-3.5,5) {};
\node[fill=black] (N4) at (-3.0,5) {};
\node[fill=black] (N5) at (-2.5,5) {};
\node[fill=black] (N6) at (-2.0,5) {};
\node[fill=black] (N7) at (-1.5000000000000002,5) {};
\node[fill=black] (N8) at (-0.9999999999999998,5) {};
\node[fill=black] (N9) at (-0.4999999999999999,5) {};

\pgfonlayer{1}

\draw[white,line width=#1*5mm] (0,1.5)   arc (-90:90:1cm);
\draw[white,line width=#1*5mm] (0,1.75)   arc (-90:90:1cm);
\draw[white,line width=#1*5mm] (0,1.25)   arc (-90:90:1cm);
\draw[] (0,1) arc (-90:90:1cm);
\draw[] (0,2) arc (-90:90:1cm);
\draw[] (1,2)--(1,3);

\node[fill=black] (A11) at (0,1.5) {};
\node[fill=black] (A21) at (0,3.5) {};

\endpgfonlayer

\draw[green] (S1).. controls (-3.75,0.625) and (-2.5,0.625) .. (S2);

\draw[blue] (S3).. controls (-1.25,2.5) and (-2.5,2.5).. (N5);

\draw[blue] (N9).. controls (-0.4999999999999999,4.0) and (-0.33333333333333326,3.5) .. (A21);
\draw[blue] (N6).. controls (-2.0,2.6666666666666665) and (-1.3333333333333333,1.5) .. (A11);

\pgfonlayer{1}
\draw[blue] (A21) arc (90:-90:1.0);
\endpgfonlayer

\draw[orange] (N1).. controls (-4.5,4.75) and (-4.0,4.75) .. (N2);

\draw[red,dashed] (N3).. controls (-3.5,4.75) and (-3.0,4.75) .. (N4);

\draw[red,dashed] (N7).. controls (-1.5000000000000002,4.75) and (-0.9999999999999998,4.75) .. (N8);

\begin{scope}[yshift=5cm,scale=5/9] 

\draw[] (-9,0)--(0,0);
\draw[] (-9,0)--(-9,9);
\draw[] (-9,9)--(0,9);
\draw[] (0,0) --(0,1);
\draw[] (0,2) -- (0,3); 
\node[fill=black,inner sep =0] (A1) at (0,1) {};
\node[fill=black,inner sep =0] (B1) at (0,2) {};
\draw[] (0,4) -- (0,5); 
\node[fill=black,inner sep =0] (A2) at (0,3) {};
\node[fill=black,inner sep =0] (B2) at (0,4) {};
\draw[] (0,6) -- (0,7); 
\node[fill=black,inner sep =0] (A3) at (0,5) {};
\node[fill=black,inner sep =0] (B3) at (0,6) {};
\draw[] (0,8) -- (0,9); 
\node[fill=black,inner sep =0] (A4) at (0,7) {};
\node[fill=black,inner sep =0] (B4) at (0,8) {};

\node[fill=blue,inner sep=1] (S1) at (-8.1,0) {};
\node[fill=blue,inner sep=1] (S2) at (-7.2,0) {};
\node[fill=blue,inner sep=1] (S3) at (-6.3,0) {};
\node[fill=blue,inner sep=1] (S4) at (-5.3999999999999995,0) {};
\node[fill=blue,inner sep=1] (S5) at (-4.5,0) {};
\node[fill=blue,inner sep=1] (S6) at (-3.6,0) {};
\node[fill=blue,inner sep=1] (S7) at (-2.7,0) {};
\node[fill=blue,inner sep=1] (S8) at (-1.7999999999999996,0) {};
\node[fill=blue,inner sep=1] (S9) at (-0.8999999999999998,0) {};

\node[fill=black] (N1) at (-7.5,9) {};
\node[fill=black] (N2) at (-6.000000000000001,9) {};
\node[fill=black] (N3) at (-4.5,9) {};
\node[fill=black] (N4) at (-3.0000000000000004,9) {};
\node[fill=black] (N5) at (-1.4999999999999996,9) {};

\pgfonlayer{2}

\draw[white,line width=#1*5/9*10mm] (0,1.5)   arc (-90:90:2cm);
\draw[] (0,1) arc (-90:90:2.5cm);
\draw[] (0,2) arc (-90:90:1.5cm);

\node[fill=black] (A11) at (0,1.3333333333333333) {};
\node[fill=black] (A31) at (0,5.333333333333333) {};
\node[fill=black] (A12) at (0,1.6666666666666665) {};
\node[fill=black] (A32) at (0,5.666666666666667) {};

\endpgfonlayer

\pgfonlayer{1}

\draw[white,line width=#1*5/9*10mm] (0,3.5)   arc (-90:90:2cm);
\draw[] (0,3) arc (-90:90:2.5cm);
\draw[] (0,4) arc (-90:90:1.5cm);

\node[fill=black] (A21) at (0,3.5) {};
\node[fill=black] (A41) at (0,7.5) {};

\endpgfonlayer

\draw[purple] (N2).. controls (-6.000000000000001,5.333333333333334) and (-4.000000000000001,3.5) .. (A21);
\draw[purple] (N5).. controls (-1.4999999999999996,8.0) and (-0.9999999999999997,7.5) .. (A41);

\pgfonlayer{1}
\draw[purple] (A21) arc (-90:90:2.0);
\endpgfonlayer

\draw[orange] (N1).. controls (-7.5,4.5) and (-8.1,4.5).. (S1);

\draw[orange] (S2).. controls (-7.2,1.111111111111111) and (-4.8,1.6666666666666665) .. (A12);
\draw[orange] (N3).. controls (-4.5,6.555555555555555) and (-3.0,5.333333333333333) .. (A31);

\pgfonlayer{2}
\draw[orange] (A12) arc (-90:90:1.8333333333333333);
\endpgfonlayer

\draw[blue] (N4).. controls (-3.0000000000000004,6.777777777777779) and (-2.0000000000000004,5.666666666666667) .. (A32);
\draw[blue] (S9).. controls (-0.8999999999999998,0.8888888888888888) and (-0.5999999999999999,1.3333333333333333) .. (A11);

\pgfonlayer{2}
\draw[blue] (A32) arc (90:-90:2.1666666666666665);
\endpgfonlayer

\draw[red,dashed] (S3).. controls (-6.3,1.5) and (-1.7999999999999996,1.5) .. (S8);

\draw[red,dashed] (S4).. controls (-5.3999999999999995,1) and (-2.7,1) .. (S7);

\draw[blue] (S5).. controls (-4.5,0.45) and (-3.6,0.45) .. (S6); 

\end{scope}

\end{tikzpicture}}

\def\vjuxtegg#1{\begin{tikzpicture}[every node/.style={shape=circle},inner sep=0,scale=#1,anchor=base,baseline=(current bounding box.center)]

\pgfdeclarelayer{1}
\pgfdeclarelayer{2}
\pgfdeclarelayer{3}

\pgfsetlayers{main,
1,
2,
3} 

\draw[] (-13,0)--(0,0);
\draw[] (-13,0)--(-13,13);
\draw[] (-13,13)--(0,13);
\draw[] (0,0) --(0,1);
\draw[] (0,2) -- (0,3); 
\node[fill=black,inner sep =0] (A1) at (0,1) {};
\node[fill=black,inner sep =0] (B1) at (0,2) {};
\draw[] (0,4) -- (0,5); 
\node[fill=black,inner sep =0] (A2) at (0,3) {};
\node[fill=black,inner sep =0] (B2) at (0,4) {};
\draw[] (0,6) -- (0,7); 
\node[fill=black,inner sep =0] (A3) at (0,5) {};
\node[fill=black,inner sep =0] (B3) at (0,6) {};
\draw[] (0,8) -- (0,9); 
\node[fill=black,inner sep =0] (A4) at (0,7) {};
\node[fill=black,inner sep =0] (B4) at (0,8) {};
\draw[] (0,10) -- (0,11); 
\node[fill=black,inner sep =0] (A5) at (0,9) {};
\node[fill=black,inner sep =0] (B5) at (0,10) {};
\draw[] (0,12) -- (0,13); 
\node[fill=black,inner sep =0] (A6) at (0,11) {};
\node[fill=black,inner sep =0] (B6) at (0,12) {};

\node[fill=black] (S1) at (-9.75,0) {};
\node[fill=black] (S2) at (-6.5,0) {};
\node[fill=black] (S3) at (-3.25,0) {};

\node[fill=black] (N1) at (-10.833333333333334,13) {};
\node[fill=black] (N2) at (-8.666666666666668,13) {};
\node[fill=black] (N3) at (-6.5,13) {};
\node[fill=black] (N4) at (-4.333333333333334,13) {};
\node[fill=black] (N5) at (-2.166666666666666,13) {};

\pgfonlayer{3}

\draw[white,line width=#1*5mm] (0,1.5)   arc (-90:90:1cm);
\draw[white,line width=#1*5mm] (0,1.75)   arc (-90:90:1cm);
\draw[white,line width=#1*5mm] (0,1.25)   arc (-90:90:1cm);
\draw[] (0,1) arc (-90:90:1cm);
\draw[] (0,2) arc (-90:90:1cm);
\draw[] (1,2)--(1,3);

\node[fill=black] (A11) at (0,1.5) {};
\node[fill=black] (A21) at (0,3.5) {};

\endpgfonlayer

\pgfonlayer{2}

\draw[white,line width=#1*10mm] (0,5.5)   arc (-90:90:2cm);
\draw[] (0,5) arc (-90:90:2.5cm);
\draw[] (0,6) arc (-90:90:1.5cm);

\node[fill=black] (A31) at (0,5.333333333333333) {};
\node[fill=black] (A51) at (0,9.333333333333334) {};
\node[fill=black] (A32) at (0,5.666666666666667) {};
\node[fill=black] (A52) at (0,9.666666666666666) {};

\endpgfonlayer

\pgfonlayer{1}

\draw[white,line width=#1*10mm] (0,7.5)   arc (-90:90:2cm);
\draw[] (0,7) arc (-90:90:2.5cm);
\draw[] (0,8) arc (-90:90:1.5cm);

\node[fill=black] (A41) at (0,7.5) {};
\node[fill=black] (A61) at (0,11.5) {};

\endpgfonlayer

\draw[purple] (N2).. controls (-8.666666666666668,9.333333333333334) and (-5.777777777777779,7.5) .. (A41);
\draw[purple] (N5).. controls (-2.166666666666666,12.0) and (-1.444444444444444,11.5) .. (A61);

\pgfonlayer{1}
\draw[purple] (A41) arc (-90:90:2.0);
\endpgfonlayer

\draw[orange] (N1).. controls (-10.833333333333334,8.11111111111111) and (-7.222222222222222,5.666666666666667) .. (A32);
\draw[orange] (N3).. controls (-6.5,10.555555555555555) and (-4.333333333333333,9.333333333333334) .. (A51);

\pgfonlayer{2}
\draw[orange] (A32) arc (-90:90:1.8333333333333333);
\endpgfonlayer

\draw[blue] (N4).. controls (-4.333333333333334,10.777777777777779) and (-2.8888888888888893,9.666666666666666) .. (A52);
\draw[blue] (A31).. controls (-0.9166666666666667,5.333333333333333) and (-0.9166666666666667,3.5) .. (A21);
\draw[blue] (S3).. controls (-3.25,1.0) and (-2.1666666666666665,1.5) .. (A11);

\pgfonlayer{2}
\draw[blue] (A52) arc (90:-90:2.1666666666666665);
\endpgfonlayer

\pgfonlayer{3}
\draw[blue] (A21) arc (90:-90:1.0);
\endpgfonlayer

\draw[green] (S1).. controls (-9.75,1.625) and (-6.5,1.625) .. (S2);

\end{tikzpicture}}

\def\vjuxtcompa#1{\begin{tikzpicture}[every node/.style={shape=circle},inner sep=0,scale=#1,anchor=base,baseline=(current bounding box.center)]
\draw[] (0,0) -- (3,0);
\draw[dashed] (0,1) -- (3,1);
\draw[dashed] (0,-1) -- (3,-1);
\draw[] (0,-1) -- (0,1);
\draw[] (3,-1) -- (3,1);

\node[fill=blue, inner sep=1pt] (G1) at (3/10,0) {};
\node[fill=blue, inner sep=1pt] (G2) at (6/10,0) {};
\node[fill=blue, inner sep=1pt] (G3) at (9/10,0) {};
\node[fill=blue, inner sep=1pt] (G4) at (12/10,0) {};
\node[fill=blue, inner sep=1pt] (G5) at (15/10,0) {};
\node[fill=blue, inner sep=1pt] (G6) at (18/10,0) {};
\node[fill=blue, inner sep=1pt] (G7) at (21/10,0) {};
\node[fill=blue, inner sep=1pt] (G8) at (24/10,0) {};
\node[fill=blue, inner sep=1pt] (G9) at (27/10,0) {};

\node (G05) at (15/10,-1) {};
\node (G06) at (18/10,-1) {};
\node (G09) at (27/10,-1) {};

\node (G11) at (3/10,1) {};
\node (G12) at (6/10,1) {};
\node (G19) at (27/10,1) {};

\draw[very thick,orange] (G11) -- (G1);
\draw[very thick,orange] (G12) -- (G2);

\draw[very thick,orange] (G1).. controls ([yshift=-0.25cm]G1) and ([yshift=-0.25cm]G2)  ..(G2);

\draw[very thick,red] (G3).. controls ([yshift=-0.25cm]G3) and ([yshift=-0.25cm]G4)  ..(G4);
\draw[very thick,red] (G7).. controls ([yshift=-0.25cm]G7) and ([yshift=-0.25cm]G8)  ..(G8);
\draw[very thick,red] (G3).. controls ([yshift=0.75cm]G3) and ([yshift=0.75cm]G8)  ..(G8);
\draw[very thick,red] (G4).. controls ([yshift=0.5cm]G4) and ([yshift=0.5cm]G7)  ..(G7);

\draw[very thick,blue] (G5).. controls ([yshift=0.25cm]G5) and ([yshift=0.25cm]G6)  ..(G6);

\draw[very thick,blue] (G19) -- (G9);
\draw[very thick,blue] (G09) -- (G9);
\draw[very thick,blue] (G05) -- (G5);
\draw[very thick,blue] (G06) -- (G6);

\node[fill=black, inner sep=1pt] (C1) at (1.5/10,0) {};
\node[fill=black, inner sep=1pt] (C2) at (4.5/10,0) {};
\node[fill=black, inner sep=1pt] (C3) at (7.5/10,0) {};
\node[fill=black, inner sep=1pt] (C4) at (10.5/10,0) {};
\node[fill=black, inner sep=1pt] (C5) at (13.5/10,0) {};
\node[fill=black, inner sep=1pt] (C6) at (16.5/10,0) {};
\node[fill=black, inner sep=1pt] (C7) at (19.5/10,0) {};
\node[fill=black, inner sep=1pt] (C8) at (22.5/10,0) {};
\node[fill=black, inner sep=1pt] (C9) at (25.5/10,0) {};
\node[fill=black, inner sep=1pt] (C10) at (28.5/10,0) {};

\node[fill=black, inner sep=1pt] (C01) at (7.5/10,-1) {};
\node[fill=black, inner sep=1pt] (C02) at (16.5/10,-1) {};
\node[fill=black, inner sep=1pt] (C03) at (22.5/10,-1) {};
\node[fill=black, inner sep=1pt] (C04) at (28.5/10,-1) {};

\node[fill=black, inner sep=1pt] (C11) at (1.5/10,1) {};
\node[fill=black, inner sep=1pt] (C12) at (4.5/10,1) {};
\node[fill=black, inner sep=1pt] (C13) at (16.5/10,1) {};
\node[fill=black, inner sep=1pt] (C14) at (28.5/10,1) {};

\draw[dashed] (C11) -- (C1);
\draw[dashed] (C12) -- (C2);
\draw[dashed] (C02) -- (C6);
\draw[dashed] (C14) -- (C10);
\draw[dashed] (C04) -- (C10);

\draw[dashed] (C3).. controls ([yshift=0.9cm]C3) and ([yshift=0.9cm]C9) ..(C9);
\draw[dashed] (C4).. controls ([yshift=0.6cm]C4) and ([yshift=0.6cm]C8) ..(C8);
\draw[dashed] (C5).. controls ([yshift=0.35cm]C5) and ([yshift=0.35cm]C7) ..(C7);
\draw[dashed] (C13).. controls ([yshift=0.8cm]C3) ..(C3);
\draw[dashed] (C13).. controls ([yshift=0.8cm]C9) ..(C9);

\draw[dashed] (C1).. controls ([yshift=-0.35cm]C1) and ([yshift=-0.35cm]C3) ..(C3);
\draw[dashed] (C3).. controls ([yshift=-0.35cm]C3) and ([yshift=-0.35cm]C5) ..(C5);
\draw[dashed] (C1).. controls ([yshift=-0.6cm]C1) and ([yshift=-0.6cm]C5) ..(C5);

\draw[dashed] (C01) -- (C3);
\draw[dashed] (C01).. controls ([yshift=-0.5cm]C1)..(C1);
\draw[dashed] (C01).. controls ([yshift=-0.5cm]C5)..(C5);

\draw[dashed] (C03).. controls ([yshift=-0.5cm]C7)..(C7);
\draw[dashed] (C03).. controls ([yshift=-0.5cm]C9)..(C9);
\draw[dashed] (C7).. controls ([yshift=-0.35cm]C7) and ([yshift=-0.35cm]C9) ..(C9);

\end{tikzpicture}}

\def\vjuxtcompb#1{\begin{tikzpicture}[every node/.style={shape=circle},inner sep=0,scale=#1,anchor=base,baseline=(current bounding box.center)]

\draw[dashed] (0,1) -- (3,1);
\draw[dashed] (0,-1) -- (3,-1);
\draw[] (0,-1) -- (0,1);
\draw[] (3,-1) -- (3,1);

\node (G05) at (15/10,-1) {};
\node (G06) at (18/10,-1) {};
\node (G09) at (27/10,-1) {};

\node (G11) at (3/10,1) {};
\node (G12) at (6/10,1) {};
\node (G19) at (27/10,1) {};

\draw[very thick,orange] (G11).. controls ([yshift=-0.25cm]G11) and ([yshift=-0.25cm]G12)  ..(G12);

\draw[very thick,blue] (G05).. controls ([yshift=0.25cm]G05) and ([yshift=0.25cm]G06)  ..(G06);

\draw[very thick,blue] (G09) -- (G19);

\node[fill=black, inner sep=1pt] (C01) at (7.5/10,-1) {};
\node[fill=black, inner sep=1pt] (C02) at (16.5/10,-1) {};
\node[fill=black, inner sep=1pt] (C03) at (22.5/10,-1) {};
\node[fill=black, inner sep=1pt] (C04) at (28.5/10,-1) {};

\node[fill=black, inner sep=1pt] (C11) at (1.5/10,1) {};
\node[fill=black, inner sep=1pt] (C12) at (4.5/10,1) {};
\node[fill=black, inner sep=1pt] (C13) at (16.5/10,1) {};
\node[fill=black, inner sep=1pt] (C14) at (28.5/10,1) {};

\draw[dashed] (C01) -- (C11);
\draw[dashed] (C01) -- (C13);
\draw[dashed] (C03) -- (C13);
\draw[dashed] (C03) -- (C11);

\draw[dashed] (C01).. controls ([yshift=0.6cm]C01) and ([yshift=0.6cm]C03)  ..(C03);
\draw[dashed] (C11).. controls ([yshift=-0.6cm]C11) and ([yshift=-0.6cm]C13)  ..(C13);

\draw[dashed] (C04) -- (C14);

\end{tikzpicture}}

\def\SWBiso#1{\begin{tikzpicture}[every node/.style={shape=circle},inner sep=0,scale=#1,anchor=base,baseline=(current bounding box.center)]
\pgfdeclarelayer{1}
\pgfdeclarelayer{2}
\pgfdeclarelayer{3}

\pgfsetlayers{main,
1,
2,
3} 

\draw[] (-13,0)--(0,0);
\draw[] (-13,0)--(-13,13);
\draw[] (-13,13)--(0,13);
\draw[] (0,0) --(0,1);
\draw[] (0,2) -- (0,3); 
\node[fill=black,inner sep =0] (C1) at (0,1) {};
\node[fill=black,inner sep =0] (D1) at (0,2) {};
\draw[] (0,4) -- (0,5); 
\node[fill=black,inner sep =0] (C2) at (0,3) {};
\node[fill=black,inner sep =0] (D2) at (0,4) {};
\draw[] (0,6) -- (0,7); 
\node[fill=black,inner sep =0] (C3) at (0,5) {};
\node[fill=black,inner sep =0] (D3) at (0,6) {};
\draw[] (0,8) -- (0,9); 
\node[fill=black,inner sep =0] (C4) at (0,7) {};
\node[fill=black,inner sep =0] (D4) at (0,8) {};
\draw[] (0,10) -- (0,11); 
\node[fill=black,inner sep =0] (C5) at (0,9) {};
\node[fill=black,inner sep =0] (D5) at (0,10) {};
\draw[] (0,12) -- (0,13); 
\node[fill=black,inner sep =0] (C6) at (0,11) {};
\node[fill=black,inner sep =0] (D6) at (0,12) {};

\node[fill=black] (So1) at (-8.666666666666668,0) {};
\node[fill=black] (So2) at (-4.333333333333334,0) {};

\node[fill=black] (Nor1) at (-10.4,13) {};
\node[fill=black] (Nor2) at (-7.8,13) {};
\node[fill=black] (Nor3) at (-5.2,13) {};
\node[fill=black] (Nor4) at (-2.5999999999999996,13) {};

\pgfonlayer{3}

\draw[white,line width=#1*5mm] (0,3.5)   arc (-90:90:4cm);
\draw[white,line width=#1*5mm] (0,3.75)   arc (-90:90:4cm);
\draw[white,line width=#1*5mm] (0,3.25)   arc (-90:90:4cm);
\draw[] (0,3) arc (-90:90:4cm);
\draw[] (0,4) arc (-90:90:4cm);
\draw[] (4,7)--(4,8);

\node[fill=black] (C21) at (0,3.3333333333333335) {};
\node[fill=black] (C61) at (0,11.333333333333334) {};
\node[fill=black] (C22) at (0,3.6666666666666665) {};
\node[fill=black] (C62) at (0,11.666666666666666) {};

\endpgfonlayer

\pgfonlayer{2}

\draw[white,line width=#1*10mm] (0,5.5)   arc (-90:90:2cm);
\draw[] (0,5) arc (-90:90:2.5cm);
\draw[] (0,6) arc (-90:90:1.5cm);

\node[fill=black] (C31) at (0,5.5) {};
\node[fill=black] (C51) at (0,9.5) {};

\endpgfonlayer

\pgfonlayer{1}

\draw[white,line width=#1*10mm] (0,1.5)   arc (-90:90:3cm);
\draw[] (0,1) arc (-90:90:3.5cm);
\draw[] (0,2) arc (-90:90:2.5cm);

\node[fill=black] (C11) at (0,1.5) {};
\node[fill=black] (C41) at (0,7.5) {};

\endpgfonlayer

\draw[blue] (So2).. controls (-4.333333333333334,2.444444444444444) and (-2.8888888888888893,3.6666666666666665) .. (C22);
\draw[blue] (Nor4).. controls (-2.5999999999999996,12.11111111111111) and (-1.7333333333333332,11.666666666666666) .. (C62);

\pgfonlayer{3}
\draw[blue] (C22) arc (-90:90:4.0);
\endpgfonlayer

\draw[red] (So1).. controls (-8.666666666666668,6.333333333333333) and (-5.777777777777779,9.5) .. (C51);
\draw[red] (C31).. controls (-1.0,5.5) and (-1.0,7.5) .. (C41);
\draw[red] (C11).. controls (-0.9166666666666666,1.5) and (-0.9166666666666666,3.3333333333333335) .. (C21);
\draw[red] (Nor3).. controls (-5.2,11.88888888888889) and (-3.466666666666667,11.333333333333334) .. (C61);

\pgfonlayer{2}
\draw[red] (C51) arc (90:-90:2.0);
\endpgfonlayer

\pgfonlayer{1}
\draw[red] (C41) arc (90:-90:3.0);
\endpgfonlayer

\pgfonlayer{3}
\draw[red] (C21) arc (-90:90:4.0);
\endpgfonlayer

\draw[green] (Nor1).. controls (-10.4,11.7) and (-7.8,11.7) .. (Nor2); 

\end{tikzpicture}}

\def\isoa#1{\begin{tikzpicture}[every node/.style={shape=circle},inner sep =0.7, scale=#1,anchor=base,baseline=(current bounding box.center)]

\draw[dashed] (-2,0) -- (3,0);
\draw[dashed] (-2,5) -- (3,5); 
\draw[dashed] (-2,0) -- (-2,5); 
\draw[dashed] (3,0) -- (3,5);

\pgfdeclarelayer{1}
\pgfdeclarelayer{2}
\pgfsetlayers{main,1,2}
\draw (0,0) -- (0,0.5);
 \draw (0,1.5) -- (0,2); 
 \draw (0,3) -- (0,3.5); 
 \draw (0,4.5) -- (0,5);
 
\draw (-1,0) -- (-1,0.8);
\draw (-1,0.8) -- (-2,0.8);
\draw (-1,1.2) -- (-1,5);
\draw (-1,1.2) -- (-2,1.2);

\node[] at (0,2.3) {$\vdots $};
\node at (-1.5,2.5) {$E'$};
\node at (-1.5,0.25) {$E'$};
\pgfonlayer{1}

\draw[white,line width=#1*10mm] (0,1)   arc (-90:90:1.5cm);
\draw[] (0,0.5) arc (-90:90:2cm);
\draw[] (0,1.5) arc (-90:90:1cm);

\pgfonlayer{2}
\node[] (A11) at (0,1.2-0.5) {};
\node[] (A31) at (0,5.2-1.5) {};
\node[fill=green] (A12) at (0,1.4-0.5) {};
\node[fill=green] (A32) at (0,5.4-1.5) {};
\node[fill=green] (A13) at (0,1.6-0.5) {};
\node[fill=green] (A33) at (0,5.6-1.5) {};
\node[] (A14) at (0,1.8-0.5) {};
\node[] (A34) at (0,5.8-1.5) {};
\endpgfonlayer
\endpgfonlayer

\pgfonlayer{1}
\draw[blue,double distance=1pt] (A11) arc (-90:90:1.8);
\endpgfonlayer
\draw[blue,double distance=1pt] (A11) -- +(-1,0);
\draw[blue,double distance=1pt] (A34) -- +(-1,0);

\pgfonlayer{1}
\draw[blue,double distance=1pt] (A14) arc (-90:90:1.2);
\endpgfonlayer
\draw[blue,double distance=1pt] (A14) -- +(-1,0);
\draw[blue,double distance=1pt] (A31) -- +(-1,0);

\draw[red] (A12) -- +(-2,0);
\draw[red] (A13) -- +(-2,0);
\draw[red] (A33).. controls ([xshift=-0.2cm]A33) and ([xshift=-0.2cm]A32) .. (A32);

\node[fill=red] (A) at ([xshift=-2cm]A12) {};
\node[fill=red] (B) at ([xshift=-2cm]A13) {};
\node[anchor=north east] () at (A) {\scriptsize $v'$};
\node[anchor=south east] () at (B) {\scriptsize $u'$};

\pgfonlayer{1}
\draw[red] (A12) arc (-90:90:1.6);
\endpgfonlayer

\pgfonlayer{1}
\draw[red] (A32) arc (90:-90:1.4);
\endpgfonlayer

\node[anchor=east] at ([xshift=-0.2cm]A32) {\scriptsize ${(i,a)}$};

\draw[blue,double distance=2pt] (0,2.5) -- +(-1,0);
\end{tikzpicture}}

\def\isoab#1{\begin{tikzpicture}[every node/.style={shape=circle},inner sep =0.7, scale=#1,anchor=base,baseline=(current bounding box.center)]

\draw[dashed] (-2,0) -- (3,0);
\draw[dashed] (-2,5) -- (3,5); 
\draw[dashed] (-2,0) -- (-2,5); 
\draw[dashed] (3,0) -- (3,5);

\pgfdeclarelayer{1}
\pgfdeclarelayer{2}
\pgfsetlayers{main,1,2}
\draw (0,0) -- (0,0.5);
 \draw (0,1.5) -- (0,2); 
 \draw (0,3) -- (0,3.5); 
 \draw (0,4.5) -- (0,5);
 
\draw (-1,0) -- (-1,0.8);
\draw (-1,0.8) -- (-2,0.8);
\draw (-1,1.2) -- (-1,5);
\draw (-1,1.2) -- (-2,1.2);

\node[] at (0,2.3) {$\vdots $};
\node at (-1.5,2.5) {$E'$};
\node at (-1.5,0.25) {$E'$};
\pgfonlayer{1}

\draw[white,line width=#1*10mm] (0,1)   arc (-90:90:1.5cm);
\draw[] (0,0.5) arc (-90:90:2cm);
\draw[] (0,1.5) arc (-90:90:1cm);

\pgfonlayer{2}
\node[] (A11) at (0,1.2-0.5) {};
\node[] (A31) at (0,5.2-1.5) {};
\node[] (A12) at (0,1.4-0.5) {};
\node[] (A32) at (0,5.4-1.5) {};
\node[] (A13) at (0,1.6-0.5) {};
\node[] (A33) at (0,5.6-1.5) {};
\node[] (A14) at (0,1.8-0.5) {};
\node[] (A34) at (0,5.8-1.5) {};
\endpgfonlayer
\endpgfonlayer

\pgfonlayer{1}
\draw[blue,double distance=1pt] (A11) arc (-90:90:1.8);
\endpgfonlayer
\draw[blue,double distance=1pt] (A11) -- +(-1,0);
\draw[blue,double distance=1pt] (A34) -- +(-1,0);

\pgfonlayer{1}
\draw[blue,double distance=1pt] (A14) arc (-90:90:1.2);
\endpgfonlayer
\draw[blue,double distance=1pt] (A14) -- +(-1,0);
\draw[blue,double distance=1pt] (A31) -- +(-1,0);

\node[fill=red] (A) at ([xshift=-2cm]A12) {};
\node[fill=red] (B) at ([xshift=-2cm]A13) {};
\node[anchor=north east] () at (A) {\scriptsize $v'$};
\node[anchor=south east] () at (B) {\scriptsize $u'$};

\draw[red] (A).. controls ([xshift=0.2cm]A) and ([xshift=0.2cm]B) .. (B);

\draw[blue,double distance=2pt] (0,2.5) -- +(-1,0);
\end{tikzpicture}}

\def\isob#1{\begin{tikzpicture}[every node/.style={shape=circle},inner sep =0, scale=#1,anchor=base,baseline=(current bounding box.center)]

\pgfdeclarelayer{1}
\pgfsetlayers{main,1}

\draw (0,0) -- (0,0.5);
 \draw (0,1.5) -- (0,2); 
 \draw (0,3) -- (0,3.5); 
 \draw (0,4.5) -- (0,5);
 
\draw (-1.25,0) -- (-1.25,0.5+9/16);
\draw (-1.25,0.5+9/16) -- (-2.25,0.5+9/16);
\draw (-1.25,0.5+13/16) -- (-1.25,3.5+9/16);
\draw (-1.25,0.5+13/16) -- (-2.25,0.5+13/16);
\draw (-1.25,3.5+9/16) -- (-2.25,3.5+9/16);
\draw (-1.25,3.5+13/16) -- (-1.25,5);
\draw (-1.25,3.5+13/16) -- (-2.25,3.5+13/16);
\node[] at (0,2.3) {$\vdots $};
\node at (-1.75,2.5) {$E'$};
\node at (-1.75,0.25) {$E'$};
\node at (-1.75,4.5) {$E'$};

\draw[dashed] (-2.25,0) -- (2.5,0);
\draw[dashed] (-2.25,5) -- (2.5,5); 
\draw[dashed] (-2.25,0) -- (-2.25,5); 
\draw[dashed] (2.5,0) -- (2.5,5);

\pgfonlayer{1}

\draw[white,line width=#1*10mm] (0,1)   arc (-90:90:1.5cm);
\draw[] (0,0.5) arc (-90:90:2cm);
\draw[] (0,1.5) arc (-90:90:1cm);

\node[fill=black] (A11) at (0,1.125-0.5) {};
\node[fill=black] (A31) at (0,5.125-1.5) {};
\node[fill=black] (A12) at (0,1.25-0.5) {};
\node[fill=black] (A32) at (0,5.25-1.5) {};
\node[fill=black] (A13) at (0,1.375-0.5) {};
\node[fill=black] (A33) at (0,5.375-1.5) {};
\node[fill=black] (A14) at (0,1.5-0.5) {};
\node[fill=black] (A34) at (0,5.5-1.5) {};
\node[fill=black] (A15) at (0,1.625-0.5) {};
\node[fill=black] (A35) at (0,5.625-1.5) {};
\node[fill=black] (A16) at (0,1.75-0.5) {};
\node[fill=black] (A36) at (0,5.75-1.5) {};
\node[] (AA15) at (-2.25,1.625-0.5) {};
\node[] (AA35) at (-2.25,5.625-1.5) {};
\node[] (AA16) at (-2.25,1.75-0.5) {};
\node[] (AA36) at (-2.25,5.75-1.5) {};
\node[fill=black] (A17) at (0,1.875-0.5) {};
\node[fill=black] (A37) at (0,5.875-1.5) {};
\endpgfonlayer

\pgfonlayer{1}
\draw[blue,double distance=1pt] (A11) arc (-90:90:2.375-0.5);
\endpgfonlayer
\draw[blue,double distance=1pt] (A11) -- +(-1.25,0);
\draw[blue,double distance=1pt] (A37) -- +(-1.25,0);

\pgfonlayer{1}
\draw[blue,double distance=1pt] (A14) arc (-90:90:2.0-0.5);
\endpgfonlayer
\draw[blue,double distance=1pt] (A14) -- +(-1.25,0);
\draw[blue,double distance=1pt] (A34) -- +(-1.25,0);

\pgfonlayer{1}
\draw[blue,double distance=1pt] (A17) arc (-90:90:1.625-0.5);
\endpgfonlayer
\draw[blue,double distance=1pt] (A17) -- +(-1.25,0);
\draw[blue,double distance=1pt] (A31) -- +(-1.25,0);

\draw[blue,double distance=2pt] (0,2.5) -- +(-1.25,0);

\draw[green] (A16) -- +(-2.25,0);
\draw[green] (A15) -- +(-2.25,0);
\draw[green] (A32).. controls ([xshift=-0.2cm]A32) and ([xshift=-0.2cm]A33) .. (A33);

\pgfonlayer{1}
\draw[green] (A16) arc (-90:90:1.75-0.5);
\endpgfonlayer

\pgfonlayer{1}
\draw[green] (A33) arc (90:-90:1.875-0.5);
\endpgfonlayer

\draw[red] (A36) -- +(-2.25,0);
\draw[red] (A35) -- +(-2.25,0);
\draw[red] (A12).. controls ([xshift=-0.2cm]A12) and ([xshift=-0.2cm]A13) .. (A13);

\pgfonlayer{1}
\draw[red] (A36) arc (90:-90:2.25-0.5);
\endpgfonlayer

\pgfonlayer{1}
\draw[red] (A13) arc (-90:90:2.125-0.5);
\endpgfonlayer

\node[anchor=east] at ([xshift=-0.2cm]A12) {\scriptsize \contour{white}{${(i_1,a_1)}$}};
\node[anchor=east] at ([xshift=-0.2cm]A32) {\scriptsize \contour{white}{${(i_2,a_2)}$}};

\draw[blue,double distance=2pt] (0,3.5) -- +(-1.25,0);

\end{tikzpicture}}

\def\isoba#1{\begin{tikzpicture}[every node/.style={shape=circle},inner sep =0, scale=#1,anchor=base,baseline=(current bounding box.center)]
\pgfdeclarelayer{1}
\pgfsetlayers{main,1}

\draw (0,0) -- (0,0.5);
 \draw (0,1.5) -- (0,2); 
 \draw (0,3) -- (0,3.5); 
 \draw (0,4.5) -- (0,5);
 
\draw (-1.25,0) -- (-1.25,0.5+9/16);
\draw (-1.25,0.5+9/16) -- (-2.25,0.5+9/16);
\draw (-1.25,0.5+13/16) -- (-1.25,3.5+9/16);
\draw (-1.25,0.5+13/16) -- (-2.25,0.5+13/16);
\draw (-1.25,3.5+9/16) -- (-2.25,3.5+9/16);
\draw (-1.25,3.5+13/16) -- (-1.25,5);
\draw (-1.25,3.5+13/16) -- (-2.25,3.5+13/16);
\node[] at (0,2.3) {$\vdots $};
\node at (-1.75,2.5) {$E'$};
\node at (-1.75,0.25) {$E'$};
\node at (-1.75,4.5) {$E'$};

\draw[dashed] (-2.25,0) -- (2.5,0);
\draw[dashed] (-2.25,5) -- (2.5,5); 
\draw[dashed] (-2.25,0) -- (-2.25,5); 
\draw[dashed] (2.5,0) -- (2.5,5);

\pgfonlayer{1}

\draw[white,line width=#1*10mm] (0,1)   arc (-90:90:1.5cm);
\draw[] (0,0.5) arc (-90:90:2cm);
\draw[] (0,1.5) arc (-90:90:1cm);

\node[fill=black] (A11) at (0,1.125-0.5) {};
\node[fill=black] (A31) at (0,5.125-1.5) {};
\node[] (A12) at (0,1.25-0.5) {};
\node[fill=black] (A32) at (0,5.25-1.5) {};
\node[] (A13) at (0,1.375-0.5) {};
\node[fill=black] (A33) at (0,5.375-1.5) {};
\node[fill=black] (A14) at (0,1.5-0.5) {};
\node[fill=black] (A34) at (0,5.5-1.5) {};
\node[fill=black] (A15) at (0,1.625-0.5) {};
\node[] (A35) at (0,5.625-1.5) {};
\node[fill=black] (A16) at (0,1.75-0.5) {};
\node[] (A36) at (0,5.75-1.5) {};
\node[] (AA15) at (-2.25,1.625-0.5) {};
\node[] (AA35) at (-2.25,5.625-1.5) {};
\node[] (AA16) at (-2.25,1.75-0.5) {};
\node[] (AA36) at (-2.25,5.75-1.5) {};
\node[fill=black] (A17) at (0,1.875-0.5) {};
\node[fill=black] (A37) at (0,5.875-1.5) {};
\endpgfonlayer

\pgfonlayer{1}
\draw[blue,double distance=1pt] (A11) arc (-90:90:2.375-0.5);
\endpgfonlayer
\draw[blue,double distance=1pt] (A11) -- +(-1.25,0);
\draw[blue,double distance=1pt] (A37) -- +(-1.25,0);

\pgfonlayer{1}
\draw[blue,double distance=1pt] (A14) arc (-90:90:2.0-0.5);
\endpgfonlayer
\draw[blue,double distance=1pt] (A14) -- +(-1.25,0);
\draw[blue,double distance=1pt] (A34) -- +(-1.25,0);

\pgfonlayer{1}
\draw[blue,double distance=1pt] (A17) arc (-90:90:1.625-0.5);
\endpgfonlayer
\draw[blue,double distance=1pt] (A17) -- +(-1.25,0);
\draw[blue,double distance=1pt] (A31) -- +(-1.25,0);

\draw[blue,double distance=2pt] (0,2.5) -- +(-1.25,0);

\draw[green] (A16) -- +(-2.25,0);
\draw[green] (A15) -- +(-2.25,0);
\draw[green] (A32).. controls ([xshift=-0.2cm]A32) and ([xshift=-0.2cm]A33) .. (A33);

\pgfonlayer{1}
\draw[green] (A16) arc (-90:90:1.75-0.5);
\endpgfonlayer

\pgfonlayer{1}
\draw[green] (A33) arc (90:-90:1.875-0.5);
\endpgfonlayer

\draw[red] (AA35).. controls ([xshift=-2cm]A35) and ([xshift=-2cm]A36) ..(AA36);

\node[anchor=east] at ([xshift=-0.2cm]A32) {\scriptsize \contour{white}{${(i_2,a_2')}$}};

\draw[blue,double distance=2pt] (0,3.5) -- +(-1.25,0);
\end{tikzpicture}}

\def\isobb#1{\begin{tikzpicture}[every node/.style={shape=circle},inner sep =0, scale=#1,anchor=base,baseline=(current bounding box.center)]
\pgfdeclarelayer{1}
\pgfsetlayers{main,1}

\draw (0,0) -- (0,0.5);
 \draw (0,1.5) -- (0,2); 
 \draw (0,3) -- (0,3.5); 
 \draw (0,4.5) -- (0,5);
 
\draw (-1.25,0) -- (-1.25,0.5+9/16);
\draw (-1.25,0.5+9/16) -- (-2.25,0.5+9/16);
\draw (-1.25,0.5+13/16) -- (-1.25,3.5+9/16);
\draw (-1.25,0.5+13/16) -- (-2.25,0.5+13/16);
\draw (-1.25,3.5+9/16) -- (-2.25,3.5+9/16);
\draw (-1.25,3.5+13/16) -- (-1.25,5);
\draw (-1.25,3.5+13/16) -- (-2.25,3.5+13/16);
\node[] at (0,2.3) {$\vdots $};
\node at (-1.75,2.5) {$E'$};
\node at (-1.75,0.25) {$E'$};
\node at (-1.75,4.5) {$E'$};

\draw[dashed] (-2.25,0) -- (2.5,0);
\draw[dashed] (-2.25,5) -- (2.5,5); 
\draw[dashed] (-2.25,0) -- (-2.25,5); 
\draw[dashed] (2.5,0) -- (2.5,5);

\pgfonlayer{1}

\draw[white,line width=#1*10mm] (0,1)   arc (-90:90:1.5cm);
\draw[] (0,0.5) arc (-90:90:2cm);
\draw[] (0,1.5) arc (-90:90:1cm);

\node[fill=black] (A11) at (0,1.125-0.5) {};
\node[fill=black] (A31) at (0,5.125-1.5) {};
\node[fill=black] (A12) at (0,1.25-0.5) {};
\node[] (A32) at (0,5.25-1.5) {};
\node[fill=black] (A13) at (0,1.375-0.5) {};
\node[] (A33) at (0,5.375-1.5) {};
\node[fill=black] (A14) at (0,1.5-0.5) {};
\node[fill=black] (A34) at (0,5.5-1.5) {};
\node[] (A15) at (0,1.625-0.5) {};
\node[fill=black] (A35) at (0,5.625-1.5) {};
\node[] (A16) at (0,1.75-0.5) {};
\node[fill=black] (A36) at (0,5.75-1.5) {};
\node[] (AA15) at (-2.25,1.625-0.5) {};
\node[] (AA35) at (-2.25,5.625-1.5) {};
\node[] (AA16) at (-2.25,1.75-0.5) {};
\node[] (AA36) at (-2.25,5.75-1.5) {};
\node[fill=black] (A17) at (0,1.875-0.5) {};
\node[fill=black] (A37) at (0,5.875-1.5) {};
\endpgfonlayer

\pgfonlayer{1}
\draw[blue,double distance=1pt] (A11) arc (-90:90:2.375-0.5);
\endpgfonlayer
\draw[blue,double distance=1pt] (A11) -- +(-1.25,0);
\draw[blue,double distance=1pt] (A37) -- +(-1.25,0);

\pgfonlayer{1}
\draw[blue,double distance=1pt] (A14) arc (-90:90:2.0-0.5);
\endpgfonlayer
\draw[blue,double distance=1pt] (A14) -- +(-1.25,0);
\draw[blue,double distance=1pt] (A34) -- +(-1.25,0);

\pgfonlayer{1}
\draw[blue,double distance=1pt] (A17) arc (-90:90:1.625-0.5);
\endpgfonlayer
\draw[blue,double distance=1pt] (A17) -- +(-1.25,0);
\draw[blue,double distance=1pt] (A31) -- +(-1.25,0);

\draw[blue,double distance=2pt] (0,2.5) -- +(-1.25,0);

\draw[red] (A36) -- +(-2.25,0);
\draw[red] (A35) -- +(-2.25,0);
\draw[red] (A12).. controls ([xshift=-0.2cm]A12) and ([xshift=-0.2cm]A13) .. (A13);

\pgfonlayer{1}
\draw[red] (A36) arc (90:-90:2.25-0.5);
\endpgfonlayer

\pgfonlayer{1}
\draw[red] (A13) arc (-90:90:2.125-0.5);
\endpgfonlayer
\draw[green] (-2.25,1.625-0.5).. controls (-2,1.625-0.5) and (-2,1.75-0.5) .. (-2.25,1.75-0.5);

\node[anchor=east] at ([xshift=-0.2cm]A12) {\scriptsize \contour{white}{${(i_1,a_1')}$}};

\end{tikzpicture}}

\def\isobc#1{\begin{tikzpicture}[every node/.style={shape=circle},inner sep =0, scale=#1,anchor=base,baseline=(current bounding box.center)]\pgfdeclarelayer{1}
\pgfsetlayers{main,1}

\draw (0,0) -- (0,0.5);
 \draw (0,1.5) -- (0,2); 
 \draw (0,3) -- (0,3.5); 
 \draw (0,4.5) -- (0,5);
 
\draw (-1.25,0) -- (-1.25,0.5+9/16);
\draw (-1.25,0.5+9/16) -- (-2.25,0.5+9/16);
\draw (-1.25,0.5+13/16) -- (-1.25,3.5+9/16);
\draw (-1.25,0.5+13/16) -- (-2.25,0.5+13/16);
\draw (-1.25,3.5+9/16) -- (-2.25,3.5+9/16);
\draw (-1.25,3.5+13/16) -- (-1.25,5);
\draw (-1.25,3.5+13/16) -- (-2.25,3.5+13/16);
\node[] at (0,2.3) {$\vdots $};
\node at (-1.75,2.5) {$E'$};
\node at (-1.75,0.25) {$E'$};
\node at (-1.75,4.5) {$E'$};

\draw[dashed] (-2.25,0) -- (2.5,0);
\draw[dashed] (-2.25,5) -- (2.5,5); 
\draw[dashed] (-2.25,0) -- (-2.25,5); 
\draw[dashed] (2.5,0) -- (2.5,5);

\pgfonlayer{1}

\draw[white,line width=#1*10mm] (0,1)   arc (-90:90:1.5cm);
\draw[] (0,0.5) arc (-90:90:2cm);
\draw[] (0,1.5) arc (-90:90:1cm);

\node[fill=black] (A11) at (0,1.125-0.5) {};
\node[fill=black] (A31) at (0,5.125-1.5) {};
\node[] (A12) at (0,1.25-0.5) {};
\node[] (A32) at (0,5.25-1.5) {};
\node[] (A13) at (0,1.375-0.5) {};
\node[] (A33) at (0,5.375-1.5) {};
\node[fill=black] (A14) at (0,1.5-0.5) {};
\node[fill=black] (A34) at (0,5.5-1.5) {};
\node[] (A15) at (0,1.625-0.5) {};
\node[] (A35) at (0,5.625-1.5) {};
\node[] (A16) at (0,1.75-0.5) {};
\node[] (A36) at (0,5.75-1.5) {};
\node[] (AA15) at (-2.25,1.625-0.5) {};
\node[] (AA35) at (-2.25,5.625-1.5) {};
\node[] (AA16) at (-2.25,1.75-0.5) {};
\node[] (AA36) at (-2.25,5.75-1.5) {};
\node[fill=black] (A17) at (0,1.875-0.5) {};
\node[fill=black] (A37) at (0,5.875-1.5) {};
\endpgfonlayer

\pgfonlayer{1}
\draw[blue,double distance=1pt] (A11) arc (-90:90:2.375-0.5);
\endpgfonlayer
\draw[blue,double distance=1pt] (A11) -- +(-1.25,0);
\draw[blue,double distance=1pt] (A37) -- +(-1.25,0);

\pgfonlayer{1}
\draw[blue,double distance=1pt] (A14) arc (-90:90:2.0-0.5);
\endpgfonlayer
\draw[blue,double distance=1pt] (A14) -- +(-1.25,0);
\draw[blue,double distance=1pt] (A34) -- +(-1.25,0);

\pgfonlayer{1}
\draw[blue,double distance=1pt] (A17) arc (-90:90:1.625-0.5);
\endpgfonlayer
\draw[blue,double distance=1pt] (A17) -- +(-1.25,0);
\draw[blue,double distance=1pt] (A31) -- +(-1.25,0);

\draw[blue,double distance=2pt] (0,2.5) -- +(-1.25,0);

\draw[green] (AA15).. controls ([xshift=-2cm]A15) and ([xshift=-2cm]A16) ..(AA16);

\draw[red] (AA35).. controls ([xshift=-2cm]A35) and ([xshift=-2cm]A36) ..(AA36);

\end{tikzpicture}}

\def\isoc#1{\begin{tikzpicture}[every node/.style={shape=circle},inner sep =0, scale=#1,anchor=base,baseline=(current bounding box.center)]

\pgfdeclarelayer{1}
\pgfsetlayers{main,1}
\draw (0,0) -- (0,0.5);
 \draw (0,1.5) -- (0,2); 
 \draw (0,3) -- (0,3.5); 
 \draw (0,4.5) -- (0,5);
\draw (-1.25,0) -- (-1.25,0.5+3/12);
\draw (-1.25,0.5+3/12) -- (-2.25,0.5+3/12);

\draw (-1.25,0.5+9/12) -- (-1.25,3.5+3/12);
\draw (-1.25,0.5+9/12) -- (-2.25,0.5+9/12);
\draw (-1.25,3.5+3/12) -- (-2.25,3.5+3/12);

\draw (-1.25,3.5+9/12) -- (-1.25,5);
\draw (-1.25,3.5+9/12) -- (-2.25,3.5+9/12);
\node[] at (0,2.3) {$\vdots $};
\node at (-1.75,2.5) {$E'$};
\node at (-1.75,0.25) {$E'$};
\node at (-1.75,4.5) {$E'$};

\draw[dashed] (-2.25,0) -- (2.5,0);
\draw[dashed] (-2.25,5) -- (2.5,5); 
\draw[dashed] (-2.25,0) -- (-2.25,5); 
\draw[dashed] (2.5,0) -- (2.5,5);

\pgfonlayer{1}

\draw[white,line width=#1*10mm] (0,1)   arc (-90:90:1.5cm);
\draw[] (0,0.5) arc (-90:90:2cm);
\draw[] (0,1.5) arc (-90:90:1cm);

\node[fill=black] (A11) at (0,0.5+1/6) {};
\node[fill=black] (A31) at (0,3.5+1/6) {};
\node[fill=black] (A12) at (0,0.5+2/6) {};
\node[fill=black] (A32) at (0,3.5+2/6) {};
\node[fill=black] (A13) at (0,0.5+3/6) {};
\node[fill=black] (A33) at (0,3.5+3/6) {};
\node[fill=black] (A14) at (0,0.5+4/6) {};
\node[fill=black] (A34) at (0,3.5+4/6) {};
\node[fill=black] (A15) at (0,0.5+5/6) {};
\node[fill=black] (A35) at (0,3.5+5/6) {};
\endpgfonlayer

\pgfonlayer{1}
\draw[blue,double distance=1pt] (A11) arc (-90:90:2-1/6);
\endpgfonlayer
\draw[blue,double distance=1pt] (A11) -- +(-1.25,0);
\draw[blue,double distance=1pt] (A31) -- +(-1.25,0);

\pgfonlayer{1}
\draw[blue,double distance=1pt] (A15) arc (-90:90:1+1/6);
\endpgfonlayer
\draw[blue,double distance=1pt] (A15) -- +(-1.25,0);
\draw[blue,double distance=1pt] (A35) -- +(-1.25,0);

\draw[red] (A12) -- +(-2.25,0);
\draw[red] (A32) -- +(-2.25,0);
\draw[red] (A34).. controls ([xshift=-0.2cm]A34) and ([xshift=-0.2cm]A33) .. (A33);
\draw[red] (A14).. controls ([xshift=-0.2cm]A14) and ([xshift=-0.2cm]A13) .. (A13);

\pgfonlayer{1}
\draw[red] (A12) arc (-90:90:1+4/6);
\endpgfonlayer

\pgfonlayer{1}
\draw[red] (A33) arc (90:-90:1.5);
\endpgfonlayer

\pgfonlayer{1}
\draw[red] (A14) arc (-90:90:1+2/6);
\endpgfonlayer

\node[anchor=east] at ([xshift=-0.2cm]A13) {\scriptsize ${(i_1,a_1)}$};
\node[anchor=east] at ([xshift=-0.2cm]A33) {\scriptsize ${(i_2,a_2)}$};

\draw[blue,double distance=2pt] (0,2.5) -- +(-1.25,0);

\end{tikzpicture}}

\def\isoca#1{\begin{tikzpicture}[every node/.style={shape=circle},inner sep =0, scale=#1,anchor=base,baseline=(current bounding box.center)]

\pgfdeclarelayer{1}
\pgfsetlayers{main,1}
\draw (0,0) -- (0,0.5);
 \draw (0,1.5) -- (0,2); 
 \draw (0,3) -- (0,3.5); 
 \draw (0,4.5) -- (0,5);
\draw (-1.25,0) -- (-1.25,0.5+3/12);
\draw (-1.25,0.5+3/12) -- (-2.25,0.5+3/12);

\draw (-1.25,0.5+9/12) -- (-1.25,3.5+3/12);
\draw (-1.25,0.5+9/12) -- (-2.25,0.5+9/12);
\draw (-1.25,3.5+3/12) -- (-2.25,3.5+3/12);

\draw (-1.25,3.5+9/12) -- (-1.25,5);
\draw (-1.25,3.5+9/12) -- (-2.25,3.5+9/12);
\node[] at (0,2.3) {$\vdots $};
\node at (-1.75,2.5) {$E'$};
\node at (-1.75,0.25) {$E'$};
\node at (-1.75,4.5) {$E'$};

\draw[dashed] (-2.25,0) -- (2.5,0);
\draw[dashed] (-2.25,5) -- (2.5,5); 
\draw[dashed] (-2.25,0) -- (-2.25,5); 
\draw[dashed] (2.5,0) -- (2.5,5);

\pgfonlayer{1}

\draw[white,line width=#1*10mm] (0,1)   arc (-90:90:1.5cm);
\draw[] (0,0.5) arc (-90:90:2cm);
\draw[] (0,1.5) arc (-90:90:1cm);

\node[fill=black] (A11) at (0,0.5+1/6) {};
\node[fill=black] (A31) at (0,3.5+1/6) {};
\node[fill=black] (A12) at (0,0.5+2/6) {};
\node[] (A32) at (0,3.5+2/6) {};
\node[] (A13) at (0,0.5+3/6) {};
\node[] (A33) at (0,3.5+3/6) {};
\node[] (A14) at (0,0.5+4/6) {};
\node[fill=black] (A34) at (0,3.5+4/6) {};
\node[fill=black] (A15) at (0,0.5+5/6) {};
\node[fill=black] (A35) at (0,3.5+5/6) {};
\endpgfonlayer

\pgfonlayer{1}
\draw[blue,double distance=1pt] (A11) arc (-90:90:2-1/6);
\endpgfonlayer
\draw[blue,double distance=1pt] (A11) -- +(-1.25,0);
\draw[blue,double distance=1pt] (A31) -- +(-1.25,0);

\pgfonlayer{1}
\draw[blue,double distance=1pt] (A15) arc (-90:90:1+1/6);
\endpgfonlayer
\draw[blue,double distance=1pt] (A15) -- +(-1.25,0);
\draw[blue,double distance=1pt] (A35) -- +(-1.25,0);

\draw[red] (A12) -- +(-2.25,0);
\draw[red] (A34) -- +(-1,0);
\draw[red] ([xshift=-1cm]A34) arc (90:270:0.08333);
\draw[red] ([xshift=-1cm]A33) arc (90:-90:0.08333);
\draw[red] (-1,5.3333-1.5) -- +(-1.25,0);

\pgfonlayer{1}
\draw[red] (A12) arc (-90:90:1+4/6);
\endpgfonlayer

\draw[blue,double distance=2pt] (0,2.5) -- +(-1.25,0);

\end{tikzpicture}}

\def\isocb#1{\begin{tikzpicture}[every node/.style={shape=circle},inner sep =0, scale=#1,anchor=base,baseline=(current bounding box.center)]

\pgfdeclarelayer{1}
\pgfsetlayers{main,1}
\draw (0,0) -- (0,0.5);
 \draw (0,1.5) -- (0,2); 
 \draw (0,3) -- (0,3.5); 
 \draw (0,4.5) -- (0,5);
\draw (-1.25,0) -- (-1.25,0.5+3/12);
\draw (-1.25,0.5+3/12) -- (-2.25,0.5+3/12);

\draw (-1.25,0.5+9/12) -- (-1.25,3.5+3/12);
\draw (-1.25,0.5+9/12) -- (-2.25,0.5+9/12);
\draw (-1.25,3.5+3/12) -- (-2.25,3.5+3/12);

\draw (-1.25,3.5+9/12) -- (-1.25,5);
\draw (-1.25,3.5+9/12) -- (-2.25,3.5+9/12);
\node[] at (0,2.3) {$\vdots $};
\node at (-1.75,2.5) {$E'$};
\node at (-1.75,0.25) {$E'$};
\node at (-1.75,4.5) {$E'$};

\draw[dashed] (-2.25,0) -- (2.5,0);
\draw[dashed] (-2.25,5) -- (2.5,5); 
\draw[dashed] (-2.25,0) -- (-2.25,5); 
\draw[dashed] (2.5,0) -- (2.5,5);

\pgfonlayer{1}

\draw[white,line width=#1*10mm] (0,1)   arc (-90:90:1.5cm);
\draw[] (0,0.5) arc (-90:90:2cm);
\draw[] (0,1.5) arc (-90:90:1cm);

\node[fill=black] (A11) at (0,0.5+1/6) {};
\node[fill=black] (A31) at (0,3.5+1/6) {};
\node[] (A12) at (0,0.5+2/6) {};
\node[fill=black] (A32) at (0,3.5+2/6) {};
\node[] (A13) at (0,0.5+3/6) {};
\node[] (A33) at (0,3.5+3/6) {};
\node[fill=black] (A14) at (0,0.5+4/6) {};
\node[] (A34) at (0,3.5+4/6) {};
\node[fill=black] (A15) at (0,0.5+5/6) {};
\node[fill=black] (A35) at (0,3.5+5/6) {};
\endpgfonlayer

\pgfonlayer{1}
\draw[blue,double distance=1pt] (A11) arc (-90:90:2-1/6);
\endpgfonlayer
\draw[blue,double distance=1pt] (A11) -- +(-1.25,0);
\draw[blue,double distance=1pt] (A31) -- +(-1.25,0);

\pgfonlayer{1}
\draw[blue,double distance=1pt] (A15) arc (-90:90:1+1/6);
\endpgfonlayer
\draw[blue,double distance=1pt] (A15) -- +(-1.25,0);
\draw[blue,double distance=1pt] (A35) -- +(-1.25,0);

\draw[red] (A32) -- +(-2.25,0);
\draw[red] (A14) -- +(-1,0);
\draw[red] ([xshift=-1cm]A14) arc (90:270:0.08333);
\draw[red] ([xshift=-1cm]A13) arc (90:-90:0.08333);
\draw[red] (-1,1.3333-0.5) -- +(-1.25,0);

\pgfonlayer{1}
\draw[red] (A14) arc (-90:90:1+2/6);
\endpgfonlayer

\draw[blue,double distance=2pt] (0,2.5) -- +(-1.25,0);

\end{tikzpicture}}

\def\hsvset#1{\begin{tikzpicture}[every node/.style={shape=circle},inner sep =0.8, scale=#1,anchor=base,baseline=(current bounding box.center)]
\pgfdeclarelayer{1}
\pgfdeclarelayer{2}
\pgfsetlayers{main,1,2}
\draw (0,0) -- (0,0.5);
 \draw (0,1.5) -- (0,2); 
 \draw (0,3) -- (0,3.5); 
 \draw (0,4.5) -- (0,5.5); 
 \draw (0,6.5) -- (0,7); 
 \draw (0,8) -- (0,8.5); 
 \draw (0,9.5) -- (0,10);
\node[] at (0,2.3) {$\vdots $};
 \node[] at (0,7.3) {$\vdots $};
\pgfonlayer{2}

\draw[dashed] (-2,0) -- (3,0);
\draw[dashed] (-2,10) -- (3,10);
\draw[dashed] (-2,0) -- (-2,10);
\draw[dashed] (3,0) -- (3,10);

\draw[white,line width=#1*10mm] (0,1)   arc (-90:90:1.5cm);
\draw[] (0,0.5) arc (-90:90:2cm);
\draw[] (0,1.5) arc (-90:90:1cm);

\node[fill=red] (A11) at (0,1.25-0.5) {};
\node[fill=red] (A31) at (0,5.25-1.5) {};
\node[] (A12) at (0,1.5-0.5) {};
\node[] (A32) at (0,5.5-1.5) {};
\node[fill=red] (A13) at (0,1.75-0.5) {};
\node[fill=red] (A33) at (0,5.75-1.5) {};
\endpgfonlayer
\pgfonlayer{1}

\draw[white,line width=#1*10mm] (0,6)   arc (-90:90:1.5cm);
\draw[] (0,5.5) arc (-90:90:2cm);
\draw[] (0,6.5) arc (-90:90:1cm);

\node[fill=blue] (A41) at (0,7.25-1.5) {};
\node[fill=blue] (A61) at (0,11.25-2.5) {};
\node[] (A42) at (0,7.5-1.5) {};
\node[] (A62) at (0,11.5-2.5) {};
\node[fill=blue] (A43) at (0,7.75-1.5) {};
\node[fill=blue] (A63) at (0,11.75-2.5) {};

\pgfonlayer{2}
\draw[white] (0,6)   arc (-90:90:1.5cm) node[inner sep=5pt,fill=white,text=black,draw=black,thick,shape=rectangle,pos=0.5] {$s(q)$};

\node[anchor=west] at (A31) {$i$};
\node[anchor=west] at (A61) {$i'$};
\node[anchor=west] at (A41) {\contour{white}{$i+\e$}};
\node[anchor=west] at (0,6.25-1.5) {$\uparrow$};
\draw[dotted,thick] (A11) -- (A13);
\draw[dotted,thick] (A31) -- (A33);
\draw[dotted,thick] (A41) -- (A43);
\draw[dotted,thick] (A61) -- (A63);\endpgfonlayer

\endpgfonlayer\end{tikzpicture}}

\def\hsvseta#1{\begin{tikzpicture}[every node/.style={shape=circle},inner sep =0.8, scale=#1,anchor=base,baseline=(current bounding box.center)]
\pgfdeclarelayer{1}
\pgfdeclarelayer{2}
\pgfsetlayers{main,1,2}
\draw (0,0) -- (0,0.5);
 \draw (0,1.5) -- (0,2); 
 \draw (0,3) -- (0,3.5); 
 \draw (0,4.5) -- (0,5); 
 \draw (0,6) -- (0,6.5); 
 \draw (0,7.5) -- (0,8.5); 
 \draw (0,9.5) -- (0,10);
\node[] at (0,2.3) {$\vdots $};
\node[] at (0,5.3) {$\vdots $};
\pgfonlayer{2}

\draw[dashed] (-3,0) -- (5,0);
\draw[dashed] (-3,10) -- (5,10);
\draw[dashed] (-3,0) -- (-3,10);
\draw[dashed] (5,0) -- (5,10);

\draw[white,line width=#1*10mm] (0,1)   arc (-90:90:4cm);
\draw[] (0,0.5) arc (-90:90:4.5cm);
\draw[] (0,1.5) arc (-90:90:3.5cm);

\node[fill=red] (A11) at (0,1.25-0.5) {};
\node[fill=green] (A61) at (0,11.25-2.5) {};
\node[] (A12) at (0,1.5-0.5) {};
\node[] (A62) at (0,11.5-2.5) {};
\node[fill=red] (A13) at (0,1.75-0.5) {};
\node[fill=green] (A63) at (0,11.75-2.5) {};
\endpgfonlayer
\pgfonlayer{1}

\draw[white,line width=#1*10mm] (0,5.5-1.5)   arc (-90:90:1.5cm);
\draw[] (0,3.5) arc (-90:90:2cm);
\draw[] (0,4.5) arc (-90:90:1cm);

\node[fill=red] (A31) at (0,5.142857142857143-1.5) {};
\node[fill=blue] (A51) at (0,9.142857142857142-2.5) {};
\node[] (A32) at (0,5.285714285714286-1.5) {};
\node[] (A52) at (0,9.285714285714286-2.5) {};
\node[fill=red] (A33) at (0,5.428571428571429-1.5) {};
\node[fill=blue] (A53) at (0,9.428571428571429-2.5) {};
\node[fill=blue] (A34) at (0,5.571428571428571-1.5) {};
\node[fill=green] (A54) at (0,9.571428571428571-2.5) {};
\node[] (A35) at (0,5.714285714285714-1.5) {};
\node[] (A55) at (0,9.714285714285714-2.5) {};
\node[fill=blue] (A36) at (0,5.857142857142857-1.5) {};
\node[fill=green] (A56) at (0,9.857142857142858-2.5) {};
\endpgfonlayer

\pgfonlayer{2}\node[anchor=west] at (A61) {$\sigma(i)$};
\node[anchor=west] at (A52) {$\sigma(i')$};
\node[anchor=west] at (A33) {\contour{white}{$\sigma(i+\e)$}};
\draw[dotted,thick] (A11) -- (A13);
\draw[dotted,thick] (A31) -- (A33);
\draw[dotted,thick] (A34) -- (A36);
\draw[dotted,thick] (A51) -- (A53);
\draw[dotted,thick] (A54) -- (A56);
\draw[dotted,thick] (A61) -- (A63);
\draw[decorate,decoration={brace,amplitude=5pt,raise=1ex}] (A54) -- (A63) node[pos=0.5,anchor=east]{$U_{(i,\e)} \ \ \ $};\endpgfonlayer

\end{tikzpicture}}

\def\hsvsetb#1{\begin{tikzpicture}[every node/.style={shape=circle},inner sep =0.8, scale=#1,anchor=base,baseline=(current bounding box.center)]

\pgfdeclarelayer{1}
\pgfdeclarelayer{2}
\pgfsetlayers{main,1,2}
\draw (0,0) -- (0,0.5);
 \draw (0,1.5) -- (0,2); 
 \draw (0,3) -- (0,3.5); 
 \draw (0,4.5) -- (0,5); 
 \draw (0,6) -- (0,6.5); 
 \draw (0,7.5) -- (0,8.5); 
 \draw (0,9.5) -- (0,10);
\node[] at (0,2.3) {$\vdots $};
\node[] at (0,5.3) {$\vdots $};
\pgfonlayer{2}

\draw[dashed] (-3,0) -- (3.5,0);
\draw[dashed] (-3,10) -- (3.5,10);
\draw[dashed] (-3,0) -- (-3,10);
\draw[dashed] (3.5,0) -- (3.5,10);

\draw[white,line width=#1*5mm] (0,1.5-0.5)   arc (-90:90:3cm);
\draw[white,line width=#1*5mm] (0,1.75-0.5)   arc (-90:90:3cm);
\draw[white,line width=#1*5mm] (0,1.25-0.5)   arc (-90:90:3cm);
\draw[] (0,1-0.5) arc (-90:90:3cm);
\draw[] (0,2-0.5) arc (-90:90:3cm);
\draw[] (3,3.5)--(3,4.5);

\node[fill=red] (A11) at (0,1.25-0.5) {};
\node[fill=green] (A51) at (0,9.25-2.5) {};
\node[] (A12) at (0,1.5-0.5) {};
\node[] (A52) at (0,9.5-2.5) {};
\node[fill=red] (A13) at (0,1.75-0.5) {};
\node[fill=green] (A53) at (0,9.75-2.5) {};
\endpgfonlayer
\pgfonlayer{1}

\draw[white,line width=#1*5mm] (0,5.5-1.5)   arc (-90:90:2.5cm);
\draw[white,line width=#1*5mm] (0,5.75-1.5)   arc (-90:90:2.5cm);
\draw[white,line width=#1*5mm] (0,5.25-1.5)   arc (-90:90:2.5cm);
\draw[] (0,5-1.5) arc (-90:90:2.5cm);
\draw[] (0,6-1.5) arc (-90:90:2.5cm);
\draw[] (2.5,6)--(2.5,7);

\node[fill=red] (A31) at (0,5.142857142857143-1.5) {};
\node[fill=green] (A61) at (0,11.142857142857142-2.5) {};
\node[] (A32) at (0,5.285714285714286-1.5) {};
\node[] (A62) at (0,11.285714285714286-2.5) {};
\node[fill=red] (A33) at (0,5.428571428571429-1.5) {};
\node[fill=green] (A63) at (0,11.428571428571429-2.5) {};
\node[fill=blue] (A34) at (0,5.571428571428571-1.5) {};
\node[fill=blue] (A64) at (0,11.571428571428571-2.5) {};
\node[] (A35) at (0,5.714285714285714-1.5) {};
\node[] (A65) at (0,11.714285714285714-2.5) {};
\node[fill=blue] (A36) at (0,5.857142857142857-1.5) {};
\node[fill=blue] (A66) at (0,11.857142857142858-2.5) {};
\endpgfonlayer

\pgfonlayer{2}\node[anchor=west] at (A61) {\contour{white}{$\sigma(i')$}};
\node[anchor=west] at (A52) {\contour{white}{$\sigma(i)$}};
\node[anchor=west] at (A33) {\contour{white}{$\sigma(i+\e)$}};
\draw[dotted,thick] (A11) -- (A13);
\draw[dotted,thick] (A31) -- (A33);
\draw[dotted,thick] (A34) -- (A36);
\draw[dotted,thick] (A51) -- (A53);
\draw[dotted,thick] (A64) -- (A66);
\draw[dotted,thick] (A61) -- (A63);
\draw[decorate,decoration={brace,amplitude=5pt,raise=1ex}] (A51) -- (A63) node[pos=0.5,anchor=east]{$U_{(i,\e)} \ \ \ $};\endpgfonlayer

\end{tikzpicture}}

\def\hsa#1{\begin{tikzpicture}[every node/.style={shape=circle},inner sep =0, scale=#1,anchor=base,baseline=(current bounding box.center)]
\pgfdeclarelayer{1}
\pgfdeclarelayer{2}
\pgfsetlayers{main,1,2}
\draw (0,0) -- (0,0.5);
 \draw (0,1.5) -- (0,2); 
 \draw (0,3) -- (0,3.5); 
 \draw (0,4.5) -- (0,5.5); 
 \draw (0,6.5) -- (0,7); 
 \draw (0,8) -- (0,8.5); 
 \draw (0,9.5) -- (0,10);
\draw (-1,0) -- (-1,10);
\node[] at (0,2.3) {$\vdots $};
 \node[] at (0,7.3) {$\vdots $};
\pgfonlayer{2}

\draw[dashed] (-2.5,0) -- (3,0);
\draw[dashed] (-2.5,10) -- (3,10);
\draw[dashed] (-2.5,0) -- (-2.5,10);
\draw[dashed] (3,0) -- (3,10);

\draw[white,line width=#1*10mm] (0,1)   arc (-90:90:1.5cm);
\draw[] (0,0.5) arc (-90:90:2cm);
\draw[] (0,1.5) arc (-90:90:1cm);

\node[fill=black] (A11) at (0,1.25-0.5) {};
\node[fill=black,inner sep=1,label=below left:{\contour{white}{$u$}}] (A31) at (0,5.25-1.5) {};
\node[] (A12) at (0,1.5-0.5) {};
\node[] (A32) at (0,5.5-1.5) {};
\node[fill=black,inner sep=1,label=above left:{\contour{white}{$v$}}] (A13) at (0,1.75-0.5) {};
\node[fill=black] (A33) at (0,5.75-1.5) {};
\endpgfonlayer
\pgfonlayer{1}

\draw[white,line width=#1*10mm] (0,6)   arc (-90:90:1.5cm);
\draw[] (0,5.5) arc (-90:90:2cm);
\draw[] (0,6.5) arc (-90:90:1cm);

\node[fill=black] (A41) at (0,7.25-1.5) {};
\node[fill=black] (A61) at (0,11.25-2.5) {};
\node[] (A42) at (0,7.5-1.5) {};
\node[] (A62) at (0,11.5-2.5) {};
\node[fill=black] (A43) at (0,7.75-1.5) {};
\node[fill=black] (A63) at (0,11.75-2.5) {};

\pgfonlayer{1}

\draw[blue] (A41) arc (-90:90:1.75cm);
\draw[blue] (A43) arc (-90:90:1.25cm);

\endpgfonlayer
\pgfonlayer{2}

\draw[red] (A11) arc (-90:90:1.75cm);
\draw[red] (A13) arc (-90:90:1.25cm);

\draw[white] (0,6)   arc (-90:90:1.5cm) node[inner sep=5pt,fill=white,text=black,draw=black,thick,shape=rectangle,pos=0.5] {$s(q)$};

\draw[red] (A11) -- +(-1,0);
\draw[red] (A13) -- +(-1,0);
\draw[red] (A31) -- +(-1,0);
\draw[red] (A33) -- +(-1,0);

\draw[blue] (A41) -- +(-1,0);
\draw[blue] (A43) -- +(-1,0);
\draw[blue] (A61) -- +(-1,0);
\draw[blue] (A63) -- +(-1,0);

\node[anchor=west] at (0,6.25-1.5) {\contour{white}{$\uparrow$}};
\draw[dotted,thick] (A11) -- (A13);
\draw[dotted,thick] (A31) -- (A33);
\draw[dotted,thick] (A41) -- (A43);
\draw[dotted,thick] (A61) -- (A63);\endpgfonlayer

\draw[blue,double distance=2pt] (0,2.5) -- +(-1,0);
\draw[blue,double distance=2pt] (0,7.5) -- +(-1,0);

\endpgfonlayer\end{tikzpicture}}

\def\hsaa#1{\begin{tikzpicture}[every node/.style={shape=circle},inner sep =0, scale=#1,anchor=base,baseline=(current bounding box.center)]

\pgfdeclarelayer{1}
\pgfdeclarelayer{2}
\pgfsetlayers{main,1,2}
\draw (0,0) -- (0,0.5);
 \draw (0,1.5) -- (0,2); 
 \draw (0,3) -- (0,3.5); 
 \draw (0,4.5) -- (0,5); 
 \draw (0,6) -- (0,6.5); 
 \draw (0,7.5) -- (0,8.5); 
 \draw (0,9.5) -- (0,10);
\draw (-1,0) -- (-1,10);
\node[] at (0,2.3) {$\vdots $};
\node[] at (0,5.3) {$\vdots $};
\node at (-2,5) {$E$};

\draw[dashed] (-2.5,0) -- (5,0);
\draw[dashed] (-2.5,10) -- (5,10);
\draw[dashed] (-2.5,0) -- (-2.5,10);
\draw[dashed] (5,0) -- (5,10);

\pgfonlayer{2}

\draw[white,line width=#1*10mm] (0,1.5)   arc (-90:90:4cm);
\draw[] (0,1-0.5) arc (-90:90:4.5cm);
\draw[] (0,2-0.5) arc (-90:90:3.5cm);

\node[fill=black] (A11) at (0,1.25-0.5) {};
\node[fill=black,inner sep=1,label=below right:{\contour{white}{$v_2$}}] (A61) at (0,11.25-2.5) {};
\node[] (A12) at (0,1.5-0.5) {};
\node[] (A62) at (0,11.5-2.5) {};
\node[fill=black,inner sep=1,label=above left:{\contour{white}{$v_1$}}] (A13) at (0,1.75-0.5) {};
\node[fill=black] (A63) at (0,11.75-2.5) {};
\endpgfonlayer
\pgfonlayer{1}

\draw[white,line width=#1*10mm] (0,4)   arc (-90:90:1.5cm);
\draw[] (0,5-1.5) arc (-90:90:2cm);
\draw[] (0,6-1.5) arc (-90:90:1cm);

\node[fill=black,inner sep=1,label=below left:{\contour{white}{$u_1$}}] (A31) at (0,5.142857142857143-1.5) {};
\node[fill=black] (A51) at (0,9.142857142857142-2.5) {};
\node[] (A32) at (0,5.285714285714286-1.5) {};
\node[] (A52) at (0,9.285714285714286-2.5) {};
\node[fill=black] (A33) at (0,5.428571428571429-1.5) {};
\node[fill=black] (A53) at (0,9.428571428571429-2.5) {};
\node[fill=black] (A34) at (0,5.571428571428571-1.5) {};
\node[fill=black] (A54) at (0,9.571428571428571-2.5) {};
\node[] (A35) at (0,5.714285714285714-1.5) {};
\node[] (A55) at (0,9.714285714285714-2.5) {};
\node[fill=black] (A36) at (0,5.857142857142857-1.5) {};
\node[fill=black,inner sep=1,label=above right:{\contour{white}{$u_2$}}] (A56) at (0,9.857142857142858-2.5) {};
\endpgfonlayer

\pgfonlayer{1}
\draw[blue] (A34) arc (-90:90:1.9285714285714286-0.5);
\endpgfonlayer
\draw[blue] (A34) -- +(-1,0);
\draw[blue] (A53) -- +(-1,0);

\pgfonlayer{1}
\draw[blue] (A36) arc (-90:90:1.6428571428571428-0.5);
\endpgfonlayer
\draw[blue] (A36) -- +(-1,0);
\draw[blue] (A51) -- +(-1,0);

\draw[red] (A11) -- +(-1,0);
\draw[red] (A33) -- +(-1,0);
\draw[red] (A63).. controls ([xshift=-1cm]A63) and ([xshift=-1cm]A54) .. (A54);

\pgfonlayer{2}
\draw[red] (A11) arc (-90:90:5.25-1);
\endpgfonlayer

\pgfonlayer{1}
\draw[red] (A54) arc (90:-90:2.0714285714285716-0.5);
\endpgfonlayer
\draw[red] (A13) -- +(-1,0);
\draw[red] (A31) -- +(-1,0);
\draw[red] (A61).. controls ([xshift=-0.5cm]A61) and ([xshift=-0.5cm]A56) .. (A56);

\pgfonlayer{2}
\draw[red] (A13) arc (-90:90:4.75-1);
\endpgfonlayer

\pgfonlayer{1}
\draw[red] (A56) arc (90:-90:2.357142857142857-0.5);
\endpgfonlayer

\pgfonlayer{2}
\draw[dotted,thick] (A11) -- (A13);
\draw[dotted,thick] (A31) -- (A33);
\draw[dotted,thick] (A34) -- (A36);
\draw[dotted,thick] (A51) -- (A53);
\draw[dotted,thick] (A54) -- (A56);
\draw[dotted,thick] (A61) -- (A63);\endpgfonlayer

\draw[blue,double distance=2pt] (0,2.5) -- +(-1,0);
\draw[blue,double distance=2pt] (0,5.5) -- +(-1,0);

\end{tikzpicture}}

\def\hsab#1{\begin{tikzpicture}[every node/.style={shape=circle},inner sep =0, scale=#1,anchor=base,baseline=(current bounding box.center)]

\pgfdeclarelayer{1}
\pgfdeclarelayer{2}
\pgfsetlayers{main,1,2}
\draw (0,0) -- (0,0.5);
 \draw (0,1.5) -- (0,2); 
 \draw (0,3) -- (0,3.5); 
 \draw (0,4.5) -- (0,5); 
 \draw (0,6) -- (0,6.5); 
 \draw (0,7.5) -- (0,8.5); 
 \draw (0,9.5) -- (0,10);
\draw (-1,0) -- (-1,10);
\node[] at (0,2.3) {$\vdots $};
\node[] at (0,5.3) {$\vdots $};
\node at (-2,5) {$E$};

\draw[dashed] (-2.5,0) -- (5,0);
\draw[dashed] (-2.5,10) -- (5,10);
\draw[dashed] (-2.5,0) -- (-2.5,10);
\draw[dashed] (5,0) -- (5,10);

\pgfonlayer{2}

\draw[white,line width=#1*5mm] (0,1.5-0.5)   arc (-90:90:3cm);
\draw[white,line width=#1*5mm] (0,1.75-0.5)   arc (-90:90:3cm);
\draw[white,line width=#1*5mm] (0,1.25-0.5)   arc (-90:90:3cm);
\draw[] (0,1-0.5) arc (-90:90:3cm);
\draw[] (0,2-0.5) arc (-90:90:3cm);
\draw[] (3,3.5)--(3,4.5);

\node[fill=black] (A11) at (0,1.25-0.5) {};
\node[fill=black] (A51) at (0,9.25-2.5) {};
\node[] (A12) at (0,1.5-0.5) {};
\node[] (A52) at (0,9.5-2.5) {};
\node[fill=black,inner sep=1,label=above left:{\contour{white}{$v_1$}}] (A13) at (0,1.75-0.5) {};
\node[fill=black,inner sep=1,label=above right:{\contour{white}{$v_2$}}] (A53) at (0,9.75-2.5) {};
\endpgfonlayer
\pgfonlayer{1}

\draw[white,line width=#1*5mm] (0,5.5-1.5)   arc (-90:90:2.5cm);
\draw[white,line width=#1*5mm] (0,5.75-1.5)   arc (-90:90:2.5cm);
\draw[white,line width=#1*5mm] (0,5.25-1.5)   arc (-90:90:2.5cm);
\draw[] (0,5-1.5) arc (-90:90:2.5cm);
\draw[] (0,6-1.5) arc (-90:90:2.5cm);
\draw[] (2.5,6)--(2.5,7);

\node[fill=black,inner sep=1,label=below left:{\contour{white}{$u_1$}}] (A31) at (0,5.142857142857143-1.5) {};
\node[fill=black,inner sep=1,label=below right:{\contour{white}{$u_2$}}] (A61) at (0,11.142857142857142-2.5) {};
\node[] (A32) at (0,5.285714285714286-1.5) {};
\node[] (A62) at (0,11.285714285714286-2.5) {};
\node[fill=black] (A33) at (0,5.428571428571429-1.5) {};
\node[fill=black] (A63) at (0,11.428571428571429-2.5) {};
\node[fill=black] (A34) at (0,5.571428571428571-1.5) {};
\node[fill=black] (A64) at (0,11.571428571428571-2.5) {};
\node[] (A35) at (0,5.714285714285714-1.5) {};
\node[] (A65) at (0,11.714285714285714-2.5) {};
\node[fill=black] (A36) at (0,5.857142857142857-1.5) {};
\node[fill=black] (A66) at (0,11.857142857142858-2.5) {};
\endpgfonlayer

\pgfonlayer{1}
\draw[blue] (A34) arc (-90:90:2.5);
\endpgfonlayer
\draw[blue] (A34) -- +(-1,0);
\draw[blue] (A64) -- +(-1,0);

\pgfonlayer{1}
\draw[blue] (A36) arc (-90:90:2.5);
\endpgfonlayer
\draw[blue] (A36) -- +(-1,0);
\draw[blue] (A66) -- +(-1,0);

\draw[red] (A11) -- +(-1,0);
\draw[red] (A33) -- +(-1,0);
\draw[red] (A51).. controls ([xshift=-1cm]A51) and ([xshift=-1cm]A63) .. (A63);

\pgfonlayer{2}
\draw[red] (A11) arc (-90:90:3.0);
\endpgfonlayer

\pgfonlayer{1}
\draw[red] (A63) arc (90:-90:2.5);
\endpgfonlayer
\draw[red] (A13) -- +(-1,0);
\draw[red] (A31) -- +(-1,0);
\draw[red] (A53).. controls ([xshift=-0.5cm]A53) and ([xshift=-0.5cm]A61) .. (A61);

\pgfonlayer{2}
\draw[red] (A13) arc (-90:90:3.0);
\endpgfonlayer

\pgfonlayer{1}
\draw[red] (A61) arc (90:-90:2.5);
\endpgfonlayer

\pgfonlayer{2}
\draw[dotted,thick] (A11) -- (A13);
\draw[dotted,thick] (A31) -- (A33);
\draw[dotted,thick] (A34) -- (A36);
\draw[dotted,thick] (A51) -- (A53);
\draw[dotted,thick] (A64) -- (A66);
\draw[dotted,thick] (A61) -- (A63);\endpgfonlayer

\draw[blue,double distance=2pt] (0,2.5) -- +(-1,0);
\draw[blue,double distance=2pt] (0,5.5) -- +(-1,0);
\end{tikzpicture}}

\def\exaahs#1{\begin{tikzpicture}[every node/.style={shape=circle},inner sep =0, scale=#1,anchor=base,baseline=(current bounding box.center)]
\pgfdeclarelayer{1}
\pgfdeclarelayer{2}

\pgfsetlayers{main,
1,
2} 

\draw[] (-9,0)--(0,0);
\draw[] (-9,0)--(-9,9);
\draw[] (-9,9)--(0,9);
\draw[] (0,0) --(0,1);
\draw[] (0,2) -- (0,3); 
\node[fill=black,inner sep =0] (A1) at (0,1) {};
\node[fill=black,inner sep =0] (B1) at (0,2) {};
\draw[] (0,4) -- (0,5); 
\node[fill=black,inner sep =0] (A2) at (0,3) {};
\node[fill=black,inner sep =0] (B2) at (0,4) {};
\draw[] (0,6) -- (0,7); 
\node[fill=black,inner sep =0] (A3) at (0,5) {};
\node[fill=black,inner sep =0] (B3) at (0,6) {};
\draw[] (0,8) -- (0,9); 
\node[fill=black,inner sep =0] (A4) at (0,7) {};
\node[fill=black,inner sep =0] (B4) at (0,8) {};

\node[fill=black] (S1) at (-7.5,0) {};
\node[fill=black] (S2) at (-6.000000000000001,0) {};
\node[fill=black] (S3) at (-4.5,0) {};
\node[fill=black] (S4) at (-3.0000000000000004,0) {};
\node[fill=black] (S5) at (-1.4999999999999996,0) {};

\node[fill=black] (N1) at (-7.5,9) {};
\node[fill=black] (N2) at (-6.000000000000001,9) {};
\node[fill=black] (N3) at (-4.5,9) {};
\node[fill=black] (N4) at (-3.0000000000000004,9) {};
\node[fill=black] (N5) at (-1.4999999999999996,9) {};

\pgfonlayer{2}

\draw[white,line width=#1*10mm] (0,3.5)   arc (-90:90:2cm);
\draw[] (0,3) arc (-90:90:2.5cm);
\draw[] (0,4) arc (-90:90:1.5cm);

\node[fill=black] (A21) at (0,3.3333333333333335) {};
\node[fill=black] (A41) at (0,7.333333333333333) {};
\node[fill=black] (A22) at (0,3.6666666666666665) {};
\node[fill=black] (A42) at (0,7.666666666666667) {};

\endpgfonlayer

\pgfonlayer{1}

\draw[white,line width=#1*10mm] (0,1.5)   arc (-90:90:2cm);
\draw[] (0,1) arc (-90:90:2.5cm);
\draw[] (0,2) arc (-90:90:1.5cm);

\node[fill=black] (A11) at (0,1.3333333333333333) {};
\node[fill=black] (A31) at (0,5.333333333333333) {};
\node[fill=black] (A12) at (0,1.6666666666666665) {};
\node[fill=black] (A32) at (0,5.666666666666667) {};

\endpgfonlayer

\draw[orange] (S5).. controls (-1.4999999999999996,0.8888888888888888) and (-0.9999999999999997,1.3333333333333333) .. (A11);
\draw[orange] (S2).. controls (-6.000000000000001,3.777777777777778) and (-4.000000000000001,5.666666666666667) .. (A32);

\pgfonlayer{1}
\draw[orange] (A11) arc (-90:90:2.1666666666666665);
\endpgfonlayer

\draw[purple] (S3).. controls (-4.5,3.5555555555555554) and (-3.0,5.333333333333333) .. (A31);
\draw[purple] (A12).. controls (-0.8333333333333334,1.6666666666666665) and (-0.8333333333333334,3.3333333333333335) .. (A21);
\draw[purple] (N5).. controls (-1.4999999999999996,8.11111111111111) and (-0.9999999999999997,7.666666666666667) .. (A42);

\pgfonlayer{1}
\draw[purple] (A31) arc (90:-90:1.8333333333333333);
\endpgfonlayer

\pgfonlayer{2}
\draw[purple] (A21) arc (-90:90:2.1666666666666665);
\endpgfonlayer

\draw[green] (S4).. controls (-3.0000000000000004,2.444444444444444) and (-2.0000000000000004,3.6666666666666665) .. (A22);
\draw[green] (N4).. controls (-3.0000000000000004,7.888888888888888) and (-2.0000000000000004,7.333333333333333) .. (A41);

\pgfonlayer{2}
\draw[green] (A22) arc (-90:90:1.8333333333333333);
\endpgfonlayer

\draw[blue] (S1).. controls (-7.5,4.5) and (-4.5,4.5).. (N3);

\draw[red] (N1).. controls (-7.5,8.25) and (-6.000000000000001,8.25) .. (N2);

\node[anchor=north west] at (A2) {$\downarrow$};
\end{tikzpicture}}

\def\exabhs#1{\begin{tikzpicture}[every node/.style={shape=circle},inner sep =0, scale=#1,anchor=base,baseline=(current bounding box.center)]
\pgfdeclarelayer{1}
\pgfdeclarelayer{2}

\pgfsetlayers{main,
1,
2} 

\draw[] (-9,0)--(0,0);
\draw[] (-9,0)--(-9,9);
\draw[] (-9,9)--(0,9);
\draw[] (0,0) --(0,1);
\draw[] (0,2) -- (0,3); 
\node[fill=black,inner sep =0] (A1) at (0,1) {};
\node[fill=black,inner sep =0] (B1) at (0,2) {};
\draw[] (0,4) -- (0,5); 
\node[fill=black,inner sep =0] (A2) at (0,3) {};
\node[fill=black,inner sep =0] (B2) at (0,4) {};
\draw[] (0,6) -- (0,7); 
\node[fill=black,inner sep =0] (A3) at (0,5) {};
\node[fill=black,inner sep =0] (B3) at (0,6) {};
\draw[] (0,8) -- (0,9); 
\node[fill=black,inner sep =0] (A4) at (0,7) {};
\node[fill=black,inner sep =0] (B4) at (0,8) {};

\node[fill=black] (S1) at (-7.5,0) {};
\node[fill=black] (S2) at (-6.000000000000001,0) {};
\node[fill=black] (S3) at (-4.5,0) {};
\node[fill=black] (S4) at (-3.0000000000000004,0) {};
\node[fill=black] (S5) at (-1.4999999999999996,0) {};

\node[fill=black] (N1) at (-7.5,9) {};
\node[fill=black] (N2) at (-6.000000000000001,9) {};
\node[fill=black] (N3) at (-4.5,9) {};
\node[fill=black] (N4) at (-3.0000000000000004,9) {};
\node[fill=black] (N5) at (-1.4999999999999996,9) {};

\pgfonlayer{2}

\draw[white,line width=#1*10mm] (0,1.5)   arc (-90:90:2cm);
\draw[] (0,1) arc (-90:90:2.5cm);
\draw[] (0,2) arc (-90:90:1.5cm);

\node[fill=black] (A11) at (0,1.2) {};
\node[fill=black] (A31) at (0,5.2) {};
\node[fill=black] (A12) at (0,1.4) {};
\node[fill=black] (A32) at (0,5.4) {};
\node[fill=black] (A13) at (0,1.6) {};
\node[fill=black] (A33) at (0,5.6) {};
\node[fill=black] (A14) at (0,1.8) {};
\node[fill=black] (A34) at (0,5.8) {};

\endpgfonlayer

\pgfonlayer{1}

\draw[white,line width=#1*10mm] (0,3.5)   arc (-90:90:2cm);
\draw[] (0,3) arc (-90:90:2.5cm);
\draw[] (0,4) arc (-90:90:1.5cm);

\node[fill=black] (A21) at (0,3.3333333333333335) {};
\node[fill=black] (A41) at (0,7.333333333333333) {};
\node[fill=black] (A22) at (0,3.6666666666666665) {};
\node[fill=black] (A42) at (0,7.666666666666667) {};

\endpgfonlayer

\draw[blue] (S1).. controls (-7.5,4.5) and (-4.5,4.5).. (N3);

\draw[red] (N1).. controls (-7.5,8.25) and (-6.000000000000001,8.25) .. (N2);

\draw[orange] (S5).. controls (-1.4999999999999996,0.7999999999999999) and (-0.9999999999999997,1.2) .. (A11);
\draw[orange] (S2).. controls (-6.000000000000001,3.8666666666666667) and (-4.000000000000001,5.8) .. (A34);

\pgfonlayer{2}
\draw[orange] (A11) arc (-90:90:2.3);
\endpgfonlayer

\draw[green] (S4).. controls (-3.0000000000000004,1.2) and (-2.0000000000000004,1.8) .. (A14);
\draw[green] (A31).. controls (-0.7666666666666668,5.2) and (-0.7666666666666668,3.6666666666666665) .. (A22);
\draw[green] (N4).. controls (-3.0000000000000004,7.888888888888888) and (-2.0000000000000004,7.333333333333333) .. (A41);

\pgfonlayer{2}
\draw[green] (A14) arc (-90:90:1.7);
\endpgfonlayer

\pgfonlayer{1}
\draw[green] (A22) arc (-90:90:1.8333333333333333);
\endpgfonlayer

\draw[purple] (S3).. controls (-4.5,3.733333333333333) and (-3.0,5.6) .. (A33);
\draw[purple] (A12).. controls (-0.09999999999999998,1.4) and (-0.09999999999999998,1.6) .. (A13);
\draw[purple] (A32).. controls (-1.0333333333333332,5.4) and (-1.0333333333333332,3.3333333333333335) .. (A21);
\draw[purple] (N5).. controls (-1.4999999999999996,8.11111111111111) and (-0.9999999999999997,7.666666666666667) .. (A42);

\pgfonlayer{2}
\draw[purple] (A33) arc (90:-90:2.1);
\endpgfonlayer

\pgfonlayer{2}
\draw[purple] (A13) arc (-90:90:1.9);
\endpgfonlayer

\pgfonlayer{1}
\draw[purple] (A21) arc (-90:90:2.1666666666666665);
\endpgfonlayer

\end{tikzpicture}}

\def\exachs#1{\begin{tikzpicture}[every node/.style={shape=circle},inner sep =0, scale=#1,anchor=base,baseline=(current bounding box.center)]
\pgfdeclarelayer{1}
\pgfdeclarelayer{2}

\pgfsetlayers{main,
1,
2} 

\draw[] (-9,0)--(0,0);
\draw[] (-9,0)--(-9,9);
\draw[] (-9,9)--(0,9);
\draw[] (0,0) --(0,1);
\draw[] (0,2) -- (0,3); 
\node[fill=black,inner sep =0] (A1) at (0,1) {};
\node[fill=black,inner sep =0] (B1) at (0,2) {};
\draw[] (0,4) -- (0,5); 
\node[fill=black,inner sep =0] (A2) at (0,3) {};
\node[fill=black,inner sep =0] (B2) at (0,4) {};
\draw[] (0,6) -- (0,7); 
\node[fill=black,inner sep =0] (A3) at (0,5) {};
\node[fill=black,inner sep =0] (B3) at (0,6) {};
\draw[] (0,8) -- (0,9); 
\node[fill=black,inner sep =0] (A4) at (0,7) {};
\node[fill=black,inner sep =0] (B4) at (0,8) {};

\node[fill=black] (S1) at (-7.5,0) {};
\node[fill=black] (S2) at (-6.000000000000001,0) {};
\node[fill=black] (S3) at (-4.5,0) {};
\node[fill=black] (S4) at (-3.0000000000000004,0) {};
\node[fill=black] (S5) at (-1.4999999999999996,0) {};

\node[fill=black] (N1) at (-7.5,9) {};
\node[fill=black] (N2) at (-6.000000000000001,9) {};
\node[fill=black] (N3) at (-4.5,9) {};
\node[fill=black] (N4) at (-3.0000000000000004,9) {};
\node[fill=black] (N5) at (-1.4999999999999996,9) {};

\pgfonlayer{2}

\draw[white,line width=#1*10mm] (0,1.5)   arc (-90:90:2cm);
\draw[] (0,1) arc (-90:90:2.5cm);
\draw[] (0,2) arc (-90:90:1.5cm);

\node[fill=black] (A11) at (0,1.3333333333333333) {};
\node[fill=black] (A31) at (0,5.333333333333333) {};
\node[fill=black] (A12) at (0,1.6666666666666665) {};
\node[fill=black] (A32) at (0,5.666666666666667) {};

\endpgfonlayer

\pgfonlayer{1}

\draw[white,line width=#1*10mm] (0,3.5)   arc (-90:90:2cm);
\draw[] (0,3) arc (-90:90:2.5cm);
\draw[] (0,4) arc (-90:90:1.5cm);

\node[fill=black] (A21) at (0,3.3333333333333335) {};
\node[fill=black] (A41) at (0,7.333333333333333) {};
\node[fill=black] (A22) at (0,3.6666666666666665) {};
\node[fill=black] (A42) at (0,7.666666666666667) {};

\endpgfonlayer

\draw[blue] (S1).. controls (-7.5,4.5) and (-4.5,4.5).. (N3);

\draw[red] (N1).. controls (-7.5,8.25) and (-6.000000000000001,8.25) .. (N2);

\draw[orange] (S5).. controls (-1.4999999999999996,0.8888888888888888) and (-0.9999999999999997,1.3333333333333333) .. (A11);
\draw[orange] (S2).. controls (-6.000000000000001,3.777777777777778) and (-4.000000000000001,5.666666666666667) .. (A32);

\pgfonlayer{2}
\draw[orange] (A11) arc (-90:90:2.1666666666666665);
\endpgfonlayer

\draw[green] (S4).. controls (-3.0000000000000004,1.111111111111111) and (-2.0000000000000004,1.6666666666666665) .. (A12);
\draw[green] (A31).. controls (-0.8333333333333335,5.333333333333333) and (-0.8333333333333335,3.6666666666666665) .. (A22);
\draw[green] (N4).. controls (-3.0000000000000004,7.888888888888888) and (-2.0000000000000004,7.333333333333333) .. (A41);

\pgfonlayer{2}
\draw[green] (A12) arc (-90:90:1.8333333333333333);
\endpgfonlayer

\pgfonlayer{1}
\draw[green] (A22) arc (-90:90:1.8333333333333333);
\endpgfonlayer

\draw[purple] (S3).. controls (-4.5,2.2222222222222223) and (-3.0,3.3333333333333335) .. (A21);
\draw[purple] (N5).. controls (-1.4999999999999996,8.11111111111111) and (-0.9999999999999997,7.666666666666667) .. (A42);

\pgfonlayer{1}
\draw[purple] (A21) arc (-90:90:2.1666666666666665);
\endpgfonlayer
\end{tikzpicture}}

\def\hsisoa#1{\begin{tikzpicture}[every node/.style={shape=circle},inner sep =0, scale=#1,anchor=base,baseline=(current bounding box.center)]

\pgfdeclarelayer{1}
\pgfdeclarelayer{2}
\pgfsetlayers{main,1,2}
\draw (0,0) -- (0,0.5);
 \draw (0,1.5) -- (0,2); 
 \draw (0,3) -- (0,3.5); 
 \draw (0,4.5) -- (0,5.5); 
 \draw (0,6.5) -- (0,7); 
 \draw (0,8) -- (0,8.5); 
 \draw (0,9.5) -- (0,10);
 
\draw (-1,0) -- (-1,1.3-0.5);
\draw (-1,1.3-0.5) -- (-2.5,1.3-0.5);
\draw (-2.5,1.7-0.5) -- (-1,1.7-0.5);
\draw (-1,1.7-0.5) -- (-1,10);
\node[] at (0,2.3) {$\vdots $};
\node[] at (0,7.3) {$\vdots $};
\node at (-1.75,5) {$E$};
\node at (-1.75,0.1) {$E$};

\draw[dashed] (-2.5,0) -- (3,0);
\draw[dashed] (-2.5,10) -- (3,10);
\draw[dashed] (-2.5,0) -- (-2.5,10);
\draw[dashed] (3,0) -- (3,10);

\pgfonlayer{2}

\draw[white,line width=#1*10mm] (0,1)   arc (-90:90:1.5cm);
\draw[] (0,0.5) arc (-90:90:2cm);
\draw[] (0,1.5) arc (-90:90:1cm);

\node[fill=black] (A11) at (0,1.2-0.5) {};
\node[fill=black] (A31) at (0,5.2-1.5) {};
\node[fill=black] (A12) at (0,1.4-0.5) {};
\node[fill=black] (A32) at (0,5.4-1.5) {};
\node[fill=black] (A13) at (0,1.6-0.5) {};
\node[fill=black] (A33) at (0,5.6-1.5) {};
\node[fill=black] (A14) at (0,1.8-0.5) {};
\node[fill=black] (A34) at (0,5.8-1.5) {};
\endpgfonlayer
\pgfonlayer{1}

\draw[white,line width=#1*10mm] (0,6)   arc (-90:90:1.5cm);
\draw[] (0,5.5) arc (-90:90:2cm);
\draw[] (0,6.5) arc (-90:90:1cm);

\node[fill=black] (A41) at (0,7.5-1.5) {};
\node[fill=black] (A61) at (0,11.5-2.5) {};
\endpgfonlayer

\pgfonlayer{2}
\draw[blue,double distance=1pt] (A11) arc (-90:90:2.3-0.5);
\endpgfonlayer
\draw[blue,double distance=1pt] (A11) -- +(-1,0);
\draw[blue,double distance=1pt] (A34) -- +(-1,0);

\pgfonlayer{2}
\draw[blue,double distance=1pt] (A14) arc (-90:90:1.7-0.5);
\endpgfonlayer
\draw[blue,double distance=1pt] (A14) -- +(-1,0);
\draw[blue,double distance=1pt] (A31) -- +(-1,0);

\draw[red] (A12) -- +(-2.5,0);
\draw[red] (A13) -- +(-2.5,0);
\draw[red] (A33).. controls ([xshift=-0.2cm]A33) and ([xshift=-0.2cm]A32) .. (A32);

\pgfonlayer{2}
\draw[red] (A12) arc (-90:90:2.1-0.5);
\endpgfonlayer

\pgfonlayer{2}
\draw[red] (A32) arc (90:-90:1.9-0.5);
\endpgfonlayer
\draw[green,double distance=2pt] (A41) -- +(-1,0);
\draw[green,double distance=2pt] (A61) -- +(-1,0);

\pgfonlayer{1}
\draw[green,double distance=2pt] (A41) arc (-90:90:1.5);
\endpgfonlayer

\pgfonlayer{2}
\node[anchor=south west] at ([xshift=0.5cm,yshift=0.5cm] A32) {\footnotesize \contour{white}{$(j,a)$}};
\draw[black,<-] (A32) to ([xshift=0.5cm,yshift=0.5cm] A32);
\endpgfonlayer

\draw[blue,double distance=2pt] (0,2.5) -- +(-1,0);
\draw[blue,double distance=2pt] (0,7.5) -- +(-1,0);

\end{tikzpicture}}

\def\hsisoaa#1{\begin{tikzpicture}[every node/.style={shape=circle},inner sep =0, scale=#1,anchor=base,baseline=(current bounding box.center)]

\pgfdeclarelayer{1}
\pgfdeclarelayer{2}
\pgfsetlayers{main,1,2}
\draw (0,0) -- (0,0.5);
 \draw (0,1.5) -- (0,2); 
 \draw (0,3) -- (0,3.5); 
 \draw (0,4.5) -- (0,5.5); 
 \draw (0,6.5) -- (0,7); 
 \draw (0,8) -- (0,8.5); 
 \draw (0,9.5) -- (0,10);
 
\draw (-1,0) -- (-1,1.3-0.5);
\draw (-1,1.3-0.5) -- (-2.5,1.3-0.5);
\draw (-2.5,1.7-0.5) -- (-1,1.7-0.5);
\draw (-1,1.7-0.5) -- (-1,10);
\node[] at (0,2.3) {$\vdots $};
\node[] at (0,7.3) {$\vdots $};
\node at (-1.75,5) {$E$};
\node at (-1.75,0.1) {$E$};

\draw[dashed] (-2.5,0) -- (3,0);
\draw[dashed] (-2.5,10) -- (3,10);
\draw[dashed] (-2.5,0) -- (-2.5,10);
\draw[dashed] (3,0) -- (3,10);

\pgfonlayer{2}

\draw[white,line width=#1*10mm] (0,1)   arc (-90:90:1.5cm);
\draw[] (0,0.5) arc (-90:90:2cm);
\draw[] (0,1.5) arc (-90:90:1cm);

\node[fill=black] (A11) at (0,1.2-0.5) {};
\node[fill=black] (A31) at (0,5.2-1.5) {};
\node[] (A12) at (0,1.4-0.5) {};
\node[] (A32) at (0,5.4-1.5) {};
\node[] (A13) at (0,1.6-0.5) {};
\node[] (A33) at (0,5.6-1.5) {};
\node[fill=black] (A14) at (0,1.8-0.5) {};
\node[fill=black] (A34) at (0,5.8-1.5) {};
\endpgfonlayer
\pgfonlayer{1}

\draw[white,line width=#1*10mm] (0,6)   arc (-90:90:1.5cm);
\draw[] (0,5.5) arc (-90:90:2cm);
\draw[] (0,6.5) arc (-90:90:1cm);

\node[fill=black] (A41) at (0,7.5-1.5) {};
\node[fill=black] (A61) at (0,11.5-2.5) {};
\endpgfonlayer

\pgfonlayer{2}
\draw[blue,double distance=1pt] (A11) arc (-90:90:2.3-0.5);
\endpgfonlayer
\draw[blue,double distance=1pt] (A11) -- +(-1,0);
\draw[blue,double distance=1pt] (A34) -- +(-1,0);

\pgfonlayer{2}
\draw[blue,double distance=1pt] (A14) arc (-90:90:1.7-0.5);
\endpgfonlayer
\draw[blue,double distance=1pt] (A14) -- +(-1,0);
\draw[blue,double distance=1pt] (A31) -- +(-1,0);

\draw[green,double distance=2pt] (A41) -- +(-1,0);
\draw[green,double distance=2pt] (A61) -- +(-1,0);

\pgfonlayer{1}
\draw[green,double distance=2pt] (A41) arc (-90:90:2.0-0.5);
\endpgfonlayer
\draw[red] (-2.5,1.4-0.5).. controls (-2.25,1.4-0.5) and (-2.25,1.6-0.5) .. (-2.5,1.6-0.5);

\draw[blue,double distance=2pt] (0,2.5) -- +(-1,0);
\draw[blue,double distance=2pt] (0,7.5) -- +(-1,0);

\end{tikzpicture}}

\def\hsisoab#1{\begin{tikzpicture}[every node/.style={shape=circle},inner sep =0, scale=#1,anchor=base,baseline=(current bounding box.center)]

\pgfdeclarelayer{1}
\pgfdeclarelayer{2}
\pgfsetlayers{main,1,2}
\draw (0,0) -- (0,0.5);
 \draw (0,1.5) -- (0,2); 
 \draw (0,3) -- (0,3.5); 
 \draw (0,4.5) -- (0,5); 
 \draw (0,6) -- (0,6.5); 
 \draw (0,7.5) -- (0,8.5); 
 \draw (0,9.5) -- (0,10);
\draw (-1,0) -- (-1,1.3-0.5);
\draw (-1,1.3-0.5) -- (-2.5,1.3-0.5);
\draw (-2.5,1.7-0.5) -- (-1,1.7-0.5);
\draw (-1,1.7-0.5) -- (-1,10);
\node[] at (0,2.3) {$\vdots $};
\node[] at (0,5.3) {$\vdots $};
\node at (-1.75,5) {$E$};
\node at (-1.75,0.1) {$E$};

\draw[dashed] (-2.5,0) -- (5,0);
\draw[dashed] (-2.5,10) -- (5,10);
\draw[dashed] (-2.5,0) -- (-2.5,10);
\draw[dashed] (5,0) -- (5,10);

\pgfonlayer{2}

\draw[white,line width=#1*10mm] (0,1)   arc (-90:90:4cm);
\draw[] (0,0.5) arc (-90:90:4.5cm);
\draw[] (0,2-0.5) arc (-90:90:3.5cm);

\node[fill=black] (A11) at (0,1.2-0.5) {};
\node[fill=black] (A61) at (0,11.2-2.5) {};
\node[fill=black] (A12) at (0,1.4-0.5) {};
\node[fill=black] (A62) at (0,11.4-2.5) {};
\node[fill=black] (A13) at (0,1.6-0.5) {};
\node[fill=black] (A63) at (0,11.6-2.5) {};
\node[fill=black] (A14) at (0,1.8-0.5) {};
\node[fill=black] (A64) at (0,11.8-2.5) {};
\endpgfonlayer
\pgfonlayer{1}

\draw[white,line width=#1*10mm] (0,4)   arc (-90:90:1.5cm);
\draw[] (0,5-1.5) arc (-90:90:2cm);
\draw[] (0,6-1.5) arc (-90:90:1cm);

\node[fill=black] (A31) at (0,5.166666666666667-1.5) {};
\node[fill=black] (A51) at (0,9.166666666666666-2.5) {};
\node[fill=black] (A32) at (0,5.333333333333333-1.5) {};
\node[fill=black] (A52) at (0,9.333333333333334-2.5) {};
\node[fill=black] (A33) at (0,5.5-1.5) {};
\node[fill=black] (A53) at (0,9.5-2.5) {};
\node[fill=black] (A34) at (0,5.666666666666667-1.5) {};
\node[fill=black] (A54) at (0,9.666666666666666-2.5) {};
\node[fill=black] (A35) at (0,5.833333333333333-1.5) {};
\node[fill=black] (A55) at (0,9.833333333333334-2.5) {};
\endpgfonlayer
\draw[red] (A12) -- +(-2.5,0);
\draw[red] (A13) -- +(-2.5,0);
\draw[red] (A63).. controls ([xshift=-0.8cm]A63) and ([xshift=-0.8cm]A53) .. (A53);
\draw[red] (A33).. controls ([xshift=-0.2cm]A33) and ([xshift=-0.2cm]A32) .. (A32);
\draw[red] (A54).. controls ([xshift=-0.6cm]A54) and ([xshift=-0.6cm]A62) .. (A62);

\pgfonlayer{2}
\draw[red] (A12) arc (-90:90:5.1-1);
\endpgfonlayer

\pgfonlayer{1}
\draw[red] (A53) arc (90:-90:2.0-0.5);
\endpgfonlayer

\pgfonlayer{1}
\draw[red] (A32) arc (-90:90:2.1666666666666665-0.5);
\endpgfonlayer

\pgfonlayer{2}
\draw[red] (A62) arc (90:-90:4.9-1);
\endpgfonlayer
\draw[blue,double distance=1pt] (A14) -- +(-1,0);
\draw[blue,double distance=1pt] (A31) -- +(-1,0);
\draw[blue,double distance=1pt] (A61).. controls ([xshift=-0.4cm]A61) and ([xshift=-0.4cm]A55) .. (A55);

\pgfonlayer{2}
\draw[blue,double distance=1pt] (A14) arc (-90:90:4.7-1);
\endpgfonlayer

\pgfonlayer{1}
\draw[blue,double distance=1pt] (A55) arc (90:-90:2.3333333333333335-0.5);
\endpgfonlayer
\draw[blue,double distance=1pt] (A11) -- +(-1,0);
\draw[blue,double distance=1pt] (A34) -- +(-1,0);
\draw[blue,double distance=1pt] (A64).. controls ([xshift=-1cm]A64) and ([xshift=-1cm]A52) .. (A52);

\pgfonlayer{2}
\draw[blue,double distance=1pt] (A11) arc (-90:90:5.3-1);
\endpgfonlayer

\pgfonlayer{1}
\draw[blue,double distance=1pt] (A52) arc (90:-90:1.8333333333333333-0.5);
\endpgfonlayer
\draw[green,double distance=1pt] (A35) -- +(-1,0);
\draw[green,double distance=1pt] (A51) -- +(-1,0);

\pgfonlayer{1}
\draw[green,double distance=1pt] (A35) arc (-90:90:1.6666666666666667-0.5);
\endpgfonlayer

\pgfonlayer{2}
\draw[black,<-] (A32) to ([xshift=0.5cm,yshift=-0.5cm] A32);
\node[anchor=west] at ([xshift=0.5cm,yshift=-0.5cm] A32) {\footnotesize \contour{white}{$(j_1,a_1)$}};
\draw[black,<-] (A62) to ([xshift=0.5cm,yshift=-0.5cm] A62);
\node[anchor=north west] at ([xshift=0.5cm,yshift=-0.5cm] A62) {\footnotesize \contour{white}{$(j_2,a_2)$}};
\endpgfonlayer

\draw[blue,double distance=2pt] (0,2.5) -- +(-1,0);
\draw[blue,double distance=2pt] (0,5.5) -- +(-1,0);

\end{tikzpicture}}

\def\hsisoac#1{\begin{tikzpicture}[every node/.style={shape=circle},inner sep =0, scale=#1,anchor=base,baseline=(current bounding box.center)]

\pgfdeclarelayer{1}
\pgfdeclarelayer{2}
\pgfsetlayers{main,1,2}
\draw (0,0) -- (0,0.5);
 \draw (0,1.5) -- (0,2); 
 \draw (0,3) -- (0,3.5); 
 \draw (0,4.5) -- (0,5); 
 \draw (0,6) -- (0,6.5); 
 \draw (0,7.5) -- (0,8.5); 
 \draw (0,9.5) -- (0,10);
\draw (-1,0) -- (-1,1.3-0.5);
\draw (-1,1.3-0.5) -- (-2.5,1.3-0.5);
\draw (-2.5,1.7-0.5) -- (-1,1.7-0.5);
\draw (-1,1.7-0.5) -- (-1,10);
\node[] at (0,2.3) {$\vdots $};
\node[] at (0,5.3) {$\vdots $};
\node at (-1.75,5) {$E$};
\node at (-1.75,0.1) {$E$};

\draw[dashed] (-2.5,0) -- (5,0);
\draw[dashed] (-2.5,10) -- (5,10);
\draw[dashed] (-2.5,0) -- (-2.5,10);
\draw[dashed] (5,0) -- (5,10);

\pgfonlayer{2}

\draw[white,line width=#1*10mm] (0,1)   arc (-90:90:4cm);
\draw[] (0,0.5) arc (-90:90:4.5cm);
\draw[] (0,2-0.5) arc (-90:90:3.5cm);

\node[fill=black] (A11) at (0,1.2-0.5) {};
\node[fill=black] (A61) at (0,11.2-2.5) {};
\node[] (A12) at (0,1.4-0.5) {};
\node[] (A62) at (0,11.4-2.5) {};
\node[] (A13) at (0,1.6-0.5) {};
\node[] (A63) at (0,11.6-2.5) {};
\node[fill=black] (A14) at (0,1.8-0.5) {};
\node[fill=black] (A64) at (0,11.8-2.5) {};
\endpgfonlayer
\pgfonlayer{1}

\draw[white,line width=#1*10mm] (0,4)   arc (-90:90:1.5cm);
\draw[] (0,5-1.5) arc (-90:90:2cm);
\draw[] (0,6-1.5) arc (-90:90:1cm);

\node[fill=black] (A31) at (0,5.166666666666667-1.5) {};
\node[fill=black] (A51) at (0,9.166666666666666-2.5) {};
\node[] (A32) at (0,5.333333333333333-1.5) {};
\node[fill=black] (A52) at (0,9.333333333333334-2.5) {};
\node[] (A33) at (0,5.5-1.5) {};
\node[] (A53) at (0,9.5-2.5) {};
\node[fill=black] (A34) at (0,5.666666666666667-1.5) {};
\node[] (A54) at (0,9.666666666666666-2.5) {};
\node[fill=black] (A35) at (0,5.833333333333333-1.5) {};
\node[fill=black] (A55) at (0,9.833333333333334-2.5) {};
\endpgfonlayer

\draw[blue,double distance=1pt] (A14) -- +(-1,0);
\draw[blue,double distance=1pt] (A31) -- +(-1,0);
\draw[blue,double distance=1pt] (A61).. controls ([xshift=-0.4cm]A61) and ([xshift=-0.4cm]A55) .. (A55);

\pgfonlayer{2}
\draw[blue,double distance=1pt] (A14) arc (-90:90:4.7-1);
\endpgfonlayer

\pgfonlayer{1}
\draw[blue,double distance=1pt] (A55) arc (90:-90:2.3333333333333335-0.5);
\endpgfonlayer
\draw[blue,double distance=1pt] (A11) -- +(-1,0);
\draw[blue,double distance=1pt] (A34) -- +(-1,0);
\draw[blue,double distance=1pt] (A64).. controls ([xshift=-1cm]A64) and ([xshift=-1cm]A52) .. (A52);

\pgfonlayer{2}
\draw[blue,double distance=1pt] (A11) arc (-90:90:5.3-1);
\endpgfonlayer

\pgfonlayer{1}
\draw[blue,double distance=1pt] (A52) arc (90:-90:1.8333333333333333-0.5);
\endpgfonlayer
\draw[green,double distance=1pt] (A35) -- +(-1,0);
\draw[green,double distance=1pt] (A51) -- +(-1,0);

\pgfonlayer{1}
\draw[green,double distance=1pt] (A35) arc (-90:90:1.6666666666666667-0.5);
\endpgfonlayer

\draw[red] (-2.5,1.4-0.5).. controls (-2.25,1.4-0.5) and (-2.25,1.6-0.5) .. (-2.5,1.6-0.5);

\draw[blue,double distance=2pt] (0,2.5) -- +(-1,0);
\draw[blue,double distance=2pt] (0,5.5) -- +(-1,0);

\end{tikzpicture}}

\def\hseqa#1{\begin{tikzpicture}[every node/.style={shape=circle},inner sep =0, scale=#1,anchor=base,baseline=(current bounding box.center)]
\pgfdeclarelayer{1}
\pgfdeclarelayer{2}
\pgfsetlayers{main,1,2}
\draw (0,0) -- (0,0.5);
 \draw (0,1.5) -- (0,2); 
 \draw (0,3) -- (0,3.5); 
 \draw (0,4.5) -- (0,5.5); 
 \draw (0,6.5) -- (0,7); 
 \draw (0,8) -- (0,8.5); 
 \draw (0,9.5) -- (0,10);
\draw (-1,0) -- (-1,10);
\node[] at (0,2.3) {$\vdots $};
 \node[] at (0,7.3) {$\vdots $};
\pgfonlayer{2}

\draw[dashed] (-2.5,0) -- (2.5,0);
\draw[dashed] (-2.5,10) -- (2.5,10);
\draw[dashed] (-2.5,0) -- (-2.5,10);
\draw[dashed] (2.5,0) -- (2.5,10);

\draw[white,line width=#1*10mm] (0,1)   arc (-90:90:1.5cm);
\draw[] (0,0.5) arc (-90:90:2cm);
\draw[] (0,1.5) arc (-90:90:1cm);

\node[fill=black] (A11) at (0,1.5-0.5) {};
\node[fill=black] (A31) at (0,5.5-1.5) {};
\endpgfonlayer
\pgfonlayer{1}

\draw[white,line width=#1*10mm] (0,6)   arc (-90:90:1.5cm);
\draw[] (0,5.5) arc (-90:90:2cm);
\draw[] (0,6.5) arc (-90:90:1cm);

\node[fill=black] (A41) at (0,7.5-1.5) {};
\node[fill=black] (A61) at (0,11.5-2.5) {};

\pgfonlayer{1}

\draw[blue, double distance=2pt] (A41) arc (-90:90:1.5cm);

\endpgfonlayer
\pgfonlayer{2}

\draw[red, double distance=2pt] (A11) arc (-90:90:1.5cm);

\draw[red, double distance=2pt] (A11) -- +(-1,0);
\draw[red, double distance=2pt] (A31) -- +(-1,0);

\draw[blue, double distance=2pt] (A41) -- +(-1,0);
\draw[blue, double distance=2pt] (A61) -- +(-1,0);

\node[anchor=west] at (0,3.2) {\contour{white}{$i$}};
\node[anchor=west] at (0,8.2) {\contour{white}{$i'$}};
\node[anchor=west] at (0,5.2) {\contour{white}{$i+\e$}};
\node[anchor=west] at (0,6.25-1.5) {\contour{white}{$\uparrow$}};\endpgfonlayer

\draw[blue,double distance=2pt] (0,2.5) -- +(-1,0);
\draw[blue,double distance=2pt] (0,7.5) -- +(-1,0);

\endpgfonlayer\end{tikzpicture}}

\def\hseqA#1{\begin{tikzpicture}[every node/.style={shape=circle},inner sep =0, scale=#1,anchor=base,baseline=(current bounding box.center)]
\pgfdeclarelayer{1}
\pgfdeclarelayer{2}
\pgfsetlayers{main,1,2}
\draw (0,0) -- (0,0.5);
 \draw (0,1.5) -- (0,2); 
 \draw (0,3) -- (0,3.5); 
 \draw (0,4.5) -- (0,5.5); 
 \draw (0,6.5) -- (0,7); 
 \draw (0,8) -- (0,8.5); 
 \draw (0,9.5) -- (0,10);
\draw (-1,0) -- (-1,10);
\node[] at (0,2.3) {$\vdots $};
 \node[] at (0,7.3) {$\vdots $};
\pgfonlayer{2}

\draw[dashed] (-2.5,0) -- (2.5,0);
\draw[dashed] (-2.5,10) -- (2.5,10);
\draw[dashed] (-2.5,0) -- (-2.5,10);
\draw[dashed] (2.5,0) -- (2.5,10);

\draw[white,line width=#1*10mm] (0,1)   arc (-90:90:1.5cm);
\draw[] (0,0.5) arc (-90:90:2cm);
\draw[] (0,1.5) arc (-90:90:1cm);

\node[fill=black] (A11) at (0,1.5-0.5) {};
\node[fill=black] (A31) at (0,5.5-1.5) {};
\endpgfonlayer
\pgfonlayer{1}

\draw[white,line width=#1*10mm] (0,6)   arc (-90:90:1.5cm);
\draw[] (0,5.5) arc (-90:90:2cm);
\draw[] (0,6.5) arc (-90:90:1cm);

\node[fill=black] (A41) at (0,7.5-1.5) {};
\node[fill=black] (A61) at (0,11.5-2.5) {};

\pgfonlayer{1}

\draw[blue, double distance=2pt] (A41) arc (-90:90:1.5cm);

\endpgfonlayer
\pgfonlayer{2}

\draw[red, double distance=2pt] (A11) arc (-90:90:1.5cm);

\draw[red, double distance=2pt] (A11) -- +(-1,0);
\draw[red, double distance=2pt] (A31) -- +(-1,0);

\draw[blue, double distance=2pt] (A41) -- +(-1,0);
\draw[blue, double distance=2pt] (A61) -- +(-1,0);

\node[anchor=west] at (0,3.2) {\contour{white}{$i$}};
\node[anchor=west] at (0,8.2) {\contour{white}{$i'$}};
\node[anchor=west] at (0,5.2) {\contour{white}{$i+\e$}};\endpgfonlayer

\draw[blue,double distance=2pt] (0,2.5) -- +(-1,0);
\draw[blue,double distance=2pt] (0,7.5) -- +(-1,0);

\endpgfonlayer\end{tikzpicture}}

\def\hseqaa#1{\begin{tikzpicture}[every node/.style={shape=circle},inner sep =0, scale=#1,anchor=base,baseline=(current bounding box.center)]

\pgfdeclarelayer{1}
\pgfdeclarelayer{2}
\pgfsetlayers{main,1,2}
\draw (0,0) -- (0,0.5);
 \draw (0,1.5) -- (0,2); 
 \draw (0,3) -- (0,3.5); 
 \draw (0,4.5) -- (0,5); 
 \draw (0,6) -- (0,6.5); 
 \draw (0,7.5) -- (0,8.5); 
 \draw (0,9.5) -- (0,10);
\draw (-1,0) -- (-1,10);
\node[] at (0,2.3) {$\vdots $};
\node[] at (0,5.3) {$\vdots $};
\node at (-2,5) {$E$};

\draw[dashed] (-2.5,0) -- (5,0);
\draw[dashed] (-2.5,10) -- (5,10);
\draw[dashed] (-2.5,0) -- (-2.5,10);
\draw[dashed] (5,0) -- (5,10);

\pgfonlayer{2}

\draw[white,line width=#1*10mm] (0,1.5)   arc (-90:90:4cm);
\draw[] (0,1-0.5) arc (-90:90:4.5cm);
\draw[] (0,2-0.5) arc (-90:90:3.5cm);

\node[fill=black] (A11) at (0,1.5-0.5) {};
\node[fill=black] (A61) at (0,11.5-2.5) {};
\endpgfonlayer
\pgfonlayer{1}

\draw[white,line width=#1*10mm] (0,4)   arc (-90:90:1.5cm);
\draw[] (0,5-1.5) arc (-90:90:2cm);
\draw[] (0,6-1.5) arc (-90:90:1cm);

\node[fill=black] (A31) at (0,5+1/3-1.5) {};
\node[fill=black] (A51) at (0,9+1/3-2.5) {};
\node[] (A32) at (0,5+2/3-1.5) {};
\node[] (A52) at (0,9+2/3-2.5) {};
\endpgfonlayer

\pgfonlayer{1}
\draw[blue,double distance=2pt] (A32) arc (-90:90:1.5+1/3-0.5);
\endpgfonlayer
\draw[blue,double distance=2pt] (A32) -- +(-1,0);
\draw[blue,double distance=2pt] (A51) -- +(-1,0);

\draw[red,double distance=2pt] (A11) -- +(-1,0);
\draw[red,double distance=2pt] (A31) -- +(-1,0);
\draw[red,double distance=2pt] (A61).. controls ([xshift=-1cm]A61) and ([xshift=-1cm]A52) .. (A52);

\pgfonlayer{2}
\draw[red,double distance=2pt] (A11) arc (-90:90:4);
\endpgfonlayer

\pgfonlayer{1}
\draw[red,double distance=2pt] (A52) arc (90:-90:1.5+2/3-0.5);
\endpgfonlayer

\pgfonlayer{2}\node[anchor=west] at (0,9.7) {\footnotesize\contour{white}{$\sigma(i)$}};
\node[anchor=west] at (0,6.2) {\footnotesize\contour{white}{$\sigma(i')$}};
\node[anchor=west] at (0,3.2) {\footnotesize\contour{white}{$\sigma(i+\e)$}};
\node[anchor=west] at (0,8.2) {\contour{white}{$\downarrow$}};
\endpgfonlayer

\draw[blue,double distance=2pt] (0,2.5) -- +(-1,0);
\draw[blue,double distance=2pt] (0,5.5) -- +(-1,0);

\end{tikzpicture}}

\def\hseqab#1{\begin{tikzpicture}[every node/.style={shape=circle},inner sep =0, scale=#1,anchor=base,baseline=(current bounding box.center)]

\pgfdeclarelayer{1}
\pgfdeclarelayer{2}
\pgfsetlayers{main,1,2}
\draw (0,0) -- (0,0.5);
 \draw (0,1.5) -- (0,2); 
 \draw (0,3) -- (0,3.5); 
 \draw (0,4.5) -- (0,5.5); 
 \draw (0,6.5) -- (0,7); 
 \draw (0,8) -- (0,8.5); 
 \draw (0,9.5) -- (0,10);
\draw (-1,0) -- (-1,10);
\node[] at (0,2.3) {$\vdots $};
\node[] at (0,7.3) {$\vdots $};
\node at (-2,5) {$E$};

\draw[dashed] (-2.5,0) -- (2.5,0);
\draw[dashed] (-2.5,10) -- (2.5,10);
\draw[dashed] (-2.5,0) -- (-2.5,10);
\draw[dashed] (2.5,0) -- (2.5,10);

\pgfonlayer{2}

\draw[white,line width=#1*10mm] (0,1)   arc (-90:90:1.5cm);
\draw[] (0,0.5) arc (-90:90:2cm);
\draw[] (0,1.5) arc (-90:90:1cm);

\node[fill=black] (A11) at (0,1) {};
\node[fill=black] (A31) at (0,4) {};
\endpgfonlayer
\pgfonlayer{1}

\draw[white,line width=#1*10mm] (0,7.5-1.5)   arc (-90:90:1.5cm);
\draw[] (0,7-1.5) arc (-90:90:2cm);
\draw[] (0,8-1.5) arc (-90:90:1cm);

\node[fill=black] (A41) at (0,7.25-1.5) {};
\node[fill=black] (A61) at (0,11.25-2.5) {};
\node[fill=black] (A42) at (0,7.5-1.5) {};
\node[fill=black] (A62) at (0,11.5-2.5) {};
\node[fill=black] (A43) at (0,7.75-1.5) {};
\node[fill=black] (A63) at (0,11.75-2.5) {};
\endpgfonlayer

\pgfonlayer{1}
\draw[blue,double distance=1pt] (A43) arc (-90:90:1.25);
\endpgfonlayer
\draw[blue,double distance=1pt] (A43) -- +(-1,0);
\draw[blue,double distance=1pt] (A61) -- +(-1,0);

\draw[red,double distance=1pt] (A11) -- +(-1,0);
\draw[red,double distance=1pt] (A31).. controls ([xshift=-1cm]A31) and ([xshift=-1cm]A41) .. (A41);
\draw[red,double distance=1pt] (A63).. controls ([xshift=-0.2cm]A63) and ([xshift=-0.2cm]A62) .. (A62);

\pgfonlayer{2}
\draw[red,double distance=1pt] (A11) arc (-90:90:1.5);
\endpgfonlayer

\pgfonlayer{1}
\draw[red,double distance=1pt] (A41) arc (-90:90:1.75);
\draw[red,double distance=1pt] (A42) arc (-90:90:1.5);
\endpgfonlayer

\draw[red,double distance=1pt] (A11) -- +(-1,0);
\draw[red,double distance=1pt] (A42) -- +(-1,0);

\pgfonlayer{2}

\node[anchor=west] at (0,3.2) {\footnotesize\contour{white}{$i$}};
\node[anchor=west] at (0,8.2) {\footnotesize\contour{white}{$i'$}};
\node[anchor=west] at (0,5.2) {\footnotesize\contour{white}{$i+\e$}};
\endpgfonlayer

\draw[blue,double distance=2pt] (0,2.5) -- +(-1,0);
\draw[blue,double distance=2pt] (0,7.5) -- +(-1,0);

\end{tikzpicture}}

\def\ILexa#1{\begin{tikzpicture}[every node/.style={shape=circle},inner sep=0,scale=#1,anchor=base,baseline=(current bounding box.center)]\pgfdeclarelayer{1}
\pgfdeclarelayer{2}

\pgfsetlayers{main,
1,
2} 

\draw[] (-9,0)--(0,0);
\draw[] (-9,0)--(-9,9);
\draw[] (-9,9)--(0,9);
\draw[] (0,0) --(0,1);
\draw[] (0,2) -- (0,3); 
\node[fill=black,inner sep =0] (A1) at (0,1) {};
\node[fill=black,inner sep =0] (B1) at (0,2) {};
\draw[] (0,4) -- (0,5); 
\node[fill=black,inner sep =0] (A2) at (0,3) {};
\node[fill=black,inner sep =0] (B2) at (0,4) {};
\draw[] (0,6) -- (0,7); 
\node[fill=black,inner sep =0] (A3) at (0,5) {};
\node[fill=black,inner sep =0] (B3) at (0,6) {};
\draw[] (0,8) -- (0,9); 
\node[fill=black,inner sep =0] (A4) at (0,7) {};
\node[fill=black,inner sep =0] (B4) at (0,8) {};

\node[fill=black] (S1) at (-7.2,0) {};
\node[fill=black] (S2) at (-5.3999999999999995,0) {};
\node[fill=black] (S3) at (-3.6,0) {};
\node[fill=black] (S4) at (-1.7999999999999996,0) {};

\node[fill=black] (N1) at (-6.000000000000001,9) {};
\node[fill=black] (N2) at (-3.0000000000000004,9) {};

\pgfonlayer{2}

\draw[white,line width=#1*10mm] (0,1.5)   arc (-90:90:2cm);
\draw[] (0,1) arc (-90:90:2.5cm);
\draw[] (0,2) arc (-90:90:1.5cm);

\node[fill=black] (A31) at (0,5.333333333333333) {};
\node[fill=black] (A11) at (0,1.3333333333333333) {};
\node[fill=black] (A32) at (0,5.666666666666667) {};
\node[fill=black] (A12) at (0,1.6666666666666665) {};

\endpgfonlayer

\pgfonlayer{1}

\draw[white,line width=#1*5mm] (0,3.5)   arc (-90:90:2cm);
\draw[white,line width=#1*5mm] (0,3.75)   arc (-90:90:2cm);
\draw[white,line width=#1*5mm] (0,3.25)   arc (-90:90:2cm);
\draw[] (0,3) arc (-90:90:2cm);
\draw[] (0,4) arc (-90:90:2cm);
\draw[] (2,5)--(2,6);

\node[fill=black] (A21) at (0,3.5) {};
\node[fill=black] (A41) at (0,7.5) {};

\endpgfonlayer

\draw[red] (S3).. controls (-3.6,3.5555555555555554) and (-2.4,5.333333333333333) .. (A31);
\draw[red] (A12).. controls (-0.9166666666666667,1.6666666666666665) and (-0.9166666666666667,3.5) .. (A21);
\draw[red] (N2).. controls (-3.0000000000000004,8.0) and (-2.0000000000000004,7.5) .. (A41);

\pgfonlayer{2}
\draw[red] (A31) arc (90:-90:1.8333333333333333);
\endpgfonlayer

\pgfonlayer{1}
\draw[red] (A21) arc (-90:90:2.0);
\endpgfonlayer

\draw[blue] (S2).. controls (-5.3999999999999995,3.777777777777778) and (-3.5999999999999996,5.666666666666667) .. (A32);
\draw[blue] (S4).. controls (-1.7999999999999996,0.8888888888888888) and (-1.1999999999999997,1.3333333333333333) .. (A11);

\pgfonlayer{2}
\draw[blue] (A32) arc (90:-90:2.1666666666666665);
\endpgfonlayer

\draw[green] (N1).. controls (-6.000000000000001,4.5) and (-7.2,4.5).. (S1);  \end{tikzpicture}}

\def\IRexa#1{\begin{tikzpicture}[every node/.style={shape=circle},inner sep=0,scale=#1,anchor=base,baseline=(current bounding box.center)]

\pgfdeclarelayer{1}
\pgfdeclarelayer{2}

\pgfsetlayers{main,
1,
2} 

\draw[] (-9,0)--(0,0);
\draw[] (-9,0)--(-9,9);
\draw[] (-9,9)--(0,9);
\draw[] (0,0) --(0,1);
\draw[] (0,2) -- (0,3); 
\node[fill=black,inner sep =0] (A1) at (0,1) {};
\node[fill=black,inner sep =0] (B1) at (0,2) {};
\draw[] (0,4) -- (0,5); 
\node[fill=black,inner sep =0] (A2) at (0,3) {};
\node[fill=black,inner sep =0] (B2) at (0,4) {};
\draw[] (0,6) -- (0,7); 
\node[fill=black,inner sep =0] (A3) at (0,5) {};
\node[fill=black,inner sep =0] (B3) at (0,6) {};
\draw[] (0,8) -- (0,9); 
\node[fill=black,inner sep =0] (A4) at (0,7) {};
\node[fill=black,inner sep =0] (B4) at (0,8) {};

\node[fill=black] (S1) at (-7.2,0) {};
\node[fill=black] (S2) at (-5.3999999999999995,0) {};
\node[fill=black] (S3) at (-3.6,0) {};
\node[fill=black] (S4) at (-1.7999999999999996,0) {};

\node[fill=black] (N1) at (-6.000000000000001,9) {};
\node[fill=black] (N2) at (-3.0000000000000004,9) {};

\pgfonlayer{2}

\draw[white,line width=#1*10mm] (0,1.5)   arc (-90:90:2cm);
\draw[] (0,1) arc (-90:90:2.5cm);
\draw[] (0,2) arc (-90:90:1.5cm);

\node[fill=black] (A31) at (0,5.2) {};
\node[fill=black] (A11) at (0,1.2) {};
\node[fill=black] (A32) at (0,5.4) {};
\node[fill=black] (A12) at (0,1.4) {};
\node[fill=black] (A33) at (0,5.6) {};
\node[fill=black] (A13) at (0,1.6) {};
\node[fill=black] (A34) at (0,5.8) {};
\node[fill=black] (A14) at (0,1.8) {};

\endpgfonlayer

\pgfonlayer{1}

\draw[white,line width=#1*5mm] (0,3.5)   arc (-90:90:2cm);
\draw[white,line width=#1*5mm] (0,3.75)   arc (-90:90:2cm);
\draw[white,line width=#1*5mm] (0,3.25)   arc (-90:90:2cm);
\draw[] (0,3) arc (-90:90:2cm);
\draw[] (0,4) arc (-90:90:2cm);
\draw[] (2,5)--(2,6);

\node[fill=black] (A21) at (0,3.25) {};
\node[fill=black] (A41) at (0,7.25) {};
\node[fill=black] (A22) at (0,3.5) {};
\node[fill=black] (A42) at (0,7.5) {};
\node[fill=black] (A23) at (0,3.75) {};
\node[fill=black] (A43) at (0,7.75) {};

\endpgfonlayer

\draw[green] (N2).. controls (-3.0000000000000004,8.166666666666666) and (-2.0000000000000004,7.75) .. (A43);
\draw[green] (A23).. controls (-0.725,3.75) and (-0.725,5.2) .. (A31);
\draw[green] (A14).. controls (-0.725,1.8) and (-0.725,3.25) .. (A21);
\draw[green] (A41).. controls (-0.725,7.25) and (-0.725,5.8) .. (A34);
\draw[green] (S4).. controls (-1.7999999999999996,0.7999999999999999) and (-1.1999999999999997,1.2) .. (A11);

\pgfonlayer{1}
\draw[green] (A43) arc (90:-90:2.0);
\endpgfonlayer

\pgfonlayer{2}
\draw[green] (A31) arc (90:-90:1.7);
\endpgfonlayer

\pgfonlayer{1}
\draw[green] (A21) arc (-90:90:2.0);
\endpgfonlayer

\pgfonlayer{2}
\draw[green] (A34) arc (90:-90:2.3);
\endpgfonlayer

\draw[red] (N1).. controls (-6.000000000000001,8.0) and (-4.000000000000001,7.5) .. (A42);
\draw[red] (A22).. controls (-0.95,3.5) and (-0.95,1.6) .. (A13);
\draw[red] (S2).. controls (-5.3999999999999995,3.6) and (-3.5999999999999996,5.4) .. (A32);

\pgfonlayer{1}
\draw[red] (A42) arc (90:-90:2.0);
\endpgfonlayer

\pgfonlayer{2}
\draw[red] (A13) arc (-90:90:1.9);
\endpgfonlayer

\draw[blue] (S1).. controls (-7.2,3.733333333333333) and (-4.8,5.6) .. (A33);
\draw[blue] (S3).. controls (-3.6,0.9333333333333332) and (-2.4,1.4) .. (A12);

\pgfonlayer{2}
\draw[blue] (A33) arc (90:-90:2.1);
\endpgfonlayer

\end{tikzpicture}}

\def\SWBtwR#1{\begin{tikzpicture}[every node/.style={shape=circle},inner sep=0,scale=#1,anchor=base,baseline=(current bounding box.center)]

\pgfdeclarelayer{1}
\pgfdeclarelayer{2}
\pgfdeclarelayer{3}

\pgfsetlayers{main,
1,
2,
3} 

\draw[] (-13,0)--(0,0);
\draw[] (-13,0)--(-13,13);
\draw[] (-13,13)--(0,13);
\draw[] (0,0) --(0,1);
\draw[] (0,2) -- (0,3); 
\node[fill=black,inner sep =0] (A1) at (0,1) {};
\node[fill=black,inner sep =0] (B1) at (0,2) {};
\draw[] (0,4) -- (0,5); 
\node[fill=black,inner sep =0] (A2) at (0,3) {};
\node[fill=black,inner sep =0] (B2) at (0,4) {};
\draw[] (0,6) -- (0,7); 
\node[fill=black,inner sep =0] (A3) at (0,5) {};
\node[fill=black,inner sep =0] (B3) at (0,6) {};
\draw[] (0,8) -- (0,9); 
\node[fill=black,inner sep =0] (A4) at (0,7) {};
\node[fill=black,inner sep =0] (B4) at (0,8) {};
\draw[] (0,10) -- (0,11); 
\node[fill=black,inner sep =0] (A5) at (0,9) {};
\node[fill=black,inner sep =0] (B5) at (0,10) {};
\draw[] (0,12) -- (0,13); 
\node[fill=black,inner sep =0] (A6) at (0,11) {};
\node[fill=black,inner sep =0] (B6) at (0,12) {};

\node[fill=black] (S1) at (-13+13/4,0) {};
\node[fill=black] (S2) at (-13+13/2,0) {};
\node[fill=black] (S3) at (-13+3*13/4,0) {};

\node[fill=black] (N1) at (-13+13/6,13) {};
\node[fill=black] (N2) at (-13+2*13/6,13) {};
\node[fill=black] (N3) at (-13+3*13/6,13) {};
\node[fill=black] (N4) at (-13+4*13/6,13) {};
\node[fill=black] (N5) at (-13+5*13/6,13) {};

\pgfonlayer{3}

\draw[white,line width=#1*10mm] (0,1.5)   arc (-90:90:3cm);
\draw[] (0,1) arc (-90:90:3.5cm);
\draw[] (0,2) arc (-90:90:2.5cm);

\node[fill=black] (A11) at (0,1.25) {};
\node[fill=black] (A12) at (0,1.5) {};
\node[fill=black] (A13) at (0,1.75) {};
\node[fill=black] (A41) at (0,7.25) {};
\node[fill=black] (A42) at (0,7.5) {};
\node[fill=black] (A43) at (0,7.75) {};

\endpgfonlayer

\pgfonlayer{2}

\draw[white,line width=#1*5mm] (0,3.5)   arc (-90:90:4cm);
\draw[white,line width=#1*5mm] (0,3.75)   arc (-90:90:4cm);
\draw[white,line width=#1*5mm] (0,3.25)   arc (-90:90:4cm);
\draw[] (0,3) arc (-90:90:4cm);
\draw[] (0,4) arc (-90:90:4cm);
\draw[] (4,7)--(4,8);

\node[fill=black] (A21) at (0,3.2) {};
\node[fill=black] (A61) at (0,11.2) {};
\node[fill=black] (A22) at (0,3.4) {};
\node[fill=black] (A62) at (0,11.4) {};
\node[fill=black] (A23) at (0,3.6) {};
\node[fill=black] (A63) at (0,11.6) {};
\node[fill=black] (A24) at (0,3.8) {};
\node[fill=black] (A64) at (0,11.8) {};

\endpgfonlayer

\pgfonlayer{1}

\draw[white,line width=#1*10mm] (0,5.5)   arc (-90:90:2cm);
\draw[] (0,5) arc (-90:90:2.5cm);
\draw[] (0,6) arc (-90:90:1.5cm);

\node[fill=black] (A31) at (0,5.16666666667) {};
\node[fill=black] (A51) at (0,9.16666666667) {};
\node[fill=black] (A32) at (0,5.33333333333) {};
\node[fill=black] (A52) at (0,9.33333333333) {};
\node[fill=black] (A33) at (0,5.5) {};
\node[fill=black] (A53) at (0,9.5) {};
\node[fill=black] (A34) at (0,5.66666666667) {};
\node[fill=black] (A54) at (0,9.66666666667) {};
\node[fill=black] (A35) at (0,5.83333333333) {};
\node[fill=black] (A55) at (0,9.83333333333) {};

\endpgfonlayer

\draw[blue] (S1) -- +(0,1);
\draw[blue] (S2) -- +(0,1);

\draw[blue] (N1) -- +(0,-1);
\draw[blue] (N2) -- +(0,-1);
\draw[blue] (N3) -- +(0,-1);
\draw[blue] (N4) -- +(0,-1);

\pgfonlayer{3}
\draw[blue] (A12) arc (-90:90:3.0);
\endpgfonlayer
\draw[blue] (A12) -- +(-1,0);
\draw[blue] (A42) -- +(-1,0);

\pgfonlayer{2}
\draw[blue] (A22) arc (-90:90:4.0);
\endpgfonlayer
\draw[blue] (A22) -- +(-1,0);
\draw[blue] (A62) -- +(-1,0);

\pgfonlayer{2}
\draw[blue] (A23) arc (-90:90:4.0);
\endpgfonlayer
\draw[blue] (A23) -- +(-1,0);
\draw[blue] (A63) -- +(-1,0);

\pgfonlayer{1}
\draw[blue] (A32) arc (-90:90:2.16666666667);
\endpgfonlayer
\draw[blue] (A32) -- +(-1,0);
\draw[blue] (A54) -- +(-1,0);

\pgfonlayer{1}
\draw[blue] (A33) arc (-90:90:2.0);
\endpgfonlayer
\draw[blue] (A33) -- +(-1,0);
\draw[blue] (A53) -- +(-1,0);

\pgfonlayer{1}
\draw[blue] (A34) arc (-90:90:1.83333333333);
\endpgfonlayer
\draw[blue] (A34) -- +(-1,0);
\draw[blue] (A52) -- +(-1,0);

\draw (-12,1) -- (-1,1);
\draw (-12,1) -- (-12,12);
\draw (-1,1) -- (-1,12);
\draw (-12,12) -- (-1,12);
\node (boxlabel) at (-6.5,6.5) {\Huge{$E_2$}};

\pgfonlayer{1}
\draw[green] (A31) arc (-90:90:2.33333333333);
\draw[green] (A35) arc (-90:90:1.66666666667);
\endpgfonlayer

\pgfonlayer{2}
\draw[green] (A21) arc (-90:90:4.0);
\draw[green] (A24) arc (-90:90:4.0);
\endpgfonlayer

\pgfonlayer{3}
\draw[green] (A11) arc (-90:90:3.25);
\draw[green] (A13) arc (-90:90:2.75);
\endpgfonlayer

\draw[green] (S3).. controls ([yshift=1cm]S3) and (-0.5,0) ..(-0.5,1);
\draw[green] (A11).. controls ([xshift=-0.5cm]A11) ..(-0.5,1);
\draw[green] (A13).. controls ([xshift=-0.5cm]A13) and ([xshift=-0.5cm]A21) ..(A21);
\draw[green] (A24).. controls ([xshift=-0.5cm]A24) and ([xshift=-0.5cm]A31) ..(A31);
\draw[green] (A35).. controls ([xshift=-0.5cm]A35) and ([xshift=-0.5cm]A41) ..(A41);
\draw[green] (A43).. controls ([xshift=-0.5cm]A43) and ([xshift=-0.5cm]A51) ..(A51);
\draw[green] (A55).. controls ([xshift=-0.5cm]A55) and ([xshift=-0.5cm]A61) ..(A61);
\draw[green] (N5).. controls ([yshift=-1cm]N5) and (-0.5,13) ..(-0.5,12);
\draw[green] (A64).. controls ([xshift=-0.5cm]A64) ..(-0.5,12);

\end{tikzpicture}
}

\def\SWBtwL#1{\begin{tikzpicture}[every node/.style={shape=circle},inner sep=0,scale=#1,anchor=base,baseline=(current bounding box.center)]

\pgfdeclarelayer{1}
\pgfdeclarelayer{2}
\pgfdeclarelayer{3}

\pgfsetlayers{main,
1,
2,
3} 

\draw[] (-13,0)--(0,0);
\draw[] (-13,0)--(-13,13);
\draw[] (-13,13)--(0,13);
\draw[] (0,0) --(0,1);
\draw[] (0,2) -- (0,3); 
\node[fill=black,inner sep =0] (A1) at (0,1) {};
\node[fill=black,inner sep =0] (B1) at (0,2) {};
\draw[] (0,4) -- (0,5); 
\node[fill=black,inner sep =0] (A2) at (0,3) {};
\node[fill=black,inner sep =0] (B2) at (0,4) {};
\draw[] (0,6) -- (0,7); 
\node[fill=black,inner sep =0] (A3) at (0,5) {};
\node[fill=black,inner sep =0] (B3) at (0,6) {};
\draw[] (0,8) -- (0,9); 
\node[fill=black,inner sep =0] (A4) at (0,7) {};
\node[fill=black,inner sep =0] (B4) at (0,8) {};
\draw[] (0,10) -- (0,11); 
\node[fill=black,inner sep =0] (A5) at (0,9) {};
\node[fill=black,inner sep =0] (B5) at (0,10) {};
\draw[] (0,12) -- (0,13); 
\node[fill=black,inner sep =0] (A6) at (0,11) {};
\node[fill=black,inner sep =0] (B6) at (0,12) {};

\node[fill=black] (S1) at (-13+13/4,0) {};
\node[fill=black] (S2) at (-13+13/2,0) {};
\node[fill=black] (S3) at (-13+3*13/4,0) {};

\node[fill=black] (N1) at (-13+13/6,13) {};
\node[fill=black] (N2) at (-13+2*13/6,13) {};
\node[fill=black] (N3) at (-13+3*13/6,13) {};
\node[fill=black] (N4) at (-13+4*13/6,13) {};
\node[fill=black] (N5) at (-13+5*13/6,13) {};

\pgfonlayer{3}

\draw[white,line width=#1*10mm] (0,1.5)   arc (-90:90:3cm);
\draw[] (0,1) arc (-90:90:3.5cm);
\draw[] (0,2) arc (-90:90:2.5cm);

\node[fill=black] (A11) at (0,1.5) {};
\node[fill=black] (A41) at (0,7.5) {};

\endpgfonlayer

\pgfonlayer{2}

\draw[white,line width=#1*5mm] (0,3.5)   arc (-90:90:4cm);
\draw[white,line width=#1*5mm] (0,3.75)   arc (-90:90:4cm);
\draw[white,line width=#1*5mm] (0,3.25)   arc (-90:90:4cm);
\draw[] (0,3) arc (-90:90:4cm);
\draw[] (0,4) arc (-90:90:4cm);
\draw[] (4,7)--(4,8);

\node[fill=black] (A21) at (0,3.3333333333333335) {};
\node[fill=black] (A61) at (0,11.333333333333334) {};
\node[fill=black] (A22) at (0,3.6666666666666665) {};
\node[fill=black] (A62) at (0,11.666666666666666) {};

\endpgfonlayer

\pgfonlayer{1}

\draw[white,line width=#1*10mm] (0,5.5)   arc (-90:90:2cm);
\draw[] (0,5) arc (-90:90:2.5cm);
\draw[] (0,6) arc (-90:90:1.5cm);

\node[fill=black] (A31) at (0,5.25) {};
\node[fill=black] (A51) at (0,9.25) {};
\node[fill=black] (A32) at (0,5.5) {};
\node[fill=black] (A52) at (0,9.5) {};
\node[fill=black] (A33) at (0,5.75) {};
\node[fill=black] (A53) at (0,9.75) {};

\endpgfonlayer

\draw[blue] (S2) -- +(0,1);
\draw[blue] (S3) -- +(0,1);

\draw[blue] (N2) -- +(0,-1);
\draw[blue] (N3) -- +(0,-1);
\draw[blue] (N4) -- +(0,-1);
\draw[blue] (N5) -- +(0,-1);

\pgfonlayer{3}
\draw[blue] (A11) arc (-90:90:3.0);
\endpgfonlayer
\draw[blue] (A11) -- (-1,1.5);
\draw[blue] (A41) -- (-1,7.5);

\pgfonlayer{2}
\draw[blue] (A21) arc (-90:90:4.0);
\endpgfonlayer
\draw[blue] (A21) -- (-1,3.3333333333333335);
\draw[blue] (A61) -- (-1,11.333333333333334);

\pgfonlayer{2}
\draw[blue] (A22) arc (-90:90:4.0);
\endpgfonlayer
\draw[blue] (A22) -- (-1,3.6666666666666665);
\draw[blue] (A62) -- (-1,11.666666666666666);

\pgfonlayer{1}
\draw[blue] (A31) arc (-90:90:2.25);
\endpgfonlayer
\draw[blue] (A31) -- (-1,5.25);
\draw[blue] (A53) -- (-1,9.75);

\pgfonlayer{1}
\draw[blue] (A32) arc (-90:90:2.0);
\endpgfonlayer
\draw[blue] (A32) -- (-1,5.5);
\draw[blue] (A52) -- (-1,9.5);

\pgfonlayer{1}
\draw[blue] (A33) arc (-90:90:1.75);
\endpgfonlayer
\draw[blue] (A33) -- (-1,5.75);
\draw[blue] (A51) -- (-1,9.25);

\draw (-12,1) -- (-1,1);
\draw (-12,1) -- (-12,12);
\draw (-1,1) -- (-1,12);
\draw (-12,12) -- (-1,12);
\node (boxlabel) at (-6.5,6.5) {\Huge{$E_2$}};

\draw[green] (N1)..controls ([yshift=-1cm]N1) and (-12.5,13) ..(-12.5,12);
\draw[green] (-12.5,12) -- (-12.5,1);
\draw[green] (-12.5,1).. controls (-12.5,0) and ([yshift=1cm]S1) ..(S1);

\end{tikzpicture}
}

\def\genSWBL#1{\begin{tikzpicture}[every node/.style={shape=circle},inner sep=0,scale=#1,anchor=base,baseline=(current bounding box.center)]

\draw[] (0,2.25) -- (0,2.75);
\draw[] (0,0.25) -- (0,0.75);
\draw[] (0,0.25) -- (-2.5,0.25);
\draw[] (0,2.75) -- (-2.5,2.75);
\draw[] (-2.5,0.25) -- (-2.5,2.75);

\node[] (S0) at (-2,0.25) {};
\node[] (S1) at (-1.5,0.25) {};
\node[] (S3) at (-0.5,0.25) {};
\node[] (SS1) at (-5/3-1/2+0.375,0.75) {};
\node[] (SS3) at (-5/3-1/2+3*+0.375,0.75) {};

\node[] (N0) at (-2,2.75) {};
\node[] (N1) at (-1.5,2.75) {};
\node[] (N3) at (-0.5,2.75) {};
\node[] (NN1) at (-5/3-1/2+0.375,2.25) {};
\node[] (NN3) at (-5/3-1/2+3*+0.375,2.25) {};

\draw[] (-1/3,1) -- (2/3,1);
\draw[] (-1/3,2) -- (2/3,2);
\draw[] (-1/3,1) -- (-1/3,2);
\draw[] (2/3,1) -- (2/3,2);
\draw[] (2/3+1/4,1) -- (2/3+1/4,2);
\draw[] (2/3,0.75) -- (0,0.75);
\draw[] (2/3,2.25) -- (0,2.25);
\draw[] (2/3,2.25) arc (90:0:0.25cm);
\draw[] (2/3,0.75) arc (-90:0:0.25cm);
\node[] (box) at (1/6,1.5) { $(P,s)$};

\draw[] (-2/3,0.75) -- (-2/3,2.25);
\draw[] (-5/3-1/2,1-1/4) -- (-5/3-1/2,2+1/4);
\draw[] (-2/3,1-1/4) -- (-5/3-1/2,1-1/4);
\draw[] (-2/3,2+1/4) -- (-5/3-1/2,2+1/4);
\node[] (L1) at (-7/6-1/4,1.5) {\large $E$};

\draw[blue] (-2/3,1.25) -- (-1/3,1.25);
\draw[blue] (-2/3,1.75) -- (-1/3,1.75);
\node[] () at (-0.5,1.35) {$\vdots$};
\node[] () at (-0.5,1.875) {\footnotesize $2 C$};

\draw[blue] (S1).. controls ([yshift=0.5cm]S1) and ([yshift=-0.5cm]SS1)..(SS1);
\draw[blue] (S3).. controls ([yshift=0.5cm]S3) and ([yshift=-0.5cm]SS3)..(SS3);
\node[] () at (-5/3-1/2+2*0.375,0.625) {$\dots$};
\node[] () at (-1,0.33) {$n$};
\draw[blue] (N1).. controls ([yshift=-0.5cm]N1) and ([yshift=0.5cm]NN1)..(NN1);
\draw[blue] (N3).. controls ([yshift=-0.5cm]N3) and ([yshift=0.5cm]NN3)..(NN3);
\node[] () at (-5/3-1/2+2*0.375,2.375) {$\dots$};
\node[] () at (-1,2.55) {$m$};

\draw[green] (S0).. controls ([yshift=0.5cm]S0) and (-14/6,0.25) ..(-14/6,0.75);
\draw[green] (-14/6,0.75) -- (-14/6,2.25);
\draw[green] (N0).. controls ([yshift=-0.5cm]N0) and (-14/6,2.75) ..(-14/6,2.25);

\end{tikzpicture}}

\def\genSWBS#1{\begin{tikzpicture}[every node/.style={shape=circle},inner sep=0,scale=#1,anchor=base,baseline=(current bounding box.center)]

\draw[] (0,2.25) -- (0,2.75);
\draw[] (0,0.25) -- (0,0.75);
\draw[] (0,0.25) -- (-2.5,0.25);
\draw[] (0,2.75) -- (-2.5,2.75);
\draw[] (-2.5,0.25) -- (-2.5,2.75);

\node[] (S0) at (-2,0.25) {};
\node[] (S1) at (-1.5,0.25) {};
\node[] (S3) at (-0.5,0.25) {};
\node[] (SS1) at (-5/3-1/2+0.375,0.75) {};
\node[] (SS3) at (-5/3-1/2+3*0.375,0.75) {};

\node[] (N1) at (-2.5*3/4,2.75) {};
\node[] (N3) at (-2.5*1/4,2.75) {};
\node[] (NN1) at (-5/3-1/2+0.3,2.25) {};
\node[] (NN2) at (-5/3-1/2+0.6,2.25) {};
\node[] (NN4) at (-5/3-1/2+4*0.3,2.25) {};
\node[] (NN0) at (-14/6,2.25) {};

\draw[] (-1/3,1) -- (2/3,1);
\draw[] (-1/3,2) -- (2/3,2);
\draw[] (-1/3,1) -- (-1/3,2);
\draw[] (2/3,1) -- (2/3,2);
\draw[] (2/3+1/4,1) -- (2/3+1/4,2);
\draw[] (2/3,0.75) -- (0,0.75);
\draw[] (2/3,2.25) -- (0,2.25);
\draw[] (2/3,2.25) arc (90:0:0.25cm);
\draw[] (2/3,0.75) arc (-90:0:0.25cm);
\node[] (box) at (1/6,1.5) { $(P,s)$};

\draw[] (-2/3,0.75) -- (-2/3,2.25);
\draw[] (-5/3-1/2,1-1/4) -- (-5/3-1/2,2+1/4);
\draw[] (-2/3,1-1/4) -- (-5/3-1/2,1-1/4);
\draw[] (-2/3,2+1/4) -- (-5/3-1/2,2+1/4);
\node[] (L1) at (-7/6-1/4,1.5) {\large $E$};

\draw[blue] (-2/3,1.25) -- (-1/3,1.25);
\draw[blue] (-2/3,1.75) -- (-1/3,1.75);
\node[] () at (-0.5,1.35) {$\vdots$};
\node[] () at (-0.5,1.875) {\footnotesize $2 C$};

\draw[blue] (S1).. controls ([yshift=0.5cm]S1) and ([yshift=-0.5cm]SS1)..(SS1);
\draw[blue] (S3).. controls ([yshift=0.5cm]S3) and ([yshift=-0.5cm]SS3)..(SS3);
\node[] () at (-5/3-1/2+2*0.375,0.625) {$\dots$};
\node[] () at (-1,0.29) {$n+1$};
\draw[blue] (N1).. controls ([yshift=-0.5cm]N1) and ([yshift=0.5cm]NN2)..(NN2);
\draw[blue] (N3).. controls ([yshift=-0.5cm]N3) and ([yshift=0.5cm]NN3)..(NN3);
\node[] () at (-5/3-1/2+3*0.3,2.375) {$\dots$};
\node[] () at (-1.25,2.55) {$m-1$};

\draw[blue] (S0).. controls ([yshift=0.5cm]S0) and (-14/6,0.25) ..(-14/6,0.75);
\draw[blue] (-14/6,0.75) -- (-14/6,2.25);
\draw[blue] (NN0).. controls ([yshift=0.5cm]NN0) and ([yshift=0.5cm]NN1) ..(NN1);

\end{tikzpicture}}

\def\genSWBR#1{\begin{tikzpicture}[every node/.style={shape=circle},inner sep=0,scale=#1,anchor=base,baseline=(current bounding box.center)]

\draw[] (0,2.25) -- (0,2.75);
\draw[] (0,0.25) -- (0,0.75);
\draw[] (0,0.25) -- (-2.5,0.25);
\draw[] (0,2.75) -- (-2.5,2.75);
\draw[] (-2.5,0.25) -- (-2.5,2.75);

\node[] (S1) at (-2,0.25) {};
\node[] (S3) at (-1,0.25) {};
\node[] (S4) at (-0.5,0.25) {};
\node[] (SS1) at (-5/3-1/2+0.375,0.75) {};
\node[] (SS3) at (-5/3-1/2+3*+0.375,0.75) {};

\node[] (N1) at (-2,2.75) {};
\node[] (N3) at (-1,2.75) {};
\node[] (N4) at (-0.5,2.75) {};
\node[] (NN1) at (-5/3-1/2+0.375,2.25) {};
\node[] (NN3) at (-5/3-1/2+3*+0.375,2.25) {};

\draw[] (-1/3,1) -- (2/3,1);
\draw[] (-1/3,2) -- (2/3,2);
\draw[] (-1/3,1) -- (-1/3,2);
\draw[] (2/3,1) -- (2/3,2);
\draw[] (2/3+1/4,1) -- (2/3+1/4,2);
\draw[] (2/3,0.75) -- (0,0.75);
\draw[] (2/3,2.25) -- (0,2.25);
\draw[] (2/3,2.25) arc (90:0:0.25cm);
\draw[] (2/3,0.75) arc (-90:0:0.25cm);
\node[] (box) at (1/6,1.5) { $(P,s)$};

\draw[] (-2/3,0.75) -- (-2/3,2.25);
\draw[] (-5/3-1/2,1-1/4) -- (-5/3-1/2,2+1/4);
\draw[] (-2/3,1-1/4) -- (-5/3-1/2,1-1/4);
\draw[] (-2/3,2+1/4) -- (-5/3-1/2,2+1/4);
\node[] (L1) at (-7/6-1/4,1.5) {\large $E$};

\draw[blue] (-2/3,1.25) -- (-1/3,1.25);
\draw[blue] (-2/3,1.75) -- (-1/3,1.75);
\node[] () at (-0.5,1.35) {$\vdots$};
\node[] () at (-0.5,1.875) {\footnotesize $2 C$};

\draw[blue] (S1).. controls ([yshift=0.5cm]S1) and ([yshift=-0.5cm]SS1)..(SS1);
\draw[blue] (S3).. controls ([yshift=0.5cm]S3) and ([yshift=-0.5cm]SS3)..(SS3);
\node[] () at (-5/3-1/2+2*0.375,0.625) {$\dots$};
\node[] () at (-1.5,0.375) {$n$};
\draw[blue] (N1).. controls ([yshift=-0.5cm]N1) and ([yshift=0.5cm]NN1)..(NN1);
\draw[blue] (N3).. controls ([yshift=-0.5cm]N3) and ([yshift=0.5cm]NN3)..(NN3);
\node[] () at (-5/3-1/2+2*0.375,2.375) {$\dots$};
\node[] () at (-1.5,2.55) {$m$};

\draw[green] (S4)..controls ([yshift=0.625cm]S4) ..(0,0.875);
\draw[green] (0,0.875) -- +(2/3,0);
\draw[green] (2/3,0.875) arc (-90:0:0.125);
\draw[green] (2/3+0.125,1) -- +(0,1);
\draw[green] (2/3+0.125,2) arc (0:90:0.125);
\draw[green] (2/3,2.125) -- +(-2/3,0);
\draw[green] (N4)..controls ([yshift=-0.625cm]N4) ..(0,2.125);

\end{tikzpicture}}

\def\lemoneintA#1{\begin{tikzpicture}[every node/.style={shape=circle},inner sep=0,scale=#1,anchor=base,baseline=#1*3.75 cm]

\draw[dashed] (2.25,0) -- (2.25,7.5);
\draw[dashed] (2.25,0) -- (-2,0);
\draw[dashed] (2.25,7.5) -- (-2,7.5);
\draw[dashed] (-2,0) -- (-2,7.5);

\draw[] (0,2.25) -- (0,3);
\draw[] (0,0) -- (0,0.75);
\draw[] (0,4) -- (0,4.5);
\draw[] (0,5.5) -- (0,6);
\draw[] (0,7) -- (0,7.5);

\draw[] (-1/3,1) -- (2/3,1);
\draw[] (-1/3,2) -- (2/3,2);
\draw[] (-1/3,1) -- (-1/3,2);
\draw[] (2/3,1) -- (2/3,2);
\draw[] (2/3+1/4,1) -- (2/3+1/4,2);
\draw[] (2/3,0.75) -- (0,0.75);
\draw[] (2/3,2.25) -- (0,2.25);
\draw[] (2/3,2.25) arc (90:0:0.25cm);
\draw[] (2/3,0.75) arc (-90:0:0.25cm);
\node[] (box) at (1/6,1.5) {\tiny \contour{white}{$(P_1,s_1)$} };

\draw[] (0,3) arc (-90:90:2cm);
\draw[] (0,4) arc (-90:90:1cm);
\draw[green,double distance=1pt] (0,3.25) arc (-90:90:1.75cm);
\draw[red,double distance=1pt] (0,3.5) arc (-90:90:1.5cm);
\draw[blue,double distance=1pt] (0,3.75) arc (-90:90:1.25cm);
\node[] (L2) at (0,4.85) {\Large $\vdots$};

\draw[] (-2/3,1) -- (-2/3,2);
\draw[] (-5/3,1) -- (-5/3,2);
\draw[] (-2/3,1) -- (-5/3,1);
\draw[] (-2/3,2) -- (-5/3,2);
\node[] (L1) at (-7/6,1.5) {\large \contour{white}{$E_1$}};

\draw[] (-2/3,4.5) -- (-2/3,5.5);
\draw[] (-2/3,4.5) -- (-2,4.5);
\draw[] (-2/3,5.5) -- (-2,5.5);
\node[] (L2) at (-7/6,5) {\large $\overline{E}$};

\draw[green, double distance=1pt] (-5/3+0.75,2).. controls (-5/3+0.75,3.25) ..(0,3.25) node [pos=0.3,right] {$m_3$};
\draw[green, double distance=1pt] (0,6.75).. controls (-1,6.75) ..(-1,7.5);
\node[fill=green, inner sep =1.5pt] (U''') at (-1,7.5) {};
\node[] () at ([yshift=0.2cm] U''') {\contour{white}{$U'''$}};

\draw[red, double distance=1pt] (-5/3+0.5,2).. controls (-5/3+0.5,3.5) ..(0,3.5) node [pos=0.6,above] {$m_2$};
\draw[red, double distance=1pt] (0,6.5) arc (90:180:0.5cm);
\draw[red, double distance=1pt] (-0.5,6).. controls (-0.5,5.75) ..(-2,5.75);
\node[fill=red, inner sep =1.5pt] (U'') at (-2,5.75) {};
\node[anchor=east] () at ([xshift=-0.1cm] U'') {\contour{white}{$U''$}};

\draw[blue, double distance=1pt] (0,3.75) arc (-90:-180:0.25);
\draw[blue, double distance=1pt] (-0.25,4) -- (-0.25,4.5);
\draw[blue, double distance=1pt] (-0.25,4.5) arc (0:90:0.25);
\draw[blue, double distance=1pt] (-0.5,4.75) -- (-2/3,4.75);
\draw[blue, double distance=1pt] (0,6.25) arc (90:180:0.25);
\draw[blue, double distance=1pt] (-0.25,6) -- (-0.25,5.5);
\draw[blue, double distance=1pt] (-0.25,5.5) arc (0:-90:0.25);
\draw[blue, double distance=1pt] (-0.5,5.25) -- (-2/3,5.25);

\draw[blue] (-2/3,1.25) -- (-1/3,1.25);
\draw[blue] (-2/3,1.75) -- (-1/3,1.75);
\node[] () at (-0.5,1.35) {$\vdots$};

\node[fill=black, shape=ellipse,rotate=90, inner sep=1.2pt,label=below:{\contour{white}{$V$}}] () at (0,3.375) {$\_\_$};
\node[fill=black, shape=ellipse,rotate=90, inner sep=1.2pt,label=below:{\contour{white}{$\iota_F(V)$}}] () at (0,6.625) {$\_\_$};

\draw[purple, double distance=1pt] (-5/3+0.5,0) -- (-5/3+0.5,1) node [midway,right] {$n$};
\node[fill=purple, inner sep =1.5pt] (U) at (-5/3+0.5,0) {};
\node[anchor=north] () at ([yshift=-0.1cm]U) {\contour{white}{$U$}};

\draw[orange, double distance=1pt] (-5/3+0.25,2).. controls (-5/3+0.25,4.25) ..(-2,4.25) node [pos=0.3,left] {$m_1$};
\node[fill=orange, inner sep =1.5pt] (U') at (-2,4.25) {};
\node[anchor=east] () at ([xshift=-0.1cm]U') {\contour{white}{$U'$}};

\draw[blue, double distance = 2pt] (0,5) -- (-2/3,5);

\end{tikzpicture}}

\def\lemoneA#1{\begin{tikzpicture}[every node/.style={shape=circle},inner sep=0,scale=#1,anchor=base,baseline=#1*3.75 cm]

\draw[dashed] (2.25,0) -- (2.25,7.5);
\draw[dashed] (2.25,0) -- (-2,0);
\draw[dashed] (2.25,7.5) -- (-2,7.5);
\draw[dashed] (-2,0) -- (-2,7.5);

\draw[] (0,2.25) -- (0,3);
\draw[] (0,0) -- (0,0.75);
\draw[] (0,4) -- (0,4.5);
\draw[] (0,5.5) -- (0,6);
\draw[] (0,7) -- (0,7.5);

\draw[] (-1/3,1) -- (2/3,1);
\draw[] (-1/3,2) -- (2/3,2);
\draw[] (-1/3,1) -- (-1/3,2);
\draw[] (2/3,1) -- (2/3,2);
\draw[] (2/3+1/4,1) -- (2/3+1/4,2);
\draw[] (2/3,0.75) -- (0,0.75);
\draw[] (2/3,2.25) -- (0,2.25);
\draw[] (2/3,2.25) arc (90:0:0.25cm);
\draw[] (2/3,0.75) arc (-90:0:0.25cm);
\node[] (box) at (1/6,1.5) {\tiny $(P_1,s_1)$};

\draw[] (0,3) arc (-90:90:2cm);
\draw[] (0,4) arc (-90:90:1cm);
\draw[green,double distance=1pt] (0,3.25) arc (-90:90:1.75cm);
\draw[red,double distance=1pt] (0,3.5) arc (-90:90:1.5cm);
\draw[blue,double distance=1pt] (0,3.75) arc (-90:90:1.25cm);
\node[] (L2) at (0,4.85) {\Large $\vdots$};

\draw[] (-2/3,1) -- (-2/3,2);
\draw[] (-5/3,1) -- (-5/3,2);
\draw[] (-2/3,1) -- (-5/3,1);
\draw[] (-2/3,2) -- (-5/3,2);
\node[] (L1) at (-8/6,1.5) { $E_1$};

\draw[] (-2/3,4.5) -- (-2/3,5.5);
\draw[] (-2/3,4.5) -- (-2,4.5);
\draw[] (-2/3,5.5) -- (-2,5.5);
\node[] (L2) at (-7/6,5-0.15) {\large $\overline{E}$};

\draw[green, double distance=1pt] (-5/3+0.75,2).. controls (-5/3+0.75,3.25) ..(0,3.25) node [pos=0.3,right] {$m_3$};
\draw[green, double distance=1pt] (0,6.75).. controls (-1,6.75) ..(-1,7.5);
\node[fill=green, inner sep =1.5pt] (U''') at (-1,7.5) {};
\node[] () at ([yshift=0.2cm] U''') {$U'''$};

\draw[red, double distance=1pt] (-5/3+0.5,2).. controls (-5/3+0.5,3.5) ..(0,3.5) node [pos=0.6,above] {$m_2$};
\draw[red, double distance=1pt] (0,6.5) arc (90:180:0.5cm);
\draw[red, double distance=1pt] (-0.5,6).. controls (-0.5,5.75) ..(-2,5.75);
\node[fill=red, inner sep =1.5pt] (U'') at (-2,5.75) {};
\node[anchor=east] () at ([xshift=-0.1cm] U'') {$U''$};

\draw[blue, double distance=1pt] (0,3.75) arc (-90:-180:0.25);
\draw[blue, double distance=1pt] (-0.25,4) -- (-0.25,4.5);
\draw[blue, double distance=1pt] (-0.25,4.5) arc (0:90:0.25);
\draw[blue, double distance=1pt] (-0.5,4.75) -- (-2/3,4.75);
\draw[blue, double distance=1pt] (0,6.25) arc (90:180:0.25);
\draw[blue, double distance=1pt] (-0.25,6) -- (-0.25,5.5);
\draw[blue, double distance=1pt] (-0.25,5.5) arc (0:-90:0.25);
\draw[blue, double distance=1pt] (-0.5,5.25) -- (-2/3,5.25);

\draw[blue] (-2/3,1.25) -- (-1/3,1.25);
\draw[blue] (-2/3,1.75) -- (-1/3,1.75);
\draw[blue] (-2/3,1.75).. controls (-11/12,1.75) ..(-11/12,2);
\draw[blue] (-2/3,1.25).. controls (-7/6,1.25) ..(-7/6,1);

\node[] () at (-0.5,1.35) {$\vdots$};

\draw[purple, double distance=1pt] (-5/3+0.5,0) -- (-5/3+0.5,1) node [midway,right] {$n$};
\node[fill=purple, inner sep =1.5pt] (U) at (-5/3+0.5,0) {};
\node[anchor=north] () at ([yshift=-0.1cm]U) {$U$};

\draw[orange, double distance=1pt] (-5/3+0.25,2).. controls (-5/3+0.25,4.25) ..(-2,4.25) node [pos=0.3,left] {$m_1$};
\node[fill=orange, inner sep =1.5pt] (U') at (-2,4.25) {};
\node[anchor=east] () at ([xshift=-0.1cm]U') {$U'$};

\draw[blue, double distance = 2pt] (0,5) -- (-2/3,5);

\node[anchor=west] () at (0,2.25) {$\uparrow$};
\end{tikzpicture}}

\def\lemoneB#1{\begin{tikzpicture}[every node/.style={shape=circle},inner sep=0,scale=#1,anchor=base,baseline=#1*3.75 cm]

\draw[dashed] (2.25,0) -- (2.25,7.5);
\draw[dashed] (2.25,0) -- (-2.5,0);
\draw[dashed] (2.25,7.5) -- (-2.5,7.5);
\draw[dashed] (-2.5,0) -- (-2.5,7.5);

\draw[] (0,0) -- (0,0.5);
\draw[] (0,1.5) -- (0,2);
\draw[] (0,3) -- (0,3.5);
\draw[] (0,4.5) -- (0,5.25);
\draw[] (0,6.75) -- (0,7.5);

\draw[] (-1/3,1+4.5) -- (2/3,1+4.5);
\draw[] (-1/3,2+4.5) -- (2/3,2+4.5);
\draw[] (-1/3,1+4.5) -- (-1/3,2+4.5);
\draw[] (2/3,1+4.5) -- (2/3,2+4.5);
\draw[] (2/3+1/4,1+4.5) -- (2/3+1/4,2+4.5);
\draw[] (2/3,0.75+4.5) -- (0,0.75+4.5);
\draw[] (2/3,2.25+4.5) -- (0,2.25+4.5);
\draw[] (2/3,2.25+4.5) arc (90:0:0.25cm);
\draw[] (2/3,0.75+4.5) arc (-90:0:0.25cm);
\node[] (box) at (1/6,1.5+4.5) {\tiny $(P_1,s_1)$};

\draw[] (0,3-2.5) arc (-90:90:2cm);
\draw[] (0,4-2.5) arc (-90:90:1cm);
\draw[purple,double distance=1pt] (0,3.25-2.5) arc (-90:90:1.75cm);
\draw[orange,double distance=1pt] (0,3.5-2.5) arc (-90:90:1.5cm);
\draw[blue,double distance=1pt] (0,3.75-2.5) arc (-90:90:1.25cm);
\node[] (L2) at (0,4.85-2.5) {\Large $\vdots$};

\draw[] (-2/3,1+4.5) -- (-2/3,2+4.5);
\draw[] (-5/3,1+4.5) -- (-5/3,2+4.5);
\draw[] (-2/3,1+4.5) -- (-5/3,1+4.5);
\draw[] (-2/3,2+4.5) -- (-5/3,2+4.5);
\node[] (L1) at (-7/6,1.5+4.5) {\large $E_1$};

\draw[] (-2/3,4.5-2.5) -- (-2/3,5.5-2.5);
\draw[] (-2/3,4.5-2.5) -- (-2.5,4.5-2.5);
\draw[] (-2/3,5.5-2.5) -- (-2.5,5.5-2.5);
\node[] (L2) at (-7/6,5-0.15-2.5) {\large $\overline{E}$};

\draw[blue, double distance=1pt] (0,3.75-2.5) arc (-90:-180:0.25);
\draw[blue, double distance=1pt] (-0.25,4-2.5) -- (-0.25,4.5-2.5);
\draw[blue, double distance=1pt] (-0.25,4.5-2.5) arc (0:90:0.25);
\draw[blue, double distance=1pt] (-0.5,4.75-2.5) -- (-2/3,4.75-2.5);
\draw[blue, double distance=1pt] (0,6.25-2.5) arc (90:180:0.25);
\draw[blue, double distance=1pt] (-0.25,6-2.5) -- (-0.25,5.5-2.5);
\draw[blue, double distance=1pt] (-0.25,5.5-2.5) arc (0:-90:0.25);
\draw[blue, double distance=1pt] (-0.5,5.25-2.5) -- (-2/3,5.25-2.5);

\draw[blue] (-2/3,1.75+4.5).. controls (-11/12,1.75+4.5) ..(-11/12,2+4.5);
\draw[blue] (-2/3,1.25+4.5).. controls (-7/6,1.25+4.5) ..(-7/6,1+4.5);
\node[anchor=south west,inner sep =0pt,outer sep=-5pt] () at (-2/3,2+4.5) {\contour{white}{$\tau_P(e)$}};
\node[anchor=north west,inner sep =0pt,outer sep=-5pt] () at (-2/3,1+4.5) {\contour{white}{$o_P(Q_u)$}};

\draw[blue] (-2/3,1.25+4.5) -- (-1/3,1.25+4.5);
\draw[blue] (-2/3,1.75+4.5) -- (-1/3,1.75+4.5);
\node[] () at (-0.5,1.35+4.5) {$\vdots$};

\draw[purple, double distance=1pt] (-5/3+0.5,1+4.5).. controls (-5/3+0.5,4.25) ..(0,4.25);
\draw[purple, double distance=1pt] (0,0.75).. controls (-1,0.75) ..(-1,0);
\node[fill=purple, inner sep =1.5pt] (U) at (-1,0) {};
\node[anchor=north] () at (U) {$o''(U)$};

\draw[orange, double distance=1pt] (-5/3+0.25,2+4.5) arc (0:180:0.25cm);
\draw[orange, double distance=1pt] (-5/3-0.25,2+4.5).. controls (-5/3-0.25,4) ..(0,4);
\draw[orange, double distance=1pt] (0,1) arc (-90:-180:0.5cm);
\draw[orange, double distance=1pt] (-0.5,1.5).. controls (-0.5,1.75) ..(-2.5,1.75);
\node[fill=orange, inner sep =1.5pt] (U') at (-2.5,1.75) {};
\node[anchor=east] () at ([xshift=-0.1cm] U') {$o''(U')$};

\draw[red, double distance=1pt] (-5/3+0.5,2+4.5) arc (0:180:0.5cm);
\draw[red, double distance=1pt] (-5/3-0.5,2+4.5).. controls (-5/3-0.5,3.25) ..(-2.5,3.25);
\node[fill=red, inner sep =1.5pt] (U'') at (-2.5,3.25) {};
\node[anchor=east] () at ([xshift=-0.1cm] U'') {$o''(U'')$};

\draw[green, double distance=1pt] (-5/3+0.75,6.5) -- (-5/3+0.75,7.5);
\node[fill=green, inner sep =1.5pt] (U''') at (-5/3+0.75,7.5) {};
\node[anchor=south] () at (U''') {$o''(U''')$};

\node[fill=black, shape=ellipse,rotate=90, inner sep=1.2pt,label=below:{\contour{white}{$\iota_{F'}(V')$}}] () at (0,3.375-2.5) {$\_\_$};
\node[fill=black, shape=ellipse,rotate=90, inner sep=1.2pt,label=below:{\contour{white}{$V'$}}] () at (0,6.625-2.5) {$\_\_$};

\draw[blue, double distance = 2pt] (0,2.5) -- (-2/3,2.5);

\end{tikzpicture}}

\def\lemoneBB#1{\begin{tikzpicture}[every node/.style={shape=circle},inner sep=0,scale=#1,anchor=base,baseline=-#1cm+#1*4.25 cm]

\draw[dashed] (2.25,-1) -- (2.25,7.5);
\draw[dashed] (2.25,-1) -- (-2.5,-1);
\draw[dashed] (2.25,7.5) -- (-2.5,7.5);
\draw[dashed] (-2.5,-1) -- (-2.5,7.5);

\draw[] (0,-1) -- (0,0.5);
\draw[] (0,1.5) -- (0,2);
\draw[] (0,3) -- (0,3.5);
\draw[] (0,4.5) -- (0,5.25);
\draw[] (0,6.75) -- (0,7.5);

\draw[] (-1/3,1+4.5) -- (2/3,1+4.5);
\draw[] (-1/3,2+4.5) -- (2/3,2+4.5);
\draw[] (-1/3,1+4.5) -- (-1/3,2+4.5);
\draw[] (2/3,1+4.5) -- (2/3,2+4.5);
\draw[] (2/3+1/4,1+4.5) -- (2/3+1/4,2+4.5);
\draw[] (2/3,0.75+4.5) -- (0,0.75+4.5);
\draw[] (2/3,2.25+4.5) -- (0,2.25+4.5);
\draw[] (2/3,2.25+4.5) arc (90:0:0.25cm);
\draw[] (2/3,0.75+4.5) arc (-90:0:0.25cm);
\node[] (box) at (1/6,1.5+4.5) {\tiny \contour{white}{$(P_1,s_1)$}};

\draw[] (0,3-2.5) arc (-90:90:2cm);
\draw[] (0,4-2.5) arc (-90:90:1cm);
\draw[green,double distance=1pt] (0,3+4/7-2.5) arc (-90:90:1.42857 cm);
\draw[red,double distance=1pt] (0,3+5/7-2.5) arc (-90:90:1.28571 cm);
\draw[blue,double distance=1pt] (0,3+6/7-2.5) arc (-90:90:1.14286 cm);
\draw[blue] (0,3+3/7-2.5) arc (-90:90:1.57143 cm);
\draw[blue] (0,3+1/7-2.5) arc (-90:90:1.85714 cm);
\node[] (L2) at (0,4.85-2.5) {\Large $\vdots$};

\draw[] (-2/3,1-1.5) -- (-2/3,2-1.5);
\draw[] (-5/3,1-1.5) -- (-5/3,2-1.5);
\draw[] (-2/3,1-1.5) -- (-5/3,1-1.5);
\draw[] (-2/3,2-1.5) -- (-5/3,2-1.5);
\node[] (L1) at (-8/6,1.5-1.5) {$E_1$};

\draw[] (-2/3,4.5-2.5) -- (-2/3,5.5-2.5);
\draw[] (-2/3,4.5-2.5) -- (-2.5,4.5-2.5);
\draw[] (-2/3,5.5-2.5) -- (-2.5,5.5-2.5);
\node[] (L2) at (-7/6,5-0.15-2.5) {\large $\overline{E}$};

\draw[blue, double distance=1pt] (0,3+6/7-2.5) arc (-90:-180:  0.14286);
\draw[blue, double distance=1pt] (-0.14286,4-2.5) -- (-0.14286,4.5-2.5);
\draw[blue, double distance=1pt] (-0.14286,4.5-2.5) arc (0:90:0.25);
\draw[blue, double distance=1pt] (-0.5+0.14286,4.75-2.5) -- (-2/3,4.75-2.5);
\draw[blue, double distance=1pt] (0,6+1/7-2.5) arc (90:180:0.14286);
\draw[blue, double distance=1pt] (-0.14286,6-2.5) -- (-0.14286,5.5-2.5);
\draw[blue, double distance=1pt] (-0.14286,5.5-2.5) arc (0:-90:0.25);
\draw[blue, double distance=1pt] (-0.5+0.14286,5.25-2.5) -- (-2/3,5.25-2.5);

\draw[blue] (-2/3,1.75-1.5).. controls (-11/12,1.75-1.5) ..(-11/12,2-1.5);
\draw[blue] (-2/3,1.25-1.5).. controls (-7/6,1.25-1.5) ..(-7/6,1-1.5);
\node[anchor=south,outer sep =-2pt] () at (-2/3,2-1.5) {$P_e$};
\node[anchor=north,outer sep =-2pt] () at (-2/3,1-1.5) {$Q_u$};

\draw[blue] (-1/3,1.25+4.5) arc (90:180:0.25);
\draw[blue] (-1/3-0.25,1+4.5).. controls (-1/3-0.25,3.5+6/7) ..(0,3.5+6/7);
\draw[blue] (0,0.5+0.142857) arc (90:180:0.142857);
\draw[blue] (-0.142857,0.5).. controls (-0.142857,-0.5+1/4) ..(-2/3,-0.5+1/4);

\draw[blue] (-1/3,1.75+4.5) arc (90:180:0.75);
\draw[blue] (-1/3-0.75,1+4.5).. controls (-1/3-0.75,3.5+4/7) ..(0,3.5+4/7);
\draw[blue] (0,0.5+0.42857) arc (90:180:0.42857);
\draw[blue] (-0.42857,0.5).. controls (-0.42857,-0.5+3/4) ..(-2/3,-0.5+3/4);
\node[] () at (-0.5,1.35+4.5) {$\vdots$};
\node[] () at (-0.5,1.35-1.5) {$\vdots$};

\draw[orange, double distance=1pt] (-5/3+0.25,2-1.5).. controls (-5/3+0.25,1.75) ..(-2.5,1.75);
\node[fill=orange, inner sep =1.5pt] (U') at (-2.5,1.75) {};
\node[anchor=east] () at ([xshift=-0.1cm]U') {$o_H(U')$};

\draw[green, double distance=1pt] (-5/3+0.75,2-1.5).. controls (-5/3+0.75,0.5+4/7) ..(0,0.5+4/7);
\draw[green, double distance=1pt] (0,3.5+3/7).. controls (-2,3.5+3/7) ..(-2,7.5);
\node[fill=green, inner sep =1.5pt] (U''') at (-2,7.5) {};
\node[] () at ([yshift=0.2cm] U''') {$o_H(U''')$};

\draw[red, double distance=1pt] (-5/3+0.5,2-1.5).. controls (-5/3+0.5,0.5+5/7) ..(0,0.5+5/7);
\draw[red, double distance=1pt] (0,3.5+2/7) arc (90:180:0.28571cm);
\draw[red, double distance=1pt] (-0.28571,3.5).. controls (-0.28571,3.25) ..(-2.5,3.25);
\node[fill=red, inner sep =1.5pt] (U'') at (-2.5,3.25) {};
\node[anchor=east] () at ([xshift=-0.1cm] U'') {$o_H(U'')$};

\draw[purple, double distance=1pt] (-5/3+0.5,-1) -- (-5/3+0.5,1-1.5);
\node[fill=purple, inner sep =1.5pt] (U) at (-5/3+0.5,-1) {};
\node[anchor=north] () at ([yshift=-0.1cm]U) {$o_H(U)$};

\draw[blue, double distance = 2pt] (0,2.5) -- (-2/3,2.5);

\end{tikzpicture}}

\def\lemtwintA#1{\begin{tikzpicture}[every node/.style={shape=circle},inner sep=0,scale=#1,anchor=base,baseline=#1*4cm]

\draw[dashed] (2,0) -- (2,8);
\draw[dashed] (2,0) -- (-2.5,0);
\draw[dashed] (2,8) -- (-2.5,8);
\draw[dashed] (-2.5,0) -- (-2.5,8);

\draw[] (0,2.25) -- (0,3);
\draw[] (0,0) -- (0,0.75);
\draw[] (0,4) -- (0,4.5);
\draw[] (0,5.5) -- (0,6);
\draw[] (0,7) -- (0,8);

\draw[] (-1/3,1) -- (2/3,1);
\draw[] (-1/3,2) -- (2/3,2);
\draw[] (-1/3,1) -- (-1/3,2);
\draw[] (2/3,1) -- (2/3,2);
\draw[] (2/3+1/4,1) -- (2/3+1/4,2);
\draw[] (2/3,0.75) -- (0,0.75);
\draw[] (2/3,2.25) -- (0,2.25);
\draw[] (2/3,2.25) arc (90:0:0.25cm);
\draw[] (2/3,0.75) arc (-90:0:0.25cm);
\node[] (box) at (1/6,1.5) {\tiny \contour{white}{$(P_1,s_1)$}};

\draw[] (0,3) arc (-90:90:1.5cm);
\draw[] (0,4) arc (-90:90:1.5cm);
\draw[green,double distance=1pt] (0,3.25) arc (-90:90:1.5cm);
\draw[red,double distance=1pt] (0,3.5) arc (-90:90:1.5cm);
\draw[blue,double distance=1pt] (0,3.75) arc (-90:90:1.5cm);
\node[] (L2) at (0,4.85) {\Large $\vdots$};

\draw[] (-2/3,1) -- (-2/3,2);
\draw[] (-5/3,1) -- (-5/3,2);
\draw[] (-2/3,1) -- (-5/3,1);
\draw[] (-2/3,2) -- (-5/3,2);
\node[] (L1) at (-7/6,1.5) {\large $E_1$};

\draw[] (-1/3,7.25) -- (-4/3,7.25);
\draw[] (-1/3,7.25) -- (-1/3,8);
\draw[] (-4/3,7.25) -- (-4/3,8);

\draw[] (-2/3,4.5) -- (-2/3,5.5);
\draw[] (-2/3,4.5) -- (-2.5,4.5);
\draw[] (-2/3,5.5) -- (-2.5,5.5);
\node[] (L2) at (-7/6,5) {\large $\overline{E}$};

\draw[red, double distance=1pt] (-5/3+0.5,2).. controls (-5/3+0.5,3.5) ..(0,3.5) node [pos=0.6,above] {$m_3$};
\draw[red, double distance=1pt] (0,6.5).. controls (-23/12,6.5) ..(-23/12,8);
\node[fill=red, inner sep =1.5pt] (U''') at (-23/12,8) {};
\node[] () at ([yshift=0.2cm] U''') {$U'''$};

\draw[green, double distance=1pt] (-5/3+0.75,2).. controls (-5/3+0.75,3.25) ..(0,3.25) node [pos=0.3,right] {$m_2$};
\draw[green, double distance=1pt] (0,6.25) arc (90:180:0.25cm);
\draw[green, double distance=1pt] (-0.25,6).. controls (-0.25,5.75) ..(-2.5,5.75);
\node[fill=green, inner sep =1.5pt] (U'') at (-2.5,5.75) {};
\node[anchor=east] () at ([xshift=-0.1cm] U'') {$U''$};

\draw[blue, double distance=1pt] (0,3.75) arc (-90:-180:0.25);
\draw[blue, double distance=1pt] (-0.25,4) -- (-0.25,4.5);
\draw[blue, double distance=1pt] (-0.25,4.5)..controls (-0.25,4.5+1/3) ..(-2/3,4.5+1/3);
\draw[blue, double distance=1pt] (0,6.75).. controls (-5/6,6.75) ..(-5/6,7.25);
\node (E'') at (-5/6,7.5) {$\overline{E}$};

\draw[blue] (-2/3,1.25) -- (-1/3,1.25);
\draw[blue] (-2/3,1.75) -- (-1/3,1.75);
\node[] () at (-0.5,1.35) {$\vdots$};

\draw[blue, double distance = 2pt] (0,4.5+2/3) -- (-2/3,4.5+2/3);

\draw[purple, double distance=1pt] (-5/3+0.5,0) -- (-5/3+0.5,1) node [pos=0.5,right] {$n$};
\node[fill=purple, inner sep =1.5pt] (U) at (-5/3+0.5,0) {};
\node[anchor=north] () at ([yshift=-0.1cm]U) {$U$};

\draw[orange, double distance=1pt] (-5/3+0.25,2).. controls (-5/3+0.25,4.25) ..(-2.5,4.25) node [pos=0.3,left] {$m_1$};
\node[fill=orange, inner sep =1.5pt] (U') at (-2.5,4.25) {};
\node[anchor=east] () at ([xshift=-0.1cm]U') {$U'$};

\node[fill=black, shape=ellipse,rotate=90, inner sep=1.2pt,label=below:{\contour{white}{$V$}}] () at (0,3.375) {$\_\_$};
\node[fill=black, shape=ellipse,rotate=90, inner sep=1.2pt,label=below:{\contour{white}{$\iota_F(V)$}}] () at (0,6.375) {$\_\_$};

\end{tikzpicture}}

\def\lemtwA#1{\begin{tikzpicture}[every node/.style={shape=circle},inner sep=0,scale=#1,anchor=base,baseline=#1*4cm]

\draw[dashed] (2.25,0) -- (2.25,8);
\draw[dashed] (2.25,0) -- (-2.5,0);
\draw[dashed] (2.25,8) -- (-2.5,8);
\draw[dashed] (-2.5,0) -- (-2.5,8);

\draw[] (0,2.25) -- (0,3);
\draw[] (0,0) -- (0,0.75);
\draw[] (0,4) -- (0,4.5);
\draw[] (0,5.5) -- (0,6);
\draw[] (0,7) -- (0,8);

\draw[] (-1/3,1) -- (2/3,1);
\draw[] (-1/3,2) -- (2/3,2);
\draw[] (-1/3,1) -- (-1/3,2);
\draw[] (2/3,1) -- (2/3,2);
\draw[] (2/3+1/4,1) -- (2/3+1/4,2);
\draw[] (2/3,0.75) -- (0,0.75);
\draw[] (2/3,2.25) -- (0,2.25);
\draw[] (2/3,2.25) arc (90:0:0.25cm);
\draw[] (2/3,0.75) arc (-90:0:0.25cm);
\node[] (box) at (1/6,1.5) {\tiny $(P_1,s_1)$};

\draw[] (0,3) arc (-90:90:1.5cm);
\draw[] (0,4) arc (-90:90:1.5cm);
\draw[green,double distance=1pt] (0,3.25) arc (-90:90:1.5cm);
\draw[red,double distance=1pt] (0,3.5) arc (-90:90:1.5cm);
\draw[blue,double distance=1pt] (0,3.75) arc (-90:90:1.5cm);
\node[] (L2) at (0,4.85) {\Large $\vdots$};

\draw[] (-2/3,1) -- (-2/3,2);
\draw[] (-5/3,1) -- (-5/3,2);
\draw[] (-2/3,1) -- (-5/3,1);
\draw[] (-2/3,2) -- (-5/3,2);
\node[] (L1) at (-7/6,1.5) {\large $E_1$};

\draw[] (-1/3,7.25) -- (-4/3,7.25);
\draw[] (-1/3,7.25) -- (-1/3,8);
\draw[] (-4/3,7.25) -- (-4/3,8);

\draw[] (-2/3,4.5) -- (-2/3,5.5);
\draw[] (-2/3,4.5) -- (-2.5,4.5);
\draw[] (-2/3,5.5) -- (-2.5,5.5);
\node[] (L2) at (-7/6,5-0.15) {\large $\overline{E}$};

\draw[red, double distance=1pt] (-5/3+0.5,2).. controls (-5/3+0.5,3.5) ..(0,3.5) node [pos=0.6,above] {$m_3$};
\draw[red, double distance=1pt] (0,6.5).. controls (-23/12,6.5) ..(-23/12,8);
\node[fill=red, inner sep =1.5pt] (U''') at (-23/12,8) {};
\node[] () at ([yshift=0.2cm] U''') {$U'''$};

\draw[green, double distance=1pt] (-5/3+0.75,2).. controls (-5/3+0.75,3.25) ..(0,3.25) node [pos=0.3,right] {$m_2$};
\draw[green, double distance=1pt] (0,6.25) arc (90:180:0.25cm);
\draw[green, double distance=1pt] (-0.25,6).. controls (-0.25,5.75) ..(-2.5,5.75);
\node[fill=green, inner sep =1.5pt] (U'') at (-2.5,5.75) {};
\node[anchor=east] () at ([xshift=-0.1cm] U'') {$U''$};

\draw[blue, double distance=1pt] (0,3.75) arc (-90:-180:0.25);
\draw[blue, double distance=1pt] (-0.25,4) -- (-0.25,4.5);
\draw[blue, double distance=1pt] (-0.25,4.5)..controls (-0.25,4.5+1/3) ..(-2/3,4.5+1/3);
\draw[blue, double distance=1pt] (0,6.75).. controls (-5/6,6.75) ..(-5/6,7.25);
\node (E'') at (-5/6,7.5) {$\overline{E}$};

\draw[blue] (-2/3,1.25) -- (-1/3,1.25);
\draw[blue] (-2/3,1.75) -- (-1/3,1.75);
\node[] () at (-0.5,1.35) {$\vdots$};

\draw[blue, double distance = 2pt] (0,4.5+2/3) -- (-2/3,4.5+2/3);

\draw[purple, double distance=1pt] (-5/3+0.5,0) -- (-5/3+0.5,1) node [pos=0.5,right] {$n$};
\node[fill=purple, inner sep =1.5pt] (U) at (-5/3+0.5,0) {};
\node[anchor=north] () at ([yshift=-0.1cm]U) {$U$};

\draw[orange, double distance=1pt] (-5/3+0.25,2).. controls (-5/3+0.25,4.25) ..(-2.5,4.25) node [pos=0.3,left] {$m_1$};
\node[fill=orange, inner sep =1.5pt] (U') at (-2.5,4.25) {};
\node[anchor=east] () at ([xshift=-0.1cm]U') {$U'$};

\end{tikzpicture}}

\def\lemtwB#1{\begin{tikzpicture}[every node/.style={shape=circle},inner sep=0,scale=#1,anchor=base,baseline=#1*4cm]

\draw[dashed] (3,0) -- (3,8);
\draw[dashed] (3,0) -- (-2.5,0);
\draw[dashed] (3,8) -- (-2.5,8);
\draw[dashed] (-2.5,0) -- (-2.5,8);

\draw[] (0,0) -- (0,0.5);
\draw[] (0,1.5) -- (0,2);
\draw[] (0,3) -- (0,3.75);
\draw[] (0,5.25) -- (0,6);
\draw[] (0,7) -- (0,8);

\draw[] (-1/3,1+3) -- (2/3,1+3);
\draw[] (-1/3,2+3) -- (2/3,2+3);
\draw[] (-1/3,1+3) -- (-1/3,2+3);
\draw[] (2/3,1+3) -- (2/3,2+3);
\draw[] (2/3+1/4,1+3) -- (2/3+1/4,2+3);
\draw[] (2/3,0.75+3) -- (0,0.75+3);
\draw[] (2/3,2.25+3) -- (0,2.25+3);
\draw[] (2/3,2.25+3) arc (90:0:0.25cm);
\draw[] (2/3,0.75+3) arc (-90:0:0.25cm);
\node[] (box) at (1/6,1.5+3) {\tiny \contour{white}{$(P_1,s_1)^*$}};

\draw[] (0,0.5) arc (-90:90:2.75cm);
\draw[] (0,1.5) arc (-90:90:2.75cm);
\draw[purple,double distance=1pt] (0,3.25-2.5) arc (-90:90:2.75cm);
\draw[orange,double distance=1pt] (0,3.5-2.5) arc (-90:90:2.75cm);
\draw[blue,double distance=1pt] (0,3.75-2.5) arc (-90:90:2.75cm);
\node[] (L2) at (0,4.85-2.5) {\Large $\vdots$};

\draw[] (-2/3,1+3) -- (-2/3,2+3);
\draw[] (-5/3,1+3) -- (-5/3,2+3);
\draw[] (-2/3,1+3) -- (-5/3,1+3);
\draw[] (-2/3,2+3) -- (-5/3,2+3);
\node[] (L1) at (-7/6,1.5+3) {\large ${E_1}^*$};

\draw[] (-1/3,7.25) -- (-4/3,7.25);
\draw[] (-1/3,7.25) -- (-1/3,8);
\draw[] (-4/3,7.25) -- (-4/3,8);
\node (E'') at (-5/6,7.5) {$\overline{E}$};

\draw[] (-2/3,4.5-2.5) -- (-2/3,5.5-2.5);
\draw[] (-2/3,4.5-2.5) -- (-2.5,4.5-2.5);
\draw[] (-2/3,5.5-2.5) -- (-2.5,5.5-2.5);
\node[] (L2) at (-7/6,5-2.5) {\large $\overline{E}$};

\draw[red, double distance=1pt] (-5/3+0.5,4) arc (0:-180:0.5cm);
\draw[red, double distance=1pt] (-5/3-0.5,4) -- (-5/3-0.5,8);
\node[fill=red, inner sep =1.5pt] (U''') at (-5/3-0.5,8) {};
\node[] () at ([yshift=0.2cm] U''') {$o''(U''')$};

\draw[green, double distance=1pt] (-5/3+0.75,4).. controls (-5/3+0.75,3.25) ..(-2.5,3.25);
\node[fill=green, inner sep =1.5pt] (U'') at (-2.5,3.25) {};
\node[anchor=east] () at ([xshift=-0.1cm] U'') {$o''(U'')$};

\draw[blue, double distance=1pt] (0,3.75-2.5) arc (-90:-180:0.25);
\draw[blue, double distance=1pt] (-0.25,4-2.5) -- (-0.25,4.5-2.5);
\draw[blue, double distance=1pt] (-0.25,4.5-2.5)..controls (-0.25,4.5+1/3-2.5) ..(-2/3,4.5+1/3-2.5);
\draw[blue, double distance=1pt] (0,6.75).. controls (-5/6,6.75) ..(-5/6,7.25);

\draw[blue, double distance = 2pt] (0,4.5+2/3-2.5) -- (-2/3,4.5+2/3-2.5);

\draw[blue] (-2/3,1.25+3) -- (-1/3,1.25+3);
\draw[blue] (-2/3,1.75+3) -- (-1/3,1.75+3);
\node[] () at (-0.5,1.35+3) {$\vdots$};

\draw[purple, double distance=1pt] (-5/3+0.5,5).. controls (-5/3+0.5,6.25).. (0,6.25);
\draw[purple, double distance=1pt] (0,0.75).. controls (-1.25,0.75).. (-1.25,0);
\node[fill=purple, inner sep =1.5pt] (U) at (-1.25,0) {};
\node[anchor=north] () at (U) {$o''(U)$};

\draw[orange, double distance=1pt] (-5/3+0.25,4) arc (0:-180:0.25cm);
\draw[orange, double distance=1pt] (-5/3-0.25,4).. controls (-5/3-0.25,6.5) ..(0,6.5);
\draw[orange, double distance=1pt] (0,1) arc (-90:-180:0.5cm);
\draw[orange, double distance=1pt] (-0.5,1.5).. controls (-0.5,1.75)  ..(-2.5,1.75);
\node[fill=orange, inner sep =1.5pt] (U') at (-2.5,1.75) {};
\node[anchor=east] () at ([xshift=-0.1cm]U') {$o''(U')$};

\end{tikzpicture}}

\def\lemthA#1{\begin{tikzpicture}[every node/.style={shape=circle},inner sep=0,scale=#1,anchor=base,baseline=#1*2.5cm]

\draw[dashed] (1,0) -- (1,5);
\draw[dashed] (1,0) -- (-2.5,0);
\draw[dashed] (1,5) -- (-2.5,5);
\draw[dashed] (-2.5,0) -- (-2.5,5);

\draw[] (0,2.25) -- (0,3);
\draw[] (0,0) -- (0,0.75);
\draw[] (0,4) -- (0,5);
\draw[] (0,3) -- (1,3);
\draw[] (0,4) -- (1,4);

\draw[] (-1/3,1) -- (2/3,1);
\draw[] (-1/3,2) -- (2/3,2);
\draw[] (-1/3,1) -- (-1/3,2);
\draw[] (2/3,1) -- (2/3,2);
\draw[] (2/3+1/4,1) -- (2/3+1/4,2);
\draw[] (2/3,0.75) -- (0,0.75);
\draw[] (2/3,2.25) -- (0,2.25);
\draw[] (2/3,2.25) arc (90:0:0.25cm);
\draw[] (2/3,0.75) arc (-90:0:0.25cm);
\node[] (box) at (1/6,1.5) {\tiny $(P_1,s_1)$};

\draw[] (-2/3,1) -- (-2/3,2);
\draw[] (-5/3,1) -- (-5/3,2);
\draw[] (-2/3,1) -- (-5/3,1);
\draw[] (-2/3,2) -- (-5/3,2);
\node[] (L1) at (-7/6,1.5) {\large $E_1$};

\draw[green, double distance=1pt] (0,3+1/3) -- (1,3+1/3);
\draw[green, double distance=1pt] (0,3+1/3).. controls (-25/12,3+1/3)  ..(-25/12,0) node [pos=0.6,right] {$k$};

\node[fill=green, inner sep =1.5pt] (X) at (-25/12,0) {};
\node[anchor=north] () at ([yshift=-0.1cm]X) {$X$};

\draw[blue, double distance=1pt] (0,3+2/3) -- (1,3+2/3);
\draw[blue, double distance=1pt] (0,3+2/3).. controls (-7/6,3+2/3)  ..(-7/6,5) node [pos=0.7,above right] {$l$};
\node[fill=blue,inner sep=1.5pt] (X') at (-7/6,5) {};
\node[anchor=south] () at ([yshift=0.1cm]X') {$X'$};

\draw[blue] (-2/3,1.25) -- (-1/3,1.25);
\draw[blue] (-2/3,1.75) -- (-1/3,1.75);
\node[] () at (-0.5,1.35) {$\vdots$};

\draw[purple, double distance=1pt] (-5/3+0.5,0) -- (-5/3+0.5,1) node [pos=0.5,right] {$n$};
\node[fill=purple, inner sep =1.5pt] (U) at (-5/3+0.5,0) {};
\node[anchor=north] () at ([yshift=-0.1cm]U) {$U$};

\draw[orange, double distance=1pt] (-5/3+0.5,2).. controls (-5/3+0.5,3.5) and (-25/12,3.5)  ..(-25/12,5) node [pos=0.7,right] {$m$};
\node[fill=orange, inner sep =1.5pt] (U') at (-25/12,5) {};
\node[anchor=south] () at ([yshift=0.1cm]U') {$U'$};

\end{tikzpicture}}

\def\lemthB#1{\begin{tikzpicture}[every node/.style={shape=circle},inner sep=0,scale=#1,anchor=base,baseline=#1*2.5 cm]

\draw[dashed] (1,0) -- (1,5);
\draw[dashed] (1,0) -- (-2.5,0);
\draw[dashed] (1,5) -- (-2.5,5);
\draw[dashed] (-2.5,0) -- (-2.5,5);

\draw[] (0,2.25+1.5) -- (0,3.5+1.5);
\draw[] (0,0+1.5) -- (0,0.75+1.5);
\draw[] (0,0) -- (0,0.5);
\draw[] (0,0.5) -- (1,0.5);
\draw[] (0,1.5) -- (1,1.5);

\draw[] (-1/3,1+1.5) -- (2/3,1+1.5);
\draw[] (-1/3,2+1.5) -- (2/3,2+1.5);
\draw[] (-1/3,1+1.5) -- (-1/3,2+1.5);
\draw[] (2/3,1+1.5) -- (2/3,2+1.5);
\draw[] (2/3+1/4,1+1.5) -- (2/3+1/4,2+1.5);
\draw[] (2/3,0.75+1.5) -- (0,0.75+1.5);
\draw[] (2/3,2.25+1.5) -- (0,2.25+1.5);
\draw[] (2/3,2.25+1.5) arc (90:0:0.25cm);
\draw[] (2/3,0.75+1.5) arc (-90:0:0.25cm);
\node[] (box) at (1/6,1.5+1.5) {\tiny $(P_1,s_1)$};

\draw[] (-2/3,1+1.5) -- (-2/3,2+1.5);
\draw[] (-5/3,1+1.5) -- (-5/3,2+1.5);
\draw[] (-2/3,1+1.5) -- (-5/3,1+1.5);
\draw[] (-2/3,2+1.5) -- (-5/3,2+1.5);
\node[] (L1) at (-7/6,1.5+1.5) {\large $E_1$};

\draw[blue, double distance=1pt] (0,1+1/6) -- (1,1+1/6);
\draw[blue, double distance=1pt] (-5/6,5).. controls (-5/6,3.5+0.16666666666) ..(0,3.5+0.16666666666);
\draw[blue, double distance=1pt] (0,3.5 +0.16666666666) -- (2/3, 3.5+0.16666666666);
\draw[blue, double distance=1pt] (2/3, 3.5+0.16666666666) arc (90:0:0.16666666666);
\draw[blue, double distance=1pt] (2/3+0.16666666666, 3.5) -- (2/3+0.16666666666, 2.5);
\draw[blue,double distance=1pt] (2/3+0.16666666666, 2.5) arc (0:-90:0.16666666666);
\draw[blue, double distance=1pt] (0,2.5-0.16666666666) -- (2/3, 2.5-0.16666666666);
\draw[blue, double distance=1pt] (0,2.5-0.16666666666).. controls (-0.5,2.5-0.16666666666) and (-0.5,1+1/6) ..(0,1+1/6);

\node[fill=blue,inner sep=1.5pt] (X') at (-5/6,5) {};
\node[anchor=south west] () at ([yshift=0.1cm]X') {$o''(X')$};

\draw[blue] (-2/3,1.25+1.5) -- (-1/3,1.25+1.5);
\draw[blue] (-2/3,1.75+1.5) -- (-1/3,1.75+1.5);
\node[] () at (-0.5,1.35+1.5) {$\vdots$};

\draw[green, double distance=1pt] (0,1-1/6) -- (1,1-1/6);
\draw[green, double distance=1pt] (-5/3,3.5+0.08333333333) -- (0,3.5+0.08333333333);
\draw[green, double distance=1pt] (-5/3,3.5+0.08333333333).. controls (-25/12,3.5+0.08333333333) .. (-25/12,0);
\draw[green, double distance=1pt] (0,3.5 +0.08333333333) -- (2/3, 3.5+0.08333333333);
\draw[green, double distance=1pt] (2/3, 3.5+0.08333333333) arc (90:0:0.08333333333);
\draw[green, double distance=1pt] (2/3+0.08333333333, 3.5) -- (2/3+0.08333333333, 2.5);
\draw[green,double distance=1pt] (2/3+0.08333333333, 2.5) arc (0:-90:0.08333333333);
\draw[green, double distance=1pt] (0,2.5-0.08333333333) -- (2/3, 2.5-0.08333333333);
\draw[green, double distance=1pt] (0,2.5-0.08333333333).. controls (-1,2.5-0.08333333333) and (-1,1-1/6) ..(0,1-1/6);

\node[fill=green, inner sep =1.5pt] (X) at (-25/12,0) {};
\node[anchor=north east] () at ([yshift=-0.1cm]X) {$o''(X)$};

\draw[purple, double distance=1pt] (-5/3+0.5,0) -- (-5/3+0.5,2.5);
\node[fill=purple, inner sep =1.5pt] (U) at (-5/3+0.5,0) {};
\node[anchor=north west] () at ([yshift=-0.1cm]U) {$o''(U)$};

\draw[orange, double distance=1pt] (-5/3+0.5,3.5)..controls (-5/3+0.5,4.25) and (-23/12,4.25) ..(-23/12,5);

\node[fill=orange, inner sep =1.5pt] (U') at (-23/12,5) {};
\node[anchor=south east] () at ([yshift=0.1cm]U') {$o''(U')$};

\end{tikzpicture}}

\def\IRdesa#1{\begin{tikzpicture}[every node/.style={shape=circle},inner sep =0, scale=#1,anchor=base,baseline=(current bounding box.center)]

\pgfdeclarelayer{1}
\pgfdeclarelayer{2}
\pgfsetlayers{main,1,2}
\draw (0,0) -- (0,0.5);
 \draw (0,1.5) -- (0,2); 
 \draw (0,3) -- (0,3.5); 
 \draw (0,4.5) -- (0,5.5); 
 \draw (0,6.5) -- (0,7); 
 \draw (0,8) -- (0,8.5); 
 \draw (0,9.5) -- (0,10);
\draw (-1,0) -- (-1,10);
\node[] at (0,2.3) {$\vdots $};
\node[] at (0,7.3) {$\vdots $};
\node at (-2,5) {$E$};
\pgfonlayer{2}

\draw[dashed] (-2.5,0) -- (2.5,0);
\draw[dashed] (-2.5,10) -- (2.5,10);
\draw[dashed] (-2.5,0) -- (-2.5,10);
\draw[dashed] (2.5,0) -- (2.5,10);

\draw[white,line width=#1*10mm] (0,1)   arc (-90:90:1.5cm);
\draw[] (0,0.5) arc (-90:90:2cm);
\draw[] (0,1.5) arc (-90:90:1cm);

\node[fill=black] (A11) at (0,1.5-0.5) {};
\node[fill=black] (A31) at (0,5.5-1.5) {};
\endpgfonlayer
\pgfonlayer{1}

\draw[white,line width=#1*10mm] (0,6)   arc (-90:90:1.5cm);
\draw[] (0,5.5) arc (-90:90:2cm);
\draw[] (0,6.5) arc (-90:90:1cm);

\node[fill=black] (A41) at (0,6) {};
\node[fill=black] (A61) at (0,9) {};
\endpgfonlayer

\pgfonlayer{1}
\draw[blue,double distance=1] (A41) arc (-90:90:1.5);
\endpgfonlayer
\draw[blue,double distance=1] (A41) -- +(-1,0);
\draw[blue,double distance=1] (A61) -- +(-1,0);

\draw[red,double distance=1] (A11) -- +(-1,0);
\draw[red,double distance=1] (A31) -- +(-1,0);

\pgfonlayer{2}
\draw[red,double distance=1] (A11) arc (-90:90:1.5);
\endpgfonlayer

\pgfonlayer{2}

\node[anchor=west] at (0,4.75) {$\uparrow$};
\endpgfonlayer

\draw[blue,double distance=2pt] (0,2.5) -- +(-1,0);
\draw[blue,double distance=2pt] (0,7.5) -- +(-1,0);

\end{tikzpicture}}

\def\IRdesab#1{\begin{tikzpicture}[every node/.style={shape=circle},inner sep =0, scale=#1,anchor=base,baseline=(current bounding box.center)]

\pgfdeclarelayer{1}
\pgfdeclarelayer{2}
\pgfsetlayers{main,1,2}

\draw (0,0) -- (0,0.5);
 \draw (0,1.5) -- (0,2); 
 \draw (0,3) -- (0,3.5); 
 \draw (0,4.5) -- (0,5.5); 
 \draw (0,6.5) -- (0,7); 
 \draw (0,8) -- (0,8.5); 
 \draw (0,9.5) -- (0,10);
\draw (-1,0) -- (-1,10);
\node[] at (0,2.3) {$\vdots $};
\node[] at (0,7.3) {$\vdots $};
\node at (-2,5) {$E$};

\pgfonlayer{2}

\draw[dashed] (-2.5,0) -- (2.5,0);
\draw[dashed] (-2.5,10) -- (2.5,10);
\draw[dashed] (-2.5,0) -- (-2.5,10);
\draw[dashed] (2.5,0) -- (2.5,10);

\draw[white,line width=#1*10mm] (0,1)   arc (-90:90:1.5cm);
\draw[] (0,0.5) arc (-90:90:2cm);
\draw[] (0,1.5) arc (-90:90:1cm);

\node[fill=black] (A11) at (0,1.25-0.5) {};
\node[fill=black] (A31) at (0,5.25-1.5) {};
\node[fill=black] (A12) at (0,1.5-0.5) {};
\node[fill=black] (A32) at (0,5.5-1.5) {};
\node[fill=black] (A13) at (0,1.75-0.5) {};
\node[fill=black] (A33) at (0,5.75-1.5) {};
\endpgfonlayer
\pgfonlayer{1}

\node[] (A21) at (0,3.25-1) {};
\node[] (A23) at (0,3.75-1) {};

\node[] (A51) at (0,9.25-2) {};
\node[] (A53) at (0,9.75-2) {};
\endpgfonlayer

\pgfonlayer{1}

\draw[white,line width=#1*10mm] (0,6)   arc (-90:90:1.5cm);
\draw[] (0,5.5) arc (-90:90:2cm);
\draw[] (0,6.5) arc (-90:90:1cm);

\node[fill=black] (A42) at (0,7.5-1.5) {};
\node[fill=black] (A62) at (0,11.5-2.5) {};
\node[fill=black] (A41) at (0,7.25-1.5) {};
\node[fill=black] (A61) at (0,11.25-2.5) {};
\node[fill=black] (A43) at (0,7.75-1.5) {};
\node[fill=black] (A63) at (0,11.75-2.5) {};
\endpgfonlayer

\pgfonlayer{1}
\draw[blue,double distance=1] (A42) arc (-90:90:1.5);
\draw[green] (A41) arc (-90:90:2.25-0.5);
\draw[green] (A43) arc (-90:90:1.75-0.5);
\endpgfonlayer
\draw[green] (A11).. controls ([xshift=-0.5cm]A11) ..(-0.5,0);
\draw[green] (A33).. controls ([xshift=-0.5cm]A33) and ([xshift=-0.5cm]A41) ..(A41);
\draw[green] (A13).. controls ([xshift=-0.5cm]A13) and ([xshift=-0.5cm]A21) ..(A21);
\draw[green] (A31).. controls ([xshift=-0.5cm]A31) and ([xshift=-0.5cm]A23) ..(A23);
\draw[green] (A43).. controls ([xshift=-0.5cm]A43) and ([xshift=-0.5cm]A51) ..(A51);
\draw[green] (A61).. controls ([xshift=-0.5cm]A61) and ([xshift=-0.5cm]A53) ..(A53);
\draw[green] (A11).. controls ([xshift=-0.5cm]A11) ..(-0.5,0);
\draw[green] (A63).. controls ([xshift=-0.5cm]A63) ..(-0.5,10);

\draw[blue,double distance=1] (A42) -- +(-1,0);
\draw[blue,double distance=1] (A62) -- +(-1,0);

\draw[red,double distance=1] (A12) -- +(-1,0);
\draw[red,double distance=1] (A32) -- +(-1,0);

\pgfonlayer{2}
\draw[red,double distance=1] (A12) arc (-90:90:1.5);
\draw[green] (A11) arc (-90:90:2.25-0.5);
\draw[green] (A13) arc (-90:90:1.75-0.5);
\endpgfonlayer

\pgfonlayer{2}

\node[anchor=west] at (0,4.75) {$\uparrow$};
\endpgfonlayer

\draw[blue,double distance=2pt] (0,2.5) -- +(-1,0);
\draw[blue,double distance=2pt] (0,7.5) -- +(-1,0);

\end{tikzpicture}}

\def\IRdesaa#1{\begin{tikzpicture}[every node/.style={shape=circle},inner sep =0, scale=#1,anchor=base,baseline=(current bounding box.center)]

\pgfdeclarelayer{1}
\pgfdeclarelayer{2}
\pgfsetlayers{main,1,2}
\draw (0,0) -- (0,0.5);
 \draw (0,1.5) -- (0,2); 
 \draw (0,3) -- (0,3.5); 
 \draw (0,4.5) -- (0,5); 
 \draw (0,6) -- (0,6.5); 
 \draw (0,7.5) -- (0,8.5); 
 \draw (0,9.5) -- (0,10);
\draw (-1,0) -- (-1,10);
\node[] at (0,2.3) {$\vdots $};
\node[] at (0,5.3) {$\vdots $};
\node at (-2,5) {$E$};

\draw[dashed] (-2.5,0) -- (5,0);
\draw[dashed] (-2.5,10) -- (5,10);
\draw[dashed] (-2.5,0) -- (-2.5,10);
\draw[dashed] (5,0) -- (5,10);

\pgfonlayer{2}

\draw[white,line width=#1*10mm] (0,1)   arc (-90:90:4cm);
\draw[] (0,0.5) arc (-90:90:4.5cm);
\draw[] (0,1.5) arc (-90:90:3.5cm);

\node[fill=black] (A11) at (0,1.5-0.5) {};
\node[fill=black] (A61) at (0,11.5-2.5) {};
\endpgfonlayer
\pgfonlayer{1}

\draw[white,line width=#1*10mm] (0,4)   arc (-90:90:1.5cm);
\draw[] (0,3.5) arc (-90:90:2cm);
\draw[] (0,4.5) arc (-90:90:1cm);

\node[fill=black] (A31) at (0,5.3333333333-1.5) {};
\node[fill=black] (A51) at (0,9.333333333-2.5) {};
\node[] (A32) at (0,5.6666666667-1.5) {};
\node[fill=black] (A52) at (0,9.666666667-2.5) {};
\endpgfonlayer

\pgfonlayer{1}
\draw[blue,double distance=1pt] (A32) arc (-90:90:1.833333333-0.5);
\endpgfonlayer

\draw[blue,double distance=1pt] (A32) -- +(-1,0);
\draw[blue,double distance=1pt] (A51) -- +(-1,0);

\pgfonlayer{1}
\draw[red,double distance=1pt] (A11) -- +(-1,0);
\draw[red,double distance=1pt] (A31) -- +(-1,0);
\draw[red,double distance=1pt] (A61).. controls ([xshift=-1cm]A61) and ([xshift=-1cm]A52) .. (A52);
\endpgfonlayer

\pgfonlayer{1}
\draw[red,double distance=1pt] (A52) arc (90:-90:2.1666667-0.5);
\endpgfonlayer

\draw[red,double distance=1pt] (A11) -- +(-1,0);
\draw[red,double distance=1pt] (A31) -- +(-1,0);

\pgfonlayer{2}
\draw[red,double distance=1pt] (A11) arc (-90:90:4);
\endpgfonlayer

\pgfonlayer{2}
\endpgfonlayer

\draw[blue,double distance=2pt] (0,2.5) -- +(-1,0);
\draw[blue,double distance=2pt] (0,5.5) -- +(-1,0);

\end{tikzpicture}}

\def\IRdesaaa#1{\begin{tikzpicture}[every node/.style={shape=circle},inner sep =0, scale=#1,anchor=base,baseline=(current bounding box.center)]

\pgfdeclarelayer{1}
\pgfdeclarelayer{2}
\pgfsetlayers{main,1,2}
\draw (0,0) -- (0,0.5);
 \draw (0,1.5) -- (0,2); 
 \draw (0,3) -- (0,3.5); 
 \draw (0,4.5) -- (0,5); 
 \draw (0,6) -- (0,6.5); 
 \draw (0,7.5) -- (0,8.5); 
 \draw (0,9.5) -- (0,10);
\draw (-1,0) -- (-1,10);
\node[] at (0,2.3) {$\vdots $};
\node[] at (0,5.3) {$\vdots $};
\node at (-2,5) {$E$};

\draw[dashed] (-2.5,0) -- (5,0);
\draw[dashed] (-2.5,10) -- (5,10);
\draw[dashed] (-2.5,0) -- (-2.5,10);
\draw[dashed] (5,0) -- (5,10);

\pgfonlayer{2}

\draw[white,line width=#1*10mm] (0,1)   arc (-90:90:4cm);
\draw[] (0,0.5) arc (-90:90:4.5cm);
\draw[] (0,1.5) arc (-90:90:3.5cm);

\node[fill=black] (A12) at (0,1.5-0.5) {};
\node[fill=black] (A62) at (0,11.5-2.5) {};
\node[fill=black] (A11) at (0,1.25-0.5) {};
\node[fill=black] (A61) at (0,11.25-2.5) {};
\node[fill=black] (A13) at (0,1.75-0.5) {};
\node[fill=black] (A63) at (0,11.75-2.5) {};
\endpgfonlayer

\node[] (A21) at (0,3.25-1) {};
\node[] (A23) at (0,3.75-1) {};
\node[] (A41) at (0,7.25-2) {};
\node[] (A43) at (0,7.75-2) {};

\pgfonlayer{1}

\draw[white,line width=#1*10mm] (0,4)   arc (-90:90:1.5cm);
\draw[] (0,3.5) arc (-90:90:2cm);
\draw[] (0,4.5) arc (-90:90:1cm);

\node[fill=black] (A31) at (0,5.2-1.5) {};
\node[fill=black] (A51) at (0,9.2-2.5) {};
\node[fill=black] (A32) at (0,5.4-1.5) {};
\node[fill=black] (A52) at (0,9.4-2.5) {};
\node[fill=black] (A33) at (0,5.6-1.5) {};
\node[fill=black] (A53) at (0,9.6-2.5) {};
\node[fill=black] (A34) at (0,5.8-1.5) {};
\node[fill=black] (A54) at (0,9.8-2.5) {};

\endpgfonlayer

\pgfonlayer{1}
\draw[blue,double distance=0.5pt] (A33) arc (-90:90:1.9-0.5);
\endpgfonlayer

\draw[blue,double distance=0.5pt] (A33) -- +(-1,0);
\draw[blue,double distance=0.5pt] (A52) -- +(-1,0);

\pgfonlayer{1}
\draw[red,double distance=0.5pt] (A12) -- +(-1,0);
\draw[red,double distance=0.5pt] (A32) -- +(-1,0);
\draw[red,double distance=0.5pt] (A62).. controls ([xshift=-1cm]A62) and ([xshift=-1cm]A53) .. (A53);
\endpgfonlayer

\pgfonlayer{1}
\draw[red,double distance=0.5pt] (A53) arc (90:-90:2.1-0.5);
\draw[green] (A31) arc (-90:90:2.3-0.5);
\draw[green] (A34) arc (-90:90:1.7-0.5);
\endpgfonlayer

\pgfonlayer{2}
\draw[red,double distance=0.5pt] (A12) arc (-90:90:5-1);
\draw[green] (A11) arc (-90:90:5.25-1);
\draw[green] (A13) arc (-90:90:4.75-1);
\endpgfonlayer

\draw[green] (A11).. controls ([xshift=-0.5cm]A11) ..(-0.5,0);
\draw[green] (A34).. controls ([xshift=-0.5cm]A34) and ([xshift=-0.5cm]A41) ..(A41);
\draw[green] (A13).. controls ([xshift=-0.5cm]A13) and ([xshift=-0.5cm]A21) ..(A21);
\draw[green] (A31).. controls ([xshift=-0.5cm]A31) and ([xshift=-0.5cm]A23) ..(A23);
\draw[green] (A43).. controls ([xshift=-0.5cm]A43) and ([xshift=-0.5cm]A51) ..(A51);
\draw[green] (A61).. controls ([xshift=-0.5cm]A61) and ([xshift=-0.5cm]A54) ..(A54);
\draw[green] (A11).. controls ([xshift=-0.5cm]A11) ..(-0.5,0);
\draw[green] (A63).. controls ([xshift=-0.5cm]A63) ..(-0.5,10);

\pgfonlayer{2}
\endpgfonlayer

\draw[blue,double distance=2pt] (0,3.5-1) -- +(-1,0);
\draw[blue,double distance=2pt] (0,7.5-2) -- +(-1,0);

\end{tikzpicture}}

\def\IRdesaba#1{\begin{tikzpicture}[every node/.style={shape=circle},inner sep =0, scale=#1,anchor=base,baseline=(current bounding box.center)]

\pgfdeclarelayer{1}
\pgfdeclarelayer{2}
\pgfsetlayers{main,1,2}
\draw (0,0) -- (0,0.5);
 \draw (0,1.5) -- (0,2); 
 \draw (0,3) -- (0,3.5); 
 \draw (0,4.5) -- (0,5); 
 \draw (0,6) -- (0,6.5); 
 \draw (0,7.5) -- (0,8.5); 
 \draw (0,9.5) -- (0,10);
\draw (-1,0) -- (-1,10);
\node[] at (0,2.3) {$\vdots $};
\node[] at (0,5.3) {$\vdots $};
\node at (-2,5) {$E$};

\draw[dashed] (-2.5,0) -- (5,0);
\draw[dashed] (-2.5,10) -- (5,10);
\draw[dashed] (-2.5,0) -- (-2.5,10);
\draw[dashed] (5,0) -- (5,10);

\pgfonlayer{2}

\draw[white,line width=#1*10mm] (0,1)   arc (-90:90:4cm);
\draw[] (0,0.5) arc (-90:90:4.5cm);
\draw[] (0,1.5) arc (-90:90:3.5cm);

\node[fill=black] (A12) at (0,1.5-0.5) {};
\node[fill=black] (A62) at (0,11.5-2.5) {};
\node[fill=black] (A11) at (0,1.25-0.5) {};
\node[fill=black] (A61) at (0,11.25-2.5) {};
\node[fill=black] (A13) at (0,1.75-0.5) {};
\node[fill=black] (A63) at (0,11.75-2.5) {};
\endpgfonlayer

\node[] (A21) at (0,3.25-1) {};
\node[] (A23) at (0,3.75-1) {};
\node[] (A41) at (0,7.25-2) {};
\node[] (A43) at (0,7.75-2) {};

\pgfonlayer{1}

\draw[white,line width=#1*10mm] (0,4)   arc (-90:90:1.5cm);
\draw[] (0,3.5) arc (-90:90:2cm);
\draw[] (0,4.5) arc (-90:90:1cm);

\node[fill=black] (A31) at (0,5+1/7-1.5) {};
\node[fill=black] (A51) at (0,9+1/7-2.5) {};
\node[fill=black] (A32) at (0,5+2/7-1.5) {};
\node[fill=black] (A52) at (0,9+2/7-2.5) {};
\node[fill=black] (A33) at (0,5+3/7-1.5) {};
\node[fill=black] (A53) at (0,9+3/7-2.5) {};
\node[fill=black] (A34) at (0,5+4/7-1.5) {};
\node[fill=black] (A54) at (0,9+4/7-2.5) {};
\node[fill=black] (A35) at (0,5+5/7-1.5) {};
\node[fill=black] (A55) at (0,9+5/7-2.5) {};
\node[fill=black] (A36) at (0,5+6/7-1.5) {};
\node[fill=black] (A56) at (0,9+6/7-2.5) {};

\endpgfonlayer

\draw[blue,double distance=0.5pt] (A35) -- +(-1,0);
\draw[blue,double distance=0.5pt] (A52) -- +(-1,0);

\pgfonlayer{1}
\draw[red,double distance=0.5pt] (A12) -- +(-1,0);
\draw[red,double distance=0.5pt] (A32) -- +(-1,0);
\draw[red,double distance=0.5pt] (A62).. controls ([xshift=-0.75cm]A62) and ([xshift=-0.75cm]A55) .. (A55);
\endpgfonlayer

\pgfonlayer{1}
\draw[blue,double distance=0.5pt] (A35) arc (-90:90:1.78571428571-0.5);
\draw[red,double distance=0.5pt] (A32) arc (-90:90:2.21428571429-0.5);
\draw[green] (A31) arc (-90:90:2.35714285714-0.5);
\draw[green] (A36) arc (-90:90:1.64285714286-0.5);
\draw[green] (A33) arc (-90:90:2.07142857143-0.5);
\draw[green] (A34) arc (-90:90:1.92857142857-0.5);
\endpgfonlayer

\pgfonlayer{2}
\draw[red,double distance=0.5pt] (A12) arc (-90:90:4);
\draw[green] (A11) arc (-90:90:5.25-1);
\draw[green] (A13) arc (-90:90:4.75-1);
\endpgfonlayer

\draw[green] (A11).. controls ([xshift=-0.5cm]A11) ..(-0.5,0);
\draw[green] (A36).. controls ([xshift=-0.5cm]A36) and ([xshift=-0.5cm]A41) ..(A41);
\draw[green] (A13).. controls ([xshift=-0.5cm]A13) and ([xshift=-0.5cm]A21) ..(A21);
\draw[green] (A31).. controls ([xshift=-0.5cm]A31) and ([xshift=-0.5cm]A23) ..(A23);
\draw[green] (A43).. controls ([xshift=-0.5cm]A43) and ([xshift=-0.5cm]A51) ..(A51);
\draw[green] (A33).. controls ([xshift=-0.3cm]A33) and ([xshift=-0.3cm]A34) ..(A34);
\draw[green] (A61).. controls ([xshift=-0.5cm]A61) and ([xshift=-0.5cm]A56) ..(A56);
\draw[green] (A63).. controls ([xshift=-1cm]A63) and ([xshift=-1cm]A54) .. (A54);
\draw[green] (A11).. controls ([xshift=-0.5cm]A11) ..(-0.5,0);
\draw[green] (A53).. controls ([xshift=-0.9cm]A53) .. (-0.9,8);
\draw[green] (-0.9,8).. controls (-0.9,9) and (-0.5,9) ..(-0.5,10);

\pgfonlayer{2}
\endpgfonlayer

\draw[blue,double distance=2pt] (0,2.5) -- +(-1,0);
\draw[blue,double distance=2pt] (0,5.5) -- +(-1,0);

\end{tikzpicture}}

\def\UU#1{\begin{tikzpicture}[every node/.style={shape=circle},inner sep =0, scale=#1,anchor=base,baseline=#1*0.5cm]

\draw[] (-1,0)--(0,0);
\draw[] (-1,0)--(-1,1);
\draw[] (-1,1)--(0,1);
\draw[] (0,0) --(0,1);

\node[fill=black] (N1) at (-0.8888888888888888,1) {};
\node[fill=black] (N2) at (-0.7777777777777778,1) {};
\node[fill=black] (N3) at (-0.6666666666666667,1) {};
\node[fill=black] (N4) at (-0.5555555555555556,1) {};
\node[fill=black] (N5) at (-0.4444444444444444,1) {};
\node[fill=black] (N6) at (-0.33333333333333337,1) {};
\node[fill=black] (N7) at (-0.2222222222222222,1) {};
\node[fill=black] (N8) at (-0.11111111111111116,1) {};

\draw[very thick, dotted] (-0.7777777777777778+0.03,1-0.03) -- (-0.5555555555555556-+0.03,1-0.03);
\draw[very thick, dotted] (-0.4444444444444444+0.03,1-0.03) -- (-0.2222222222222222-+0.03,1-0.03);

\draw[thick,decorate,decoration={brace,amplitude=5pt}] (N1.north)--(N4.north) node[pos=0.5,anchor=south,inner sep=5pt] {$n$};

\draw[blue] (N1).. controls (-0.8888888888888888,1-7/9) and (-0.11111111111111116,1-7/9) .. (N8);

\draw[blue] (N2).. controls (-0.7777777777777778,1-5/9) and (-0.2222222222222222,1-5/9) .. (N7);

\draw[blue] (N5).. controls (-0.4444444444444444,8/9) and (-0.5555555555555556,8/9) .. (N4);

\end{tikzpicture}}

\def\VV#1{\begin{tikzpicture}[every node/.style={shape=circle},inner sep =0, scale=#1,anchor=base,baseline=#1*0.5cm]

\draw[] (-1,0)--(0,0);
\draw[] (-1,0)--(-1,1);
\draw[] (-1,1)--(0,1);
\draw[] (0,0) --(0,1);

\node[fill=black] (S1) at (-0.8888888888888888,0) {};
\node[fill=black] (S2) at (-0.7777777777777778,0) {};
\node[fill=black] (S3) at (-0.6666666666666667,0) {};
\node[fill=black] (S4) at (-0.5555555555555556,0) {};
\node[fill=black] (S5) at (-0.4444444444444444,0) {};
\node[fill=black] (S6) at (-0.33333333333333337,0) {};
\node[fill=black] (S7) at (-0.2222222222222222,0) {};
\node[fill=black] (S8) at (-0.11111111111111116,0) {};

\draw[very thick, dotted] (-0.7777777777777778+0.03,0.03) -- (-0.5555555555555556-+0.03,0.03);
\draw[very thick, dotted] (-0.4444444444444444+0.03,0.03) -- (-0.2222222222222222-+0.03,0.03);

\draw[thick,decorate,decoration={brace,amplitude=5pt,mirror}] (S1.south)--(S4.south) node[pos=0.5,anchor=north,inner sep=5pt] {$n$};

\draw[blue] (S4).. controls (-0.5555555555555556,1/9) and (-0.4444444444444444,1/9) .. (S5);

\draw[blue] (S7).. controls (-0.2222222222222222,5/9) and (-0.7777777777777778,5/9) .. (S2);

\draw[blue] (S1).. controls (-0.8888888888888888,7/9) and (-0.11111111111111116,7/9) .. (S8);

\end{tikzpicture}}

\def\II#1{\begin{tikzpicture}[every node/.style={shape=circle},inner sep =0, scale=#1,anchor=base,baseline=#1*0.5cm]

\draw[] (-1,0)--(0,0);
\draw[] (-1,0)--(-1,1);
\draw[] (-1,1)--(0,1);
\draw[] (0,0) --(0,1);

\node[fill=black] (S1) at (-0.8,0) {};
\node[fill=black] (S2) at (-0.6,0) {};
\node[fill=black] (S3) at (-0.4,0) {};
\node[fill=black] (S4) at (-0.19999999999999996,0) {};

\node[fill=black] (N1) at (-0.8,1) {};
\node[fill=black] (N2) at (-0.6,1) {};
\node[fill=black] (N3) at (-0.4,1) {};
\node[fill=black] (N4) at (-0.19999999999999996,1) {};

\draw[very thick, dotted] (-0.6+0.1,1-0.1) -- (-0.2-+0.1,1-0.1);
\draw[very thick, dotted] (-0.6+0.1,0.1) -- (-0.2-+0.1,0.1);

\draw[thick,decorate,decoration={brace,amplitude=5pt}] (N1.north)--(N4.north) node[pos=0.5,anchor=south,inner sep=5pt] {$n$};

\draw[blue] (S4).. controls (-0.19999999999999996,0.5) and (-0.19999999999999996,0.5).. (N4);

\draw[blue] (S1).. controls (-0.8,0.5) and (-0.8,0.5).. (N1);

\draw[blue] (S2).. controls (-0.6,0.5) and (-0.6,0.5).. (N2);

\end{tikzpicture}}

\def\TTpic#1{\begin{tikzpicture}[every node/.style={shape=circle},inner sep =0, scale=#1,anchor=base,baseline=#1*4.5cm]  
\pgfdeclarelayer{1}
\pgfdeclarelayer{2}

\pgfsetlayers{main,
1,
2} 

\draw[] (-9,0)--(0,0);
\draw[] (-9,0)--(-9,9);
\draw[] (-9,9)--(0,9);
\draw[] (0,0) --(0,1);
\draw[] (0,2) -- (0,3); 
\node[fill=black,inner sep =0] (A1) at (0,1) {};
\node[fill=black,inner sep =0] (B1) at (0,2) {};
\draw[] (0,4) -- (0,5); 
\node[fill=black,inner sep =0] (A2) at (0,3) {};
\node[fill=black,inner sep =0] (B2) at (0,4) {};
\draw[] (0,6) -- (0,7); 
\node[fill=black,inner sep =0] (A3) at (0,5) {};
\node[fill=black,inner sep =0] (B3) at (0,6) {};
\draw[] (0,8) -- (0,9); 
\node[fill=black,inner sep =0] (A4) at (0,7) {};
\node[fill=black,inner sep =0] (B4) at (0,8) {};

\node[fill=black] (S1) at (-8.0,0) {};
\node[fill=black] (S2) at (-7.0,0) {};

\node[] (S22) at (9*0.05-7,5*0.05) {};

\node[] (S44) at (-9*0.05-5,5*0.05) {};

\node[fill=black] (S4) at (-5.0,0) {};
\node[fill=black] (S5) at (-4.0,0) {};
\node[fill=black] (S6) at (-3.0000000000000004,0) {};

\node[] (S66) at (9*0.05-3,5*0.05) {};

\node[] (S88) at (-9*0.05-1,5*0.05) {};

\node[fill=black] (S8) at (-1.0000000000000004,0) {};

\draw[very thick, dotted] (S22) -- (S44);
\draw[very thick, dotted] (S66) -- (S88);

\node[fill=black] (N1) at (-8.0,9) {};
\node[fill=black] (N2) at (-7.0,9) {};

\node[] (N22) at (9*0.05-7,9-5*0.05) {};

\node[] (N44) at (-9*0.05-5,9-5*0.05) {};

\node[fill=black] (N4) at (-5.0,9) {};
\node[fill=black] (N5) at (-4.0,9) {};
\node[fill=black] (N6) at (-3.0000000000000004,9) {};

\node[] (N66) at (9*0.05-3,9-5*0.05) {};

\node[] (N88) at (-9*0.05-1,9-5*0.05) {};

\node[fill=black] (N8) at (-1.0000000000000004,9) {};

\draw[very thick, dotted] (N22) -- (N44);
\draw[very thick, dotted] (N66) -- (N88);

\pgfonlayer{2}

\draw[white,line width=#1*10mm] (0,1.5)   arc (-90:90:2cm);
\draw[] (0,1) arc (-90:90:2.5cm);
\draw[] (0,2) arc (-90:90:1.5cm);

\node[fill=black] (A11) at (0,1.2) {};
\node[fill=black] (A31) at (0,5.2) {};
\node[fill=black] (A12) at (0,1.4) {};
\node[fill=black] (A32) at (0,5.4) {};
\node[fill=black] (A13) at (0,1.6) {};
\node[fill=black] (A33) at (0,5.6) {};
\node[fill=black] (A14) at (0,1.8) {};
\node[fill=black] (A34) at (0,5.8) {};

\endpgfonlayer

\pgfonlayer{1}

\draw[white,line width=#1*10mm] (0,3.5)   arc (-90:90:2cm);
\draw[] (0,3) arc (-90:90:2.5cm);
\draw[] (0,4) arc (-90:90:1.5cm);

\node[fill=black] (A21) at (0,3.2) {};
\node[fill=black] (A41) at (0,7.2) {};
\node[fill=black] (A22) at (0,3.4) {};
\node[fill=black] (A42) at (0,7.4) {};
\node[fill=black] (A23) at (0,3.6) {};
\node[fill=black] (A43) at (0,7.6) {};
\node[fill=black] (A24) at (0,3.8) {};
\node[fill=black] (A44) at (0,7.8) {};

\endpgfonlayer

\draw[red] (N4).. controls (-5.0,6.866666666666667) and (-3.3333333333333335,5.8) .. (A34);
\draw[red] (S8).. controls (-1.0000000000000004,0.7999999999999999) and (-0.666666666666667,1.2) .. (A11);

\pgfonlayer{2}
\draw[red] (A34) arc (90:-90:2.3);
\endpgfonlayer

\draw[blue] (S1).. controls (-8.0,2.533333333333333) and (-5.333333333333333,3.8) .. (A24);
\draw[blue] (N5).. controls (-4.0,7.8) and (-2.6666666666666665,7.2) .. (A41);

\pgfonlayer{1}
\draw[blue] (A24) arc (-90:90:1.7);
\endpgfonlayer

\draw[red] (N2).. controls (-7.0,6.6000000000000005) and (-4.666666666666667,5.4) .. (A32);
\draw[red] (S6).. controls (-3.0000000000000004,1.0666666666666667) and (-2.0000000000000004,1.6) .. (A13);

\pgfonlayer{2}
\draw[red] (A32) arc (90:-90:1.9);
\endpgfonlayer

\draw[red] (S5).. controls (-4.0,1.2) and (-2.6666666666666665,1.8) .. (A14);
\draw[red] (N1).. controls (-8.0,6.466666666666667) and (-5.333333333333333,5.2) .. (A31);

\pgfonlayer{2}
\draw[red] (A14) arc (-90:90:1.7);
\endpgfonlayer

\draw[blue] (S2).. controls (-7.0,2.4) and (-4.666666666666667,3.6) .. (A23);
\draw[blue] (N6).. controls (-3.0000000000000004,7.933333333333334) and (-2.0000000000000004,7.4) .. (A42);

\pgfonlayer{1}
\draw[blue] (A23) arc (-90:90:1.9);
\endpgfonlayer

\draw[blue] (S4).. controls (-5.0,2.1333333333333333) and (-3.3333333333333335,3.2) .. (A21);
\draw[blue] (N8).. controls (-1.0000000000000004,8.2) and (-0.666666666666667,7.8) .. (A44);

\pgfonlayer{1}
\draw[blue] (A21) arc (-90:90:2.3);
\endpgfonlayer

\draw[thick,decorate,decoration={brace,amplitude=5pt,mirror}] (S1.south)--(S4.south) node[pos=0.5,anchor=north,inner sep=5pt] {$l$};

\draw[thick,decorate,decoration={brace,amplitude=5pt,mirror}] (S5.south)--(S8.south) node[pos=0.5,anchor=north,inner sep=5pt] {$m$};

\end{tikzpicture}}

\def\MMpic#1{\begin{tikzpicture}[every node/.style={shape=circle},inner sep =0, scale=#1,anchor=base,baseline=#1*2.5cm]  
\pgfdeclarelayer{1}

\pgfsetlayers{main,
1} 

\draw[] (-5,0)--(0,0);
\draw[] (-5,0)--(-5,5);
\draw[] (-5,5)--(0,5);
\draw[] (0,0) --(0,1);
\draw[] (0,2) -- (0,3); 
\node[fill=black,inner sep =0] (A1) at (0,1) {};
\node[fill=black,inner sep =0] (B1) at (0,2) {};
\draw[] (0,4) -- (0,5); 
\node[fill=black,inner sep =0] (A2) at (0,3) {};
\node[fill=black,inner sep =0] (B2) at (0,4) {};

\node[fill=black] (S1) at (-4.0,0) {};
\node[fill=black] (S2) at (-3.0,0) {};

\node[] (S22) at (-3.0+5*0.05,0.05*3) {};

\node[] (S44) at (-1-5*0.05,0.05*3) {};

\node[fill=black] (S4) at (-0.9999999999999998,0) {};

\draw[very thick, dotted] (S22) -- (S44);

\node[fill=black] (N1) at (-4.0,5) {};
\node[fill=black] (N3) at (-2.0,5) {};

\node[] (N11) at (-4.0+5*0.05,5-0.05*3) {};

\node[] (N33) at (-2-5*0.05,5-0.05*3) {};

\node[fill=black] (N4) at (-0.9999999999999998,5) {};

\draw[very thick, dotted] (N11) -- (N33);

\pgfonlayer{1}

\draw[white,line width=#1*5mm] (0,1.5)   arc (-90:90:1cm);
\draw[white,line width=#1*5mm] (0,1.75)   arc (-90:90:1cm);
\draw[white,line width=#1*5mm] (0,1.25)   arc (-90:90:1cm);
\draw[] (0,1) arc (-90:90:1cm);
\draw[] (0,2) arc (-90:90:1cm);
\draw[] (1,2)--(1,3);

\node[fill=black] (A11) at (0,1.2) {};
\node[fill=black] (A21) at (0,3.2) {};
\node[fill=black] (A13) at (0,1.6) {};
\node[fill=black] (A23) at (0,3.6) {};
\node[fill=black] (A14) at (0,1.8) {};
\node[fill=black] (A24) at (0,3.8) {};

\draw[thick,dotted] (A11) -- (A13);
\draw[thick,dotted] (A21) -- (A23);

\endpgfonlayer

\draw[blue] (N1).. controls (-4.0,3.8000000000000003) and (-2.6666666666666665,3.2) .. (A21);
\draw[blue] (S4).. controls (-0.9999999999999998,0.7999999999999999) and (-0.6666666666666665,1.2) .. (A11);

\pgfonlayer{1}
\draw[blue] (A21) arc (90:-90:1.0);
\endpgfonlayer

\draw[blue] (N3).. controls (-2.0,4.066666666666666) and (-1.3333333333333333,3.6) .. (A23);
\draw[blue] (S2).. controls (-3.0,1.0666666666666667) and (-2.0,1.6) .. (A13);

\pgfonlayer{1}
\draw[blue] (A23) arc (90:-90:1.0);
\endpgfonlayer

\draw[blue] (S1).. controls (-4.0,1.2) and (-2.6666666666666665,1.8) .. (A14);
\draw[blue] (N4).. controls (-0.9999999999999998,4.2) and (-0.6666666666666665,3.8) .. (A24);

\pgfonlayer{1}
\draw[blue] (A14) arc (-90:90:1.0);
\endpgfonlayer

\draw[thick,decorate,decoration={brace,amplitude=5pt,mirror}] (S1.south)--(S4.south) node[pos=0.5,anchor=north,inner sep=5pt] {$n$};

\end{tikzpicture}}

\title{Temperley-Lieb Categories on Non-orientable Surfaces}
\date{\today}

\author{Dionne Ibarra}
\address{School of Mathematics, Monash University, VIC 3800, Australia}
\email{{\rm \textcolor{blue}{dionne.ibarra@monash.edu}}}

\author{Gabriel Montoya-Vega}
\address{Department of Mathematics, University of Puerto Rico-R\'io Piedras, San Juan, PR, USA}
\email{{\rm \textcolor{blue}{gabriel.montoya@upr.edu}}}

\author{Benjamin Morris}
\address{School of Mathematics, University of Leeds, LS2 9JT, United Kingdom}
\email{{\rm \textcolor{blue}{mmbajm@leeds.ac.uk}}}

\begin{document}
\begin{abstract}
   In this paper we present the construction of a skeletal diagram category, which we call the square with bands category. This category extends the Temperley-Lieb (TL) category, where morphisms now include diagrams of embedded curves on (possibly) non-orientable bounded surfaces, and involves three parameters associated to simple closed curves. Such diagrams utilise handle decompositions for surfaces and are considered up to a handle slide equivalence. We define a tensor product on this category, extending the well-known tensor product on the TL category, and a full set of monoidal generators is given, which includes the TL generators, a family of orientable genus one diagrams, and a family of non-orientable diagrams.
   This document constitutes an initial draft of ongoing research with preliminary reporting of some results in the last section; a subsequent version including a detailed introduction and full proofs will follow.
\end{abstract}
\maketitle

\tableofcontents 

\section{Introduction} The main result in this paper is the construction of the square with bands (SWB) category, $\SQ$, given in Theorem \ref{SWBcat}. This is a skeletal category, whose objects are the natural numbers, and whose morphisms are constructed using combinatorial encodings of embedded curves in (once) bounded surfaces, known as SWB data, which are introduced in \S3. The latter describe diagrams obtained from a square by attaching two dimensional one-handles (or bands) to the right hand edge, such as in Example \ref{Ex:handleslide} (given below also). Such diagrams are then considered up to an equivalence, which we refer to as handle slide equivalence. The example below contains a sequence of transformations allowed under this equivalence. Such diagrams include Temperley-Lieb diagrams \cite{TemperleyLieb, JonesIndex}, that is, SWB diagrams with no handles attached. 

\begin{named}{Example \ref{Ex:handleslide}}
    \[\exaahs{0.4} \ \mapsto \ \exabhs{0.4}\overset{\iso}{\sim}\exachs{0.4}\]
\end{named}

In \S2 we present two preliminary definitions and notions used to define SWB data and their equivalences. In particular, we use the elementary (but possibly non-standard) notion of ``graph contraction" defined in \S2.1, and we use twisted chord diagrams, defined in \S2.2, to encode such realisations of bounded surfaces. In \S3 we define SWB data, an operation coined ``vertical juxtaposition" which prefigures a categorical composition, and the notion of handle slide equivalence. In the last subsection we define two operations, which add boundary adjacent curves to a diagram, and which will later (in \S4) be used to equip $\SQ$ with a tensor product. In \S4 we use definitions and results from \S3 to define the category $\SQ$, equip it with a tensor product, and provide it with a monoidal generating set.

\section{Preliminaries}
This section provides preliminary material, such as partitions, (ordered) graphs, and twisted chord diagrams, that are used to provide a combinatorial foundation for this paper.

Throughout this paper, we use the conventions
\[\underline{N}=\{1,2,\dots, N\}, \qquad \underline{(N+1/2)}=\{1/2,3/2,\dots, N+1/2\},\]
for $N$ a positive integer.

\mdef\label{notn}For $S$ any set, let $\mathcal{P}(S)$ denote the power set of $S$, that is, the set of all subsets of $S$. We denote by $\text{Perm}(S)$ the group of all permutations $f:S\to S$, and we use the short hand $\SS_N=\text{Perm}(\underline{N})$. Given a function $f:S\to S'$ for any set $S'$, we abuse notation and write $f:\mathcal{P}(S)\to \mathcal{P}(S')$ for the function $U\mapsto f(U):=\{f(u) \ | \ u\in U\}$. In particular, when $f\in \text{Perm}(S)$, this defines a map $f\in \text{Perm}(\mathcal{P}(S))$. 

Let $P_S\subset \mathcal{P}(\mathcal{P}(S))$ denote the set of all set-partitions of $S$ and let $B_S\subset P_S$ denote the set of all set-partitions of $S$ 
into subsets of order 2. We will refer to elements of $B_S$ as \textbf{pair-partitions} of $S$. For a partition of $S$, $P\in P_S$, we refer to elements $p\in P$ (which are subsets of $S$) as parts of $P$.

\mdef\label{order} Let $(S,\prec)$ be a totally ordered set. For a pair-partition of $S$, $P\in B_S$, define maps $m,M:P\to S$ by $m(p)=\min_{\prec}(p)$ and $M(p)=\max_{\prec}(p)$. For parts $p,q\in P$ we say that $p$ and $q$ \textbf{cross} or are \textbf{crossing} (with respect to $\prec$) if one of the following holds
\[m(p)\prec m(q)\prec M(p)\prec M(q), \qquad m(q)\prec m(p)\prec M(q)\prec M(p). \]
We call $P$ \textbf{crossing-less} if all its parts are pairwise non-crossing.

For a subset $S'\subset S$, let $(S',\prec)$ denote the unique total order inherited from $(S,\prec)$. Unless otherwise stated we will assume all subsets of totally ordered sets are ordered in this way. Furthermore, we will assume that $\mathbb{R}$ is given the unique total order $<$ such that $(\mathbb{R},<)$ is an ordered field.

\subsection{Graphs}

Here we collect some basic (but possibly non-standard) notions from graph theory that are used later.

 \mdef A (simple undirected) \textbf{graph} is a pair of sets $G=(V,E)$ where $E\subset \mathcal{P}(V)$ and $|e|=2$ for all $e\in E$ (\textit{i.e.} $E$ is a set of order two subsets of $V$). We refer to elements of $V$ as \textbf{vertices}, and elements of $E$ as \textbf{edges}. For a graph $G$, with unspecified vertex and edge sets we use the notation $G=(V(G),E(G))$. We often denote an edge $\{u,v\}\in E$ with the notation $u\edge v$. A graph $G$ is finite if $V(G)$ (and hence $E(G)$) is finite. We refer to the graph $(\emptyset,\emptyset)$ as the empty graph.

\medskip\noindent In what follows, let $G$ be a graph.

\mdef\label{inc} Let $v\in V(G)$ be a vertex. We say an edge $e\in E(G)$ is \textbf{incident} to $v$ if $v\in e$. For a vertex subset $U\subset V(G)$, we say that $e\in E(G)$ is incident to $U$ if $e$ is incident to some vertex $u\in U$, and we define
\[E(G)\cap U:=\{e \in E(G) \ | \ e \text{ is incident to } U\}\subset E(G).\]

\mdef\label{adj} We say that $v, u\in V(G)$ are \textbf{adjacent} (or $v$ is adjacent to $u$) if $\{u,v\}\in E$, and we say that $v$ is adjacent to a subset $U\subset V(G)$ if it is adjacent to some vertex $u\in U$. For $U\subset V(G)$ we define the \textbf{adjacency set} or \textbf{open neighbourhood}\footnote{we use the term adjacency set here to avoid confusion with the neighbourhood graph} of $U$ in $G$, $A_G(U)\subset V(G)$, as $A_G(U):=\{v \in V(G)\setminus U \ | \ v \text{ is adjacent to } U \}$.

\mdef\label{nbhd} For a graph $G'$, we say that $G'$ is a \textbf{subgraph} of $G$ (write $G'\subset G$) if $V(G')\subset V(G)$ and $E(G')\subset E(G)$. A subgraph $G'$ is said to be \textbf{full} if for any edge $\{u,v\}\in E(G)$, we have that $\{u,v\}\subset V(G')$ implies $\{u,v\}\in E(G')$. For a subset $U\subset V(G)$ we define $G[U]\subset G$ as the unique full subgraph of $G$ with vertex set $U$, and define the \textbf{neighbourhood graph} of $U$ in $G$, $N_G(U)\subset G$, as $N_G(U)=G[U\cup A_G(U)]$.

\mdef The \textbf{degree} or \textbf{valence} of a vertex $v\in V(G)$ is the integer 
\[d(v)=\left|\{ e\in E(G) \ | \ e \text{ incident to } v \}\right|.\]

\mdef We define the union and intersection of graphs as $G\cup G'=(V(G)\cup V(G') , E(G)\cup E(G'))$ and $G\cap G'=(V(G)\cap V(G'), E(G)\cap E(G'))$. We say that a union of graphs $G\cup G'$ is disjoint if $G\cap G'$ is the empty graph, and write $G\sqcup G'$.

\mdef An isomorphism of graphs $G$ and $G'$ is a bijection $f:V(G)\to V(G')$ such that $e \in E(G)$ if and only if $f(e)\in E(G')$. We often abuse notation and write $f:G\overset{\sim}{\to} G'$. Given an injective map $g:V(G)\to V'$, define $g(G)$ as the graph with vertex set $g(V(G))\subset V'$ and edge set $g(E(G))$. Note then, that $g:G\overset{\sim}{\to} g(G)$ is an isomorphism of graphs.

\mdef\label{path} A \textbf{path} is a finite graph of the form $P=\left(\{v_1,\dots, v_r \},\{ \{v_1,v_2\}, \dots, \{v_{r-1},v_r\}\}\right)$, for some $r\in \Z_{> 0}$, where the $v_i$ are pairwise distinct. For such a path $P$, we use the shorthand $P=v_1 \edge v_2 \edge \dots \edge v_r$ and we call $v_1$ and $v_r$ the \textbf{endpoints} of $P$ (if $r=1$ we use the convention $v_1\edge v_1=(\{v_1\},\emptyset)$). For $v, v'\in V(G)$, a path between $v$ and $v'$ in $G$, is a subgraph $P\subset G$ which is a path with endpoints $v$ and $v'$.

For a subset $U\subset V(G)$, and for vertices $v,v'\in A_G(U)$, a \textbf{path through} $U$ between $v$ and $v'$ is a path $P$ between $v$ and $v'$ in the subgraph $N_G(U)$, such that $V(P)\setminus \{v,v'\}\subset U$. We write $v -_{U} v'$ if there exists a path through $U$ between $v$ and $v'$.

\mdef Define a relation $\sim_G$ on $V(G)$, by $v\sim_G v'$ if there exists a path between $v$ and $v'$ in $G$. Then $\sim_G$ is an equivalence relation on $V(G)$ and we say that $G$ is \textbf{connected} if $\left|V(G)/\sim_G\right|=1$. A \textbf{component} of $G$ is a maximal connected subgraph of $G$. We denote
\[\mathcal{C}(G)=\{\text{ components } \Gamma\subset G \ \}.\]
There is a bijection $\mathcal{C}(G)\to V(G)/\sim_G$ given by $\Gamma \mapsto V(\Gamma)$. Furthermore $G=\sqcup_{\Gamma\in \mathcal{C}(G)}\Gamma$.
 
\mdef A \textbf{cycle} is a finite graph of the form $C=\left(\{v_1,\dots, v_r \},\{ \{v_1,v_2\},  \dots, \{v_{r-1},v_r\},\{v_r,v_1\}\}\right)$, for some $r\in \Z_{>2}$ where the $v_i$ are pairwise distinct. We can write cycles in path notation \textit{e.g.} $C=v_1\edge  v_2 \edge \dots \edge v_r \edge v_1$.

\begin{lemma}\label{grlemma}
    Suppose $G=(V,E)$ is a finite graph such that $d(v)\leq 2$ for all $v\in V$. Then $G$ is a disjoint union of paths and cycles. Furthermore, $V_{(d=1)}/\sim_{G}$ is a pair partition of $V_{(d=1)}$, where $V_{(d=1)}=\{v\in V \ | \ d(v)=1 \}$.
\end{lemma} 

\mdef For a graph $G'$ define the \textbf{deletion} of $G'$ from $G$, $G\smallsetminus G'$, as the minimal subgraph of $G$ such that $E(G\smallsetminus G')=E(G)\smallsetminus E(G')$ and $V(G)\smallsetminus V(G')\subset V(G\smallsetminus G')$.

\mdef\label{con1} For a subset $U \subset V(G)$, define the \textbf{contraction} of $G$ on $U$, $G/U$, as the graph
\[G/U= \left(G\smallsetminus N_G(U)\right)\cup \left(\cup_{u -_{U} v } \, u- v\right) ,  \]
\textit{i.e.} the right-most union is over pairs of distinct vertices $u,v\in A_G(U)$ such that there is a path through $U$ between them. Define the \textbf{index} of $U$ in $G$, $[U:G]$, as the integer $[U:G]=|\mathcal{C}(G)|-|\mathcal{C}(G/U)|$. Note that if $U=V(\Gamma)$ for some component $\Gamma \subset G$, then $G/U=G\smallsetminus \Gamma$.

\begin{prop}\label{contpr} Let $G=(V,E)$ be a graph and suppose $U\subset V$. Then 
\begin{enumerate}
    \item\label{conconn} $V(G/U)=V\smallsetminus U$ and for $v, v'\in V\smallsetminus U$ we have $v\sim_G v'$ if and only if $v \sim_{G/U} v'$.
    \item\label{conassoc} Let $W\subset V$ be such that $A_G(U)\cap A_G(W)=\emptyset$. Then $(G/U)/W=G/(U\cup W)$ and $[U: G]+[W: G/U]=[U\cup W:G]$.
    \item\label{conmap} Let $f:V\hookrightarrow V'$ be an injective map. Then we have $f(G/U)=f(G)/f(U)$ and $[f(U):f(G)]=[U:G]$.
    \item\label{conunion} Let $G'$ be a graph such that $V(G')\cap U=\emptyset$. Then $(G\cup G')/U=(G/U)\cup G'$ and $[U:G\cup G']=[U:G]$.
\end{enumerate}
\end{prop}

\begin{remark}\label{compmap} It follows from \ref{contpr} part \ref{conconn}, that the map $\{\Gamma \in \mathcal{C}(G) \ | \ V(\Gamma)\not\subset U \} \to \mathcal{C}(G/U)$, given by $\Gamma \mapsto \Gamma/(V(\Gamma)\cap U)$, is a bijection. In particular, $[U:G]\geq 0$ and if $[U:G]=0$ then this map defines a bijection $\mathcal{C}(G)\to \mathcal{C}(G/U)$.
    
\end{remark}

\mdef\label{OG} An \textbf{ordered graph} $G$, is a graph with a totally ordered vertex set $(V(G), \prec)$. An isomorphism of ordered graphs, $G$ and $G'$, is a graph isomorphism $V(G)\to V(G')$, which is also order-preserving. We will write $G\simeq_O G'$ if are isomorphic as ordered graphs. Let $G^*$ denote the ordered graph obtained from $G$ by reversing the ordering on $V(G)$. 

\medskip\noindent We conclude this section with a technical lemma that we will use later on.

\begin{lemma}\label{TLarg} Let $V$ be a finite set and suppose that $V=V_1\sqcup V_2\sqcup V_3$ for some subsets $V_i$. Now suppose that $V$ has the following totally ordered subsets
\[U_{12}:=(V_1\sqcup V_2, \prec_{12}), \qquad U_{13}:=(V_1\sqcup V_3, \prec_{13}), \qquad U_{23}:=(V_2\sqcup V_3, \prec_{23}),\]
which satisfy $(V_1,\prec_{12})=(V_1,\prec_{13})$, $(V_3,\prec_{13})=(V_3,\prec_{23})$,  $(V_2,\prec_{12})=(V_2,\succ_{23})$, and $v_i\prec_{ij} v_j$ whenever $i<j$ and $v_i\in V_i, v_j\in V_j$. Now suppose that $G=(V,E)$ is a graph such that $G=G_{12}\cup G_{23}$ for subgraphs of the form
\[G_{12}=(V_{1}\sqcup V_2,E_{12}), \qquad G_{23}=(V_2\sqcup V_3,E_{23}),\]
where $E_{12}$ and $E_{23}$ are crossing-less pair partitions of $U_{12}$ and $U_{23}$, respectively. Then the contraction $G'=G/V_2$ has the form
$G'=(V_1\sqcup V_3, E')$ where $E'$ is a crossingless pair partition of $U_{13}$.  \end{lemma}

\subsection{Twisted Chord Diagrams}\ 
In this subsection we introduce twisted chord data (TCD) and their diagrams. These are well known combinatorial objects which describe handle decompositions of (possibly non-orientable) bounded surfaces \cite{Carter}. Motivated by the Morse theoretic notion of handle sliding, we define a combinatorial notion of chord sliding \cite{Kirby} on TCD. We also define here the notions of vertical juxtaposition and insertion, in order to decompose TCD, and conclude this subsection with Theorem \ref{rform}, which establishes the existence of a unique ``caravan" decomposition for TCD considered up to chord sliding, extending the orientable case from \cite{BGOnTheMelvin}.

\mdef\label{tcd} A \textbf{twisted chord datum} of rank $N\in \Z_{\geq 0}$ is a pair $(P, s)$ where $P\in B_{\underline{2N}}$ is a pair-partition of $\underline{2N}$, and $s$ is a function $s:P\to \Z_2$. We refer to parts $p\in P$ as \textbf{arcs}, and we say an arc $p$ is \textbf{twisted} if $s(p)=1$ and otherwise \textbf{untwisted}. A twisted chord datum, $(P,s)$, is said to be \textbf{orientable} if $s^{-1}\{1\}=\emptyset$.

Denote by $\mathcal{TC}_N$, the set of all twisted chord data of rank $N$, and by $\mathcal{TC}_{N +}\subset\mathcal{TC}_N$, the set of all orientable such twisted chord data. Observe then, that $|\mathcal{TC}_N|=2^N (2N-1)!!$ and $|\mathcal{TC}_{N+}|=(2N-1)!!$. 

\mdef Twisted chord data are represented by \textbf{twisted chord diagrams} constructed as follows: mark $2N$ points (labelled by integers $1,2,\dots$ increasing upwards) evenly spaced along a vertical line and then for each arc $p=\{i,j\}\in P$, join points $i$ and $j$ by an arc to the right of the line, ensuring that distinct such arcs only intersect transversely at a single point. Label the arc $p=\{i,j\}$, with its corresponding $\Z_2$ value $s(p)$.    

\begin{example}\label{example:markedchorddiagrams}
    We illustrate in Figure \ref{fig:illustrationmarkedchord} the marked chord diagrams of four pair-partitions of $\underline{6}$, $P_1 = \{ \{1, 2\}, \{3, 4\}, \{5, 6\}\}, P_2 = \{ \{1, 2\}, \{3, 5\}, \{4, 6\}\}$, $P_3 = \{ \{1, 5\}, \{2, 6\}, \{3, 4\}\}$, and $P_4 = \{ \{1, 5\}, \{2, 4\}, \{3, 6\}\} $ paired with the functions
    \begin{eqnarray*}
        s_{1}(P_1) &=& \{0\}, \qquad
        s_{2}(p) := \left\{\begin{array}{rl}
         1,    &  \text{ if } p = \{1, 2\} \\
          0,   &  \text{ else } 
        \end{array} \right.\\
        s_{3}(P_3) &=& \{1\}, \qquad
        s_{4}(p) := \left \{\begin{array}{rl}
         1,    &  \text{ if } p = \{1, 5\}\\
          0,   &  \text{ else. }
        \end{array} \right.
    \end{eqnarray*}

    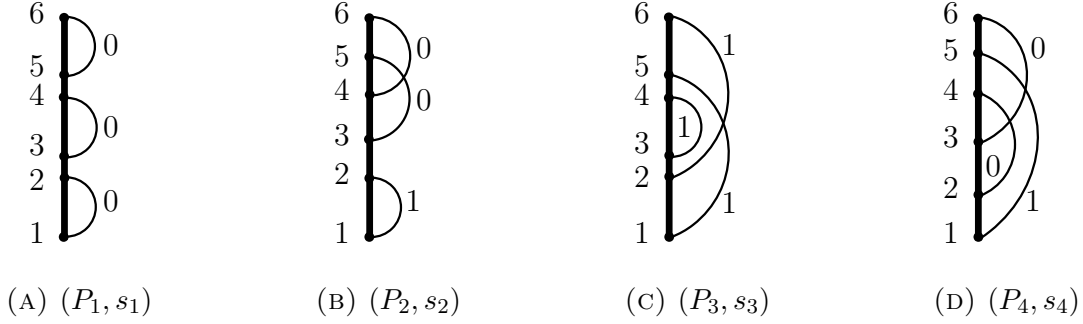
\begin{figure}[ht]
        \centering
       \begin{subfigure}{.24\textwidth}
       \centering
     $$  \vcenter{\hbox{\begin{overpic}[scale=1]{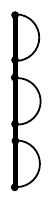}
\put(22, 76){$0$}
\put(22, 45){$0$}
\put(22, 17){$0$}
\put(-6, 5){$1$}
\put(-6, 25){$2$}
\put(-6, 38){$3$}
\put(-6, 57){$4$}
\put(-6, 68){$5$}
\put(-6, 88){$6$}
\end{overpic}}}$$
\caption{$(P_1, s_{1})$}
        \label{fig:P1ex}
       \end{subfigure}
       \begin{subfigure}{.24\textwidth}
       \centering
       $$  \vcenter{\hbox{\begin{overpic}[scale=1]{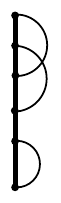}
\put(25, 75){$0$}
\put(25, 55){$0$}
\put(21, 17){$1$}
\put(-6, 5){$1$}
\put(-6, 27){$2$}
\put(-6, 42){$3$}
\put(-6, 59){$4$}
\put(-6, 72){$5$}
\put(-6, 88){$6$}
\end{overpic}}}$$
            \caption{$(P_2, s_{2})$}
        \label{fig:P2ex}
       \end{subfigure}
       \begin{subfigure}{.24\textwidth}
       \centering
       $$  \vcenter{\hbox{\begin{overpic}[scale=1]{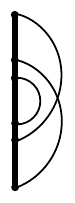}
\put(27, 76){$1$}
\put(10, 45){$1$}
\put(27, 16){$1$}
\put(-6, 5){$1$}
\put(-6, 26){$2$}
\put(-6, 39){$3$}
\put(-6, 57){$4$}
\put(-6, 68){$5$}
\put(-6, 88){$6$}
\end{overpic}}}$$
            \caption{$(P_3, s_{3})$}
        \label{fig:P3ex}
       \end{subfigure}
       \begin{subfigure}{.24\textwidth}
       \centering
       $$  \vcenter{\hbox{\begin{overpic}[scale=1]{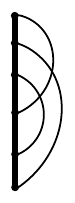}
\put(27, 75){$0$}
\put(10, 30){$0$}
\put(25, 17){$1$}
\put(-6, 5){$1$}
\put(-6, 22){$2$}
\put(-6, 41){$3$}
\put(-6, 59){$4$}
\put(-6, 74){$5$}
\put(-6, 88){$6$}
\end{overpic}}}$$
            \caption{$(P_4, s_{4})$}
        \label{fig:P4ex}
       \end{subfigure}
        \caption{Illustrations of twisted chord diagrams.}
        \label{fig:illustrationmarkedchord}
    \end{figure}
We will omit the labels $1,2,\dots$ for points in future diagrams for simplicity, and remark that a twisted chord diagram determines its corresponding datum. We use the latter fact to define data $(P,s)\in \mathcal{TC}_N$ by giving their diagrams.
\end{example}

\mdef\label{perm1} For a twisted chord datum $(P,s)\in \mathcal{TC}_N$, and a permutation, $\sigma\in \SS_{2N}$, we define $\sigma(P,s):=(\sigma(P),s\circ \sigma^{-1})$, where $\sigma^{-1}$ denotes the map $\mathcal{P}(S)\to \mathcal{P}(S)$ following \ref{notn}. Whenever it is unambiguous, we abuse notation to denote by $\sigma$ the map $\sigma:\mathcal{TC}_N\to\mathcal{TC}_N$. Henceforth, $\omega_N\in \SS_{2N}$ will denote the unique order reversing permutation, $\omega_N(i)=2N+1-i\in \SS_{2N}$, and $C_N\in \SS_{2N}$ will denote the cycle $C_N=(2N \ 2N-1 \ \dots \ 2 \ 1)$. Denote $(P,s)^*:=\omega_N(P,s)$.

\begin{example}\label{ex:perms} In Figure \ref{fig:permutationomega} we give diagrams for $(P_2, s_2)^*$ and $C_3(P_2, s_2)$, with $(P_2,s_2)$ as per  \ref{example:markedchorddiagrams}. Note that applying $\omega_N$ vertically reflects a diagrams and $C_N$ cyclically rearranges it. 
\end{example}

  \begin{figure}[ht]
        \centering
       \begin{subfigure}{.45\textwidth}
       \centering
       $$  \vcenter{\hbox{\begin{overpic}[scale=1]{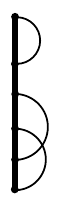}
\put(22, 75){$1$}
\put(25, 35){$0$}
\put(25, 15){$0$}
\end{overpic}}}$$
            \caption{$(P_2, s_{2})^*$}
        \label{fig:omegap2}
       \end{subfigure}
       \begin{subfigure}{.45\textwidth}
       \centering
       \begin{tikzpicture}[scale=0.55,every node/.style={shape=circle},inner sep=0pt,baseline=(current bounding box.center)  ]
      
       \draw[very thick] (0,0.5) -- (0,6.5);
\node[fill=black,label=left:{},inner sep=1pt] (1) at (0,1) {};
\node[fill=black,label=left:{},inner sep=1pt] (6) at (0,6) {};
\draw[thick] (0,6) arc (90:-90:2.5cm) node[midway,anchor=west] {\scriptsize{1}};
\node[fill=black,label=left:{},inner sep=1pt] (2) at (0,2) {};
\node[fill=black,label=left:{},inner sep=1pt] (4) at (0,4) {};
\draw[thick] (0,4) arc (90:-90:1.0cm) node[midway,anchor=west] {\scriptsize{0}};
\node[fill=black,label=left:{},inner sep=1pt] (3) at (0,3) {};
\node[fill=black,label=left:{},inner sep=1pt] (5) at (0,5) {};
\draw[thick] (0,5) arc (90:-90:1.0cm) node[midway,anchor=west] {\scriptsize{0}};

\end{tikzpicture}
            \caption{$C_3(P_2, s_{2})$}
        \label{fig:omegap4}
       \end{subfigure}
        \caption{Twisted chord diagrams under permutations.}
        \label{fig:permutationomega}
    \end{figure}
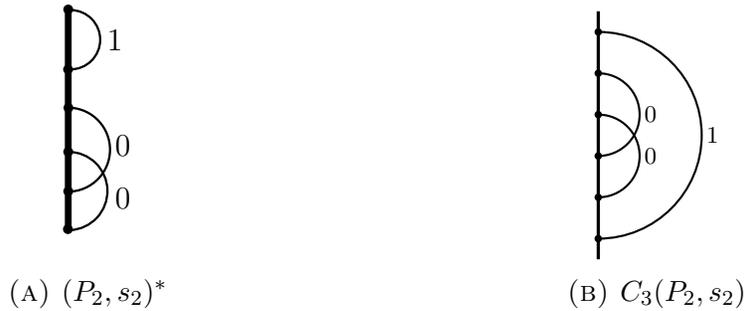

To define a chord slide operation, we require a site, $i\in \underline{2N}$, to slide, and a direction, up or down, in which to slide.
We thus define:

\mdef\label{admissibility} For a pair partition $P\in B_{\underline{2N}}$, an \textbf{admissible pair} in $P$ is a pair $(i,\epsilon)\in \underline{2N}\times \{\pm 1\}$, such that $i+\e\in \underline{2N}$ and $\{i,i+\e\}\notin P$. \label{def:chordslide} Let $(i,\e)$ be admissible in $P$, and hence let $q=\{i+\e,i'\}\in P$ for $i'\neq i$. Then for any $(P,s)\in \mathcal{TC}_N$, let $\sigma_{(i,\e,P,s)}\in \mathfrak{S}_{2N}$ denote the unique permutation which is order preserving (OP) on $\underline{2N}\smallsetminus\{i\}$, and 
\[\sigma_{(i,\e,P,s)}(i)=\sigma_{(i,\e,P,s)}(i')+\e(-1)^{s(q)},\]
(Note that the map $(-1)^{(\_)} : \Z_2 \to \{\pm 1\}$ is well defined). Explicitly, one has 
\begin{equation}\sigma=\begin{cases} \left(M(q) \ M(q)-1  \dots  i\right), &\begin{matrix} \e=1, \ i+\e =m(q), \text{ and } s(q)=0,\\
\text{or }\e=-1, \ i+\e =m(q), \text{ and } s(q)=1,\end{matrix}\\[1em]
\left( M(q)-1 \ M(q)-2 \ \dots \ i\right), & \begin{matrix}\e=1, \ i+\e =m(q), \text{ and } s(q)=1,\\
\text{or }\e=-1, \ i+\e =m(q), \text{ and } s(q)=0,\end{matrix}\\[1em]
\left( m(q)+1 \ m(q)+2  \ \dots \ i\right), & \begin{matrix}\e=1, \ i+\e =M(q), \text{ and } s(q)=0, \\
\text{or }\e=-1, \ i+\e =M(q), \text{ and } s(q)=1,\end{matrix}\\[1em]
\left( m(q) \ m(q)+1  \ \dots \ i\right), & \begin{matrix}\e=1, \ i+\e =M(q), \text{ and } s(q)=1, \\
\text{or }\e=-1, \ i+\e =M(q), \text{ and } s(q)=0,
\end{matrix} 

\end{cases}\label{perm}\end{equation}
using cycle notation.

For any $(i,\epsilon)\in \underline{2N}\times \{\pm1\}$, define a map $h_{(i,\epsilon)}:\mathcal{TC}_N\to\mathcal{TC}_N$ by
\begin{equation*}
    h_{(i,\epsilon)}(P,s)=\begin{cases}(\sigma(P), s_{(i, \epsilon)}\circ \sigma^{-1}), & (i,\e) \text{ admissible in }P,\\
(P,s), & \text{otherwise},\end{cases}
\end{equation*}
where $\sigma=\sigma_{(i,\e,P,s)} \in \mathfrak{S}_{2N}$ is as per \ref{def:chordslide}, and $s_{(i, \epsilon)}: P\to \Z_2$ is the map
\begin{equation*}
s_{(i, \epsilon)}(p)=\begin{cases}s(p)+s(q), & \text{when } i\in p,\\
s(p), & \text{otherwise}.\end{cases}    
\end{equation*}
For $(i,\pm 1)$ we call the map $h_{(i,\pm 1)}$ the \textbf{chord slide} up (+) or down (-) at position $i$.

\begin{remark} \label{rem:orient} For $(P,s)\in \mathcal{TC}_N$, $(i,\e)$ admissible in $P$, and $q$ as above, observe that if $s(q)=0$, then $s_{(i,\e)}=s$. In particular, if $(P,s)$ is orientable than any $(P',s')\in \mathcal{TC}_N$ satisfying $(P,s)\sim (P',s')$ is necessarily orientable (\textit{c.f.} Prop. \ref{S2eq}).\end{remark}

The following two lemmas describe the interplay between certain permutations and chord slides, which is used to reduce case work in further results. 

\begin{lemma}\label{cyclemma} Suppose that $(P,s)\in \mathcal{TC}_N$ and $\{i+1,i'\}\in P$ for integers $1<i<2N$ and $i'$.
Then $h_{(i,+1)}(P,s)=C_N^{-1} \circ h_{(i-1,+1)}\circ C_N(P,s)$ unless $i'=1$, in which case $h_{(i,+1)}(P,s)=C_N^{-2}\circ h_{(i-1,+1)}\circ C_N (P,s)$. 
\end{lemma}

\begin{proof} Follows by direct verification. \end{proof}

\begin{lemma}\label{fliplemma} For any $N\in \Z_{\geq 0}$, $(i,\e)\in \underline{2N}\times \{\pm 1\}$, we have $h_{(i,\e)}=\omega_N \circ h_{(\omega_N(i),-\e)}\circ\omega_N$, where $\omega_N \in \SS_{2N}$ is as per \ref{perm1}.
\end{lemma}
\begin{proof} Follows by direct verification. \end{proof}

\begin{prop}\label{S2eq}  Define a relation on $\mathcal{TC}_N$ by $(P,s)\sim (P',s')$ if $(P,s)=h (P',s')$, where $h$ is a finite (possibly empty) composite of chord slides. This is an equivalence relation.
\end{prop}
\begin{proof} 
The relation $\sim$ is manifestly reflexive and transitive. To check symmetry, it suffices to check that for any $(P,s)\in \mathcal{TC}_N$ and $(i,\e)$ admissible in $P$, we have $(P,s)\sim h_{(i,\e)} (P,s)$. Writing $h_{(i,\e)}(P,s)=(\sigma(P),s_{(i, \e)}\circ \sigma^{-1})$ as per \ref{def:chordslide}, and $i+\e \in q \in P$, let $\e'$ be given by 
\begin{equation*}
    \e'=\begin{cases} -\e, & s(q)=0, \\ \e, & s(q)=1.\end{cases}
\label{epdash}\end{equation*}
Then we claim that $(P,s)=h_{(\sigma(i),\e')}\left[ h_{(i,\e)}(P,s)\right]$ giving $(P,s)\sim h_{(i,\e)} (P,s)$ as desired.
\end{proof}

Below we define the boundary graph of a TCD, which describes the boundary of the surface formed from a TCD. 

\mdef\label{bdgrp} Given a twisted chord datum $(P,s)\in \mathcal{TC}_N$, define its \textbf{boundary graph} $\del (P,s)$ as the graph with vertex set $\underline{2N}\times\{\pm 1\}$, and edge set given by
\[
    \left\{\begin{array}{ccc}
         (i,+1) \edge (i+1,-1), & |  &\text{for all } i=1,\dots,2N-1,\\
    (i,+1) \edge (j,- 1 ),(i,-1) \edge (j,+ 1 ), & |  & \text{for all untwisted arcs } \{i,j\} \in P,\\
    (i', +1) \edge (j',+ 1 ), (i', -1) \edge (j',-1 ), &| & \text{for all twisted arcs } \{i',j'\} \in P,\end{array}\right\}\,
\]
(recall that $(i, \epsilon_1) \edge (j, \epsilon_2)$ denotes an edge between vertices $(i, \epsilon_1)$ and $ (j, \epsilon_2)$). Order this vertex set by $(i,\e)<(j,\e')$ whenever $2i+\e/2<2j+\e'/2$; if we place $(i,\pm)$ just above (+) and below (-) the vertex $i$ in a twisted chord diagram, then this order is compatible with the diagrammatic ``bottom to top" order and the graph $\del(P,s)$ can be drawn as a collection of 2 dimensional 1-handles attached to a line as per Figure \ref{fig:boundary_graph}. It will also be convenient to define $\overline{\del(P,s)}:=\del(P,s)\cup \left[ (0,+1)\edge (1,-1)\right]\cup \left[(2N,+1)\edge (2N+1,-1)\right]$, with the previous order extended to this vertex set by the same formula. 

  \begin{figure}[ht]
        \centering
       \begin{subfigure}{.3\textwidth}
       \centering\begin{tikzpicture}[scale=.9,every node/.style={shape=circle},inner sep=0pt,baseline=(current bounding box.center) ]
      
       \draw[very thick] (0,0.5) -- (0,4.5);
\node[fill=black,label=left:{},inner sep=1pt] (1) at (0,1) {};
\node[fill=black,label=left:{},inner sep=1pt] (2) at (0,2) {};
\node[fill=black,label=left:{},inner sep=1pt] (3) at (0,3) {};
\node[fill=black,label=left:{},inner sep=1pt] (4) at (0,4) {};
\draw[thick] (0,4) arc (90:-90:1.0cm) node[midway,anchor=west] {\scriptsize{$1$}};
\draw[thick] (0,3) arc (90:-90:1.0cm) node[midway,anchor=west] {\scriptsize{$0$}};      
\end{tikzpicture}
            \caption{$(P,s)$}
        \label{fig:twistedchorddiagramm}
       \end{subfigure}
       \begin{subfigure}{.3\textwidth}
       \centering\begin{tikzpicture}[scale=0.9,every node/.style={shape=circle},inner sep=0pt,baseline=(current bounding box.center) ]
      
       \draw[very thick,opacity=0.4] (-0.1,0.5) -- (-0.1,4.5);
\node[fill=black,label=left:{},inner sep=0pt,opacity=0.4] (1) at (0,1) {};
\node[fill=red,label=left:{\scriptsize $1^+$},inner sep=1pt,] (1+) at (0,1.25) {};
\node[fill=blue,label=left:{\scriptsize $1^-$},inner sep=1pt] (1-) at (0,0.75) {};
\node[fill=black,label=left:{},inner sep=0pt,opacity=0.4] (2) at (0,2) {};
\node[fill=red,label=left:{\scriptsize $2^+$},inner sep=1pt] (2+) at (0,2.25) {};
\node[fill=blue,label=left:{\scriptsize $2^-$},inner sep=1pt] (2-) at (0,1.75) {};
\node[fill=black,label=left:{},inner sep=0pt,opacity=0.4] (3) at (0,3) {};
\node[fill=red,label=left:{\scriptsize $3^+$},inner sep=1pt] (3+) at (0,3.25) {};
\node[fill=blue,label=left:{\scriptsize $3^-$},inner sep=1pt] (3-) at (0,2.75) {};
\node[fill=black,label=left:{},inner sep=0pt,opacity=0.4] (4) at (0,4) {};
\node[fill=red,label=left:{\scriptsize $4^+$},inner sep=1pt] (4+) at (0,4.25) {};
\node[fill=blue,label=left:{\scriptsize $4^-$},inner sep=1pt] (4-) at (0,3.75) {};

\draw[thick,opacity=0.4] (0,4) arc (90:-90:1.0cm) node[midway,anchor=west,inner sep=-1pt] {\scriptsize{$1$}};
\draw[thick,opacity=0.4] (0,3) arc (90:-90:1.0cm) node[midway,anchor=west,inner sep=-1pt] {\scriptsize{$0$}};

\draw[thick,opacity=0.4] (1) -- + (-0.1,0);
\draw[thick,opacity=0.4] (2) -- + (-0.1,0);
\draw[thick,opacity=0.4] (3) -- + (-0.1,0);
\draw[thick,opacity=0.4] (4) -- + (-0.1,0);

\draw[thick] (4+) arc (90:-90:1.0cm) node[midway,anchor=west] {};
\draw[thick] (4-) arc (90:-90:1.0cm) node[midway,anchor=west] {};     
\draw[thick] (3+) arc (90:-90:1.25cm) node[midway,anchor=west] {};
\draw[thick] (3-) arc (90:-90:0.75cm) node[midway,anchor=west] {};    
\draw[thick] (1+) -- (2-);    
\draw[thick] (2+) -- (3-);    
\draw[thick] (3+) -- (4-);

\end{tikzpicture}
            \caption{$\del(P,s)$.}
        \label{fig:boundary_graph}
       \end{subfigure}\begin{subfigure}{.3\textwidth}
       \centering\begin{tikzpicture}[scale=0.8,every node/.style={shape=circle},inner sep=0pt,baseline=(current bounding box.center) ]
      
       \draw[very thick,opacity=0.4] (-0.1,0.5) -- (-0.1,4.5);
\node[fill=red,label=left:{\scriptsize $0^+$},inner sep=1pt] (0+) at (0,0.25) {};
\node[fill=black,label=left:{},inner sep=0pt,opacity=0.4] (1) at (0,1) {};
\node[fill=red,label=left:{\scriptsize $1^+$},inner sep=1pt,] (1+) at (0,1.25) {};
\node[fill=blue,label=left:{\scriptsize $1^-$},inner sep=1pt] (1-) at (0,0.75) {};
\node[fill=black,label=left:{},inner sep=0pt,opacity=0.4] (2) at (0,2) {};
\node[fill=red,label=left:{\scriptsize $2^+$},inner sep=1pt] (2+) at (0,2.25) {};
\node[fill=blue,label=left:{\scriptsize $2^-$},inner sep=1pt] (2-) at (0,1.75) {};
\node[fill=black,label=left:{},inner sep=0pt,opacity=0.4] (3) at (0,3) {};
\node[fill=red,label=left:{\scriptsize $3^+$},inner sep=1pt] (3+) at (0,3.25) {};
\node[fill=blue,label=left:{\scriptsize $3^-$},inner sep=1pt] (3-) at (0,2.75) {};
\node[fill=black,label=left:{},inner sep=0pt,opacity=0.4] (4) at (0,4) {};
\node[fill=red,label=left:{\scriptsize $4^+$},inner sep=1pt] (4+) at (0,4.25) {};
\node[fill=blue,label=left:{\scriptsize $4^-$},inner sep=1pt] (4-) at (0,3.75) {};
\node[fill=blue,label=left:{\scriptsize $5^-$},inner sep=1pt] (5-) at (0,4.75) {};

\draw[thick,opacity=0.4] (0,4) arc (90:-90:1.0cm) node[midway,anchor=west,inner sep=-1pt] {\scriptsize{$1$}};
\draw[thick,opacity=0.4] (0,3) arc (90:-90:1.0cm) node[midway,anchor=west,inner sep=-1pt] {\scriptsize{$0$}};

\draw[thick,opacity=0.4] (1) -- + (-0.1,0);
\draw[thick,opacity=0.4] (2) -- + (-0.1,0);
\draw[thick,opacity=0.4] (3) -- + (-0.1,0);
\draw[thick,opacity=0.4] (4) -- + (-0.1,0);

\draw[thick] (4+) arc (90:-90:1.0cm) node[midway,anchor=west] {};
\draw[thick] (4-) arc (90:-90:1.0cm) node[midway,anchor=west] {};     
\draw[thick] (3+) arc (90:-90:1.25cm) node[midway,anchor=west] {};
\draw[thick] (3-) arc (90:-90:0.75cm) node[midway,anchor=west] {};    
\draw[thick] (0+) -- (1-);    
\draw[thick] (1+) -- (2-);    
\draw[thick] (2+) -- (3-);    
\draw[thick] (3+) -- (4-);    
\draw[thick] (4+) -- (5-);

\end{tikzpicture}
            \caption{$\overline{\del(P,s)}$.}
        \label{fig:boundary_graphh}
       \end{subfigure}
        \caption{Twisted chord diagram and its boundary graph, using the short hand $i^{\pm}=(i,\pm 1)$.}
        \label{fig:bdgraph}
    \end{figure}
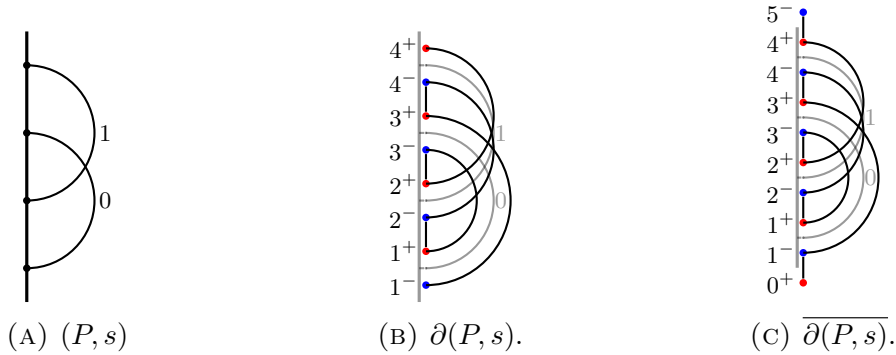

\begin{prop}\label{Prop:bondarycomponentspreserved} For any $(P,s)\in \mathcal{TC}_N$, $\del(P,s)$ contains a path between $(1,-1)$ and $(2N,+1)$. Furthermore, if $(P,s)\sim (P',s')$, then $|\CC(\del(P,S))|=|\CC(\del(P',S'))|$.
    
\end{prop}
\begin{proof}
    Existence of the path between $(1,-1)$ and $(2N,+1)$ follows immediately from Lemma \ref{grlemma} since these are the only two vertices with valence 1. For the second part, it suffices to treat the case $(P',s')=h_{(i,\e)}(P,s)$ for $(i,\e)$ admissible in $P$. Using Lemmata \ref{cyclemma} and \ref{fliplemma}, we can further reduce to the case $(i,\e)=(1,+1)$ since $(i,\e)\mapsto (\omega_N(i),-\e)$ is a graph isomorphism $\del(P,s)\overset{\sim}{\to} \del((P,s)^*)$, and $(i,\e)\mapsto (C_N(i),\e)$ is a graph isomorphism 
    \begin{equation*}
    (\del(P,s)\cup e)\overset{\sim}{\to} (\del(C_N(P,s))\cup e), \quad \text{for} \quad e=(1,-1)\edge (2N,+1),
    \end{equation*}
    and it follows from the first claim that $|\CC(\del(P,s)\cup e)|=|\CC(\del(P,s))|$. This case can be directly verified.\end{proof}

Below we define the notions of vertical juxtaposition and insertion of TCD, which are useful notions of decomposition for TCD (\textit{c.f.} Theorem \ref{rform}). 

\mdef\label{vjuxt} For two pair-partitions $P_i\in B_{\underline{2N_i}}$ ($i=1,2$), define their \textbf{vertical juxtaposition} $P_2\# P_1\in B_{\underline{2(N_1+N_2)}}$ by $P_2\# P_1=P_1\cup (P_2+2N_1)$\footnote{We regard $(\_)+N$ as a function $\Z\to \Z$ so that $P_2+2N_1=\{\{i+2N_1,i'+2N_1\} \ | \ \{i,i'\}\in P_2\}$ in accordance with \ref{notn}.}. For two twisted chord data $(P_i,s_i)\in \mathcal{TC}_{N_i}$ ($i=1,2$), define their vertical juxtaposition by $(P_2,s_2)\# (P_1,s_1)=(P_2\#P_1,s_2\#s_1)\in \mathcal{TC}_{N_1+N_2}$, where $s_2\#s_1$ is uniquely defined by $s_2\#s_1 (p)=s_1(p)$ and $s_2\# s_1 (q+2N_1) =s_2(q)$ for $p\in P_1, q\in P_2$. 

\mdef\label{ins} Define the \textbf{insertion} of $P_1$ in $P_2$ at height $d =0,\dots, 2N_2$, as the pair partition 

\begin{equation*}
P_{2} \#_d P_1:= (P_1+d) \cup o_d(P_2),
\end{equation*}
where $o_d:\underline{2N_2}\to \underline{2(N_1+N_2)}\smallsetminus(\underline{2N_1}+d)$ is the unique OP map. For twisted chord data, define their insertion $(P_2,s_2)\#_d (P_1,s_1):=(P_2\#_d P_1,s_2\#_d s_1)\in \mathcal{TC}_{N_1+N_2}$, where $s_2\#_d s_1$ is uniquely defined by $s_2\#_d s_1 (p+d)=s_1(p)$ and $s_2\#_d s_1 (o_d(q)) =s_2(q)$ for $p\in P_1, q\in P_2$. In particular, observe that $(P_2,s_2)\#_0 (P_1,s_1)=(P_2,s_2)\# (P_1,s_1)$ and $(P_2,s_2)\#_{2N_2} (P_1,s_1)=(P_1,s_1)\# (P_2,s_2)$.

\begin{example} The datum $(P_2,s_2)$ from \ref{example:markedchorddiagrams} can be written as the juxtaposition $(P_2,s_2)=\tor\#\mob$ with $\tor\in \mathcal{TC}_2$, $\mob\in \mathcal{TC}_1$ as per \ref{elem}. The datum $C_3(P_2,s_2)$ (\textit{c.f.} \ref{ex:perms}), can be written as the insertion $\mob\#_1 \tor$. We can represent the insertion of a generic datum $(P,s)\in \mathcal{TC}_N$, using the following ``box" convention:
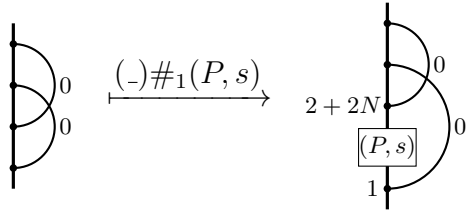
\begin{figure}[ht]
    \centering
    \[\begin{tikzpicture}[scale=0.55,every node/.style={shape=circle},inner sep=0pt,baseline=(current bounding box.center) ]
      
       \draw[very thick] (0,0.5) -- (0,4.5);
\node[fill=black,label=left:{},inner sep=1pt] (1) at (0,1) {};
\node[fill=black,label=left:{},inner sep=1pt] (2) at (0,2) {};
\node[fill=black,label=left:{},inner sep=1pt] (3) at (0,3) {};
\node[fill=black,label=left:{},inner sep=1pt] (4) at (0,4) {};
\draw[thick] (0,4) arc (90:-90:1.0cm) node[midway,anchor=west] {\scriptsize{$0$}};
\draw[thick] (0,3) arc (90:-90:1.0cm) node[midway,anchor=west] {\scriptsize{$0$}};      
\end{tikzpicture} \quad\xmapstoo{(\_)\#_1(P,s)}\quad \begin{tikzpicture}[scale=0.55,every node/.style={shape=circle},inner sep=0pt, baseline=(current bounding box.center)]
      
       \draw[very thick] (0,0.5) -- (0,5.5);

\node[fill=black,label=left:{\scriptsize $1$},inner sep=1pt] (1) at (0,1) {};
\node[fill=white,draw=black,style=rectangle,minimum height=0.5cm,opaque] (x) at (0,2) {\scriptsize $(P,s)$};
\node[fill=black,label=left:{\scriptsize $2+2N$},inner sep=1pt] (2) at (0,3) {};
\node[fill=black,label=left:{},inner sep=1pt] (3) at (0,4) {};
\node[fill=black,label=left:{},inner sep=1pt] (4) at (0,5) {};
\draw[thick] (0,5) arc (90:-90:1.0cm) node[midway,anchor=west] {\scriptsize{$0$}};
\draw[thick] (0,4) arc (90:-90:1.5cm) node[midway,anchor=west] {\scriptsize{$0$}};      
\end{tikzpicture}\]
    \caption{Box convention for insertion}
    \label{fig:box}
\end{figure}
\end{example}

\begin{prop}\label{assoc1} For $(P_i,s_i)\in \mathcal{TC}_{N_i}$ ($i=1,2,3$) and non-negative integers $d_1\leq 2N_2$ and $d_2\leq 2N_3$, and $d_1'<d_2'\leq 2N_3$ we have
\begin{align*} 
    (P_3, s_3)\#_{d_2}\left[(P_2, s_2)\#_{d_1}(P_1, s_1)\right]&=\left[(P_3, s_3)\#_{d_2} (P_2, s_2)\right]\#_{d_1+d_2}(P_1, s_1),\\
    \left[(P_3, s_3)\#_{d_2'}(P_2, s_2)\right]\#_{d_1'}(P_1, s_1)&=  \left[(P_3, s_3)\#_{d_1'}(P_1, s_1)\right]\#_{d_2'+2 N_1}(P_2, s_2),
    \\
    [(P_2,s_2)\#_d (P_1,s_1)]^*&=(P_2,s_2)^*\#_{2N-d} (P_1,s_1)^*.
    \end{align*}
In particular, vertical juxtaposition is an associative binary operation on twisted chord data, for which $(\_)^*$ is an order reversing involution.
\end{prop}

\begin{proof} Follows by direct verification. \end{proof}

In what remains of this subsection, definitions and results are geared towards proving Theorem \ref{rform}, which is an analogy of the classification of surfaces, establishing a unique caravan decomposition for TCD when considered up to chord sliding. We begin by defining three small rank TCD, which are the ``prime" factors in this decomposition.

\mdef\label{elem} Define three twisted chord data, $\mob \in \mathcal{TC}_{1}$, $\text{Ann}\in \mathcal{TC}_{1}$, and $\tor\in \mathcal{TC}_{2}$, by
\begin{align*}
    \mob&=(\{\{1,2\}\},s_{\mob}), &s_{\mob}(\{1,2\})&=1,\\   
    \text{Ann}&=(\{\{1,2\}\},s_{\text{Ann}}), &s_{\text{Ann}}(\{1,2\})&=0,\\ 
    \tor&=(\{\{1,3\},\{2,4\}\},s_{\tor}), &s_{\tor}(\{1,3\})&=s_{\tor}(\{2,4\})=0.
\end{align*}
They are illustrated in Figure \ref{fig:ementarytwistedchord}.

Note that these three TCD have the property that they are preserved by all chord slides, \textit{i.e.} $h_{(i,\e)}(P,s)=(P,s)$ for all $(i,\e)$ admissible in $P$. It therefore follows that any decomposition $(P,s)\sim (P_2,s_2)\#_d (P_1,s_1)$ for $(P,s)=\mob, \text{Ann},$ or $\tor$ is trivial; \textit{i.e.} one factor is the unique rank 0 datum. In this sense, these data are prime.
It follows from Theorem \ref{rform}, that these are the only data with this property. 

\begin{figure}[ht]
        \centering
       \begin{subfigure}{.3\textwidth}
       \centering
     $$  \vcenter{\hbox{\begin{overpic}[scale=1]{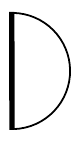}
\put(37, 30){$1$}
\end{overpic}}}$$
\caption{\mob}
        \label{fig:TMob}
       \end{subfigure}
       \begin{subfigure}{.3\textwidth}
       \centering
       $$  \vcenter{\hbox{\begin{overpic}[scale=1]{ElementaryTmob}
\put(37, 30){$0$}
\end{overpic}}}$$
            \caption{Ann}
        \label{fig:TAn}
       \end{subfigure}
       \begin{subfigure}{.3\textwidth}
       \centering
       $$  \vcenter{\hbox{\begin{overpic}[scale=1]{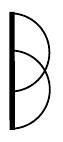}
\put(26, 20){$0$}
\put(26, 42){$0$}
\end{overpic}}}$$
            \caption{\tor}
        \label{fig:Ttor}
       \end{subfigure}
        \caption{Diagrams for $\mob$, $\text{Ann}$ and $\tor$.}
        \label{fig:ementarytwistedchord}
    \end{figure}

In order to prove Theorem \ref{rform}, it will be useful to define two types of sequences of chord slides for algorithmic purposes. In the first type of sequence, which we call an \textbf{evacuation sequence}, one slides a collection of adjacent sites one by one in the same direction. For example, suppose $(P,s)$ contains a twisted arc $\{i,j\}$ where $i<j$; the upward evacuation of sites $i+1$ to $j-1$ is given by the sequence indicated in Figure \ref{fig:evac}. \\

  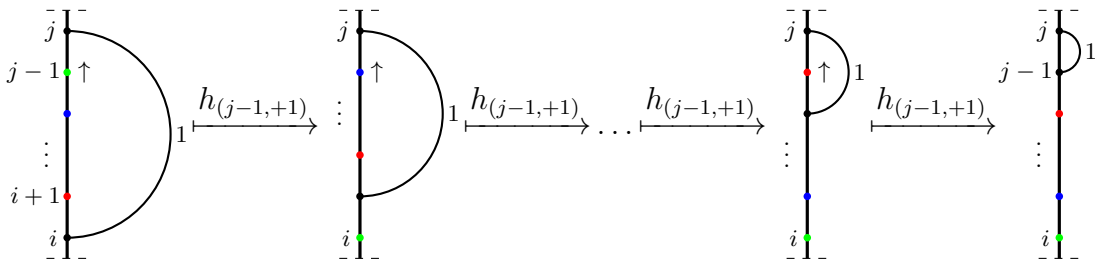
\begin{figure}[ht]
      \centering
\[\begin{tikzpicture}[scale=1.1,every node/.style={shape=circle},inner sep=0pt,baseline=(current bounding box.center) ]
      \draw[dashed] (-0.25,0.75) -- (0.25,0.75);
      \draw[dashed] (-0.25,3.75) -- (0.25,3.75);
       \draw[very thick] (0,0.75) -- (0,3.75);
\node[fill=black,label=left:{\scriptsize $i$},inner sep=1pt] (1) at (0,1) {};
\node[fill=black,label=left:{\scriptsize $j$},inner sep=1pt] (2) at (0,3.5) {};
\draw[thick] (0,3.5) arc (90:-90:1.25cm) node[midway,anchor=west] {\scriptsize{1}};
\node[fill=red,label=left:{\scriptsize $i+1$},inner sep=1pt] (a) at (0,1.5) {};
\node[anchor=east,inner sep=0pt] (b) at (0,2.1) {\scriptsize$\vdots$};
\node[fill=blue,label=left:{},inner sep=1pt] (b) at (0,2.5) {};
\node[fill=green,label=left:{\scriptsize $j-1$},inner sep=1pt] (c) at (0,3) {};
\node[label=right:{\scriptsize $\uparrow$},inner sep=1pt] () at (c) {};
       
\end{tikzpicture} \xmapstoo{h_{(j-1,+1)}} \begin{tikzpicture}[scale=1.1,every node/.style={shape=circle},inner sep=0pt,baseline=(current bounding box.center) ]
      \draw[dashed] (-0.25,0.75) -- (0.25,0.75);
      \draw[dashed] (-0.25,3.75) -- (0.25,3.75);
       \draw[very thick] (0,0.75) -- (0,3.75);
\node[fill=black,label=left:{},inner sep=1pt] (1) at (0,1.5) {};
\node[fill=black,label=left:{\scriptsize $j$},inner sep=1pt] (2) at (0,3.5) {};
\draw[thick] (0,3.5) arc (90:-90:1cm) node[midway,anchor=west] {\scriptsize{1}};
\node[fill=red,label=left:{},inner sep=1pt] (a) at (0,2) {};
\node[anchor=east,inner sep=0pt] (b) at (0,2.6) {\scriptsize$\vdots$};
\node[fill=blue,label=left:{},inner sep=1pt] (b) at (0,3) {};
\node[label=right:{\scriptsize $\uparrow$},inner sep=1pt] () at (b) {};
\node[fill=green,label=left:{\scriptsize $i$},inner sep=1pt] (c) at (0,1) {};
       
\end{tikzpicture}\xmapstoo{h_{(j-1,+1)}}\dots \xmapstoo{h_{(j-1,+1)}}\begin{tikzpicture}[scale=1.1,every node/.style={shape=circle},inner sep=0pt,baseline=(current bounding box.center)]
     \draw[dashed] (-0.25,0.75) -- (0.25,0.75);
      \draw[dashed] (-0.25,3.75) -- (0.25,3.75);
       \draw[very thick] (0,0.75) -- (0,3.75);
\node[fill=black,label=left:{},inner sep=1pt] (1) at (0,2.5) {};
\node[fill=black,label=left:{\scriptsize $j$},inner sep=1pt] (2) at (0,3.5) {};
\draw[thick] (0,3.5) arc (90:-90:0.5cm) node[midway,anchor=west] {\scriptsize{1}};
\node[fill=red,label=left:{},inner sep=1pt] (a) at (0,3) {};
\node[label=right:{\scriptsize $\uparrow$},inner sep=1pt] () at (a) {};
\node[anchor=east,inner sep=0pt] (b) at (0,2.1) {\scriptsize$\vdots$};
\node[fill=blue,label=left:{},inner sep=1pt] (b) at (0,1.5) {};
\node[fill=green,label=left:{\scriptsize $i$},inner sep=1pt] (c) at (0,1) {};
\end{tikzpicture}\xmapstoo{h_{(j-1,+1)}}\begin{tikzpicture}[scale=1.1,every node/.style={shape=circle},inner sep=0pt,baseline=(current bounding box.center) ]
      \draw[dashed] (-0.25,0.75) -- (0.25,0.75);
      \draw[dashed] (-0.25,3.75) -- (0.25,3.75);
       \draw[very thick] (0,0.75) -- (0,3.75);
\node[fill=black,label=left:{\scriptsize $j-1$},inner sep=1pt] (1) at (0,3) {};
\node[fill=black,label=left:{\scriptsize $j$},inner sep=1pt] (2) at (0,3.5) {};
\draw[thick] (0,3.5) arc (90:-90:0.25cm) node[midway,anchor=west] {\scriptsize{1}};
\node[fill=red,label=left:{},inner sep=1pt] (a) at (0,2.5) {};
\node[anchor=east,inner sep=0pt] (b) at (0,2.1) {\scriptsize$\vdots$};
\node[fill=blue,label=left:{},inner sep=1pt] (b) at (0,1.5) {};
\node[fill=green,label=left:{\scriptsize $i$},inner sep=1pt] (c) at (0,1) {};
\end{tikzpicture}  \]
      \caption{Upward evacuation sequence of sites $i+1$ to $j-1$}
      \label{fig:evac}
  \end{figure}

In the second type of sequence, known as a \textbf{boundary slide} of a site over some inserted TCD $(P,s)$, one applies chord slides at a given site and at its resultant position after each step, until this site has been ``transported over the boundary" of $(P,s)$. For example, suppose that $(P,s)$ is of the form $(P',s')\#_{i-1} \tor$; the downward slide of site $i$ over $\tor$ is then illustrated in Figure \ref{fig:slide}.\\ 
  
  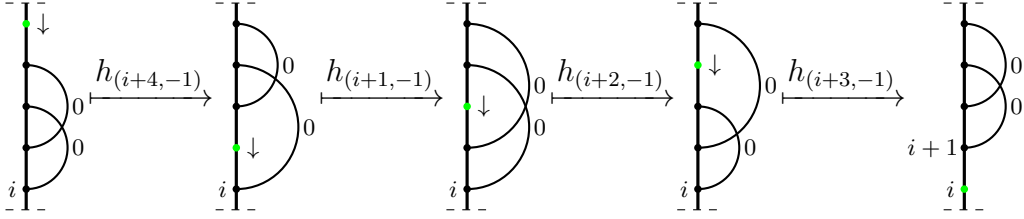
\begin{figure}[ht]
      \centering
\[\begin{tikzpicture}[scale=0.55,every node/.style={shape=circle},inner sep=0pt,baseline=(current bounding box.center)]
      \draw[dashed] (-0.5,0.5) -- (0.5,0.5);
      \draw[dashed] (-0.5,5.5) -- (0.5,5.5);
       \draw[very thick] (0,0.5) -- (0,5.5);
\node[fill=black,label=left:{\scriptsize $i$},inner sep=1pt] (1) at (0,1) {};
\node[fill=black,label=left:{},inner sep=1pt] (2) at (0,2) {};
\node[fill=black,label=left:{},inner sep=1pt] (3) at (0,3) {};
\node[fill=black,label=left:{},inner sep=1pt] (4) at (0,4) {};
\node[fill=green,label=left:{},inner sep=1pt] (5) at (0,5) {};
\node[label=right:{\scriptsize $\downarrow$},inner sep=1pt] () at (5) {};
\draw[thick] (0,4) arc (90:-90:1cm) node[midway,anchor=west] {\scriptsize{0}};
\draw[thick] (0,3) arc (90:-90:1cm) node[midway,anchor=west] {\scriptsize{0}};
       
\end{tikzpicture}\xmapstoo{h_{(i+4,-1)}}\begin{tikzpicture}[scale=0.55,every node/.style={shape=circle},inner sep=0pt,baseline=(current bounding box.center) ]
      \draw[dashed] (-0.5,0.5) -- (0.5,0.5);
      \draw[dashed] (-0.5,5.5) -- (0.5,5.5);
       \draw[very thick] (0,0.5) -- (0,5.5);
\node[fill=black,label=left:{\scriptsize $i$},inner sep=1pt] (1) at (0,1) {};
\node[fill=green,label=left:{},inner sep=1pt] (2) at (0,2) {};
\node[label=right:{\scriptsize $\downarrow$},inner sep=1pt] () at (2) {};
\node[fill=black,label=left:{},inner sep=1pt] (3) at (0,3) {};
\node[fill=black,label=left:{},inner sep=1pt] (4) at (0,4) {};
\node[fill=black,label=left:{},inner sep=1pt] (5) at (0,5) {};
\draw[thick] (0,5) arc (90:-90:1cm) node[midway,anchor=west] {\scriptsize{0}};
\draw[thick] (0,4) arc (90:-90:1.5cm) node[midway,anchor=west] {\scriptsize{0}};
       
\end{tikzpicture}\xmapstoo{h_{(i+1,-1)}}\begin{tikzpicture}[scale=0.55,every node/.style={shape=circle},inner sep=0pt,baseline=(current bounding box.center) ]
      \draw[dashed] (-0.5,0.5) -- (0.5,0.5);
      \draw[dashed] (-0.5,5.5) -- (0.5,5.5);
       \draw[very thick] (0,0.5) -- (0,5.5);
\node[fill=black,label=left:{\scriptsize $i$},inner sep=1pt] (1) at (0,1) {};
\node[fill=black,label=left:{},inner sep=1pt] (2) at (0,2) {};
\node[fill=green,label=left:{},inner sep=1pt] (3) at (0,3) {};
\node[label=right:{\scriptsize $\downarrow$},inner sep=1pt] () at (3) {};
\node[fill=black,label=left:{},inner sep=1pt] (4) at (0,4) {};
\node[fill=black,label=left:{},inner sep=1pt] (5) at (0,5) {};
\draw[thick] (0,5) arc (90:-90:1.5cm) node[midway,anchor=west] {\scriptsize{0}};
\draw[thick] (0,4) arc (90:-90:1.5cm) node[midway,anchor=west] {\scriptsize{0}};
       
\end{tikzpicture}\xmapstoo{h_{(i+2,-1)}}\begin{tikzpicture}[scale=0.55,every node/.style={shape=circle},inner sep=0pt,baseline=(current bounding box.center) ]
      \draw[dashed] (-0.5,0.5) -- (0.5,0.5);
      \draw[dashed] (-0.5,5.5) -- (0.5,5.5);
       \draw[very thick] (0,0.5) -- (0,5.5);
\node[fill=black,label=left:{\scriptsize $i$},inner sep=1pt] (1) at (0,1) {};
\node[fill=black,label=left:{},inner sep=1pt] (2) at (0,2) {};
\node[fill=black,label=left:{},inner sep=1pt] (3) at (0,3) {};
\node[fill=green,label=left:{},inner sep=1pt] (4) at (0,4) {};
\node[label=right:{\scriptsize $\downarrow$},inner sep=1pt] () at (4) {};
\node[fill=black,label=left:{},inner sep=1pt] (5) at (0,5) {};
\draw[thick] (0,5) arc (90:-90:1.5cm) node[midway,anchor=west] {\scriptsize{0}};
\draw[thick] (0,3) arc (90:-90:1cm) node[midway,anchor=west] {\scriptsize{0}};
       
\end{tikzpicture}\xmapstoo{h_{(i+3,-1)}}\begin{tikzpicture}[scale=0.55,every node/.style={shape=circle},inner sep=0pt,baseline=(current bounding box.center) ]
      \draw[dashed] (-0.5,0.5) -- (0.5,0.5);
      \draw[dashed] (-0.5,5.5) -- (0.5,5.5);
       \draw[very thick] (0,0.5) -- (0,5.5);
\node[fill=green,label=left:{\scriptsize $i$},inner sep=1pt] (1) at (0,1) {};
\node[fill=black,label=left:{\scriptsize $i+1$},inner sep=1pt] (2) at (0,2) {};
\node[fill=black,label=left:{},inner sep=1pt] (3) at (0,3) {};
\node[fill=black,label=left:{},inner sep=1pt] (4) at (0,4) {};
\node[fill=black,label=left:{},inner sep=1pt] (5) at (0,5) {};
\draw[thick] (0,5) arc (90:-90:1cm) node[midway,anchor=west] {\scriptsize{0}};
\draw[thick] (0,4) arc (90:-90:1cm) node[midway,anchor=west] {\scriptsize{0}};
       
\end{tikzpicture}\]
      \caption{Downward boundary slide of site $i$ over $\tor$}
      \label{fig:slide}
  \end{figure}

We define these notions, in more generality below:

\mdef\label{def:evac}  For $(P,s)\in \mathcal{TC}_N$, and non-negative integers $d_1 \leq d_2<2N$, and $d_1' \leq d_2'<2N$, define two sequences of chord slides $\mathit{ev}^{+}_{[d_1,d_2]}$ and $\mathit{ev}^{-}_{[d_1',d_2']}$ iteratively as follows: $\mathit{ev}_{[d_2,d_2]}^{+}=h_{(d_2,+1)}$, $\mathit{ev}_{[d_1',d_1']}^{-}=h_{(d_1',-1)}$, and for $d_1<d\leq d_2$ and $d_1'\leq d' <d_2'$
\begin{align*} \mathit{ev}^{+}_{[d-1,d_2]}&:= h_{(\tau^{+}_d(d-1),+1)}\circ \mathit{ev}^{+}_{[d,d_2]},\\
\mathit{ev}^{-}_{[d_1',d'+1]}&:= h_{(\tau^{-}_{d'}(d'+1),-1)}\circ \mathit{ev}^{-}_{[d_1',d']},
\end{align*}
where $\tau^{+}_d, \tau^{-}_{d'}\in \mathfrak{S}_{2N}$ are the permutations determined by $\mathit{ev}^{+}_{[d,d_2]}(P,s)=(\tau^+_d(P),(s^+_d)\circ(\tau^+_d)^{-1} )$ and  $\mathit{ev}^{-}_{[d_1',d']}(P,s)=(\tau^-_{d'}(P),(s^-_{d'})\circ(\tau^-_{d'})^{-1} )$, respectively. We refer to $\mathit{ev}^{+}_{[d_1,d_2]}$ (and $\mathit{ev}^{-}_{[d_1',d_2']}$) as the \textbf{upward} (and \textbf{downward}) \textbf{evacuation} of sites $d_1$ to $d_2$ (and $d_1'$ to $d_2'$) in $(P,s)$, respectively. Observe that 
\begin{equation*}
\mathit{ev}^{-}_{[d_1',d_2']}=\omega_N\circ \left(\mathit{ev}^{+}_{[\omega_N(d_2'),\omega_N(d_1')]}\right)_{(P,s)^*} \circ\omega_N,
\end{equation*}
using the subscript to indicate that the upward slide is in the chord datum $(P,s)^*\in \mathcal{TC}_N$.

\mdef\label{def:slideover} Let $(P_1,s_1)\in \mathcal{TC}_{N_1} $, $(P_2,s_2)\in \mathcal{TC}_{N_2},$ and suppose that $0< d\leq 2N_2$ and $0\leq  d'< 2N_2$ are non-negative integers. Let $\Gamma\subset \del (P_1,s_1)$ be the unique component containing vertices $(2N_1,+1)$ and $(1,-1)$. Then $\Gamma$ is a line which has the form
\[(1,-1) \edge (i_{1},\e_1) \edge (i_{1}-\e_1,-\e_1) \edge (i_{2},\e_2) \edge (i_{2}-\e_2,-\e_2) \edge \dots \]
\[\dots  \edge (i_{k},\e_k) \edge (i_{k}-\e_k,-\e_k) \edge (2N_1,+1),    \]
for $\e_j \in \{\pm 1\}$, $i_j\in \underline{2N_1}$. Define $\mathit{bs}^{+}_{(d,(P_1,s_1))}, \mathit{bs}^{-}_{(d',(P_1,s_1))}:\mathcal{TC}_{N_1+N_2}\to \mathcal{TC}_{N_1+N_2}$ to be the following sequences of handle slides:
\begin{align*}
\mathit{bs}^{+}_{(d,(P_1,s_1))}&:=  h_{(d+i_k+\delta_k,\epsilon_k)} \circ \dots  \circ h_{(d+i_2+\delta_2,\epsilon_2)} \circ h_{(d+i_1+\delta_1,\epsilon_1)}\circ h_{(d,+1)},\\
\mathit{bs}^{-}_{(d',(P_1,s_1))}&:=  h_{(d'+i_{1}+\delta_1,-\epsilon_1)}\circ\dots  \circ h_{(d'+i_{k-1}+\delta_{k-1},-\epsilon_{k-1})}\circ h_{(d'+i_k+\delta_k,-\epsilon_k)}\circ h_{(d'+2N_1+1,-1)},
\end{align*}
where $\delta_i=(\epsilon_i+1)/2$. We refer to $\mathit{bs}^+_{(d,(P_1,s_1))}$ (and $\mathit{bs}^-_{(d',(P_1,s_1))}$) as the \textbf{upward} (and \textbf{downward}) \textbf{boundary slide} of site $d$ (and $d'+1$) over $(P_1,s_1)$. Observe that
\begin{equation*}
\mathit{bs}^{-}_{(d',(P_1,s_1))}=\omega_N\circ \mathit{bs}^+_{\left(2N_2-d',(P_1,s_1)^*\right)}\circ \omega_N.
\end{equation*}

\begin{lemma}\label{moveslem} Let $(P_1,s_1)\in \mathcal{TC}_{N_1} $, $(P_2,s_2)\in \mathcal{TC}_{N_2}$, and let $d<2N_1$ be a non-negative integer. Let $p=\{d+1,i\}\in P_2$ be the unique part containing $d+1$. The following hold:
\begin{enumerate}
    \item If $s(p)=0$, then $\mathit{ev}^+_{[d+1,d+2N_1]}\left((P_2,s_2)\#_d (P_1,s_1)\right)= (P_2,s_2)\#_{i} (P_1,s_1)$.
    \item If $s(p)=1$, then $\mathit{ev}^+_{[d+1,d+2N_1]}\left((P_2,s_2)\#_d (P_1,s_1)\right)= (P_2,s_2)\#_{i-1} (P_1,s_1)^*$.
    \item $\mathit{bs}^-_{(d,(P_1,s_1))}\left((P_2,s_2)\#_d (P_1,s_1)\right)= (P_2,s_2)\#_{d+1} (P_1,s_1)$.
\end{enumerate}
\end{lemma}
\begin{proof}
    Omitted.
\end{proof}
\begin{figure}[ht]
    \centering
    \begin{subfigure}{0.4\textwidth}
    \[\begin{tikzpicture}[scale=0.6,every node/.style={shape=circle},inner sep=0pt, baseline=(current bounding box.center)]

       \draw[very thick] (0,0.5) -- (0,1.5);   
       \draw[very thick,dotted] (0,2.25) -- (0,1.5); 
       \draw[very thick] (0,2.25) -- (0,4.5);
       \draw[dashed] (-0.5,0.5) -- (0.5,0.5);    
       \draw[dashed] (-0.5,4.5) -- (0.5,4.5);

\node[fill=black,label=left:{\scriptsize $i$},inner sep=1pt] (1) at (0,1) {};
\node[fill=white,draw=black,style=rectangle,minimum height=0.5cm,opaque,label={[xshift=0.15cm,yshift=-0.1cm]\scriptsize$\uparrow$}] (x) at (0,3) {\scriptsize $(P_1,s_1)$};
\node[fill=black,label=left:{\scriptsize $D+1$},inner sep=1pt] (2) at (0,4) {};

\draw[thick] (0,4) arc (90:-90:1.5cm) node[midway,anchor=west] {\scriptsize{$0$}};      
\end{tikzpicture} \xmapstoo{\mathit{ev}^+_{[d+1,D]}} \begin{tikzpicture}[scale=0.6,every node/.style={shape=circle},inner sep=0pt, baseline=(current bounding box.center)]

       \draw[very thick] (0,0.5) -- (0,2.75);    
       \draw[very thick,dotted] (0,3.5) -- (0,2.75);
       \draw[very thick] (0,3.5) -- (0,4.5);
       \draw[dashed] (-0.5,0.5) -- (0.5,0.5);    
       \draw[dashed] (-0.5,4.5) -- (0.5,4.5);

\node[fill=black,label=left:{\scriptsize $i$},inner sep=1pt] (1) at (0,1) {};
\node[fill=white,draw=black,style=rectangle,minimum height=0.5cm,opaque] (x) at (0,2) {\scriptsize $(P_1,s_1)$};
\node[fill=black,label=left:{\scriptsize $D+1$},inner sep=1pt] (2) at (0,4) {};

\draw[thick] (0,4) arc (90:-90:1.5cm) node[midway,anchor=west] {\scriptsize{$0$}};      
\end{tikzpicture}\]\caption{ Lemma \ref{moveslem} part (1) }
    \end{subfigure}
    \begin{subfigure}{0.4\textwidth}
    \[\begin{tikzpicture}[scale=0.6,every node/.style={shape=circle},inner sep=0pt, baseline=(current bounding box.center)]

       \draw[very thick] (0,0.5) -- (0,1.5);    
       \draw[very thick, dotted] (0,2.25) -- (0,1.5);
       \draw[very thick] (0,2.25) -- (0,4.5);
       \draw[dashed] (-0.5,0.5) -- (0.5,0.5);    
       \draw[dashed] (-0.5,4.5) -- (0.5,4.5);

\node[fill=black,label=left:{\scriptsize $i$},inner sep=1pt] (1) at (0,1) {};
\node[fill=white,draw=black,style=rectangle,minimum height=0.5cm,opaque,label={[xshift=0.15cm,yshift=-0.1cm]\scriptsize$\uparrow$}] (x) at (0,3) {\scriptsize $(P_1,s_1)$};
\node[fill=black,label=left:{\scriptsize $D+1$},inner sep=1pt] (2) at (0,4) {};

\draw[thick] (0,4) arc (90:-90:1.5cm) node[midway,anchor=west] {\scriptsize{$1$}};      
\end{tikzpicture}\xmapstoo{\mathit{ev}^+_{[d+1,D]}} \begin{tikzpicture}[scale=0.6,every node/.style={shape=circle},inner sep=0pt, baseline=(current bounding box.center)]

       \draw[very thick] (0,0.5) -- (0,2.75);    
       \draw[very thick,dotted] (0,3.5) -- (0,2.75);  
       \draw[very thick] (0,3.5) -- (0,4.5);
       \draw[dashed] (-0.5,0.5) -- (0.5,0.5);    
       \draw[dashed] (-0.5,4.5) -- (0.5,4.5);

\node[fill=black,label=left:{\scriptsize $i+1+2N$},inner sep=1pt] (1) at (0,2+0.25) {};
\node[fill=white,draw=black,style=rectangle,minimum height=0.5cm,opaque] (x) at (0,1+0.25) {\scriptsize $(P_1,s_1)^*$};
\node[fill=black,label=left:{\scriptsize $D+1$},inner sep=1pt] (2) at (0,4) {};

\draw[thick] (0,4) arc (90:-90:0.875cm) node[midway,anchor=west] {\scriptsize{$1$}};      
\end{tikzpicture}\]
\caption{Lemma \ref{moveslem} part (2)}
    \end{subfigure}
    
    \begin{subfigure}{0.4\textwidth}
    \[\begin{tikzpicture}[scale=0.6,every node/.style={shape=circle},inner sep=0pt, baseline=(current bounding box.center)]

       \draw[very thick] (0,0.5) -- (0,1.5);    
       \draw[very thick,dotted] (0,2.25) -- (0,1.5);
       \draw[very thick] (0,2.25) -- (0,4.5);
       \draw[dashed] (-0.5,0.5) -- (0.5,0.5);    
       \draw[dashed] (-0.5,4.5) -- (0.5,4.5);

\node[fill=black,label=left:{\scriptsize $i$},inner sep=1pt] (1) at (0,1) {};
\node[fill=white,draw=black,style=rectangle,minimum height=0.5cm,opaque] (x) at (0,3) {\scriptsize $(P_1,s_1)$};
\node[fill=black,label=left:{\scriptsize $D+1$},inner sep=1pt] (2) at (0,4) {};
\node[label={[xshift=0.15cm,yshift=-0.3cm]\scriptsize $\downarrow$},inner sep=0pt] () at (2) {};

\draw[thick] (0,4) arc (90:-90:1.5cm) node[midway,anchor=west] {\scriptsize{$s$}};      
\end{tikzpicture} \xmapstoo{\mathit{bs}^-_{(d,(P_1,s_1))}}\begin{tikzpicture}[scale=0.6,every node/.style={shape=circle},inner sep=0pt, baseline=(current bounding box.center)]

       \draw[very thick] (0,0.5) -- (0,1.5);    
       \draw[very thick,dotted] (0,2.25) -- (0,1.5);
       \draw[very thick] (0,2.25) -- (0,4.5);
       \draw[dashed] (-0.5,0.5) -- (0.5,0.5);    
       \draw[dashed] (-0.5,4.5) -- (0.5,4.5);

\node[fill=black,label=left:{\scriptsize $i$},inner sep=1pt] (1) at (0,1) {};
\node[fill=white,draw=black,style=rectangle,minimum height=0.5cm,opaque] (x) at (0,4-0.25) {\scriptsize $(P_1,s_1)$};
\node[fill=black,label=left:{\scriptsize $d$},inner sep=1pt] (2) at (0,3-0.25) {};

\draw[thick] (0,3-0.25) arc (90:-90:0.875cm) node[midway,anchor=west] {\scriptsize{$s$}}; 

\end{tikzpicture}\]
\caption{Lemma \ref{moveslem} part (3)}
\end{subfigure}
\caption{Lemma \ref{moveslem}, using the short hand $D=d+2N_1$}
    \label{fig:moves_lem}
\end{figure}
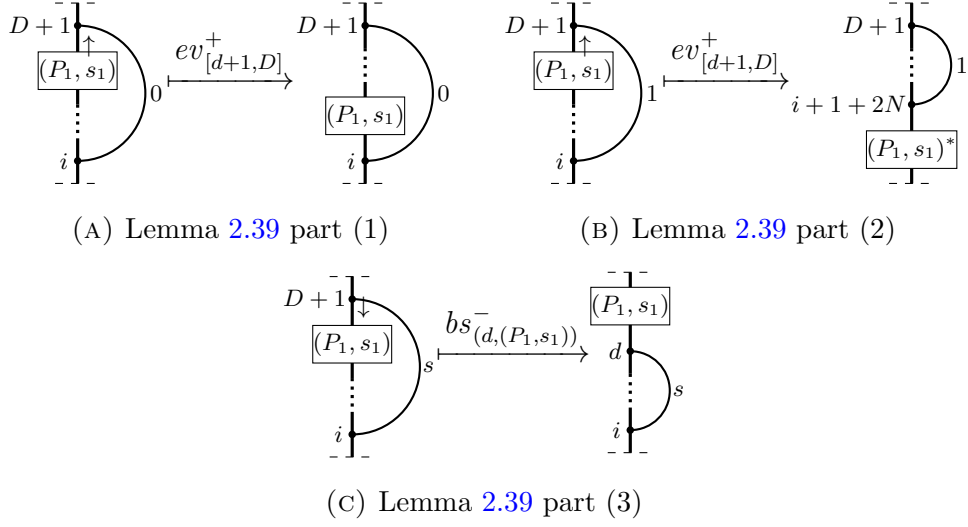
\begin{remark}\label{helpful} It follows from part 3 of \ref{moveslem}, that for any $(P_1,s_1)\in \mathcal{TC}_{N_1} $, $(P_2,s_2)\in \mathcal{TC}_{N_2}$, and non-negative integer $d\leq 2N_2$, we have
\begin{equation*}
(P_2,s_2)\#(P_1,s_1)\sim (P_2,s_2)\#_d(P_1,s_1)\sim (P_1,s_1)\#(P_2,s_2).
\end{equation*}
\end{remark}

\mdef\label{intmatrix} For a twisted chord datum, $(P,s)\in \mathcal{TC}_N$, define an $N\times N$ matrix, $\mathit{IM}(P,s)\in \mathcal{M}_{N}(\Z_2)$, with entries in $\Z_2$ labeled by parts in $P$ given by
\[\left(\mathit{IM}(P,s)\right)_{pp}=s(p), \quad \left(\mathit{IM}(P,s)\right)_{pq}=\begin{cases} 1, & \text{if $p$ and $q$ cross},\\
0, &\text{otherwise},
\end{cases}\]
for $p\neq q$. Following \cite{BGOnTheMelvin}, we call $\mathit{IM}(P,s)$ the \textbf{intersection matrix} of $(P,s)$. To write explicit matrices one requires an ordering on the arcs; our convention unless otherwise stated is to order the arcs in the order they first appear from bottom to top in a diagram. This ordering is indicated in Figure \ref{fig:diagramIM1}.

\begin{remark}\label{dsum} For $(P_1,s_1)\in \mathcal{TC}_{N_1}$ and $(P_2,s_2)\in \mathcal{TC}_{N_2}$, we have 
\begin{equation*}
\mathit{IM}((P_2,s_2)\# (P_1,s_1)) = \mathit{IM}(P_1,s_1)\oplus \mathit{IM}(P_2,s_2)
\end{equation*}
using the previous ordering on arcs.  
\end{remark}

\begin{example} The intersection matrices for the TCD from \ref{elem} are as follows
\begin{equation}\mathit{IM}(\mob)=\begin{pmatrix}1\end{pmatrix}, \qquad \mathit{IM}(\text{Ann})=\begin{pmatrix}0\end{pmatrix}, 
\qquad \mathit{IM}(\tor)= \begin{pmatrix} 
        0 & 1\\
        1 & 0
    \end{pmatrix}.\label{emats}\end{equation}
For the $(P_i,s_i)\in \mathcal{TC}_3$ as per Figure \ref{fig:illustrationmarkedchord3}, the intersection matrices $\mathit{IM}(P_i,s_i)$ are as follows:
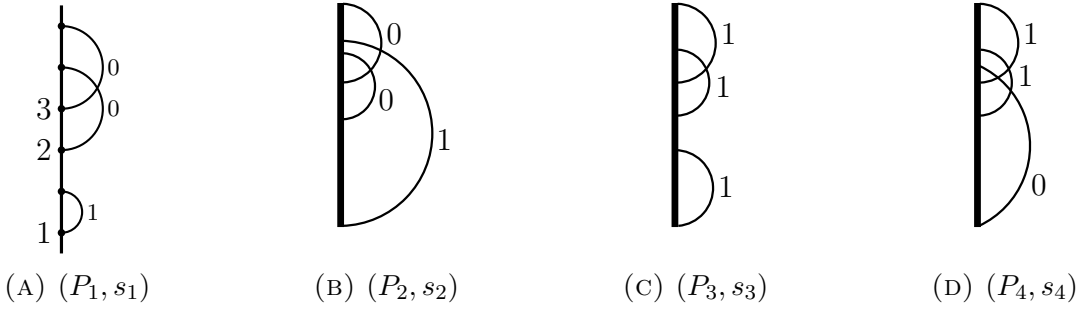
\begin{figure}[ht]
        \centering
       \begin{subfigure}{.24\textwidth}
       \centering
       \begin{tikzpicture}[scale=0.55,every node/.style={shape=circle},inner sep=0pt,baseline=(current bounding box.center) ]
      
       \draw[very thick] (0,0.5) -- (0,6.5);
\node[fill=black,label=left:{$1$},inner sep=1pt] (1) at (0,1) {};
\node[fill=black,label=left:{},inner sep=1pt] (6) at (0,6) {};
\draw[thick] (0,2) arc (90:-90:0.5cm) node[midway,anchor=west] {\scriptsize{$1$}};
\node[fill=black,label=left:{},inner sep=1pt] (2) at (0,2) {};
\node[fill=black,label=left:{$3$},inner sep=1pt] (4) at (0,4) {};
\draw[thick] (0,5) arc (90:-90:1.0cm) node[midway,anchor=west] {\scriptsize{$0$}};
\node[fill=black,label=left:{$2$},inner sep=1pt] (3) at (0,3) {};
\node[fill=black,label=left:{},inner sep=1pt] (5) at (0,5) {};
\draw[thick] (0,6) arc (90:-90:1.0cm) node[midway,anchor=west] {\scriptsize{$0$}};

\end{tikzpicture}
\caption{$(P_1, s_1)$}
        \label{fig:diagramIM1}
       \end{subfigure}
       \begin{subfigure}{.24\textwidth}
       \centering
       $$  \vcenter{\hbox{\begin{overpic}[scale=1]{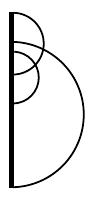}
\put(23, 75){$0$}
\put(20, 50){$0$}
\put(42, 35){$1$}
\end{overpic}}}$$
            \caption{$(P_2, s_2)$}
        \label{fig:diagramIM2}
       \end{subfigure}
       \begin{subfigure}{.24\textwidth}
       \centering
       $$ \vcenter{\hbox{\begin{overpic}[scale=1]{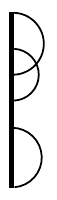}
\put(23, 75){$1$}
\put(21, 55){$1$}
\put(22, 17){$1$}
\end{overpic}}}$$
            \caption{$(P_3, s_3)$}
        \label{fig:diagramIM3}
       \end{subfigure}
       \begin{subfigure}{.24\textwidth}
       \centering
        $$   \vcenter{\hbox{\begin{overpic}[scale=1]{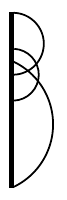}
\put(23, 75){$1$}
\put(21, 59){$1$}
\put(26, 18){$0$}
\end{overpic}}}$$
            \caption{$(P_4, s_4)$}
        \label{fig:diagramIM4}
       \end{subfigure}
        \caption{Illustrations of marked chord diagrams.}
        \label{fig:illustrationmarkedchord3}
    \end{figure}
    \[ \mathit{IM}(P_1, s_1) = \begin{pmatrix}
        1 & 0 & 0 \\
        0 & 0 & 1 \\
        0 & 1 & 0
    \end{pmatrix} = \mathit{IM}(\mob) \oplus \mathit{IM}(\tor), \]
    \[\mathit{IM}(P_2, s_2) = \begin{pmatrix}
        1 & 0 & 1 \\
        0 & 0 & 1 \\
        1 & 1 & 0
    \end{pmatrix}  \quad  \mathit{IM}(P_3, s_3) = \begin{pmatrix}
        1 & 0 & 0 \\
        0 & 1 & 1 \\
        0 & 1 & 1
    \end{pmatrix} \quad \mathit{IM}(P_4, s_4) = \begin{pmatrix}
        0 & 1 & 1 \\
        1 & 1 & 1 \\
        1 & 1 & 1
    \end{pmatrix}. \]
\end{example}

\begin{lemma}\label{matlemma} Suppose that $(P,s)$ and $(P',s')$ are two twisted chord data of rank $N$,  and that $(P,s)\sim (P',s')$. Then $\mathit{IM}(P',s')$ can be obtained from $\mathit{IM}(P,s)$ by a sequence of elementary row and column operations.  
\end{lemma}

\begin{proof} It is enough to prove the claim for the case $(P',s')=h_{(i,\e)}(P,s)$ for $(i,\e)$ admissible in $P$. Furthermore, since $p, q \in P$ cross, if and only if, $\sigma(p), \sigma(q)\in \sigma(P)$ cross, when $\sigma=C_N$ or $\omega_N$, by using Lemmata \ref{cyclemma} and \ref{fliplemma} it is enough to consider only the case $(i,\e)=(1,+1)$, since for such $\sigma \in \SS_{2N}$, the matrices $\mathit{IM}(P,s)$ and $\mathit{IM}(\sigma(P,s))$ differ only by a re-labelling of rows and columns. Further steps are omitted.
\end{proof}

\mdef\label{width} Define $w:P\to \Z_{\geq 0}$ by $w(p)=M(p)-m(p)-1$ (using the notation of \ref{order}), and call $w(p)$ the \textbf{width} of $p$. 

\begin{theorem}[Caravan Decomposition]\label{rform} For any $(P,s)\in \mathcal{TC}_N$, there exist unique integers $g,b\in \Z_{\geq 0},$ and $t\in\{0,1,2\}$, such that
\begin{equation}\label{eq:rform}
     (P,s)\sim \left(\#_{\kappa=1}^{t} \mob \right)\# \left(\#_{\lambda=1}^{g}\tor\right)\#\left( \#_{\mu=1}^{b} \text{Ann} \right).
\end{equation}\end{theorem}
\begin{proof} To prove existence, our strategy will be to reduce to the orientable case treated in \cite{BGOnTheMelvin}. Suppose that $(P,s)$ contains a twisted arc $\{i,j\}$ with $i<j$; by applying the upward evacuation sequence $\mathit{ev}_{[i+1,j-1]}^+$ we obtain some $(P',s')\sim (P,s)$ with a twisted arc $\{j-1,j\}$ (\textit{c.f.} Figure \ref{fig:evac}). Therefore, we have $(P',s')=(P'',s'')\#_{j-2}\mob$ for some $(P'',s'')\in \mathcal{TC}_{N-1}$, and so by Remark \ref{helpful}, $(P,s)\sim \mob \# (P'',s'')$. By iterating this argument on $(P'',s'')$, we eventually obtain $(P,s)\sim \left(\#_{\kappa =1}^{t} \mob\right) \# (\tilde{P},\tilde{s})$, for some $t\in \N$, and $(\tilde{P},\tilde{s})\in \mathcal{TC}_{\tilde{N}}^+$ orientable. Applying \cite[Theorem 2]{BGOnTheMelvin} to $(\tilde{P},\tilde{s})$ then gives a caravan decomposition
\[(P,s)\sim \left(\#_{\kappa=1}^{t} \mob \right)\# \left(\#_{\lambda=1}^{g}\tor\right)\#\left( \#_{\mu=1}^{b} \text{Ann} \right),\]
for $(b,g,t)\in \N^3$. By using the identity $\mob \# \mob \# \mob \sim \mob \# \tor$ (\textit{c.f.} Figure \ref{fig:ezmove}), if necessary, we can ensure $t=0,1$, or $2$ as required.

\begin{figure}[ht]
    \centering
    
       \[\begin{tikzpicture}[scale=0.55,every node/.style={shape=circle},inner sep=0pt,baseline=(current bounding box.center) ]
      
       \draw[very thick] (0,0.5) -- (0,6.5);
\node[fill=black,label=left:{},inner sep=1pt] (1) at (0,1) {};
\node[fill=black,label=left:{},inner sep=1pt] (6) at (0,6) {};
\draw[thick] (6) arc (90:-90:0.5cm) node[midway,anchor=west] {\scriptsize{$1$}};
\node[fill=black,label=left:{\scriptsize $\uparrow$},inner sep=1pt] (2) at (0,2) {};
\node[fill=black,label=left:{},inner sep=1pt] (4) at (0,4) {};
\draw[thick] (4) arc (90:-90:0.5cm) node[midway,anchor=west] {\scriptsize{$1$}};
\node[fill=black,label=left:{},inner sep=1pt] (3) at (0,3) {};
\node[fill=black,label=left:{\scriptsize $\downarrow$},inner sep=1pt] (5) at (0,5) {};
\draw[thick] (2) arc (90:-90:.5cm) node[midway,anchor=west] {\scriptsize{$1$}};

\end{tikzpicture}\quad\xmapstoo{h_{(5,-1)}\circ h_{(2,+1)}}\begin{tikzpicture}[scale=0.55,every node/.style={shape=circle},inner sep=0pt,baseline=(current bounding box.center) ]
      
       \draw[very thick] (0,0.5) -- (0,6.5);
\node[fill=black,label=left:{},inner sep=1pt] (1) at (0,1) {};
\node[fill=black,label=left:{},inner sep=1pt] (6) at (0,6) {};
\draw[thick] (6) arc (90:-90:1.5cm) node[midway,anchor=west] {\scriptsize{$0$}};
\node[fill=black,label=left:{\scriptsize $\uparrow$},inner sep=1pt] (2) at (0,2) {};
\node[fill=black,label=left:{},inner sep=1pt] (4) at (0,4) {};
\draw[thick] (5) arc (90:-90:1.5cm) node[midway,anchor=west] {\scriptsize{$1$}};
\node[fill=black,label=left:{},inner sep=1pt] (3) at (0,3) {};
\node[fill=black,label=left:{\scriptsize $\downarrow$},inner sep=1pt] (5) at (0,5) {};
\draw[thick] (4) arc (90:-90:1.5cm) node[midway,anchor=west] {\scriptsize{$0$}};

\end{tikzpicture}\quad \xmapstoo{h_{(4,-1)}\circ h_{(2,+1)}}
\begin{tikzpicture}[scale=0.55,every node/.style={shape=circle},inner sep=0pt,baseline=(current bounding box.center) ]
      
       \draw[very thick] (0,0.5) -- (0,6.5);
\node[fill=black,label=left:{\scriptsize $\uparrow$},inner sep=1pt] (1) at (0,1) {};
\node[fill=black,label=left:{},inner sep=1pt] (6) at (0,6) {};
\draw[thick] (6) arc (90:-90:2.5cm) node[midway,anchor=west] {\scriptsize{$1$}};
\node[fill=black,label=left:{},inner sep=1pt] (2) at (0,2) {};
\node[fill=black,label=left:{},inner sep=1pt] (4) at (0,4) {};
\draw[thick] (5) arc (90:-90:1cm) node[midway,anchor=west] {\scriptsize{$0$}};
\node[fill=black,label=left:{},inner sep=1pt] (3) at (0,3) {};
\node[fill=black,label=left:{},inner sep=1pt] (5) at (0,5) {};
\draw[thick] (4) arc (90:-90:1cm) node[midway,anchor=west] {\scriptsize{$0$}};
\end{tikzpicture}\quad \xmapstoo{\mathit{bs}^{+}_{(1,\tor)}}\quad\begin{tikzpicture}[scale=0.55,every node/.style={shape=circle},inner sep=0pt,baseline=(current bounding box.center) ]
      
       \draw[very thick] (0,0.5) -- (0,6.5);
\node[fill=black,label=left:{},inner sep=1pt] (1) at (0,1) {};
\node[fill=black,label=left:{},inner sep=1pt] (6) at (0,6) {};
\draw[thick] (6) arc (90:-90:0.5cm) node[midway,anchor=west] {\scriptsize{$1$}};
\node[fill=black,label=left:{},inner sep=1pt] (2) at (0,2) {};
\node[fill=black,label=left:{},inner sep=1pt] (4) at (0,4) {};
\draw[thick] (4) arc (90:-90:1cm) node[midway,anchor=west] {\scriptsize{$0$}};
\node[fill=black,label=left:{},inner sep=1pt] (3) at (0,3) {};
\node[fill=black,label=left:{},inner sep=1pt] (5) at (0,5) {};
\draw[thick] (3) arc (90:-90:1cm) node[midway,anchor=west] {\scriptsize{$0$}};

\end{tikzpicture}\]
\caption{The identity $\mob \# \mob \# \mob \sim \mob \# \tor$}
    \label{fig:ezmove}
\end{figure}

For uniqueness of the caravan decomposition, suppose $(P,s)$ satisfies equation \eqref{eq:rform} and $(P,s)\sim (\#_{\kappa=1}^{t'} \mob )\# (\#_{\lambda=1}^{g'}\tor)\#( \#_{\mu=1}^{b'} \text{Ann} )$. It follows from \ref{matlemma}, \ref{dsum}, and equation \eqref{emats}, that $b=\text{null}(\mathit{IM}(P,s))=b'$. The equalities $g=g'$ and $t=t'$ are then forced by the constraint $2g'+t'=N-b=2g+t$ as well as the orientability of $(P,s)$. \end{proof}
\begin{figure}[ht]
    \centering
    \[ (P,s) \sim \begin{tikzpicture}[scale=0.55,every node/.style={shape=circle},inner sep=0pt,baseline=0pt  ]
        \draw[very thick] (0.5,0)--(19.5,0);
        \draw[very thick,dotted] (2.5,0.2)--(3.5,0.2);
        \draw[very thick,dotted] (2,-0.4)--(3,-0.4);
        \draw[thick,decorate,decoration={brace,amplitude=5pt}] (1,1)--(5,1) node[pos=0.5,anchor=south,inner sep=5pt] {$t$};
        \draw[very thick,dotted] (9.5,0.2)--(10.5,0.2);
        \draw[thick,decorate,decoration={brace,amplitude=5pt}] (6,1.5)--(14,1.5) node[pos=0.5,anchor=south,inner sep=5pt] {$g$};
        \draw[very thick,dotted] (16.5,0.2)--(17.5,0.2);
        \draw[very thick,dotted] (16.5,-0.4)--(17.5,-0.4);
        \draw[thick,decorate,decoration={brace,amplitude=5pt}] (15,1)--(19,1) node[pos=0.5,anchor=south,inner sep=5pt] {$b$};
        
        \node[fill=black,inner sep=1pt,label=below:{\scriptsize $2N$}] (1) at (1,0) {};
        \node[fill=black,label=left:{},inner sep=1pt] (2) at (2,0) {};
        \node[fill=black,label=left:{},inner sep=1pt] (4) at (4,0) {};
        \node[fill=black,label=left:{},inner sep=1pt] (5) at (5,0) {};
        \node[fill=black,label=left:{},inner sep=1pt] (6) at (6,0) {};
        \node[fill=black,label=left:{},inner sep=1pt] (7) at (7,0) {};
        \node[fill=black,label=left:{},inner sep=1pt] (8) at (8,0) {};
        \node[fill=black,label=left:{},inner sep=1pt] (9) at (9,0) {};
        \node[fill=black,label=left:{},inner sep=1pt] (11) at (11,0) {};
        \node[fill=black,label=left:{},inner sep=1pt] (12) at (12,0) {};
        \node[fill=black,label=left:{},inner sep=1pt] (13) at (13,0) {};
        \node[fill=black,label=left:{},inner sep=1pt] (14) at (14,0) {};
        \node[fill=black,label=left:{},inner sep=1pt] (15) at (15,0) {};
        \node[fill=black,label=left:{},inner sep=1pt] (16) at (16,0) {};
        \node[fill=black,label=below:{\scriptsize$2$},inner sep=1pt] (18) at (18,0) {};
        \node[fill=black,label=below:{\scriptsize$1$},inner sep=1pt] (19) at (19,0) {};

        \draw[thick] (1) arc(180:0:0.5cm) node[pos=0.5,anchor=south,inner sep=0pt] {\scriptsize $1$};
        \draw[thick] (4) arc(180:0:0.5cm) node[pos=0.5,anchor=south,inner sep=0pt] {\scriptsize $1$};
        
        \draw[thick] (6) arc(180:0:1cm) node[pos=0.5,anchor=south,inner sep=0pt] {\scriptsize $0$};
        \draw[thick] (7) arc(180:0:1cm) node[pos=0.5,anchor=south,inner sep=0pt] {\scriptsize $0$};
        \draw[thick] (11) arc(180:0:1cm) node[pos=0.5,anchor=south,inner sep=0pt] {\scriptsize $0$};
        \draw[thick] (12) arc(180:0:1cm) node[pos=0.5,anchor=south,inner sep=0pt] {\scriptsize $0$};

        \draw[thick] (15) arc(180:0:0.5cm) node[pos=0.5,anchor=south,inner sep=0pt] {\scriptsize $0$};
        \draw[thick] (18) arc(180:0:0.5cm) node[pos=0.5,anchor=south,inner sep=0pt] {\scriptsize $0$};

    \end{tikzpicture} \]
    \caption{Caravan decomposition as per equation \eqref{eq:rform} (drawn horizontally here).}
    \label{fig:caravan}
\end{figure}
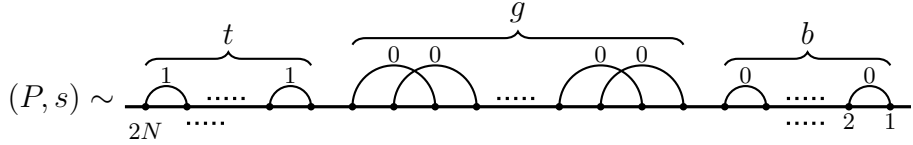

\begin{remark} Prop. \ref{rform} is an analogy of the classification surfaces; just as a (compact) topological surface is determined up to homeomorphism by three things (say for instance, its Euler characteristic, number of boundary components, and orientability), $(P,s)\in \mathcal{TC}_N$ is determined up to chord sliding by the triple of integers $(b,g,t)\in \mathbb{N}^2 \times \{0,1,2\}$ such that \eqref{eq:rform} holds. By analogy, we refer to $(b,g,t)$ as the \textbf{type} of $(P,s)$. Combining Prop. \ref{Prop:bondarycomponentspreserved}, with Theorem \ref{rform}, we have that $|\CC(\del(P,s))|-1=b=\text{null}(\mathit{IM}(P,s))$.  
\end{remark}

\mdef\label{TCDst} Define $\mathcal{TC}^*_N=\{(P,s)\in \mathcal{TC}_N \ | \ | \CC(\del(P,s))|=1\}\subset \mathcal{TC}_N $. It follows from the above remark that $\mathcal{TC}_N^*=\{(P,s) \in \mathcal{TC}_N \ | \ (P,s) \text{ has type }(0,g,t) \in \{0\}\times \N \times \{0,1,2\} \}$ and $\mathcal{TC}^*_N=\{(P,s)\in \mathcal{TC}_N \ | \ \mathit{IM}(P,s) \text{ is invertible}\}$ are two equivalent characterisations. Note that juxtaposition, $\#$, thus defines a map $\mathcal{TC}^*_{N_1}\times \mathcal{TC}^*_{N_2}\to \mathcal{TC}^*_{N_1+N_2}$. 

\section{Square with Bands Diagrams}
In this section, we introduce the main combinatorial objects used in this paper, \textbf{square with bands} (SWB) data and their associated diagrams, as well as the notion of handle slide equivalence on the same objects. These objects are intended to encode embedded curves in (once bounded, possibly non-orientable) surfaces from which we will define a skeletal (linear) monoidal category, extending the Temperley-Lieb category, in \S4.

This section is organised as follows: in \S3.1 we define SWB data, and some basic terminology and results useful in later parts. In \S3.2 we define the vertical juxtaposition of two SWB data, which will later be used to define a categorical composition in \S4. In \S3.3 we define a suitable notion of isotopy on SWB data and using this, in \S3.4 we define the notion of handle slide equivalence, which diagrammatically captures the Morse theoretic notion of handle sliding. In \S3.5 we construct components in SWB diagrams, which geometrically, correspond to boundary parallel curves, and are used to define a tensor product in \S4.
 
\subsection{Square with Bands Diagrams} In this subsection we introduce square with bands (SWB) data and their diagrams. We encourage the reader to view the latter as diagrams of embedded curves in some once bounded surface, in which the underlying surface is drawn by attaching handles or bands to a square. The latter is described by a twisted chord diagram $(P,s)\in \mathcal{TC}_N^*$. To encode such diagrams then, we first introduce frames, $F$, which record the underlying surface, and the points of intersection of the embedded curves with the boundary and the attaching spheres of the handles. Then, stipulating a rule that curves must exit any handle they enter, such diagrams are completely described by a frame $F$ and a crossing-less pairing, $E$, of the aforementioned points of intersection.

It will be helpful to associate to SWB data, certained ordered graphs (\textit{c.f.} \ref{SWBgraph} and \ref{SWBcomp}); this will allow us to give combinatorial characterisations of geometric notions, and also gives us useful criteria for proving the equality of two SWB data (namely Lemma \ref{grp1}). We conclude this section with the Insertion Lemma for SWB data, which is useful for some computations in later subsections.

\mdef\label{frame} A frame $F$ (of rank $N$), is a triple $F=(P,s,f)$ where $(P,s)\in \mathcal{TC}^*_N$ is a twisted chord datum (\textit{c.f.} \ref{TCDst}) and $f$ is a function $f:\overline{P}\to \N$, where $\overline{P}:=P\cup\{\{0\},\{2N+1\}\}$ for $P\in B_{\underline{2N}}$.

Given two frames $F_1$ and $F_2$ (of ranks $N_1$ and $N_2$, respectively), define their \textbf{vertical juxtaposition} as $F_2\# F_1:=(P_2\#P_1,s_2\#s_1,f_2\#f_1)$ (\text{c.f.} \ref{vjuxt}), where $f_2\# f_1:\overline{P_2\# P_1}\to \N$ is given by  $f_2\# f_1(p)=f_1(p)$ for $p\in P_1$, $f_2\# f_1(q+2N_1)=f_2(q)$ for $q \in P_2$, and otherwise
\[ f_2\# f_1(\{0\})=f_1(\{0\}),\quad   f_2\# f_1(\{2(N_1+N_2)+1\})=f_2(\{2N_2+1\}). \]

Define the \textbf{insertion} of $F_1$ in $F_2$ at height $d$, as $F_2\#_d F_1:=(P_2\#_d P_1,s_2\#_d s_1,f_2\#_d f_1)$ (\text{c.f.} \ref{ins}), where $f_2\#_d f_1:\overline{P_2\#_d P_1}\to \N$ is given by $f_2\#_d f_1(p+d )=f_1(p)$ for $p\in P_1$, $f_2\#_d f_1(o_d(q))=f_2(q)$ for $q \in P_2$, and otherwise
\[ f_2\#_d f_1(\{0\})=f_2(\{0\}),\quad   f_2\#_d f_1(\{2(N_1+N_2)+1\})=f_2(\{2N_2+1\}), \]
where $0\leq d\leq 2N_2$ is an integer. Note that the equalities $F_2\#_0 F_1 = F_2\# F_1$ and $F_2\#_{2N_2} F_1= F_1\# F_2$ do not hold in general.

\mdef\label{vset} Given a frame $F=(P,s,f)$ of rank $N$, define two totally ordered vertex sets $(V_F, \prec)$ and $(\V_F, \prec)$, as follows:
\[ 
V_F=\bigcup_{p\in \overline{P} } p\times \underline{f(p) }, \qquad \V_F=\bigcup_{p \in \overline{P}} p \times \underline{(f(p)+1/2)},
\]
where $(i,a)\prec (j,b)$ if $i<j$, $i=j\leq 2N$ and $a<b$, or $i=j=2N+1$ and $a>b$ in each case. Note that the same formulae define a total order on $V_F\cup \V_F$.

We refer to elements of $V_F$ or $\V_F$ of the form $(0,a)$  and $(2N+1, b)$, as \textbf{southern} and \textbf{northern} vertices, respectively, and all remaining elements as \textbf{internal} vertices. Write $(V_F)_{\sth}, (V_F)_{\nth}$, and $(V_F)_{\int}$ for the set of southern, northern, and internal vertices belonging to $V_F$, respectively, write $(V_F)_{\ext}=(V_F)_{\sth}\cup (V_F)_{\nth}$ (the set of \textbf{external} vertices belonging to $V_F$), and do likewise for $\V_F$.  For $i\in \N$ write $(V_F)_i=(\{i\}\times \R) \cap V_F$, and for $p\in \overline{P}$ write $(V_F)_p=(p\times \R )\cap V_F$, and do likewise for $\mathcal{V}_F$.

We associate a \textbf{frame diagram} to $F$ as follows: Firstly, draw a square and place the southern (northern) vertices $(i,a)$ (so $i=0,2N+1$) on the bottom (top) edge of the square, respectively, with $a$ increasing from left to right. Then, mark $2N$ evenly spaced intervals along the right hand edge of the square, place the internal vertices $(i,a)\in V_i$ on the $i$-th such interval (ordered bottom to top) with $a$ increasing upwards, and for each part $\{i,j\}\in P$ join the $i$-th and $j$-th such interval with a band (drawn as a thickened circular arc to the right of the square), which is untwisted if $s(\{i,j\})=0$ and contains a half-twist if $s(\{i,j\})=1$. When arcs in $P$ cross the corresponding bands in the frame diagram will overlap; we regard two frame diagrams as the same if they are related by changing the apparent over/under crossings of bands. 

\begin{example}\label{frex} Let $F_1=(\{p_1,p_2\},s_1,f_1)$ with $p_1=\{1,3\}$, and $F_2=(\{q_1,q_2,q_3\},s_2,f_2)$ with $q_1=\{1,4\}$ and $q_2=\{2,6\}$,  where 
\begin{align*}
    s_1(p_1, p_2)&=(0,1),& f_1(p_0,p_1,p_2,p_{3})&=(3,2,1,1),\\
    s_2(q_1, q_2,q_3)&=(0,1,0),& f_2(q_0,q_1,q_2,q_{3},q_4)&=(2,1,2,3,4),\end{align*}
using $p_0=\{0\},p_3=\{5\}$, and likewise for $q_0, q_4$. Then frame diagrams for $F_1$ and $F_2$ are below   
\[\frone{0.6} \qquad \frtwo{0.41538461538}\]
In future diagrams we will omit vertex labels.
    \end{example}

Note that the ordering on $V_F$ can obtained from a frame diagram for $F$ by regarding the bottom left most vertex and as the minimum and otherwise increasing in an anti-clockwise direction. Furthermore, one can identify vertices in $\V_F$ with certain intervals in the frame diagram such that the order $\prec$ on $\V_F$ (and indeed on $V_F\cup \V_F$) is compatible with this anti-clockwise ordering.  

In what follows, let $F=(P,s,f)$ be a frame of rank $N$. 

\mdef\label{inv} Define a map, $\iota_F : (V_F)_{\int} \to (V_F)_{\int}$, by the formula
\[\iota_F(i,a)=\begin{cases}  (j,f(p)+1-a) & \{i,j\}:=p \in P,\ \& \ s(p)=0,\\ 
(j,a) & \{i,j\}:=p \in P, \ \& \ s(p)=1. \end{cases}\]
Note that, abusing notation, the same formula defines a map $\iota_F: (\V_F)_{\int} \to (\V_F)_{\int}$ and in both cases one has $(\iota_F)^2=\id$.

By construction, one can draw a collection of $f(p)$ parallel arcs in the band corresponding to each part $p\in P$ which joins vertices $v$ and $\iota_F(v)$ for all $v\in (V_F)_p$, in the frame diagram for $F$. This is illustrated in the frame diagram for $F_2$ in \ref{frex}.

\mdef\label{cpx} Define the \textbf{complexity} of $F$,  as 
\[c(F)=\sum_{p\in P} f(p).\]
Observe that $|V_F|=f(\{0\})+f(\{2N+1\})+2c(F)$.

\mdef\label{permSWB} Given a permutation $\sigma \in \mathfrak{S}_{2N}$, define $\sigma(F)=(\sigma(P,s), f\circ \overline{\sigma}^{-1})$, where $\overline{\sigma}$ denotes the unique extension of $\sigma$ to a permutation of $\{0,1,\dots, 2N+1\}$ fixing $0$ and $2N+1$. In addition, define $F^*$ as the frame $F^*=(\omega_{N}(P), s\circ \omega_N, f\circ \omega'_N )$, where $\omega'_N$ denotes the unique order reversing (OR) permutation of $\{0,1,\dots, 2N+1\}$. Observe that $(F^*)^*=F$.

\label{flip} Let $\omega_F$ denote the unique OR map $\omega_F:V_F\to V_{F^*}$ and observe that 
\[\omega_F(i,a)=(\omega_N'(i),f(p(i) )+1-a ),\]
where $p(i)\in \overline{P}$ is defined to be the unique part containing $i$. From this formula, it is clear that a frame diagram for $F^*$ is simply a vertically reflected (upside down) frame diagram for $F$ and that $\omega_F\circ\iota_F=\iota_{F^*}\circ\omega_F: (V_F)_{\int}\to (V_{F^*})_{\int}$.

\mdef\label{SWBdef} A \textbf{square with bands datum} (or SWB-datum) of rank $N$ is a pair $\Th=(F,E)$, where $F=(P,s,f)$ is a frame of rank $N$, and $E$ is a crossingless pair-partition of $(V_F,\prec)$. We say that $\Th$ has type $(n,m)$ if $f(\{0\})=n$ and $f(\{2N+1\})=m$. 

Let $Sq_N(n,m)$ denote the set of all rank $N$, type $(n,m)$ SWB-data, and let $Sq(n,m)=\cup_{N\in \N} Sq_N(n,m)$ (note then, that $Sq(n,m)=\emptyset$ unless $n=m\mod 2$). In addition, define the following subsets of $Sq(n,m)$:
\begin{align*} Sq_{\tor}(n,m)&:= \{ \Th \ :\ (P,s)=\tor\}, &  Sq_{\mob}(n,m)&:= \{ \Th \ :\ (P,s)=\mob\},\\
Sq_+(n,m)&:=\{ \Th \ :\ (P,s) \text{ is orientable}\}. \end{align*}

We associate a \textbf{SWB-diagram} to $\Th=(F,E)$ as follows: Draw a frame diagram for $F$, including the parallel arcs in each band as per the right most diagram in \ref{frex}, and for each part $e=\{u,v\}\in E$ draw an arc connecting vertices $u\in V_F$ and $v \in V_F$ inside the square. The crossingless condition on $E$ ensures all such arcs can be drawn without intersections. 

\mdef\label{SWBdual} Given $\Th=(F,E)$, define $\Th^*=(F^*, \omega_F(E))$ and observe that $(\Th^*)^*=\Th$. From \ref{flip}, it follows that a SWB-diagram for $\Th^*$  is obtained by vertically reflecting one for $\Th$.

\begin{example}\label{SWBex} Let $F_1$ and $F_2$ be as per \ref{frex}. The SWB data $\Th_1:=(F_1,E_1)$ with
\[E_1=\{(0,1)\edge(3,2),(0,2)\edge(3,1),(0,3)\edge(1,1),(1,2)\edge(2,1),(5,1)\edge(4,1)\},\]
and $\Th_1^*$ have the following diagrams (left and right hand diagrams, respectively)
\[\SWBone{0.45} \quad \xmapstoo{(\_)^*}\quad \SWBonef{0.45}\]
Given an SWB diagram with frame $F$ one can unambiguously determine the pairing $E$ on $V_F$, \textit{i.e.} the diagram below on the left determines a unique SWB datum $\Th_2'=(F_2,E_2')$ for $F_2$ as per \ref{frex}. Indeed, we can use the following diagrammatic convention (below right) to represent a generic SWB datum $\Th_2=(F_2,E_2)$ with $E_2$ unspecified:
\[\SWBtwo{0.45} \qquad \SWBtw{0.45}\]
Furthermore, using the box convention (\textit{c.f.} Figure \ref{fig:box}) to represent a generic TCD $(P,s)\in \mathcal{TC}_N^*$, we can represent a completely generic SWB datum $(F,E)$ of type (n,m) with $F=(P,s,f)$ as follows:
\[\genSWB{2}\]
We draw the box for $(P,s)$ enclosed by the surface, since we know this surface has only one boundary component.
\end{example}

The reader will have observed that SWB diagrams resemble once bounded surfaces with embedded curves where the latter correspond to collections of vertices in $V_F$ joined successively by arcs (so far, we have distinguished distinct such curves by drawing them in colour). This motivates the following definition:

\mdef\label{SWBgraph} Given $\Th=(F,E)\in Sq(n,m)$, define the \textbf{curve graph} $G(\Th)=(V_F, E\cup D_F)$ where 
\[D_F=\{ v \edge \iota_F(v)  : \ v \in (V_F)_{\int} \},  \]
\text{i.e.} $D_F$ is a collection of edges which joins internal vertices by parallel arcs in each band. For $p\in P$, let $(D_F)_p=\{ v \edge \iota_F(v) \ | \ v\in (V_F)_p\}$.

A component $\Gamma \in \mathcal{C}(G(\Th))$ is said to be \textbf{external} if $V(\Gamma)$ contains a northern or southern vertex, and \textbf{internal} otherwise. If $V(\Gamma)$ contains no internal vertices, we say that $\Gamma$ is a \textbf{fully external} component.\\

We will often make use of local diagrams, which are ``zoomed in" versions of a diagram, showing only certain features of interest. We will draw such diagrams enclosed in dashed lines, indicating they are to be regarded as having been ``cut out" of a larger SWB diagram. For example, in the following local diagrams we show only the curves passing through one band (twisted or untwisted):
\[\LocalU{0.8} \qquad \LocalT{0.8}\]

\begin{example}\label{CPTex} For all diagrams in \ref{SWBex}, one sees that all curves correspond to external components. Below we give two SWB diagrams for some $\Th_3 \in Sq_{\tor}(3,1)$ and $\Th_4 \in Sq_{\mob}(1,3)$, whose curve graphs contain internal components.
\[\SWBth{0.45} \qquad \SWBfo{0.8} \]
    
\end{example}

\begin{lemma}
\label{grp1} Let $\Th$ and $\Th'$ be two SWB-datum. If $\Th$ and $\Th'$ have the same frame, then $\Th= \Th'$, if and only if, $G(\Th)=G(\Th')$. 
\end{lemma}
\begin{proof} Clearly $\Th= \Th'$ implies $G(\Th)=G(\Th')$. Now write, $\Th=(F,E)$, $\Th'=(F,E')$ and suppose $G(\Th)=G(\Th')$. If $e \in E\smallsetminus E'$ then $e\in D_F$ (otherwise $E(G(\Th))\neq E(G(\Th'))$), but then a vertex in $e$ has valence $2$ in $G(\Th)$ and valence $1$ in $G(\Th')$, a contradiction. Thus $E\smallsetminus E'=\emptyset$, and by symmetry $E=E'$. \end{proof}

\begin{remark} Note that we can relax the last part of the above to $G(\Th)\simeq_O G(\Th')$ (\textit{c.f.} \ref{OG}), since having the same frame guarantees both graphs have the same vertex set. We will use this fact frequently.\end{remark}

\mdef\label{twist} Given a SWB $\Th=(F,E)$, and a component $\Gamma\in \mathcal{C}(G(\Th))$, define the \textbf{twist} of $\Gamma$, $\tau_\Gamma\in \Z_2$, by the formula
\[\tau_{\Gamma}= \sum_{p \in P} \left|(D_F)_p\cap E(\Gamma)\right|s(p).\]
The twist of a component then, is simply the $\Z_2$-count of the number of twisted bands the corresponding curve passes through. Geometrically then, the twist of a component is 1 if the corresponding curve is $1$-sided, and 0 if it is $2$-sided. 

\begin{example}\label{TWex} From \ref{SWBex}, the red components in the diagrams for $\Th_1$ and $\Th_2'$ have twist 1. From \ref{CPTex}, the red component in $\Th_4$ has twist 1, whereas the blue components has twist 0. 
    
\end{example}
\mdef\label{cptfn} Given a SWB $\Th=(F,E)$, $F=(P,s,f)$, and a component $\Gamma \in \mathcal{C}(G(\Th))$, define $g_{\Gamma}:\overline{P}\to \Z$ by 
\[g_{\Gamma}(p)=\left|(V_F)_p\cap V(\Gamma)\right|.\]

In the following proposition, we establish how to ``delete curves" from SWB-data:

\begin{prop}\label{del} Let $\Th=(F,E)\in Sq_N(n,m)$, and let $\Gamma\in \mathcal{C}(G(\Th))$. There exists a unique SWB datum $\Th\smallsetminus \Gamma$ with frame $F_{-\Gamma}=(P,s,f-g_{\Gamma})$, such that $G(\Th\smallsetminus \Gamma)\simeq_{O} G(\Th)\smallsetminus \Gamma$.
\end{prop}
\begin{proof} Since $|V_F\smallsetminus V(\Gamma)|=|V_{F_{-\Gamma}}|$ by construction, we may set $o :V_F\smallsetminus V(\Gamma) \to V_{F_{-\Gamma}}$ as the unique OP map. Then $\Th':=(F_{- \Gamma}, o(E\smallsetminus E(\Gamma)  ) )$ is a SWB-datum and 
\[G(\Th')=(V_{F_{-\Gamma}}, o(E\smallsetminus E(\Gamma))\cup D_{F_{-\Gamma}} ).\]
Since $|(V_F)_i \smallsetminus V(\Gamma)|=|(V_{F_{-\Gamma}})_i|$ for all $i=0,\dots, 2N+1$, one has that $o((V_F)_i \smallsetminus V(\Gamma))=(V_{F_{-\Gamma}})_i$. 
One can futher check that for any $v\in (V_F)_\int \smallsetminus V(\Gamma)$, we have $o \circ \iota_F(v)=\iota_{F_{-\Gamma}}\circ o(v)$, from which it follows that $D_{F_{-\Gamma}}=o(D_F)$, and hence $G(\Th')=o(G(\Th)\smallsetminus \Gamma )$. Thus we may take $\Th\smallsetminus\Gamma =\Th'$. Uniqueness is guaranteed by \ref{grp1}.
\end{proof}

Henceforth, let $o_{\Gamma} :V_F\smallsetminus V(\Gamma) \to V_{F_{-\Gamma}}$ denote the unique OP map. Thus in particular $G(\Th\smallsetminus \Gamma)= o_{\Gamma}(G(\Th)\smallsetminus \Gamma)$.

\begin{prop}\label{delp} Let $\Th=(F,E)$ be a SWB and let $\Gamma\in \mathcal{C}(G(\Th))$. The following hold:
\begin{enumerate}
\item\label{delp1} For $\Gamma \neq \Lambda\in \mathcal{C}(\Th)$, set $\Gamma':=o_{\Lambda}(\Gamma)\in \mathcal{C}(G(\Th\smallsetminus \Lambda))$ and $\Lambda'=o_{\Gamma}(\Lambda)\in \mathcal{C}(G(\Th\smallsetminus \Gamma))$. Then $\tau_{\Gamma'}=\tau_{\Gamma}$, $\Gamma$ is internal if and only if $\Gamma'$ is, and
\[(\Th\smallsetminus \Gamma)\smallsetminus \Lambda'=(\Th\smallsetminus \Lambda)\smallsetminus \Gamma'.\] 
\item\label{delp2} $G(\Th^*)\simeq_O G(\Th)^*$. Furthermore, set $\Gamma^*=\omega_F(\Gamma)\in \mathcal{C}(G(\Th^*))$. Then $\tau_{\Gamma^*}=\tau_{\Gamma}$, $\Gamma^*$ is internal if and only if $\Gamma$ is, and 
\[(\Th\smallsetminus \Gamma)^*=\Th^*\smallsetminus \Gamma^*.\]
\end{enumerate}
\end{prop}
\begin{proof} For (1), since $\iota_{F_{-\Lambda}}\circ o_{\Lambda}=o_{\Lambda }\circ \iota_F$ (where defined), and since $o_{\Lambda}((V_F)_p \smallsetminus V(\Lambda))=(V_{F_{-\Lambda}})_p$ for any $p\in P$, $o_{\Lambda}$ induces a bijection $o_{\Lambda}:E(\Gamma) \cap D_F \to E(\Gamma') \cap D_{F_{-\Lambda}}$ which restricts to bijections $E(\Gamma) \cap (D_F)_p \to E(\Gamma') \cap (D_{F_{-\Lambda}})_p$. It therefore follows that $\tau_{\Gamma'}=\tau_{\Gamma}$. The second claim is clear since $o_{\Lambda}$ restricts to a bijection $(V_F)_{\int}\smallsetminus V(\Lambda) \overset{\sim}{\to} (V_{F_{-\Lambda}})_{\int}$. The third claim follows by \ref{grp1} since $g_{\Gamma'}=g_{\Gamma}$ (and likewise for $\Lambda$) and since (using \ref{contpr})
\[G\left( (\Th\smallsetminus \Gamma)\smallsetminus \Lambda'\right)\simeq_{O}G(\Th) \smallsetminus(\Gamma\cup \Lambda )\simeq_{O}G\left( (\Th\smallsetminus \Lambda)\smallsetminus \Gamma'\right).\]
For (2), one can check using the formula in \ref{flip} that $\omega_{F}\circ \iota_F =\iota_{F^*} \circ \omega_{F}$, from which it follows that $\omega_F(D_F)=D_{F^*}$ and hence $G(\Th^*)=\omega_F(G(\Th))$. Clearly, $\omega_F$ descends to a bijection $(V_F)_\int\overset{\sim}{\to} (V_{F^*})_\int$, which proves the second claim about $\Gamma^*$. For the first we note that $\omega_{F}$ induces a bijection $E(\Gamma)\cap D_F\overset{\sim}{\to} E(\Gamma^*)\cap D_{F^*}$ which descends to bijections $E(\Gamma)\cap (D_F)_p\overset{\sim}{\to} E(\Gamma^*)\cap (D_{F^*})_{\omega_N(p)}$ for all $p\in P$. The last claim follows by \ref{grp1}, since $g_{\Gamma^*}=g_{\Gamma}\circ \omega_{N}'$, and (using \ref{contpr})
\[G((\Th\smallsetminus \Gamma)^*)\simeq_O (G(\Th)\smallsetminus \Gamma)^*=G(\Th)^*\smallsetminus \Gamma^*\simeq G(\Th^*\smallsetminus \Gamma^*).\qedhere\]
\end{proof}

Just as we can describe the curves in a SWB-diagram via the curve graph $G(\Th)$, so too we can describe the complement of these curves with a graph as follows:

\mdef\label{SWBcomp} Given a SWB $\Th=(F,E)$, define the \textbf{complement graph} $\mathcal{G}(\Th)$ to be the graph with vertex set $\V_F$, and edge set $\mathcal{D}_F\cup \mathcal{E}$ where
\begin{align*}\mathcal{D}_F&=\{ v\edge \iota_F(v) \ : \ v \in (\V_F)_{\int} \}, \\
\mathcal{E}&=\{ u\edge v \ : \ u,v \in \V_F, \ \{u,v\} \text{ crosses no part in } E  \}.
\end{align*}
Here we regard $\prec$ as defining a total order on $V_F\cup \V_F$ (\textit{c.f.} \ref{vset}).

\begin{example} Below we illustrate the complement graph on the SWB diagram for $\Th_1$ (\textit{c.f.} \ref{SWBex}). The vertices in $\V_F$ (drawn as black circles) correspond to regions between the vertices in $V_F$, and the edges (drawn as dashed curves) are all possible pairings of the vertices which do not cross curves in the SWB-diagram. Thus one sees that components of $\G(\Th)$ correspond to components of the complement of the curves in a given SWB diagram, for example $|\CC(\G(\Th_1))|=1$.
    \[\CPTone{0.6}\]
Furthermore, observe that components of the subgraph $(\V_F, \mathcal{E})$ correspond to the connected components of the complement of the curves inside the square, now ignoring the bands.
\end{example}

\begin{prop}\label{comp1} Let $\Th=(F,E)$ be a SWB datum. Then $|\mathcal{C}(\mathcal{G}(\Th))|\leq |\mathcal{C}(G(\Th))|+1$.
\end{prop}
\begin{proof} For $v=(i,a)\in V_F$, let $v^{\pm}$ denote the vertex $(i,a\pm 1/2)\in \V_F$. For a path $v_1\edge v_2\edge \dots\edge v_k\subset G(\Th)$, one has that $v_1^{\pm}\edge v_2^{\pm}\edge \dots\edge v_k^{\pm}\subset \G(\Th)$, where the signs at each step are determined as follows: For $u\edge v\in E $ we have $u^{\pm}\edge v^{\mp}\in \mathcal{E}$, and for $u\edge v \in (D_F)_p$ we have $u^{\pm}\edge v^{\mp} \in \mathcal{D}_F$ if $s (p)=0$, and $u^{\pm}\edge v^{\pm}\in \mathcal{D}_F$ for $s (p)=1$. 

Now let $\mathcal{H}\subset \G(\Th)$ be a connected component, and set $\bm{v}:=\min_{\prec}V(\mathcal{H})$. We now claim that either $\bm{v}=v^+$ where $v=\min_{\prec}(V(\Gamma))$ for some connected component $\Gamma \subset G(\Th)$, or $\bm{v}=(0,1/2)$. If $\bm{v}\neq (0,1/2)$ then $\bm{v}=v^+$ for some $v\in V_F$, as otherwise, $\bm{v}=(i,1/2)$ for $i>0$ implies $(i,1/2)-\max_{\prec}((V_F)_{i-1} )^+$ contradicting minimality. Let $\Gamma \in \CC(G(\Th))$ be such that $v\in V(\Gamma)$, and suppose that $v\neq v':= \min_{\prec}V(\Gamma)$. By the previous paragraph, there exists a path between $\bm{v}=v^+$ and $(v')^\pm$ in $\mathcal{H}$ which contradicts minimality of $\bm{v}$. Thus we have proven the claim, and it follows that $\G(\Th)$ may be written as a union of components
\[\G(\Th)=\G_{(0,1/2)}\cup \left(\cup_{\Gamma \in \mathcal{C}(G(\Th)) }\G_{ \min_{\prec}(V(\Gamma))^{+} }\right),\]
where $\G_{\bm{v}}$ denotes the unique component of $\G(\Th)$ containing the vertex $\bm{v}$. Since this union is in general not disjoint, we obtain the desired bound on $|\CC(\G(\Th))|$.\end{proof}

\mdef Given a SWB $\Th=(F,E)$ and a component $\Gamma\in \CC(G(\Th))$, let $\Th_{\Gamma}=(F_\Gamma,E_\Gamma)$ be the SWB obtained from $\Th$ by deleting all components except $\Gamma$ (this is well defined by \ref{delp} part \ref{delp1}). We say that $\Gamma$ is \textbf{separating} (in $\Th$) if $\G (\Th_{\Gamma})$ has two components, and \textbf{non-separating} if it has one.

\begin{prop}\label{comp2} Let $\Th=(F,E)$ be a SWB and let $\Gamma\in \mathcal{C}(G(\Th))$. The following hold:
\begin{enumerate}
\item\label{comp21} If $\tau_\Gamma=1$, then $\Gamma$ is non-separating.

\item\label{comp22} $\G(\Th^*)\simeq_O \G(\Th)^*$. In particular, $\Gamma$ is separating in $\Th$  if and only if $\omega_F(\Gamma)$ is separating in $\Th^*$.

\item\label{comp23} For $\Gamma \neq \Lambda\in \mathcal{C}(\Th)$, $\omega_\Lambda(\Gamma)$ is separating in $\Th\smallsetminus \Lambda$, if and only if, $\Gamma$ is separating in $\Th$. 
\end{enumerate}
\end{prop}
\begin{proof} For (1), set $\mathcal{G}=\mathcal{G}(\Th_{\Gamma})$. From the proof of \ref{comp1}, it is enough to prove that $(0,1/2)\sim_{\G} v^+$, where $v=\min V_{F_\Gamma}$. If $\Gamma$ is external with $v$ southern and $w=\max V_{F_\Gamma}$ northern, it follows that $\G$ contains a path between $v^+$ and $w^{+}=\max (\V_{F_{\Gamma}})$, since $\Gamma$ consists of odd number of arcs coming from twisted bands and an even number of other arcs, and thus $(0,1/2)\sim_{\G} w^+ \sim_{\G} v^+$. If $\Gamma$ is internal with two southern vertices as end points, say $v(=\min V_{F_\Gamma})$ and $u$, then by a similar argument $(0,1/2)=v^-\sim_{G} u^-=v^+$. If $\Gamma$ is internal then $\G$ contains a path between $v^-$ and $v^+$, since $\Gamma$ now consists of an odd number of arcs through twisted bands and an odd number of other arcs, thus $(0,1/2)\sim_\G v^- \sim_\G v^+$.

For part (2), we observe that the identity $\omega_{F}\circ \iota_F =\iota_{F^*} \circ \omega_{F}$, used in the proof of \ref{delp}, remains true when both sides are regarded as maps $(\V_F)_\int\to (\V_{F^*})_\int$. Thus we have $\mathcal{D}_{F^*}=\omega_F(\mathcal{D}_F)$. Using $\omega_{F^*}\circ \omega_F=\text{id}$ one sees that 
\[\omega_F(\mathcal{E})=\{u\edge v \ : \ u,v\in \V_{F^*}, \{u,v\} \text{ crosses no part in } \omega_F(E)\}.\]
Hence $\G(\Th^*)=\omega_F(\G(\Th))$.  Part (3) is immediate.
\end{proof}

\begin{lemma}
    [Insertion Lemma]\label{insSWB} Let $\Th=(F,E)$ be a SWB with $F=(P,s,f)$, and suppose $(P,s)=(P_2,s_2)\#_{d}(P_1,s_1)$, where $(P_i,s_i)\in \mathcal{TC}_{N_i}$, and $1\leq d\leq 2N_2$. Then there exist unique frames $F_1=(P_1,s_1,f_1)$ and $F_2=(P_2,s_2,f_2)$, and a SWB, without fully external components, $\Th_1=(F_1,E_1)$, together with an OP map $o:V_{F_1}\to V_F$, obeying $o(i,a)=(i+d,a)$ for all $(i,a)\in (V_{F_1})_I$, such that 
\[\Th=(F_2\#_N F_1, o(E_1)\sqcup \left(E\smallsetminus o(E_1) \right) ).\]
\end{lemma}
\begin{proof} Let $U:=\cup_{i=0}^{d-1} (V_F)_{i},  V:=\cup_{p\in P_1} (V_F)_{p+d}, W:=V_F\smallsetminus (U\cup V)$, and define vertices $u_i$ and $w_i$, by
\begin{align*}
\{u_1\prec u_2\prec \dots \prec u_n\}&:=\{u \in U \ : \ \text{there exists } v\in V \text{ with } u-v \in E \},  \\
\{w_1\succ w_2 \succ \dots\succ w_m\}&:=\{w \in W \ : \ \text{there exists } v \in V \text{ with } w-v \in E  \}. 
\end{align*}  
Now let $v_i$ ($v^j$) be the unique vertex in $V$ such that $u_i\edge v_i\in E$ $(w_j-v^j\in E)$, respectively. Define $f_1:\overline{P_1}\to \N$ by $f_1(\{0\})=n, f_1(\{2N_1+1\})=m$ and $f_1(p)=f(p+d)$ for $p\in P_1$, define $f_2:\overline{P_2}\to \N$ by $f_2=f\circ o_d$ (using \ref{ins}), and set $F_i=(P_i,s_i,f_i)$ for $i=1,2$. Let $\tilde{o}: V \to (V_{F_1})_\int$ be the unique OP map, and define
\[E_1=\tilde{o}\left(\{e \in E \ : \ e\subset V\}\right)\cup \{(0,i)\edge \tilde{o}(v_i) \ : \ 1\leq i \leq n\} \cup \{(2 N_1+1,j)\edge \tilde{o}(v^j) \ : \ 1\leq j \leq m\}.  \]  
Then by construction $(F_1,E_1)\in Sq_{N_1}(n,m)$, and is without fully external components. Define an OP map $o':V_{F_1}\to V_{F}$ by
\[o'(v)= \begin{cases}(\tilde{o})^{-1}(v), &v \text{ is internal}, \\
u_i, &v=(0,i) \in (V_{F_1})_\sth, \\
w_j, &v=(2N_1+1,j) \in (V_{F_1})_\nth.\end{cases}\]
By construction, $o'(E_1)\subset E$ and so the result follows taking $o=o'$.\end{proof}

\begin{figure}[ht]
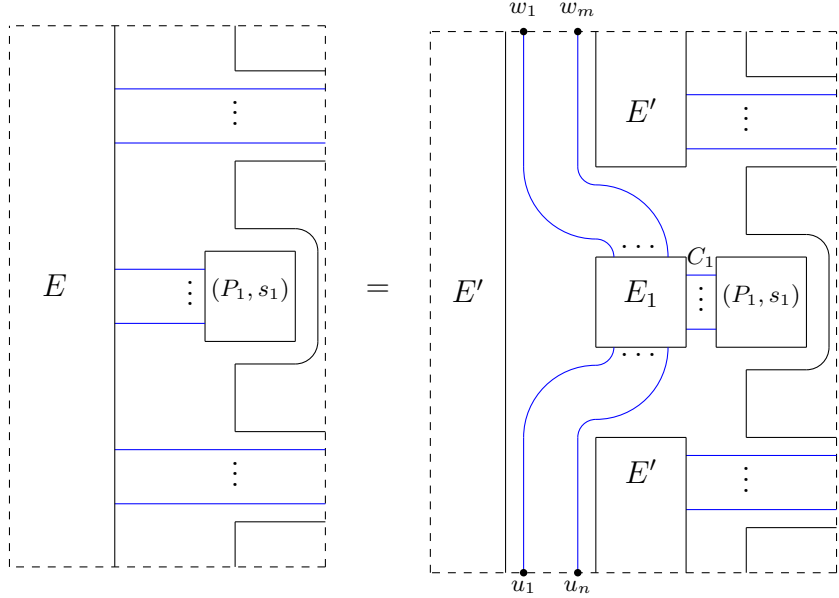

    \centering
    \[\inslone{1.2} \quad = \quad \insltw{1.2}  \]
    \caption{Local diagram for Lemma \ref{insSWB}, where $E':=E\smallsetminus o(E_1)$, and $C_1=C(F_1)$. }
    \label{fig:ins}
\end{figure}

\subsection{Vertical Juxtaposition} In this subsection we define the vertical juxtaposition of SWB data $(\_) \# (\_) : Sq(m,l)\times Sq(n,m) \to Sq(n,l)$, which will be used in \S4 to define a categorical composition. After defining this operation, we establish some of its key properties in Prop \ref{vjassoc}, and prove the Factorisation Lemma for SWB in Lemma \ref{vjfact}. 

\mdef\label{vjvset} Given SWB data $\Th_1=(F_1,E_1)\in Sq_{N_1}(n,m)$ and $\Th_2=(F_2,E_2)\in Sq_{N_2}(m,l)$, define a set $V_{F_1,F_2}$ as 
\[V_{F_1,F_2}=V_{F_1}\sqcup V_{F_2}/(V_{F_1}\ni (2N_1+1,a) \sim (0,a) \in V_{F_2} \text{ for } a=1,\dots, m).\]
Let $(\_)^{(i)}:V_{F_i}\to V_{F_1,F_2}$ denote the canonical inclusion maps, so that in particular, $(2N_1+1,a)^{(1)}=(0,a)^{(2)}$ for all $a=1,\dots, m$. Let $U_{F_1,F_2}\subset V_{F_1,F_2}$ be the subset $(V_{F_1})_{\nth}^{(1)}=(V_{F_2})_{\sth}^{(2)}$, and let $W_{F_1,F_2}=V_{F_1,F_2}\smallsetminus U_{F_1,F_2}$. Define a total order on $W_{F_1,F_2}$ by $v^{(i)}\prec u^{(j)}$ if $i<j$, or if $i=j$ and $v\prec u$ in $V_{F_i}$. Noting that $|W_{F_1,F_2}|=|V_{F_2\#F_1}|$ (\textit{c.f.} \ref{frame}), let $o_{F_1,F_2}:W_{F_1,F_2}\to V_{F_2\#F_1}$ denote the unique OP map. 

Define the \textbf{composite graph} of the pair $(\Th_1,\Th_2)$ as the graph (with vertex set $V_{F_1,F_2}$) given by
\[G(\Th_1,\Th_2):=G(\Th_1)^{(1)}\cup G(\Th_2)^{(2)},\]
and define the \textbf{linking number} $L_{\Th_1,\Th_2}$ as the index (\textit{c.f.} \ref{con1})  $L_{\Th_1,\Th_2}=[U_{F_1,F_2}, G(\Th_1,\Th_2)]$. 

Similarly, given SWB data $\Th_i=(F_i,E_i)\in Sq_{N_i}(n_i,n_{i+1})$ for $i=1,\dots, k$, define 
\[V_{\bm{F}}= \left[\sqcup_{i=1}^k V_{F_i}\right]/ \left((V_{F_i})_{\nth} \ni (2N_i+1,a) \sim (0,a) \in (V_{F_{i+1}})_{\sth} \text{ for } a=1,\dots, n_{i+1}\right),\]
(where $\bm{F}=(F_1,\dots, F_k)$). Define $U_{F_i,F_{i+1}}=(V_{F_i})_{\nth}^{(i)}$, and $W_{\bm{F}}=V_{\bm{F}}/(\cup_i U_{F_i,F_{i+1}})$. Observe, the previous formulae defines a total order on $W_{\bm{F}}$ also. Define the composite graph of the tuple $\bm{\Th}=(\Th_1,\dots, \Th_k)$ as $G(\bm{\Th})=\cup_{i} G(\Th_i)^{(i)}$ on vertex set $V_{\bm{F}}$, and define the linking number as $L_{\bm{\Th}}=[\cup_i U_{F_i,F_{i+1}} , G(\bm{\Th})]$.

\begin{example}\label{vjuxteg} The composite graph of the pair $(\Th_1,\Th_2)$ can be visualised by vertically stacking the SWB diagram for $\Th_2$ on top that for $\Th_1$, lining up the vertices $(0,a)\in V_{F_2}$ with $(2N_1+1,a)\in V_{F_1}$ for each $a$. For example, let $\Th_1 \in Sq(3,9), \Th_2\in Sq(9,5)$ be the SWB data represented by the diagrams below. 
    \[\vjuxtega{0.7} \qquad \qquad \vjuxtegb{0.4}\]
Then the composite graph $G=G(\Th_1,\Th_2)$ is visualised below. The vertices drawn as large blue dots are those comprising the set $U_{F_1,F_2}$. One reads off the linking number $L_{\Th_1,\Th_2}=1$ since there is one component of $G$ whose vertex set is contained in $U_{F_1,F_2}$ (namely, the component drawn with red dashed lines).
\[\vjuxteg{0.8} \quad \to \quad  \vjuxtegg{0.6} \]
By deleting all such components from the above diagram, it can be regarded as an SWB diagram on frame $F_2 \# F_1$ by simply ignoring the blue vertices as depicted on the right above. In particular, the ordering on $W_{F_1,F_2}\subset V_{F_1,F_2}$ is compatible with the ordering on $V_{F_1\#F_2}$ as the ``anticlockwise" ordering. This motivates what follows.  
\end{example}

\begin{prop}[Vertical Juxtaposition]\label{vjuxtSWB} Let $\Th_1=(F_1,E_1)\in Sq_{N_1}(n,m)$ and $\Th_2=(F_2,E_2)\in Sq_{N_2}(m,l)$. Then there exists a unique SWB-datum $\Th_2\# \Th_1\in Sq_{N_1+N_2}(n,l)$ with frame $F_2\# F_1$, such that $G(\Th_2\# \Th_1)\simeq_O G(\Th_1,\Th_2)/ U_{F_1,F_2}$. 
\end{prop}

\begin{proof} Let $o=o_{F_1,F_2}$ for this proof. Set $\Th=(F_2\#F_1, o(E) )$, where $E$ is uniquely defined by 
\[(W_{F_1,F_2},E)=(V_{F_1}, E_1)^{(1)}\cup (V_{F_2}, E_2)^{(2)}/U_{F_1,F_2}.\]
 Then one has that $\Th\in Sq_{N_1+N_2}(n,l)$ by \ref{TLarg}. One can now check the identities $o\circ \iota_{F_1}^{(1)}=\iota_{F_2\#F_1}\circ o$ and $o\circ \iota_{F_2}^{(2)}=\iota_{F_2\#F_1}\circ o$, wherever both sides are defined. Thus we obtain $D_{F_2\#F_1}=o(D_{F_1}^{(1)}\cup D_{F_2}^{(2)})$, from which it follows that
\begin{align*}
G(\Th)&=o\left( (V_{F_1,F_2}, E_1^{(1)}\cup E_2^{(2)})/U_{F_1,F_2}\right) \cup (V_{F_2\#F_1}, D_{F_2\#F_1})\\
&=o\left( \left( V_{F_1,F_2}, (E_1\cup D_{F_1})^{(1)}\cup (E_2\cup D_{F_2})^{(2)} \right) /U_{F_1,F_2}\right)=o(G(\Th_1,\Th_2)/ U_{F_1,F_2}),
\end{align*}
using \ref{contpr}. Thus $\Th$ has the desired properties for $\Th_2\# \Th_1$, uniqueness follows by \ref{grp1}. 
\end{proof}

\begin{example}The SWB diagram for $\Th_2\#\Th_1$ with $\Th_1$ and $\Th_2$ as per \ref{vjuxteg}, is precisely the diagram on the right of the composite graph in the same example. \end{example}

\begin{prop}[Properties of Vertical Juxtaposition]\label{vjassoc} Given $(\Th_1, \Th_2, \Th_3)\in Sq_{M_1}(n,m)\times Sq_{M_2}(m,l)\times Sq_{M_3}(l,k)$, write $\Th_i=(F_i,E_i)$ for each $i$. The following hold:
\begin{enumerate}
    \item $\mathcal{G}(\Th_2\# \Th_1)\simeq_O\mathcal{G}(\Th_1)^{(1)}\cup \mathcal{G}(\Th_2)^{(2)}/\mathcal{U}_{F_1,F_2}$ (defined analogously to in \ref{vjvset}).
    \item $\Th_3\#(\Th_2\# \Th_1)=(\Th_3\#\Th_2)\#\Th_1\in Sq_{M}(n,k)$ where $M=M_1+M_2+M_3$ and  $L_{\Th_1,\Th_2}+L_{\Th_2\#\Th_1,\Th_3}=L_{\Th_1,\Th_3\#\Th_2}+L_{\Th_2,\Th_3}$.
    \item $(\Th_2\#\Th_1)^*=\Th_1^*\#\Th_2^*\in Sq_{M_1+M_2}(l,n)$ and $L_{\Th_1,\Th_2}=L_{\Th_2^*,\Th_1^*}$.
    \item For any internal component $\Gamma\subset G(\Th_1)$, $\Gamma':=o_{F_1,F_2}(\Gamma^{(1)})\subset G(\Th_2\#\Th_1)$ is an internal component such that $\Th_2\#(\Th_1\smallsetminus \Gamma)=(\Th_2\#\Th_1)\smallsetminus \Gamma'$, and $L_{\Th_1,\Th_2}=L_{\Th_1\smallsetminus\Gamma,\Th_2}$. Furthermore, $\tau_{\Gamma}=\tau_{\Gamma'}$, and $\Gamma'$ is separating if and only if $\Gamma$ is.   \end{enumerate}\end{prop}

\begin{proof} For (1), placing elements of $\V_{F_1}^{(1)}$ and $\V_{F_2}^{(2)}$ in a diagram for the composite graph $G(\Th_1,\Th_2)$ in the obvious way, it is clear that for $v_1^{(i_1)}, v_2^{(i_2)}\in \mathcal{W}_{F_1,F_2}$ we have that $o_{F_1,F_2}(v_1^{(i_1)})$, $o_{F_1,F_2}(v_2^{(i_2)}) \in \V_{F_2\#F_1}$ \footnote{here we uniquely extend $o_{F_1,F_2}$ to an OP map $V_{F_1,F_2}\cup \V_{F_1,F_2}\smallsetminus(U_{F_1,F_2}\cup \mathcal{U}_{F_1,F_2} )\to  V_{F_2\#F_1}\cup \V_{F_2\#F_1}$.} belong to the same component of the curve complement within the square (for a diagram of $\Th_2\#\Th_1$), if and only if, there is a path (possibly of length 2) through $\mathcal{U}_{F_1,F_2}$ in $\mathcal{G}(\Th_1)^{(1)}\cup \mathcal{G}(\Th_2)^{(2)}$ (\textit{c.f.} Figure \ref{fig:vjcomp}). Thus the desired result follows.
\begin{figure}[ht]
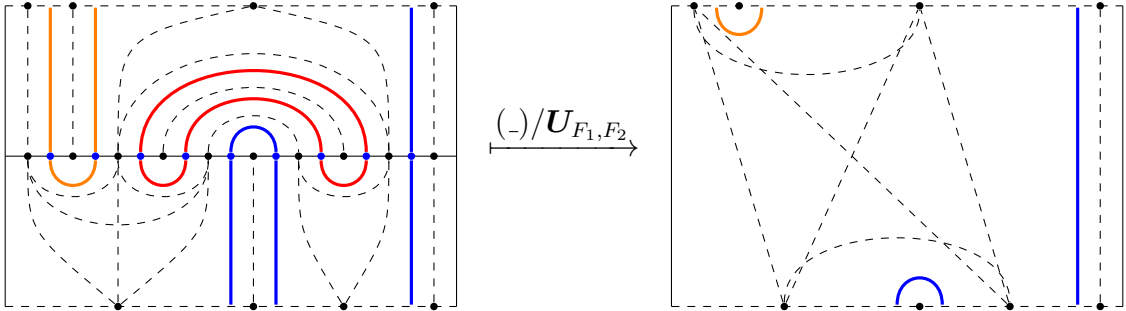


\[ \vjuxtcompa{2} \quad \xmapstoo{ (\_ )/\bm{U}_{F_1,F_2}}\quad  \vjuxtcompb{2}\]
    \caption{Local diagram for Prop. \ref{vjassoc}, part 1 with $\Th_1$ and $\Th_2$ as per \ref{vjuxteg} ($\bm{U}_{F_1,F_2}=U_{F_1,F_2}\cup \mathcal{U}_{F_1,F_2}$). }
    \label{fig:vjcomp}
\end{figure}

For (2), set $\Th_L=\Th_3\#(\Th_2\# \Th_1)$ and $\Th_R=(\Th_3\#\Th_2) \# \Th_1$. One can check that $\Th_L$ and $\Th_R$ have equal frames (\textit{c.f.} \ref{assoc1}). Furthermore, using \ref{contpr}, one has that 
\begin{align*}G(\Th_L )\simeq_{O} G(\Th_1,\Th_2,\Th_3)/(U_{F_1,F_2}\cup U_{F_2,F_3})\simeq_O G((\Th_R ),\end{align*}
hence $\Th_L=\Th_R$ by \ref{grp1}. The identity $L_{\Th_1,\Th_2}+L_{\Th_2\#\Th_1,\Th_3}=L_{\Th_1,\Th_3\#\Th_2}+L_{\Th_2,\Th_3}$ follows by \ref{contpr}; both sides are equal to $L_{(\Th_1,\Th_2,\Th_3)}$.

For (3), one can check that $(F_2\# F_1)^*=F_1^*\# F_2^*$ and using \ref{contpr}, we compute
\begin{align*}G(\Th_1^*\# \Th_2^*)&\simeq_{O} (\omega_{F_1}(G(\Th_1))^{(2)}\cup  \omega_{F_2}(G(\Th_2))^{(1)})/U_{F_2^*,F_1^*}\\
&=\omega( G(\Th_1)^{(1)} \cup G(\Th_2)^{(2)} / U_{F_1,F_2})\simeq_O G(\Th_2\#\Th_1)^*, \end{align*}
where $\omega:V_{F_1,F_2}\to V_{F_2^*,F_1^*}$ is the map given by $\omega(v^{(1)})=\omega_{F_1}(v)^{(2)}$ and $\omega(u^{(2)} )=\omega_{F_2}(u)^{(1)}$, which satisfies $\omega(U_{F_1,F_2})= U_{F_2^*,F_1^*}$ and restricts to the unique OR map $W_{F_1,F_2}\to W_{F_2^*,F_1^*}$. Furthermore, one computes 
\[L_{\Th_2^*,\Th_1^*}=[U_{F_2^*,F_1^*}, G(\Th_2^*,\Th_1^*)]=[U_{F_1,F_2}, G(\Th_1,\Th_2)] =L_{\Th_1,\Th_2}. \]
For (4), clearly $o_{F_1,F_2}$ maps $(V_{F_1})^{(1)}_{\int}$ into $(V_{F_2\# F_1})_\int$, so $\Gamma'$ is internal. Furthermore, $g_{\Gamma'}|_{P_1}=g_{\Gamma}|_{P_1}$, and $g_{\Gamma'}$ and $g_{\Gamma}$ are 0 elsewhere, hence $F_2\#(F_1)_{-\Gamma}=(F_2\#F_1)_{-\Gamma'}$ (where we have written $F_1=(P_1,s_1,f_1)$). Using \ref{contpr}, one has that 
\[G(\Th_2\#(\Th_1\smallsetminus \Gamma)) \simeq_O (G(\Th_1,\Th_2)/U_{F_1,F_2})\smallsetminus \Gamma^{(1)} \simeq_O G((\Th_2\#\Th_1)\smallsetminus \Gamma'),\]
Hence $\Th_2\#(\Th_1\smallsetminus \Gamma)=(\Th_2\#\Th_1)\smallsetminus \Gamma'$ follows by \ref{grp1}. 

To prove $\tau_{\Gamma}=\tau_{\Gamma}$, observe that $o_{F_1,F_2}^{(1)}$ induces a bijection $E(\Gamma)\cap D_{F_1}\to E(\Gamma')\cap D_{F_2\# F_1} $ which restricts to bijections $E(\Gamma)\cap (D_{F_1})_p\to E(\Gamma')\cap (D_{F_2\# F_1})_p$ for all $p\in P_1$, and that $E(\Gamma')\cap (D_{F_2\# F_1})_{q+2N_1}=\emptyset$ for all $q\in P_2$. For the last claim, we use that $(\Th_2\#\Th_1)_{\Gamma'}=\emptyset_{(P_2,s_2)} \# (\Th_1)_{\Gamma}$ with part (1), to conclude $|\G((\Th_1)_{\Gamma})|=|\G((\Th_2\#\Th_1)_{\Gamma'})|$. 
\end{proof}

\begin{lemma}[Factorisation Lemma]\label{vjfact} Let $\Th =(F,E)$ be a SWB with $F=(P,s,f)$, and suppose that $(P,s)=(P_2,s_2)\# (P_1,s_1)$ for $(P_a,s_a)\in \mathcal{TC}^*_{N_i}$. Then there exists frames $F_1=(P_1,s_1,f_1)$ and $F_2=(P_2,s_2,f_2)$, and SWB data $\Th_a=(F_{a},E_a)$, such that $\Th=\Th_2\#\Th_1$ and $L_{\Th_1, \Th_2}=0$. \end{lemma}

\begin{proof} Let $N=N_1+N_2, V_1=\cup_{i=0}^{2N_1} (V_F)_i$, $V_2=\cup_{i=2N_1+1}^{2N+1} (V_F)_i$, and partition $E$ into sets $E_1'=\{e \in E \ | \ e \subset V_1\},  E_2'=\{e \in E \ | \ e \subset V_2\},  E'=E\smallsetminus(E_1'\cup E_2')$. Let $|E'|=m$. The crossing-less condition for $E$ now implies that 
\[ E'=\{ u_i\edge v_i \ | \ u_i\in V_1, \ v_i\in V_2, \ i=1,\dots, m \}, \]
where $u_1\prec u_2 \prec \dots \prec u_m\prec v_m \prec v_{m-1}\prec \dots \prec v_1$. 

Define $f_1:\overline{P_1}\to \N$ by $f_1(p)=f(p)$ for $p\in P_1$, $f_1(\{0\})=f(\{0\})$, and $f_1(\{2N_1+1\})=m$, and define $f_2:\overline{P_2}\to \N$ by $f_2(p)=f(p+2N_1)$ for $p\in P_2$, $f_2(\{0\})=m$, and $f_2(\{2N_2+1\})=f(\{2N+1\})$. Let $o_1:V_1\to (V_{F_1})_{\sth}\cup (V_{F_1})_{\int}$ and $o_2:V_2 \to (V_{F_2})_{\nth}\cup (V_{F_2})_{\int}$ be the unique OP maps, where $F_i=(P_i,s_i,f_i)$. Now define pair-partitions
\begin{align*} E_1&=o_1(E_1')\cup \{ o(u_a)\edge (2N_1+1,a) \ | \ a=1, \dots, m\},\\
E_2&=o_1(E_2')\cup \{ o(v_a)\edge (0,a) \ | \ a=1, \dots, m\}.\end{align*}
The crossing-less condition for $E$ implies that $E_1$ and $E_2$ are crossing-less, that is, $\Th_1=(F_1,E_1)$ and $\Th_2=(F_2,E_2)$ are SWB-data, and one has $\Th=\Th_2\# \Th_1$ and $L_{\Th_1, \Th_2}=0$ by construction.  \end{proof}

\subsection{Isotopy} In this subsection we define a notion of isotopy for SWB data, which is an equivalence relation on the latter \ref{iso}. We establish some key properties of this relation in Prop. \ref{ptprops} and then prove Prop. \ref{isorep}, which establishes the existence and uniqueness of isotopy reduced representatives, under mild conditions.

\mdef\label{turnb} For $\Th=(F,E)\in Sq_N (n,m)$, a \textbf{turn-back} is a part $e\in E$ of the form $e=\{(i,a),(i,b)\}$ for $i\in \underline{2N}$, \textit{i.e.} $e\subset V_i$. For $(i,a)\in (\Z_{> 0})^2$ fix the notation $((i,a))=\{(i,a),(i,a+1)\}$. If $((i,a))\subset (V_F)_\int$, define $U_{((i,a))}=((i,a))\cup \iota_{F}((i,a))\subset (V_F)_\int$, and $V_{((i,a))}=V_F\smallsetminus U_{((i,a))}$. Setting $F=(P,s,f)$, define $\delta_i:\overline{P}\to \Z_{\geq 0}$ by
\[\delta_{i}(p)=\begin{cases}1, & \text{if }i\in p,\\
0, &\text{otherwise,}\end{cases}\]
and $F_{((i,a))}:=(P,s,f-2\delta_i)$. Observe that $|V_{((i,a))}|=|V_{F_{((i,a))}}|$ and let $o_{((i,a))}:V_{((i,a))}\to V_{F_{((i,a))}}$ be the unique OP map. 

\begin{prop}[Pull-throughs]\label{pthru} Let $\Th=(F,E)$ be a SWB-datum, and suppose that $((i,a))\in E$ for $(i,a)\in (V_F)_\int$, while $\iota_{F}((i,a))\notin E$. There exists a unique SWB-datum $\Th_{((i,a))}:=(F_{((i,a))}, E')$, such that 
\[G(\Th_{((i,a))})\simeq_O G(\Th)/U_{((i,a))},\]
Furthermore, there is a bijection $O_{((i,a))}:\mathcal{C}(G(\Th))\to \mathcal{C}(G(\Th_{((i,a))}))$ given by 
\[O_{((i,a))}(\Gamma)=o_{((i,a))}\left(\Gamma/\left(V(\Gamma)\cap U_{((i,a))}\right) \right).\] 
\end{prop}
\begin{proof} Let $u=(i,a)$ and $v=(i,a+1)$ for convenience. Since $\iota_F((i,a))\notin E$, it follows that there exist distinct $u', v' \in V_F$ such that $u'\edge \iota_F(u)\edge u\edge v\edge \iota_F(v)\edge v' \subset G(\Th)$, with $\{u',\iota_F(u)\}, \{v,\iota_F(v)\} \in E$. The crossingless condition for $E$, then implies that 
\[E''=E \smallsetminus \{\{u,v\},\{u', \iota_F(u)\},\{v', \iota_F(v)\}\} \cup\{ \{u', v'\}\}, \]
is a crossingless pair-partition of $V_{((i,a))}$, and hence $\Th'=(F_{((i,a))}, o_{((i,a))}(E'') )$ is a SWB-datum. 

Using the fact that $\iota_{F_{((i,a)}}\circ o_{((i,a))}= o_{((i,a))}\circ \iota_F$ (where defined), one can check that $G(\Th')$ is obtained from $G(\Th)$ by replacing the path $u'\edge \iota_F(u)\edge u\edge v\edge \iota_F(v)\edge v'$ with the path $u'\edge v'$ (which is precisely the effect of contracting on the vertex subset $U_{((i,a))}$), and then applying the OP map $o_{((i,a))}$. The last claim follows by \ref{contpr} since $[U_{((i,a))}: G(\Th)]=0$.
\end{proof}

\begin{figure}[ht]
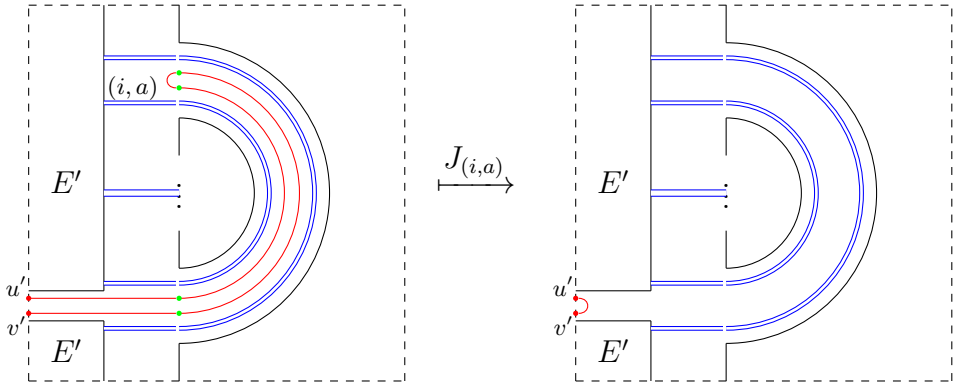

    \centering
\[\isoa{1} \quad\xmapstoo{J_{(i,a)}}\quad \isoab{1}\]
    \caption{Local diagram for a pull-through, using notation as in the proof of \ref{pthru}. We use thickened, decorated curves (as per the blue curves) to indicate an arbitrary number of parallel curves for brevity. We will use this convention throughout.  }
    \label{fig:iso_local}
\end{figure}
\mdef\label{pthru2} For $N\in \Z_{\geq 0}$, and $(i,a)\in \underline{2N}\times \Z_{>0}$ define the \textbf{pull-through}, $J^{(N)}_{(i,a)}:Sq_N (n,m)\to Sq_N (n,m)$ by
\[J^{(N)}_{(i,a)}(\Th)=\begin{cases}\Th_{((i,a))}, &((i,a))\in E \text{ and } \iota_{F}((i,a))\notin E,\\
\Th, & \text{otherwise.}
\end{cases}\]
We will often omit the superscript and write simply $J_{(i,a)}$ for a pull-through when the rank is clear from context.

\begin{example} We draw the SWB diagrams for $\Th_2'$ as per \ref{SWBex}, and $J_{(3,1)}(\Th_2')$ below. 
\[\SWBtwo{0.42} \quad\xmapstoo{J_{(3,1)}}\quad \SWBiso{0.42}\]
\end{example}

\mdef\label{iso} Let $\overset{\iso}{\sim}$ be the equivalence relation on $Sq(n,m)$ generated by $\Th_1 \overset{\iso}{\sim} \Th_2$, if $\Th_2=J^{(N)}_{(i,a)}(\Th_1)$ for some $N\in \Z_{\geq 0}$ and $(i,a)\in \underline{2N}\times \Z_{> 0}$. We call the equivalence relation $\overset{\iso}{\sim}$ \textbf{isotopy}, and if $\Th_1\overset{\iso}{\sim} \Th_2 $ we say that $\Th_1$ and $\Th_2$ are \textbf{isotopic}. We denote the set of isotopy classes by $Sq^{\iso}(n,m)=Sq(n,m)/\overset{\iso}{\sim}$.

\begin{prop}[Properties of pull-throughs]\label{ptprops} Let $\Th=(F,E)\in Sq_N(n,m)$ be such that $((i,a))\in E$ for $(i,a)\in (V_F)_\int$, whilst $\iota_F((i,a))\notin E$. Let $\Th'=(F',E')\in Sq_{N'}(m,l)$ and write $\Th''=(F'',E'')=\Th'\#\Th\in Sq_{N''}(n,l)$, and $L=L_{\Th,\Th'}$. The following hold:
\begin{enumerate}
    \item $|\CC(\mathcal{G}(\Th_{((i,a))}))|=|\CC(\mathcal{G}(\Th))|$.
    \item $(\Th_{((i,a))})^*=J_{\omega_{F}(i,j+1)}(\Th^*)$.
    \item For $\Gamma\in \mathcal{C}(G(\Th))$, set $\Gamma'=O_{((i,a))}(\Gamma)\in \mathcal{C}(G(\Th_{((i,a))}))$. If $(i,a)\in V(\Gamma)$ then $\Th\smallsetminus \Gamma=\Th_{((i,a))}\smallsetminus \Gamma'$, otherwise $J_{o_\Gamma(i,a)}(\Th\smallsetminus \Gamma)=\Th_{((i,a))}\smallsetminus \Gamma'$. Furthermore, $\tau_{\Gamma'}=\tau_{\Gamma}$, and $\Gamma'$ is internal (separating) if and only if $\Gamma$ is, respectively.
    \item If $\iota_{F''}((i,a))\notin E''$ then $J_{(i,a)}^{(N'')}(\Th'')=\Th'\#\Th_{((i,a))}$ and $L_{\Th_{((i,a))},\Th'}=L$. If $\iota_{F''}((i,a))\in E''$ then $U^{(F'')}_{((i,a))}$ is the vertex set of a component $\Gamma \in \mathcal{C}(G(\Th''))$, \textit{i.e.} $V(\Gamma)=U^{(F'')}_{((i,a))}$, which satisfies $\Th''\smallsetminus \Gamma=\Th'\#\Th_{((i,a))}$ and $L_{\Th_{((i,a))},\Th'}=L+1$. Futhermore, $\Gamma$ is internal, separating, and satisfies $\tau_\Gamma=0$.
\end{enumerate}\end{prop}

\begin{proof} For (1), set $\mathcal{G}=\mathcal{G}(\Th)$, $\mathcal{G}'=\mathcal{G}(\Th_{((i,a))})$ and $u=(i,a), v=(i,a+1)\in V_F$. One sees from \ref{fig:iso_local} that there are two components of $\mathcal{G}$ of interest, namely, the components containing $u^+$ and $u^-$, respectively. After applying $J_{(i,a)}$, the former component is changed by deleting vertex subset $\{u^+, \iota_F(u^+)\}$, and the latter is changed by identifying $u^-\sim v^+$ and $\iota_F(u^-)\sim \iota_F(v^+)$ (both up to applying an OP map). Neither of these steps changes the number of components, hence (1) follows.

For (2), one can check both SWB-datum have frame $(F_{((i,a))})^*$, and we compute 
\[G((\Th_{((i,a))})^* ) \simeq_O \omega_F ( G(\Th )/U_{((i,a))}) = G(\Th^* )/ U^{(F^*)}_{ \omega_F((i,a)) }\simeq_O G(J_{\omega_F(i,j+1)}(\Th^*)),  \]
using the fact that $\omega_F(U_{((i,a))})=U^{(F^*)}_{\omega_F ((i,a))}$. The desired equality follows from \ref{grp1}.

For (3), if $(i,a)\in V(\Gamma)$ then $g_{\Gamma}=g_{\Gamma'}+2 \delta_i$, hence $\Th\smallsetminus \Gamma$ and $\Th_{((i,a))}\smallsetminus \Gamma'$ have frame $F_{-\Gamma}$. If $(i,a)\notin V(\Gamma)$ then $g_{\Gamma}=g_{\Gamma'}$, hence $J_{o_{\Gamma}(i,a)}(\Th\smallsetminus \Gamma)$ and $\Th_{((i,a))}\smallsetminus \Gamma'$ have frame $(F_{((i,a))})_{-\Gamma'}$. In both cases we have 
\[G(\Th_{((i,a))}\smallsetminus \Gamma')\simeq_O (G(\Th)/U_{((i,a))})/o_{((i,a))}^{-1}( V(\Gamma') )=G(\Th)/(U_{((i,a))}\cup o_{((i,a))}^{-1}V(\Gamma' ) ).  \]
If $(i,a)\in V(\Gamma)$ then $U_{((i,a))}\cup o_{((i,a))}^{-1}V(\Gamma' )=V(\Gamma)$ hence $G(\Th_{((i,a))}\smallsetminus \Gamma')\simeq_O G(\Th\smallsetminus \Gamma)$. If $(i,a)\notin V(\Gamma)$ then $o_{((i,a))}^{-1}V(\Gamma' )=V(\Gamma)$, hence 
\[G(\Th_{((i,a))}\smallsetminus \Gamma')\simeq_O(G(\Th)/V(\Gamma))/U_{((i,a))} \simeq_O G(J_{o_{\Gamma}(i,a)}(\Th\smallsetminus \Gamma)).\]
In both cases, the desired equality now follows from \ref{grp1}. 

Since $o_{((i,a))}$ restricts to a bijection $(V_{F})_{\int}/U_{((i,j))}\to (V_{F_{((i,j))}})_\int$, it follows that $O_{((i,a))}$ maps internal (external) components of $G(\Th)$, to internal (external) components of $G(\Th_{((i,a))})$, respectively. If $((i,a))\subset V(\Gamma)$, then $E(\Gamma')\cap D_{F_{((i,j))}} $ contains two fewer parts both belonging to the same band, so $\tau_{\Gamma'}$ is unchanged. If $((i,a))\not\subset V(\Gamma)$, then $o_{((i,a))}$ induces a bijection $E(\Gamma)\cap D_{F}\to E(\Gamma')\cap D_{F_{((i,j))}}$ which preserves the band each part belongs to so $\tau_{\Gamma'}=\tau_{\Gamma}$. For the last claim, it follows from the first part of (3) and part (1) that $|\CC(\G(\Th_{\Gamma}))|=|\CC(\G((\Th_{((i,j))})_{\Gamma'}))|$.

For (4), in both cases using \ref{grp1} we can compute
\[ G(\Th'\#\Th_{((i,a))})\simeq_O   \left(G(\Th, \Th')/ (U_{((i,a))}^{(F)})^{(1)}\right)/U_{F,F'}\simeq_O G(\Th'')/U_{((i,a))}^{(F'')},  \]
and that $\Th'\#\Th_{((i,a))}$ has frame $F'\# F_{((i,a))}=F''_{((i,a))}$. 

If $\iota_{F''}((i,a))\notin E''$, then $J^{(N'')}_{((i,a))}(\Th'')=(\Th'')_{((i,a))}$ which has frame $F''_{((i,a))}$, and satisfies $G(\Th''_{((i,a))})\simeq_O G(\Th'')/U_{((i,a))}^{(F'')}$, hence $J^{(N'')}_{((i,a))}(\Th'')=\Th'\# \Th_{((i,a))}$ by \ref{grp1}. Furthermore,  $|\CC(G(\Th''_{((i,a))}))|=|\CC(G(\Th''))|$ so $L_{\Th_{((i,a))},\Th'}=L$.

If $\iota_{F''}((i,a))\in E''$, then $G(\Th'')$ contains the cycle $\Gamma=u\edge v\edge \iota_{F''}(v)\edge \iota_{F''}(u)\edge u$ (where we let $((i,a))=\{u,v\}$), and clearly $V(\Gamma)=U^{(F'')}_{((i,a))}$. Furthermore, $\Th''\smallsetminus \Gamma$ has frame $(P'',s'',f''-2\delta_i)=F'\# F_{((i,a))}$ (where $F''=(P'',s'',f'')$), and satisfies $G(\Th''\smallsetminus \Gamma)\simeq_O G(\Th'')/ V(\Gamma)\simeq_O G(\Th' \# \Th_{((i,a))})$. Hence $\Th''\smallsetminus \Gamma=\Th' \# \Th_{((i,a))}$ by \ref{grp1}, and since $|\CC(G(\Th''\smallsetminus \Gamma))|=|\CC(G(\Th''))|-1$, we find $F_{\Th_{((i,a))}, \Th'}=L+1$. The desired properties for $\Gamma$ are clear.
\end{proof}

\begin{prop}\label{isorep} Let $\Th$ be a SWB-datum whose curve graph has no internal components. Then the class $[\Th]_\iso$ has a unique representative $\Th'\isosim\Th$ without turnbacks. We call such a representative \textbf{isotopy reduced}.    
\end{prop}

\begin{proof} Suppose that $\Th=(F,E)$ has a turnback $e\in E$, such that $e\subset (V_F)_i$ for some integer $i$. The crossing-less condition for $E$, then implies that $((i,a))\in e$ for some $j$. Since $G(\Th)$ had no internal components, we have $\iota_F((i,a))\notin E$ hence $J_{(i,a)}(\Th)=\Th_{((i,a))}$, \textit{i.e.} we can always apply a non-trivial pull-through to such a $\Th$. The existence of an isotopy representative $\Th' \sim \Th$ without turnbacks now follows by iterating this procedure finitely many times, since the complexity (\text{c.f.} \ref{cpx}) strictly decreases with each iteration; $C(\Th_{((i,a))})=C(\Th)-2$.

For uniqueness we apply Newman's diamond lemma \cite{Newman}; suppose that $\Th_1=J_{((i_1,a_1))}(\Th)$ and $\Th_2=J_{((i_2,a_2))}(\Th)$ for $(i_1,a_1)\neq (i_2,a_2) $. We may assume WLOG that $((i_1,a_1)), ((i_2,a_2)) \in E$, and set $o_{((i_1,a_1))}(i_2,a_2)=(i_2,a_2')$ and $o_{((i_2,a_2))}(i_1,a_1)=(i_1,a_1')$, for convenience. There are two cases two consider, namely,  $|U_{((i_1,a_1))}\cap U_{((i_2,a_2))}|=0,2$. When $U_{((i_1,a_1))}$ and $U_{((i_2,a_2))}$ are disjoint, one can verify straightforwardly using the criteria from \ref{grp1} that $J_{((i_2,a_2'))}(\Th_1)= J_{((i_1,a_1'))}(\Th_2)$. Otherwise, when $|U_{((i_1,a_1))}\cap U_{((i_2,a_2))}|=2$ one can similarly verify that $\Th_1=\Th_2$. In either case, there exists some common SWB $\Th_3$ obtainable from $\Th_1$ and $\Th_2$ by a pull-through. \end{proof}

\begin{figure}[ht]
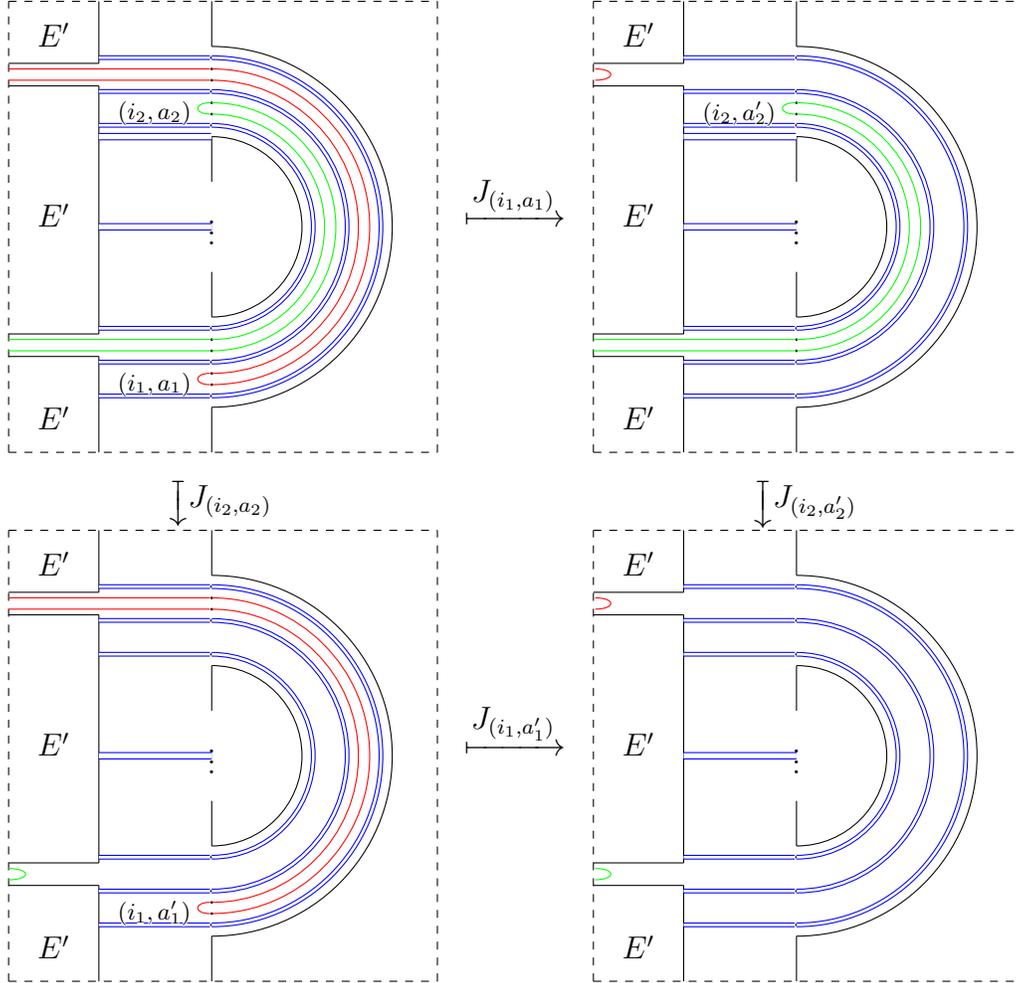

    \centering
    \[
 \begin{matrix}\isob{1.2}& \xmapstoo{J_{(i_1,a_1)}} &\isoba{1.2}\\ 
 \rotatebox{-90}{\xmapstoo{\rotatebox{90}{$J_{(i_2,a_2)}$}}}&&\rotatebox{-90}{\xmapstoo{\rotatebox{90}{$J_{(i_2,a_2')}$}}}\\
 \isobb{1.2}& \xmapstoo{J_{(i_1,a_1')}}&  \isobc{1.2}
\end{matrix}\] 
    \caption{Local diagram for the proof of \ref{isorep} showing $J_{((i_2,a_2'))}(\Th_1)= J_{((i_1,a_1'))}(\Th_2)$ when $U_{((i_1,a_1))}\cap U_{((i_2,a_2))}=\emptyset$ (here $E'=E\smallsetminus\{((i_1,a_1)),((i_2,a_2))\}$). In general, we may not have $\{i_1,i_2\}\in P$.}
    \label{fig:case1}
    \end{figure}
    
    \begin{figure}[ht]
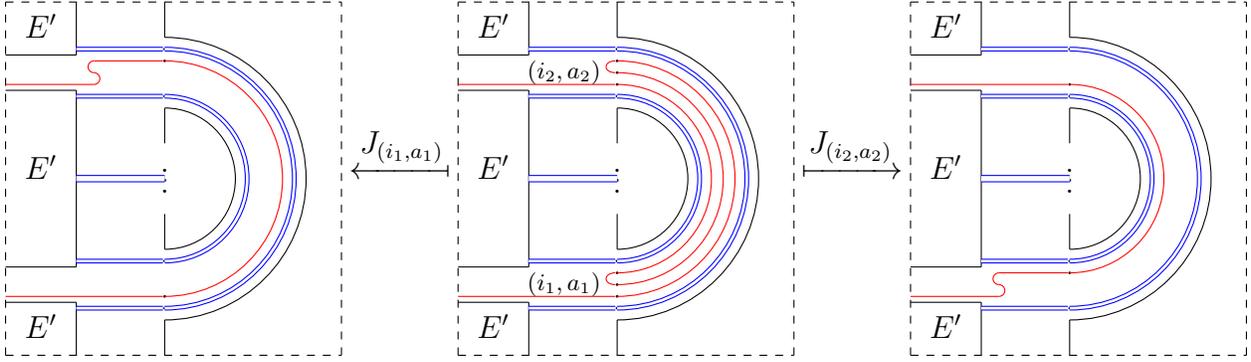

    \centering
\[\isoca{0.94} \ \xmapsfrom{J_{(i_1,a_1)}} \ \isoc{0.94} \ \xmapstoo{J_{(i_2,a_2)}}\ \isocb{0.94} \] 
    \caption{Local diagram for the proof of \ref{isorep}, showing $\Th_1=\Th_2$ when $|U_{((i_1,a_1))}\cap U_{((i_2,a_2))}|=2$ (here $E'=E\smallsetminus\{((i_1,a_1)),((i_2,a_2))\}$).}
    \label{fig:case2}
\end{figure}

\subsection{Handle slides} In this subsection we define handle slides for SWB data; these are certain transformations $Sq_N(n,m)\to Sq_N(n,m)$ (\textit{c.f.} \ref{hsdef}) whose diagrammatic interpretation corresponds to the Morse theoretic notion. In Prop. \ref{hselem} we establish basic properties of handle slides, in Prop. \ref{hseq} we use them to define handle slide equivalence of SWB data, $\hssim$, and we then prove Corollary \ref{hscor} which ensures each SWB datum (considered up to $\hssim$) has a factorisation in terms of data from $Sq_{\tor}$, $Sq_{\mob}$, and $Sq_0$.

\mdef\label{hsvset} Let $\Th=(F,E)\in Sq_N(n,m)$ with $F=(P,s,f)$, let $(i,\e)$ be an admissible pair in $P$, and let $h_{(i,\e)}(P,s)=(\sigma(P), s_{(i,\e)}\circ \sigma^{-1})$ (with $\sigma=\sigma_{(i,\e,P,s)}\in \SS_{2N}$) as per \ref{admissibility}. Suppose that $i\in p\in P$ and define $f_{(i,\e)}:\overline{P}\to \Z_{\geq 0}$ by  
\[f_{(i,\e)}(q)=\begin{cases}f(q)+f(p), &i+\e\in q,\\
f(q), & \text{otherwise}.\end{cases}\]
Now define $F_{(i,\e)}=(\sigma(P), s_{(i,\e)}\circ \sigma^{-1},f_{(i,\e)}\circ \overline{\sigma}^{-1})$ with $\overline{\sigma}$ here understood as the unique extension of $\sigma$ to a permutation of $\{0,1,\dots, 2N, 2N+1\}$ which fices 0 and $2N+1$.

Let $i'$ be the unique integer such that $\{i+\e,i'\}=:q\in P$, and define $U_{(i,\e)}\subset V_{F_{(i,\e)}}$ as the unique convex subset of order $2\cdot f(p)$ such that $V_{\sigma(i)}\subset U_{(i,\e)}\subset V_{\sigma(i)}\cup V_{\sigma(i')}$ (note that $|\sigma(i)-\sigma(i')|=1$). Then since $|V_F|=|V_{F_{(i,\e)}}\smallsetminus U_{(i,\e)}|$, set $o_{(i,\e)}:V_F\to V_{F_{(i,\e)}}\smallsetminus U_{(i,\e)}$ as the unique OP map.

\begin{figure}[ht]
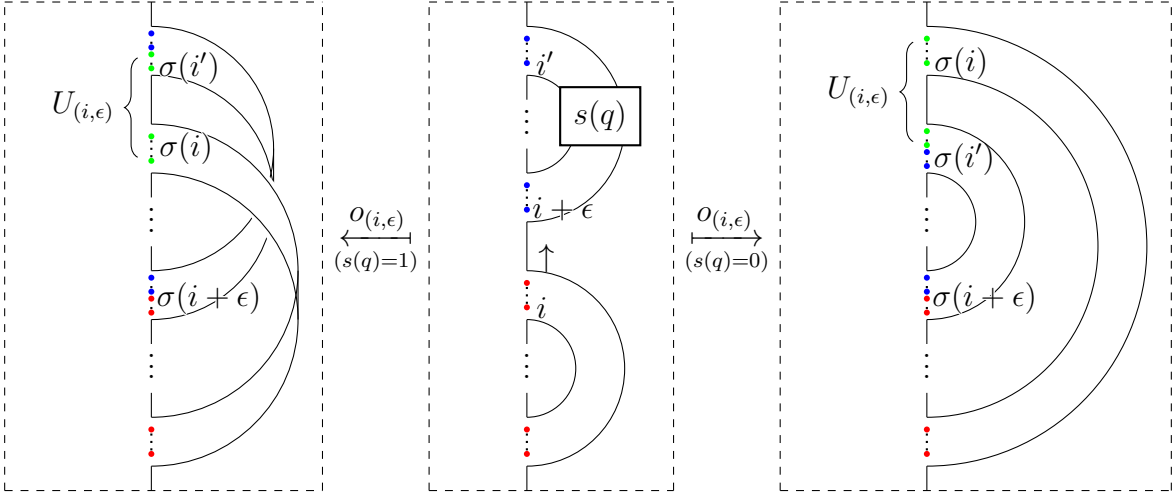

    \centering
\[   \hsvsetb{0.65} \ \underset{{(s(q)=1)}}{\xmapsfrom{o_{(i,\e)}}} \  \hsvset{0.65} \  \underset{{(s(q)=0)}}{\xmapstoo{o_{(i,\e)}}} \ \hsvseta{0.65} \] 
    \caption{Local diagram for \ref{hsvset}. Red vertices comprise the set $V_p$ and blue vertices comprise the set $V_q$ (or their images under $o_{(i,\e)}$, respectively).}
    \label{fig:hsvs}
\end{figure}

\begin{prop}[Handle slides]\label{hsprop} Let $\Th=(F,E)\in Sq_M(n,m)$ with $F=(P,s,f)$, and let $(i,\e)$ be an admissible pair in $P$. Then there exists a unique SWB-datum $\Th_{(i,\e)}\in Sq_M(n,m)$ of the form $\Th_{(i,\e)}=(F_{(i,\e)},E')$ such that
\begin{equation}G(\Th_{(i,\e)})=o_{(i,\e)}\left(G(\Th)\smallsetminus\bigcup_{u \in (U_F)_i} u\edge \iota_F(u)\right)\cup \bigcup_{u\in (V_F)_i} P_u  ,\label{hsgr}\end{equation}
where $(U_F)_i:=\{ u\in (V_F)_i \ | \ \{u,\iota_F(u)\}\notin E\}$, and $P_u$ is the path
\[P_u:=[o_{(i,\e)}(u) ] \edge [\iota_{F_{(i,\e)}}\circ o_{(i,\e)}(u)] \edge [\iota_{F_{(i,\e)}}\circ o_{(i,\e)}\circ \iota_{F}(u)]\edge [o_{(i,\e)}\circ \iota_{F}(u)].\]
Furthermore, there is a bijection $O_{(i,\e)}: \mathcal{C}(G(\Th))\to \mathcal{C}(G(\Th_{(i,\e)}))$ uniquely defined by 
\[O_{(i,\e)}(\Gamma)/(V(O_{(i,\e)}(\Gamma))\cap U_{(i,\e)})=o_{(i,\e)}(\Gamma).\] \end{prop}

\begin{proof} For this proof, fix $F'=F_{(i,\e)}, o'=o_{(i,\e)}$, $U'=U_{(i,\e)}$, set $q=\{i+\e,i'\}\in P$ and $i\in p\in P$. By direct computation, one finds the following for $(i,a) \in (V_F)_i$:
\begin{align*}
    \iota_{F'}\circ o' (i,a)&=\begin{cases} (\sigma(i'),f(p)+f(q)+1-a), & s(q)=0, \& \ \sigma(i)>\sigma(i'),\\
    (\sigma(i'),f(q)+a), & s(q)=1, \& \ \sigma(i)>\sigma(i'),\\
    (\sigma(i'),f(p)+1-a), & s(q)=0, \& \  \sigma(i')>\sigma(i),\\
    (\sigma(i'),a), & s(q)=1, \& \ \sigma(i')>\sigma(i),   \end{cases} \\
    \iota_{F'}\circ o' \circ \iota_{F} (i,a)&=\begin{cases} (\sigma(i),a), & s(q)=0,\\
    (\sigma(i),f(p)+1-a), & s(q)=1.   \end{cases} 
\end{align*}
From the above, one can verify $\iota_{F'}\circ o'((V_F)_i\cup \iota_F(V_F)_i )=U'$ and the following:
\begin{align*}E_{(i,\e)}&:=\iota_{F'}\circ o'\{ u \edge \iota_F(u) | \ u \in (V_F)_i \}\\
&=\begin{cases} \left\{ (\sigma(i),f(p)+1-a)\edge (\sigma(i'),f(q)+a)    \ | \ a=1,\dots, f(p)\right\}, &  \sigma(i')<\sigma(i),\\
\left\{ (\sigma(i),f(p)+1-a)\edge (\sigma(i'),a)    \ | \ a=1,\dots, f(p)\right\}, &  \sigma(i)<\sigma(i').
\end{cases} \end{align*}
Thus $E_{(i,\e)}$ is a crossing-less pairing of the convex subset $U'$, and so $\Th'=(F', o'(E)\cup E_{(i,\e)})$ is a SWB-datum. 

To see that $G(\Th')$ has the desired property \eqref{hsgr}, we note that the identity $o'\circ \iota_F=\iota_{F'} \circ o'$ holds on $(V_F)_\int \smallsetminus (V_F)_p$, so 
\[D_{F'}=o'(D_F\smallsetminus (D_F)_p )\cup \left\{\begin{array}{c}  {[\iota_F'\circ o'(u)] \edge [o'(u)]  ,}\\ {[\iota_F'\circ o'\circ \iota_F(u)] \edge [o'\circ \iota_F(u)],}\end{array}  \ \Big| \ u \in (V_F)_i   \right\}.\]
The identity \eqref{hsgr} now follows from \ref{SWBgraph}. The existence and bijectivity of the map $O_{(i,\e)}$ follows from \ref{contpr}, noting that $G(\Th')/ U'=o'(G(\Th))$; contraction on vertices $\iota_F'\circ o'(u)$ and $\iota_F'\circ o'\circ \iota_F(u)$ replaces the path $P_u$ with the edge $o'(u\edge \iota_F(u))$ for all $u\in (V_F)_i$. \end{proof}

\begin{figure}[ht]
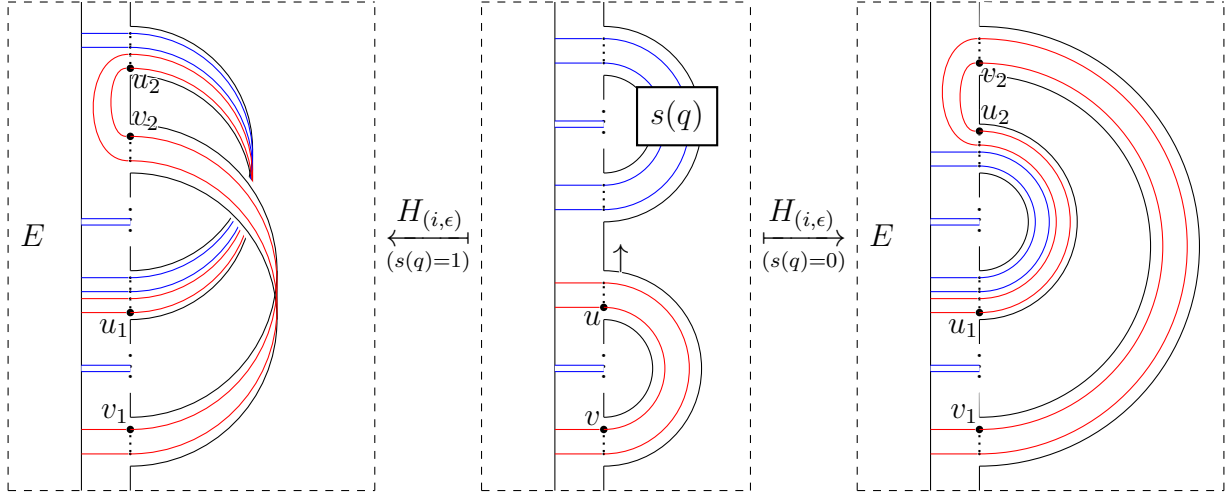

    \centering
\[  \hsab{0.65} \ \underset{(s(q)=1)}{\xmapsfrom{H_{(i,\e)}}} \ \hsa{0.65}\ \underset{(s(q)=0)}{\xmapstoo{H_{(i,\e)}}} \ \hsaa{0.65}\]
    \caption{Local diagram for a handle slide. For $u=\min(V_F)_i$, we show the path $P_u=u_1\edge u_2 \edge v_2 \edge v_1 \subset G(\Th_{((i,\e))})$ using the shorthands $v=\iota_F(u)$, $ u_1=o_{((i,\e))}(u)$, $ u_2=\iota_{F_{((i,\e))}}(u_1)$, $v_1=o_{((i,\e))}(v)$, and $v_2=\iota_{F_{((i,\e))}}(v_1)$.}
    \label{fig:hsloc}
\end{figure}

\begin{example}\label{Ex:handleslide} Below we apply a handle slide to some SWB datum $\Th\in Sq_{\tor}(5,5)$. 
\[\exaahs{0.4} \ \xmapstoo{H_{(2,-1)}} \ \exabhs{0.4}\overset{\iso}{\sim}\exachs{0.4}\]
We indicate the handle slide to be performed diagrammatically with an arrow pointing up or down and the appropriate site.
\end{example}

\mdef\label{hsdef} For any $N\in \Z_{\geq 0}$ and $(i,\e)\in \underline{2N}\times \{\pm 1\}$, define a map $H^{(N)}_{(i,\e)}:Sq_{N}(n,m)\to Sq_{N}(n,m)$, by 
\[ H^{(N)}_{(i,\e)}(\Th ) = \begin{cases} \Th_{(i,\e)}, & (i,\e) \text{ admissible in } P,\\
\Th, & \text{otherwise},\end{cases} \]
for $\Th=(F,E)$, $F=(P,s,f)$. We refer to the map $H^{(M)}_{(i,\pm 1)}$ as the \textbf{handle slide} up (+1) or down (-1) at site $i$. We will often omit the superscript and write simply $H_{(i,\e)}$ for a handle slide when the rank is clear from context.

\begin{prop}[Properties of handle-sliding]\label{hselem} Let $\Th=(F,E)\in Sq_N(n,m)$ be an SWB datum, with $F=(P,s,f)$, and suppose $(i,\e)$ is admissible in $P$. Then the following hold:
\begin{enumerate}
    \item\label{hscplt} $|\mathcal{G}(\Th_{(i,\e)})|=|\mathcal{G}(\Th)|$. 
    \item\label{hsdual} $(\Th_{(i,\e)})^*=H_{(\omega_M(i),-\e)}(\Th^*)$.
    \item\label{hscpt} For $\Gamma\in \mathcal{C}(G(\Th))$ set $\Gamma'=O_{(i,\e)}(\Gamma)$. Then  $H_{(i,\e)}(\Th\smallsetminus \Gamma)=\Th_{(i,\e)}\smallsetminus \Gamma'$. Furthermore, $\tau_{\Gamma}=\tau_{\Gamma'}$, and $\Gamma'$ is internal (separating), if and only if $\Gamma$ is, respectively.
    \item\label{hsiso}$H_{(i,\e)}\circ J_{(j,a)}(\Th)\overset{\iso}{\sim} \Th_{(i,\e)}$ for all $(j,a)\in \underline{2M}\times \Z_{>0}$.
    \item\label{hsjuxt} For any $\Th' \in Sq_{N'}(m,l)$ we have $\Th' \# \Th_{(i,\e)}=H^{(N+N')}_{(i,\e)}\left(\Th' \#\Th\right)$ and $L_{\Th_{(i,\e)},\Th'}=L_{\Th,\Th'}$.
\end{enumerate}\end{prop}
\begin{proof} As for \ref{hsprop}, set $F'=F_{(i,\e)}$, $o'=o_{(i,\e)}$ and let $i\in p\in P$ and let $q=\{i+\e, i'\}\in P$.

    For (1), by considering the curve complements in Figure \ref{fig:hsloc}, one sees that for $\e=+1$
\begin{equation}\mathcal{G}(\Th)/\{v_M^+, w_m^-, \iota_F(w_m^-) \} \simeq_O \mathcal{G}(\Th_{(i,\e)})/\{o'(v_M)^{+}\}\cup \mathcal{U}_{(i,\e)}  \label{eqhscpt}\end{equation}
where $v_M=\max((V_F)_i)$, $w_m=\min((V_F)_{i+1})$, and $\mathcal{U}_{(i,\e)}:=\left\{u^+, u^- \ | \ u \in U_{(i,\e)} \right\}$ (using the notation $v^\pm\in \V_F$ for $v\in V_F$ as in the proof of Prop. \ref{comp1}). The desired equality in the $\e=+1$ case now follows since both contractions in \eqref{eqhscpt} preserve the number of components. The case $\e=-1$ follows by combining with part \eqref{hsdual}.

    For (2), using \ref{fliplemma} one has that $H_{(\omega_N(i),-\e)}(\Th^*)=(\Th^*)_{(\omega_N(i),-\e)}$, and both SWB data have frame $(\omega(P),s_{(i,\e)}\circ\omega^{-1},f_{(i,\e)}\circ\omega^{-1}  )$, where $\omega=\omega_N\circ \sigma$. Furthermore, one computes 
    \begin{align*}
          G(\Th_{(i,\e)}^*)&=\omega_{F_{(i,\e)}}(G(\Th_{(i,\e)}))\\
          &=o_{(i,\e)}\left(G(\Th^*)\smallsetminus\bigcup_{u \in (U_{F^*})_{\omega_N(i)}} u- \iota_{F^*}(u)\right)\cup \bigcup_{u\in (V_{F^*})_{\omega_N(i)}} P^{(\Th^*)}_u\\
          &=G(H_{(\omega_N(i),-\e)}(\Th^*) ),
    \end{align*}
    where we have used the identities $\omega_{F'}\circ \iota_{F'}= \iota_{(F')^*}\circ \omega_{F'}$ (for any frame $F'$) and $\omega_{F_{(i,\e)}}\circ o_{(i,\e)}=o_{(\omega_N(i),-\e)}\circ \omega_F$, wherever defined. The result now follows by \ref{grp1}.

    For (3), we first note that $(f-g_{\Gamma})_{(i,\e)}=f_{(i,\e)}-(g_{\Gamma})_{(i,\e)}=f_{(i,\e)}-g_{\Gamma'}$, hence both SWB-data have the same frame. Their equality follows from \ref{grp1} now since both have graphs $\simeq_O G(\Th_{(i,\e)}) \smallsetminus \Gamma'$.

    Clearly $O_{(i,\e)}$ preserves internality of components, and $|\CC\left(\mathcal{G}((\Th_{(i,\e)})_{\Gamma'}  )\right)|=|\CC\left(\mathcal{G}(\Th_{\Gamma}  )\right)|$ follows by combining part (1) and (3). Since $o_{(i,\e)}$ induces a bijections $E(\Gamma)\cap (D_{F})_r\to E(\Gamma')\cap (D_{F'})_{\sigma(r)} $, for all $r\in P\smallsetminus\{p,q\}$, we now compute
    \begin{align*}\tau_{\Gamma'}&=\sum_{p' \in \sigma( P) }| E(\Gamma')\cap (D_{F'})_{p'}|s_{(i,\e)}\circ \sigma^{-1}(p')\\
    &=\sum_{r\in P\smallsetminus\{p,q\}} | E(\Gamma)\cap (D_{F})_r|s(r)+| E(\Gamma')\cap (D_{F'})_{\sigma(p)}|s_{(i,\e)}(p)+| E(\Gamma')\cap (D_{F'})_{\sigma(p)}|s_{(i,\e)}(q) \\
    &=\sum_{r\in P\smallsetminus\{p,q\}} | E(\Gamma)\cap (D_{F})_r|s(r)+| E(\Gamma)\cap (D_{F})_{p}|\left( s(p)+s(q) \right)\\
    &\quad+\left(| E(\Gamma)\cap (D_{F})_{p}|+| E(\Gamma)\cap (D_{F})_{q}|\right )s(q)=\sum_{r\in P} | E(\Gamma)\cap (D_{F})_r|s(r)=\tau_\Gamma.
    \end{align*}
    
    Part (4) is trivial unless $((j,a)) \in E$ and $\iota_F((j,a))\notin E$. If $j \in p$, we then claim that $H_{(i,\e)}\circ J_{(j,a)}(\Th)=J_{(j_2,a_2)}\circ J_{(j_1,a_1)} (\Th_{(i,\e)})$ where $((j_1,a_1))=o'((j,a))$, and $((j_2,a_2))=\iota_{F'}\circ o'\circ \iota_F((j,a))$. Otherwise, we claim that $H_{(i,\e)}\circ J_{(j,a)}(\Th)=J_{(j_1,a_1)} (\Th_{(i,\e)})$ for $(j_1,a_1)$ as before. Part (5) is verified straightforwardly using the criteria in \ref{grp1}.\end{proof}

\begin{figure}[ht]
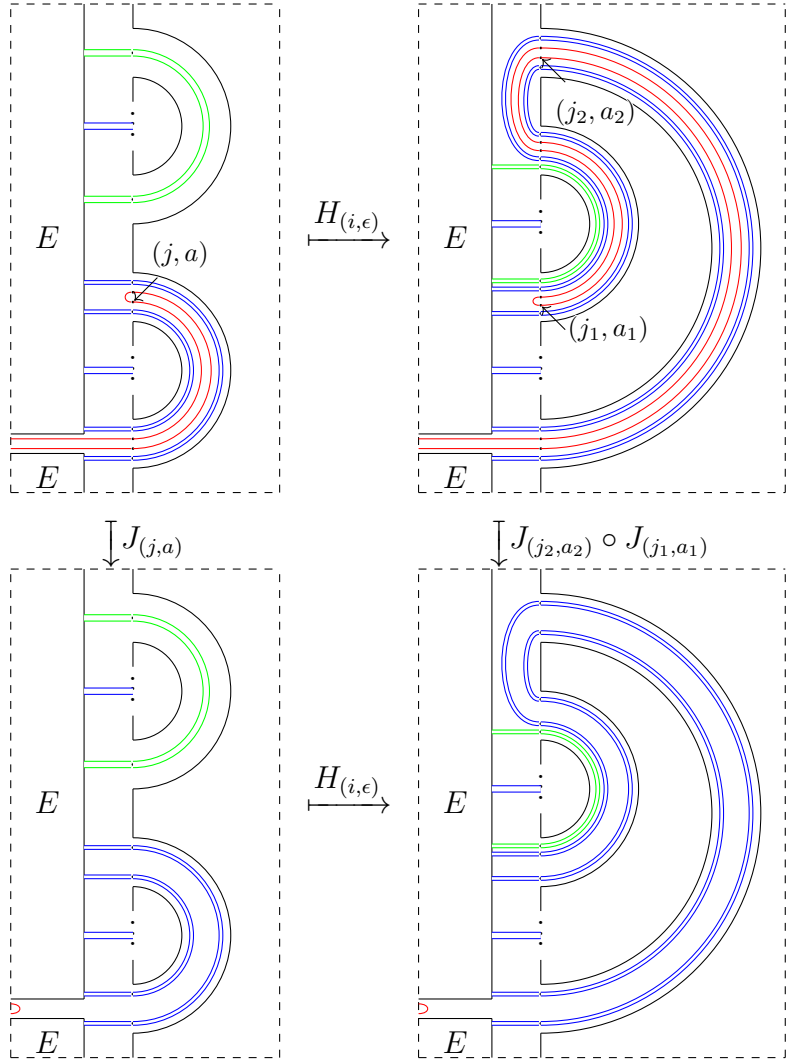

\centering
\[ \begin{matrix}\hsisoa{0.65} & \xmapstoo{H_{(i,\e)}}&\hsisoab{0.65}\\
\rotatebox{-90}{\xmapstoo{\rotatebox{90}{$J_{(j,a)}$}}}&&\rotatebox{-90}{\xmapstoo{\rotatebox{90}{$J_{(j_2,a_2)}\circ J_{(j_1,a_1)} $}}} \\
\hsisoaa{0.65}& \xmapstoo{H_{(i,\e)}}&\hsisoac{0.65}\end{matrix}\]
\caption{Local diagram showing $H_{(i,\e)}\circ J_{(j,a)}(\Th)=J_{(j_2,a_2)}\circ J_{(j_1,a_1)} (\Th_{(i,\e)})$ for part \ref{hsiso} of Prop. \ref{hselem}.}
\label{fig:hsiso}\end{figure}

\begin{remark}\label{hsdesc} By \ref{hselem} part \ref{hsiso} it follows that the handle slide $H^{(N)}_{(i,\e)}:Sq_N(n,m)\to Sq_N(n,m)$ descends to a map $Sq_N(n,m)/\isosim\to Sq_N(n,m)/\isosim$ (namely $[\Th]_{\iso}\mapsto [H^{(N)}_{(i,\e)}(\Th)]_{\iso}$). Henceforth, we abuse notation and write $H^{(N)}_{(i,\e)}$ for the latter map also.
\end{remark} 

\begin{prop}[Handle slide equivalence]\label{hseq} There is an equivalence relation on $Sq(n,m)$ defined by $\Th\overset{\hs}{\sim} \Th'$ if and only if $[\Th]_\iso =H [\Th']_{\iso}$, for $H:Sq_N(n,m)/\isosim\to Sq_N(n,m)/\isosim$ a finite (possibly empty) composite of handle slides, $H^{(N)}_{(i,\e)}$ for some $N\in \Z_{\geq 0}$. We call the relation $\overset{\hs}{\sim}$ \textbf{handle slide equivalence}.\end{prop}

\begin{proof} The relation $\overset{\hs}{\sim}$ is manifestly reflexive and transitive. For symmetry, it suffices to check that $H_{(i,\e)} (\Th) \hssim \Th$ for any rank N, SWB datum $\Th=(F,E)$, and $(i,\e)\in H_N$. Suppose that $F=(P,s,f)$, then the previous relation is trivial unless $(i,\e)$ is admissible in $P$. In this case we claim that $H_{(\sigma(i), \e')}\circ H_{(i,\e)} (\Th) \isosim \Th$, where $(\sigma(i),\e')$ is as per the proof of \ref{S2eq} (\textit{c.f.} Figure \ref{fig:hseq}).\end{proof}

\begin{figure}[ht] 
\centering
\[\hseqa{0.55}  \ \xmapstoo{H_{(i,\e)}}\ \hseqaa{0.55} \ \xmapstoo{H_{(\sigma(i),\e')}} \ \hseqab{0.55} \isosim \hseqA{0.55}
 \]
\caption{Local diagram showing $H_{(\sigma(i), \e')}\circ H_{(i,\e)} (\Th) \isosim \Th$ for Prop. \ref{hseq}.}
\label{fig:hseq}\end{figure}

We conclude this subsection with the following corollary:
 
\begin{corollary}\label{hscor}
    For any $\Th\in Sq_N(n,m)$ with $N>0$, there exist unique $g\in \Z_{\geq 0}$ and $t\in \{0,1,2\}$ subject to $2g+t=N$, and $l_i\in \Z_{\geq 0}$ for $i=1,\dots, g+t-1$, such that 
    \[\Th \hssim (\Psi_{t}\# \dots \# \Psi_{1}) \# (\Th_g \# \dots \# \Th_1)  \]
    for some $\Th_{i}\in Sq_{\tor}(l_{i-1},l_i)$, $\Psi_i\in Sq_{\mob}(l_{g+i-1}, l_{g+i})$, where we set $l_0=n$ and $l_{g+t}=m$.
\end{corollary}

\begin{proof}
    Follows by combining the Factorisation Lemma \ref{vjfact}, with Theorem \ref{rform}.
\end{proof}

\subsection{Left-most and Right-most Components} In this subsection we define two (families of) maps $\I_L,\I_R:Sq(n,m)\to Sq(n+1,m+1)$, which can be defined by adding certain components to the corresponding curve graphs. Diagrammatically, these correspond to adding curves which are isotopic to sub-intervals of the surface boundary, either on the left or right hand edge of the diagram. In Prop \ref{Ilrprop}, we prove that both maps descend to maps on handle slide equivalence classes, as well as some other key properties. In Lemmata \ref{lem1}, \ref{lem2}, and \ref{lem3} we establish three identities, which will be used in \S4 to establish a monoidal structure on the category we define in \S4. Finally we conclude this section with Prop. \ref{decomp}, from which we can deduce a monoidal generating set for said category.

\mdef For a frame $F=(P,s,f)$ of rank $N$, define frames $F_L=(P,s,f_L)$ and $F_R=(P,s,f_R)$ by $f_L(r)=f_R(r)=f(r)+1$ for $r\in \{\{0\}, \{2N+1\}\}$, and $f_L(p)=f(p)$, $f_R(p)=f(p)+2$ for $p \in P$. Define vertex subsets $U_{L}\subset V_{F_L}$ and $U_{R}\subset V_{F_R}$ by
\[U_L=\{\min(V_{F_L}), \max(V_{F_L})\}, \quad U_R=\{\max(V_{F_R})_i, \min(V_{F_R})_{i+1} \ | \ i=0,\dots, 2N\}.\]
Let $o_L:V_{F}\to V_{F_L}\smallsetminus U_L$ and $o_R: V_F \to V_{F_R}\smallsetminus U_R$ be the unique OP maps.

\begin{prop}\label{LRadd} Let $\Th\in Sq_N(n,m)$. Then there exist unique SWB-data $\Th_L=(F_L,E'), \Th_R=(F_R,E'')\in Sq(n+1,m+1)$, such that 
\[G(\Th_L)= o_L(G(\Th))\sqcup \Gamma_L, \qquad  G(\Th_R)=o_R(G(\Th))\sqcup \Gamma_R,\]
where $\Gamma_L=(0,1)\edge (2N+1,1)$, and $\Gamma_R\simeq_O \overline{\del(P,s)}$ (\textit{c.f.} \ref{bdgrp}) is a path between $(0,n+1)$ and $(2N+1,m+1)$. Furthermore, $\Gamma_L$ and $\Gamma_R$ are separating with $\tau_{\Gamma_L}=0=\tau_{\Gamma_R}$.\end{prop}
\begin{proof} Set $v_M^{(i)}=\max((V_{F_R})_i)$ and $v_m^{(i)}=\min((V_{F_R})_i)$ and take $E'=o_L(E) \cup \{ (0,1)\edge (2N+1,1) \}$ and $E''=o_R(E)\cup \{v_M^{(i)}\edge v_m^{(i+1)} \ | \ i=0,\dots 2N\}$. Then $\Th'=(F_L,E')$ and $\Th''=(F_R,E'')$ are SWB-data. It is clear that $\Th'$ is the unique such datum with desired property for $\Th_L$. Observe that $\iota_{F_R}$ preserves the vertex subset $U_R\subset V_{F_R}$, thus one may write $G(\Th'')=o_R(G(\Th))\sqcup G(\Th'')[U_R]$. Since $\iota_{F_R}(v_M^{(i)})=v_m^{(i')}$ where $\{i,i'\}\in P$, we establish an OP graph isomorphism $\overline{G(P,s)} \simeq_O G(\Th'')[U_R]$ via the map $(i,+)\mapsto v_M^{(i)}$ and $(i,-)\mapsto v_m^{(i)}$, thus $\Th''$ is the unique SWB datum with the desired property for $\Th_R$.

The component $\Gamma_L\subset \G(\Th_L)$ has $\tau_{\Gamma_L}=0$, since it contains no edges from $D_{F_L}$, and $\Gamma_R\subset \G(\Th_R)$ has $\tau_{\Gamma_R}=0$ since $|E(\Gamma_R)\cap (D_F)_p|=2$ for all $p\in P$. Furthermore, $\mathcal{G} \left((\Th_L)_{\Gamma_L}\right)$ contains the component $(0,1/2) \edge (2N+1,1/2)$, and $\mathcal{G}\left((\Th_R)_{\Gamma_R} \right)$ contains a component $\simeq_O\overline{\del(P,s)}$ on the vertex set 
\[\left\{(0,3/2),(2N+1,3/2)\right\}\cup \left\{ (i,1/2),(i,5/2) \ | \ i=1,\dots 2N\right\}.\]
Thus the components $\Gamma_L$ and $\Gamma_R$ are separating.\end{proof}

\mdef\label{Ilr} Let $\I_L, \I_R: Sq(n,m)\to Sq(n+1,m+1)$ be maps defined by $\I_L(\Th)=\Th_L$ and $\I_R(\Th)=\Th_R$, respectively. 

\begin{example} Let $\Th_1 \in Sq(3,1)$ be as per \ref{SWBex}. SWB diagrams for $\I_L(\Th_1)$ and $\I_R(\Th_2)$ are below:
\[\I_L(\Th_1)=\ILexa{0.45}, \quad \I_R(\Th_1)=\IRexa{0.45}.\]
    Let $\Th_2=(F_2, E_2)$ as per the same example with $E_2$ generic. We can draw $\I_L(\Th_1)$ and $\I_R(\Th_2)$ by adding new components to the otherwise generic diagram:
    \[ \I_L(\Th_2)=\SWBtwL{0.35}, \quad \I_R(\Th_2)=\SWBtwR{0.35} .\]
For a completely generic diagram of some $\Th=(F,E)$, drawn using the box notation as in Example \ref{SWBex}, we can similarly draw $\I_L(\Th)$ and $\I_R(\Th)$ by adding new components to the generic diagrams
\[\I_L(\Th)= \genSWBL{1.4} , \qquad \I_R(\Th)=\genSWBR{1.4}.\]
\end{example}

\begin{prop}\label{Ilrprop} Let $\Th\in Sq(n,m)$ The following hold:
\begin{enumerate}    
    \item For any component $\Gamma \subset G(\Th)$, $o_L(\Gamma)\subset G(\I_L(\Th))$ and $o_R(\Gamma)\subset G(\I_R(\Th))$ are components with the same twist, internality, and separatingness.
    \item\label{Ilrdesc} $\I_L$ and $\I_R$ descend to maps $\I_L,\I_R : Sq^{\hs}(n,m)\to Sq^{\hs}(n+1,m+1)$.
    \item For $\Th' \in Sq(m,l)$, we have $\I_L(\Th'\# \Th)=\I_L(\Th')\# \I_L(\Th)$, $\I_R(\Th'\# \Th)=\I_R(\Th')\# \I_R(\Th)$, and $L_{\Th,\Th'}=L_{\I_L(\Th),\I_L(\Th')}=L_{\I_R(\Th),\I_R(\Th')}$.
    \item $\I_L(\Th^*)=(\I_L(\Th))^*$ and $\I_R(\Th^*)=(\I_R(\Th))^*$. 
\end{enumerate}
\end{prop}
\begin{proof}
    Part (1) follows since $\Th_L\smallsetminus \Gamma_L = \Th = \Th_R\smallsetminus \Gamma_R$, since for any $\Gamma \in \mathcal{C}(G(\Th))$ and $p \in P$ we have $|E(o_L(\Gamma))\cap (D_{F_L})_p|=| E(\Gamma)\cap (D_F)_p|=| E(o_R(\Gamma))\cap (D_{F_R})_p|$, and since $o_L$ and $o_R$ preserve the internality/externality of vertices.

    For (2), it suffices to check that $\I_L\circ J_{(i,\a)}(\Th) \hssim \I_L(\Th)$, $\I_L\circ H_{(i,\e)}(\Th) \hssim \I_L(\Th)$, $\I_R\circ J_{(i,\a)}(\Th) \hssim \I_R(\Th)$, and $\I_R\circ H_{(i,\e)}(\Th) \hssim \I_R(\Th)$ for any relevant $(i,a)$ and $(i,\e)$. In the first three cases, the following identities can be proven by the criteria in Lemma \ref{grp1}:
    \[\I_L\circ J_{(i,a)}(\Th) =J_{(i,a)} \circ \I_L(\Th), \quad   \I_L\circ H_{(i,\e)}(\Th) =H_{(i,\e)} \circ \I_L(\Th),\quad \I_R\circ J_{(i,a)}(\Th)=J_{(i,a+1)}\circ \I_R (\Th).\]
    For the last case, comparing $\I_R\circ H_{(i,\e)}(\Th)$ and $H_{(i,\e)}\circ \I_R(\Th)$ we see that the latter contains an extra turn-back at level $\sigma(i+\e)$, belonging to the added component $O_{(i,\e)}(\Gamma_R)$. Removing this turn-back using the isotopy equivalence gives $\I_R\circ H_{(i,\e)}(\Th)\isosim H_{(i,\e)}\circ \I_R(\Th)$ (\textit{c.f.} Figure \ref{fig:Irdes}).

    For (3), letting $G(\Th'_L)=o'_L(G(\Th'))\cup \Gamma_L$ and likewise for $\Th_R$, we have 
    \begin{align*}G(\Th_L,\Th_L') &= (o_L(G(\Th))^{(1)} \cup o'_L(G(\Th))^{(2)}) \sqcup(\Gamma_L^{(1)} \cup {\Gamma'_L}^{(2)}),\\
    G(\Th_R,\Th_R') &= (o_R(G(\Th))^{(1)} \cup o'_R(G(\Th))^{(2)}) \sqcup(\Gamma_R^{(1)} \cup {\Gamma'_R}^{(2)}).\end{align*}
Observe that the subgraphs $\Gamma_L^{(1)} \cup {\Gamma'_L}^{(2)}$ and $\Gamma_R^{(1)} \cup {\Gamma'_R}^{(2)}$ are lines which join the left-most (right-most) southern and northern vertices in the composite graphs, respectively. After contracting the above graphs on vertices $U_{F_L,F_L'}$ and $U_{F_R,F_R'}$, and applying the OP maps $o_{F_L,F_L'}$ and $o_{F_R,F_R'}$, respectively, one obtains $\I_L(\Th)=\Th'_L \# \Th_L$ and $\I_R(\Th)=\Th'_R \# \Th_R$ by using Lemma \ref{grp1}. Futhermore, it is clear from the above that $[U_{F_L,F_L'}:G(\Th_L,\Th'_L)]=[U_{F,F'}:G(\Th, \Th')]=[U_{F_R,F_R'}:G(\Th_R,\Th'_R)]$. Part (4) follows immediately using the construction in \ref{LRadd} and the fact that $\omega_{F}(\max((V_F)_i))=\min((V_{F^*})_{\overline{\omega_N}(i)})$ for any frame $F$ and relevant integer $i$.\end{proof}

\begin{figure}[ht]
    \centering
    \[\begin{matrix}
         \IRdesa{0.65}&\xmapstoo{H_{(i,\e)}}&\IRdesaa{0.65}&\xmapstoo{\I_R}& \IRdesaaa{0.65} \\
         \rotatebox{-90}{\xmapstoo{\rotatebox{90}{$\I_R$}}}\\
         \IRdesab{0.65}&\xmapstoo{H_{(i,\e)}}&\IRdesaba{0.65}&\isosim& \IRdesaaa{0.65}
    \end{matrix}\]
    \caption{Local diagram of the relation $\I_R\circ H_{(i,\e)}(\Th)\isosim H_{(i,\e)}\circ \I_R(\Th)$, for Prop. \ref{Ilrprop} part \eqref{Ilrdesc}.}
    \label{fig:Irdes}
\end{figure}

\mdef\label{sft} Let $\mathbb{0}\in Sq_0(0,0), \mathbb{V}\in Sq_0(2,0)$ and $\mathbb{U}=\mathbb{V}^*\in Sq_0(0,2)$ be the unique such rank 0 SWB datum. For $m>0$, define a map $\mathit{Sf}=\mathit{Sf}^{(n,m)}:Sq(n,m)\to Sq(n+1,m-1)$, by
\[\mathit{Sf}(\Th)=\mathcal{I}_R^{m-1}(\mathbb{V})\# \mathcal{I}_L(\Th).\]
We usually omit superscript for the map $\mathit{Sf}$ in order to take powers of $\mathit{Sf}$. 

For $n\in \N$, define $\mathbb{I}_n:=\mathcal{I}_L^n(\mathbb{0})\in Sq_0(n,n)$, $\mathbb{V}_n:=\mathit{Sf}^n(\mathbb{I}_n)\in Sq_0(2n,0)$ and $\mathbb{U}_n:=(\mathbb{V}_n)^*\in Sq_0(0,2n)$. See Figure \ref{fig:IVU} for diagrams of each data.
\begin{figure}[ht]
\begin{subfigure}{\textwidth}
    \centering
\[ \genSWB{1.4}\quad \xmapstoo{\mathit{Sf}}\quad \genSWBS{1.4}\]
    \caption{Generic diagram for the map $\mathit{Sf}$, using box notation.}
    \label{fig:sfmap}
\end{subfigure}

\begin{subfigure}{\textwidth}
        \centering
\[ \mathbb{I}_n=\II{3} \qquad \mathbb{V}_n=\VV{3} \qquad \mathbb{U}_n= \UU{3}\]
    \caption{SWB diagrams for $\mathbb{I}_n, \mathbb{V}_n$ and $\mathbb{U}_n$.}
    \label{fig:IVU}
   \end{subfigure}
   \caption{Diagrams for \ref{sft}.}
   \label{fig:351}
\end{figure}

\mdef\label{setup} Suppose that $\Th=(F, E)$ is a SWB-datum without turn-backs and without internal components, and that $F=(P,s,f)$ where $(P,s)=(P_2,s_2)\#_d (P_1,s_1)$ for $(P_i,s_i)\in \mathcal{TC}_{N_i}$ and $d<2N_2$. Let $p=\{d+1, c\}\in P_2$ be the unique part containing $d+1$, and write $o_d(p)=:\{D+1,C\}$, where $D=o_d(d+1)=d+2N_1+1$. Then let $F_2=(P_2,f_2,s_2)$, $F_1=(P_1,s_1,f_1)$, be frames, $\Th=(F_1,E_1)\in Sq(n,m)$ a SWB-datum without fully external components, and $o:V_{F_1}\to V_{F}$ an OP map according to \ref{insSWB}, such that, $\Th=(F_2\#_d F_1, o(E_1)\sqcup \overline{E} )$.

Define the following vertex subsets of $V_F$ 
\begin{align*}
    U&:=o((V_{F_1})_{\sth}),& V&:=o(V_{F_1})\cap (V_F)_{D+1}\\
    U''&:= \{u\in (V_F)_{<C} \ | \ u\edge \iota_F(v)\in E \text{ for } v\in V  \}, & U'&:=o(V_{F_1})\cap (V_F)_{>D+1}, \\
    U'''&:= \{u\in (V_F)_{>C} \ | \ u\edge \iota_F(v) \in E \text{ for } v\in V \}, & W&:=V \cup \iota_F(V) ,
\end{align*}
and write $|U'|=m_1$, $|U''|=m_2$, $|U'''|=m_3$, hence $m=m_1+m_2+m_3$, and $|U|=n$. The condition that $\Th$ is without turnbacks implies that $V$ is the set of the least $m_2+m_3$ elements of $(V_F)_{D+1}$. See Figure \ref{fig:setup} for a diagram of this notation.
\begin{figure}[ht]
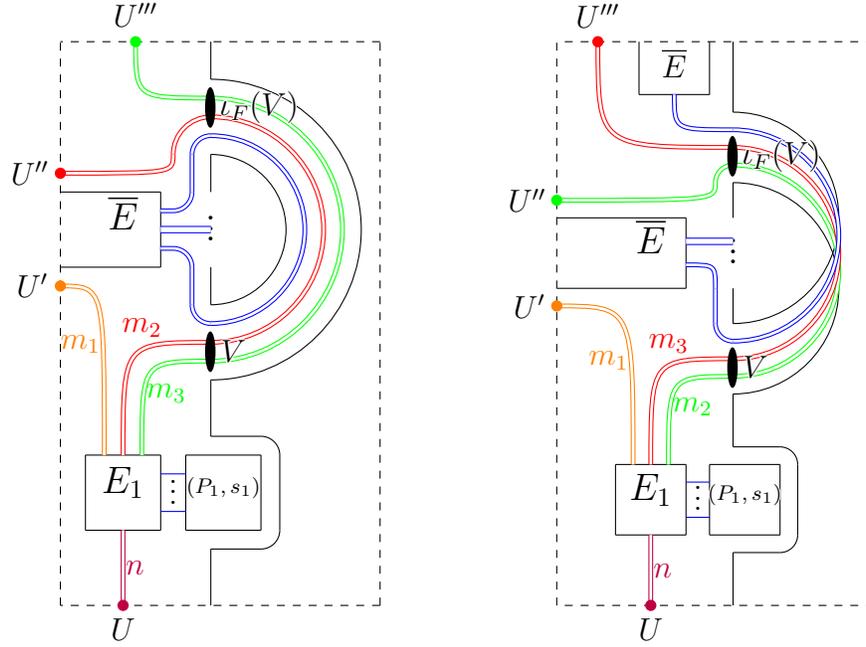

    \centering
\[\lemoneintA{1} \qquad \qquad \lemtwintA{0.9375}\]
    \caption{Diagram for the notation in \ref{setup}, with the case $s(p)=0$ on the left and $s(p)=1$ on the right.}
    \label{fig:setup}
\end{figure}

\begin{lemma}\label{lem1} Let $\Th=(F, E)=(F_2\#_d F_1, o(E_1)\sqcup \overline{E} )$ be as per \ref{setup} and suppose that $s(p)=0$. Then 
\[\Th\hssim (F', E' ):= (F_2'\#_c F_1', o'(E_1')\sqcup  o''(\overline{E}\smallsetminus [\overline{E}\cap W] )\sqcup \overline{E}'), \]
where $(F_1',E_1')=\mathit{Sf}^{m_1+m_2}(\Th_1) \in Sq(n+m_1+m_2,m_3)$ and $F_2'=(P_2,s_2,f_2')$ with $f_2'(p)=f_2(p)+n+m_1-m_2-m_3$, and otherwise $f_2'(q)=f_2(q)$ for $q\neq p \in \overline{P_2}$.
    
    Here $o''$ is the unique OP map $V_F\smallsetminus (W\cup o((V_{F_1})_\int) )\to V_{F'}\smallsetminus (W'\cup o'((V_{F_1'})_\int))$, where $W'=V'\cup \iota_{F'}(V')$ for $V'$ the set of the $n+m_1$ greatest elements of $(V_{F'})_c$, and $o':V_{F_1'}\to V_{F'} $ is the unique OP map such that $o'(i,a)=(i+c,a)$ for $(i,a)\in (V_{F_1'})_{\int}$ and $o'((V_{F_1'})_{\ext})=V'\cup o''(U''\cup U''')$. $\overline{E}'$ is the unique pairing of $\iota_{F'}(V')\cup o''(U\cup U')$ without turnbacks. 
\end{lemma}

\begin{proof} Let $\mathit{Ev}_{[d+1,D]}:Sq_N\to Sq_N$ be the sequence of handle slides obtained from the upward evacuation of sites $d+1$ to $D$, sequence of chord slides, $\mathit{ev}_{[d+1,D]}:\mathcal{TC}_N\to \mathcal{TC}_N$ (\textit{c.f.} \ref{def:evac}), by replacing each chord slide $h_{(i,\e)}$ with the corresponding handle slide $H_{(i,\e)}$. Set $\Th'':=(F'',E''):=\mathit{Ev}_{[d+1,D]}(\Th)$ with $F''=(P'',s'',f'')$. By Lemma \ref{moveslem}, it follows that $(P'',s'')=\mathit{ev}_{[d+1,D]}(P,s)=(P_2,s_2)\#_c (P_1,s_1)$ hence we can write $F''=F_2''\#_c F_1''$, with $F_2''=(P_2,s_2,f_2'')$ and $F_1''=(P_1, s_1,f_1'')$. Furthermore, $f_2''(q)$ differs from $f_2(q)$ only when $q=p \in \overline{P_2}$, in which case $f_2''(p)=f_2(p)+2 C(F_1)$. 

Let $\tau \in \mathfrak{S}_{2N}$ be the resultant permutation from applying $\mathit{ev}_{[d+1,D]}$ (\textit{i.e.} $P''=\tau(P)$), let $\tau(v)$ for any $v=(i,a)\in (V_F)_\int$ denote $(\tau(i),a)\in V_{F''}$, and let $o_H:V_F\to V_{F''}$ be the OP map which is the composite of the OP maps $o_{(i,\e)}$ (\textit{c.f.} \ref{hsvset}) for each handle slide in $\mathit{Ev}_{[d+1,D]}$. Note that $V_{F''}\smallsetminus o_{H}(V_F)=\iota_{F''}\circ o_{H}\circ o((V_{F_1})_{\int}) \cup \tau\circ o((V_{F_1})_\int)$. Let $\tau':V_{F_1}\to V_{F''}$ be the map given by 
\[
\tau'(v)=\begin{cases}\tau\circ o(v), & v\in (V_{F_1})_\int,\\
\begin{matrix} o_{H}(u) \text{ for the unique $u\neq o(v)$ s.t. }\\ u\edge \iota_{F}\circ o(v)\edge o(v) \subset G(\Th)\end{matrix} , & v \in  o^{-1}\left((V_{F})_{D+1}\right),\\ 
o_{H}\circ o(u), & \text{otherwise.}    
\end{cases}
\]
For each part $e=\{u,v\}\in E_1$ such that $e\subset (V_{F_1})_{\int}\cup o^{-1}\left((V_{F})_{D+1}\right)$, we have that $o_H\circ o(e):=\{u'',v''\}\in E''$ is a turnback in $\Th''$ and $G(\Th'')$ contains the (full) subgraph
\[P_{e}:=\left(\tau'(u)\edge \iota_{F''}(u'')\edge u'' \edge v'' \edge \iota_{F''}(v'')\edge\tau'(v)\right). \]
Since $\Th$ was without turnbacks, these are the only turnbacks in $\Th''$. For any $u \in V_{F_1}\smallsetminus \left[(V_{F_1})_{\int}\cup o^{-1}\left((V_{F})_{D+1}\right)\right]$, there is a part
$e=\{u,v_u\}\in E_1$ for $v_u\in (V_{F_1})_\int$. Letting $o_H\circ o(e):=\{u'',(v_u)''\}\in E''$ again, we have that $G(\Th'')$ contains the (full) subgraph
\[ Q_{u}:=\left(u''=\tau'(u)\edge (v_u)''\edge \iota_{F''}((v_u)'')\edge \tau'(v_u)\right).  \]
\begin{figure}[ht]
    \centering
    \[\lemoneA{1}\quad \xmapstoo{\mathit{Ev}_{[d+1,D]}}\quad \lemoneBB{0.88235294117}  \]
    \caption{Local diagram for $\mathit{Ev}_{[d+1,D]}(\Th)$.}
    \label{fig:localEV}
\end{figure}
Let $\tilde{\Th}=(\tilde{F},\tilde{E})$ be the unique isotopy reduced representative of $\Th''$. We then claim that $\tilde{\Th}=(F',E')$. By counting the number of turn-backs removed by pull-throughs one finds that $\tilde{F}=F'$ as desired. Now let $U_{F_1}=\{v \in (V_{F_1})_\int\cup o^{-1}\left( (V_F)_{D+1} \right) \ | \ \{u,v\}\in E_1 \text{ for } u \in (V_{F_1})_\int\cup o^{-1}\left( (V_F)_{D+1} \right) \}$ so that 
\[G(\tilde{\Th})\simeq_O G(\Th'')/\left[o_H\circ o(U_{F_1})\cup \iota_{F''}\circ o_H\circ o(U_{F_1})\right].   \]
The result of this contraction is simply to replace all paths $P_e$ with $\tau'(e)$. After applying the unique OP map, say $o_P$, between the vertex sets of the above graphs (right to left), we have that
\begin{align*}
    o'(E_1')&=o_P\left(\{ \tau'(e) \ | \ e \subset U_{F_1}\} \cup \{\iota_{F''}((v_u)'')\edge \tau'(v_u) \ | \ u \in V_{F_1}\smallsetminus \left[(V_{F_1})_{\int}\cup o^{-1}\left((V_{F})_{D+1}\right)\right] \}\right),\\
    \overline{E}'&= o_P\left(\{\tau'(u) \edge (v_u)'' \ | \ u \in V_{F_1}\smallsetminus \left[(V_{F_1})_{\int}\cup o^{-1}\left((V_{F})_{D+1}\right)\right]  \}\right).
\end{align*}
which combined with the fact that $o_P\circ o_H$ agrees with $o''$ when restricted to the appropriate domain, now gives $\tilde{\Th}=(F',E')$ as desired. Thus $\Th\hssim \tilde{\Th}=(F',E')$.\end{proof}

\begin{figure}[ht]
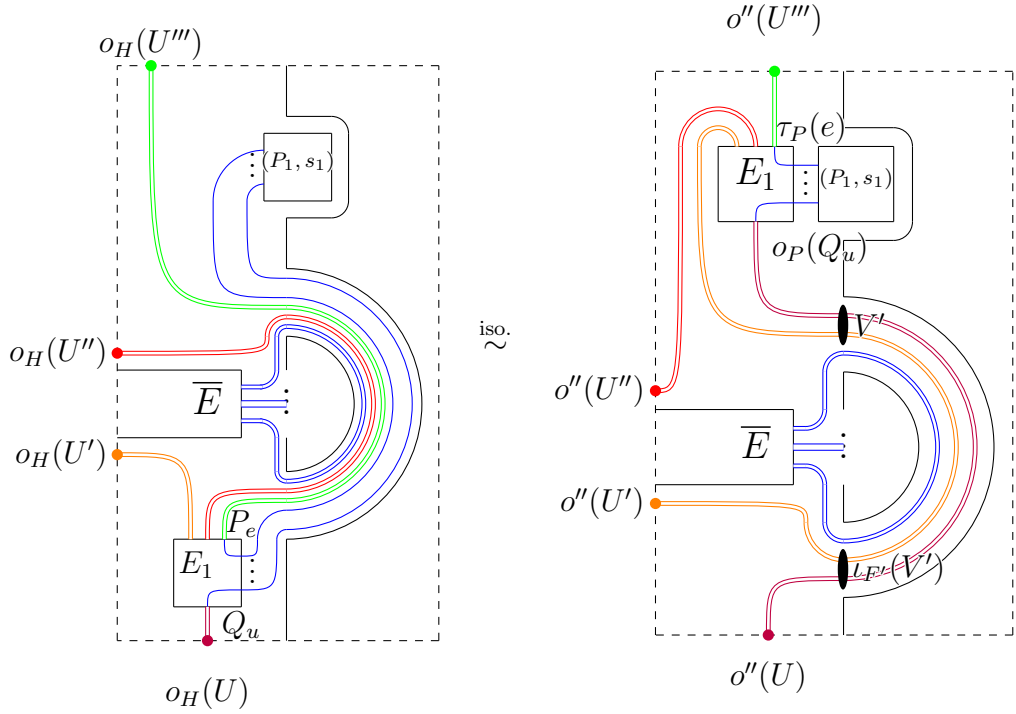

    \centering
\[\lemoneBB{0.9}\quad \isosim\quad \lemoneB{1}\]
    \caption{Local diagram of isotopy reduction for $\Th''$  
(using the short-hand $\tau_P=o_P\circ \tau'$).}
    \label{fig:lemm1}
\end{figure}

\begin{lemma}\label{lem2} Let $\Th=(F, E)=(F_2\#_d F_1, o(E_1)\sqcup \overline{E} )$ be as per \ref{setup} and suppose that $s(p)=1$. Then 
\[\Th\hssim (F', E' ):= (F_2'\#_{c-1} F_1', o'(E_1')\sqcup  o''(\overline{E}\smallsetminus (\overline{E}\cap W) )\sqcup \overline{E}'), \]
where $(F_1',E_1')=(\mathit{Sf}^{m_2+m_3}(\Th_1))^* \in Sq(m_2,n+m_1+m_3)$ and $F_2'=(P_2,s_2,f_2')$ with $f_2'(p)=f_2(p)+n+m_1-m_2-m_3$, and otherwise $f_2'(q)=f_2(q)$ for $q\neq p \in P_2$. 
    
    Here $o''$ is the unique OP map $V_F\smallsetminus (W\cup o((V_{F_1})_\int) )\to V_{F'}\smallsetminus (W'\cup o'((V_{F_1'})_\int))$, where $W'=V'\cup \iota_{F'}(V')$ for $V'$ the set of the $n+m_1$ least elements of $(V_{F'})_{c+2N_1}$, and $o':V_{F_1'}\to V_{F'} $ is the unique OP map such that $o'(i,a)=(i+c-1,a)$ for $(i,a)\in (V_{F_1'})_{\int}$, and $o'((V_{F_1'})_{\ext})=V'\cup o''(U''\cup U''')$. $\overline{E}'$ is the unique pairing of $\iota_{F'}(V')\cup o''(U\cup U')$ without turnbacks. 
\end{lemma}
\begin{proof} The proof proceeds the same as for \ref{lem1}, \textit{i.e.} $\textit{Ev}_{[d+1,D]}(\Th):=(F'',E'')\isosim (F',E')$. The key differences here are that $\textit{ev}_{[d+1,D]}(P,s)=(P_2,s_2)\#_{c-1}[\omega_{N_1}(P_1,s_1)]$ (\textit{c.f.} Lemma \ref{moveslem}), and that the map $\tau':V_{F_1}\to V_{F''}$ used in defining the paths $P_e$ and $Q_u$ becomes the map
\[
\tau'(v)=\begin{cases}\omega_{F_1}(v)+(c-1,0), & v\in (V_{F_1})_\int,\\
\begin{matrix} o_{H}(u) \text{ for the unique $u\neq o(v)$ s.t. }\\ u\edge \iota_{F}\circ o(v)\edge o(v) \subset G(\Th)\end{matrix} , & v \in  o^{-1}\left((V_{F})_{D+1}\right),\\ 
o_{H}\circ o(u), & \text{otherwise.}    
\end{cases}\qedhere\]\end{proof}

\begin{figure}[ht]
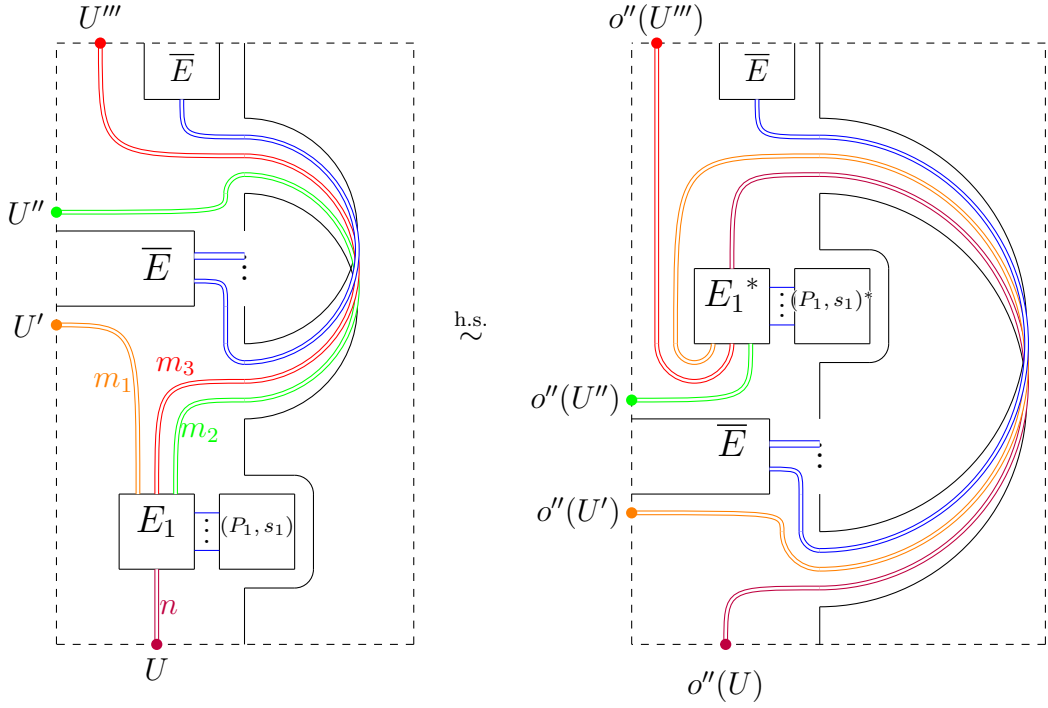

    \centering
\[\lemtwA{1}\quad \hssim\quad \lemtwB{1}\]
    \caption{Local diagram for \ref{lem2}, with the shorthand $E_1^*:=\omega_{F_1}(E_1)$.}
    \label{fig:lemm2}
\end{figure}

\begin{lemma}\label{lem3} Let $\Th=(F, E)=(F_2\#_d F_1, o(E_1)\sqcup \overline{E} )$ be as per \ref{setup}, and suppose that $m_2=m_3=0$, and that $f(p)=l+k$, where $l=|X|, k=|X'|$, for vertex subsets 
\begin{align*}
    X&:=\{v \in (V_F)_{<D} \ | \ v\edge u \in E \text{ for } u \in (V_F)_{D+1}\},\\
    X'&:=\{v \in (V_F)_{>D} \ | \ v\edge u\in E \text{ for } u \in (V_F)_{D+1}\},
\end{align*}
(note that one of $k$ and $m$ must be 0). Then
\[\Th\hssim (F',E'):=(F_2\#_{d+1} F_1',o'(E_1')\sqcup o''(\overline{E}\smallsetminus[\overline{E}\cap (V_F)_{D+1}])  ),\]
where $(F_1',E_1')=\mathit{Sf}^{k}(\I_R^{l+k}(\Th_1))$, $o'':V_F\smallsetminus(o((V_{F_1})_{\int})\cup (V_F)_{D+1})\to V_{F'}\smallsetminus(o'((V_{F_1'})_{\int})\cup (V_{F})_d)$ is the unique OP map, and $o':V_{F_1'}\to V_{F'}$ is the OP map uniquely defined by $o'(i,a)=(i+d+1,a)$ for $(i,a)\in (V_{F_1})_{\int}$, and $o'((V_{F_1'})_{\ext})=o''(X\cup U \cup X' \cup U')\cup (V_{F'})_{d+1}$.
\end{lemma}

\begin{figure}[ht]
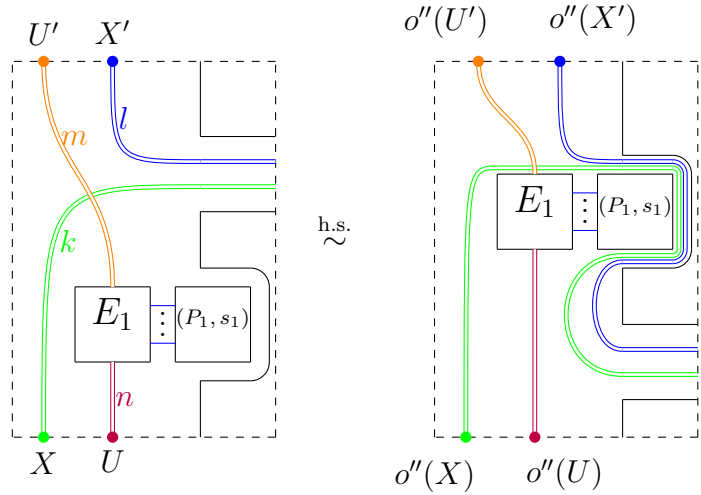

    \centering
\[\lemthA{1} \quad \hssim \quad \lemthB{1}\]
    \caption{Local diagram for \ref{lem3}.}
    \label{fig:lemm3}
\end{figure}

\begin{proof} We claim $\mathit{Bs}_{(d,(P_1,s_1))}^{-}(\Th)=(F',E')$, where $\mathit{Bs}_{(d,(P_1,s_1))}^{-}$ is the sequence of handle slides obtained from the sequence of chord slides $\mathit{bs}_{(d,(P_1,s_1))}^{-}$ (\textit{c.f.} \ref{def:slideover}) in the obvious way.\end{proof}

\mdef\label{tgens} Let $l,m,n \in \N$ and define two families of SWB data $\TT_{l,m}\in Sq_{\tor}(l+m,l+m)$ and $\MM_n \in Sq_{\mob}(n,n)$ as follows: $\TT_{l,m}=:(F_{l,m},E_{l,m})$ is uniquely defined by the properties $|(V_{F_{l,m}})_1|=m$, $|(V_{F_{l,m}})_2|=l$, $G(\TT_{l,m})$ is without fully external components; $\MM_n=:(F_{n},E_{n})$ is uniquely defined by the properties $|(V_{F_n})_1|=n$, and $G(\MM_n)$ is without fully external components. Observe that $(\TT_{l,m})^*=\TT_{m,l}$ and $(\MM_n)^*=\MM_n$.

\begin{figure}[ht]
    \centering
    \[ \TT_{l,m}=\TTpic{0.45} \qquad  \MM_n=\MMpic{0.81}\]
    \caption{SWB diagrams for $\TT_{l,m}$ and $\MM_{n}$ as per \ref{tgens}.}
    \label{fig:enter-label}
\end{figure}

\begin{prop}\label{decomp} Let $\Th \in Sq_{\tor}(n,m)$ and let $\Psi \in Sq_{\mob} (l,k)$. Then there exist non-negative integers $\l$, $\mu$, $a$,  and $\nu$, $b$, and rank 0 SWB data $T_1\in Sq_0(n,\l+\mu+a)$, $T_2\in Sq_0(\l+\mu+a,m)$, and $T_3\in Sq_0(l,\nu+b)$, $T_4\in Sq_0(\nu+b,k)$,  such that
\[\Th=T_2\#(\I_L^a(\TT_{\l,\mu}))\# T_1, \qquad \Psi=T_4\#(\I_L^b(\MM_\nu))\# T_3,  \]
with all associated linking numbers being 0, \textit{i.e.}: 
\[L_{T_1,\I_L^a(\TT_{\l,\mu}) }=L_{\I_L^a(\TT_{\l,\mu}),T_2}=0=L_{T_3,\I_L^b(\MM_\nu)}=L_{\I_L^b(\MM_\nu),T_4}.\]
\end{prop}

\begin{remark}\label{decompr} It follows from Cor. \ref{hscor} and Prop. \ref{decomp}, that using the operations $\#$ and $\I_L$, any SWB datum $\Th$, considered up to handle slide equivalence, can be decomposed into rank 0 SWB data, and data from the two families $\TT_{l,m}$ and $\MM_n$. \end{remark}

\section{Square with Bands Category} In this section we define the square with bands (SWB) category, $\SQ$, using SWB data from \S3. This category will be a linear (small) category, which contains the Temperley-Lieb category as a subcategory. Using \S3.5, we equip $\SQ$ with a monoidal structure (and indeed a pivotal structure), and provide a monoidal generating set for $\SQ(n,m)$.\\

In what follows, let $\K$ be a commutative, unital, ring, and fix $\a,\b,\gamma\in \K$.

\mdef\label{homset} For $n,m\in \N$, let $A(n,m), B(n,m), C(n,m)$ denote the following $\K$-submodules of $\K Sq^{\hs}(n,m)$:
\begin{align*}
A(n,m)&=\K \left\{ [\Th]_{\hs}-\a [\Th\smallsetminus \Gamma ]_{\hs} \ \Big| \  \begin{array}{c}\Th\in Sq(n,m), \text{ and } \Gamma \in \CC(G(\Th)) \text{ an internal,} \\ \text{separating component.} \end{array}\right\}, \\
B(n,m)&=\K \left\{ [\Th]_{\hs}-\b [\Th\smallsetminus \Gamma ]_{\hs} \ \Big| \   \begin{array}{c}\Th\in Sq(n,m), \text{ and } \Gamma \in \CC(G(\Th)) \text{ an internal,}\\ \text{twisted component.}\end{array} \right\},\\
C(n,m)&=\K \left\{ [\Th]_{\hs}-\gamma [\Th\smallsetminus \Gamma ]_{\hs} \ \Big| \  \begin{array}{c} \Th \in Sq(n,m), \text{ and } \Gamma \in \CC(G(\Th)) \text{ an internal,}\\
\text{non-separating, untwisted component.}\end{array} \right\}.
\end{align*}

\begin{theorem}[Square with Bands Category]\label{SWBcat} $\SQ=\SQ(\a,\b,\gamma)$, is a small $\K$-linear category with the following data:
\begin{itemize} 
\item[] \underline{Objects:} $\SQ$ has object set $\SQ:=\N$.
\item[] \underline{Morphisms:} For $n,m\in \SQ$, $\SQ(n,m)$ is the $\K$-module of morphisms given by
\begin{equation*}\SQ(n,m):= \K Sq^{\hs}(n,m)/(A(n,m) + B(n,m) + C(n,m) ). \end{equation*}
For $\Th\in Sq(n,m)$, write $\overline{\Th}$ for the image of $[\Th]_\hs\in \K Sq^{\hs}(n,m)$ under this quotient.
\item[] \underline{Composition:} For $n,m, l \in \SQ$, composition $\circ: \SQ(n,m) \times \SQ(m,l)\to \SQ(n,l)$ is given by linearly extending the formula 
\begin{equation*} (\overline{\Th_1},\overline{\Th_2})\mapsto \overline{\Th_2} \circ \overline{\Th_1} := \a^{L_{\Th_1,\Th_2}} \, \overline{\Th_2 \# \Th_1}, \end{equation*}
for $\Th_1 \in Sq(n,m), \Th_2 \in Sq(m,l)$, where $\Th_2\# \Th_1$ denotes the vertical juxtaposition, and $L_{\Th_1, \Th_2}$ the linking number (\textit{c.f.} \S3.2).
\end{itemize}
Furthermore, $\SQ$ has a contravariant duality $(\_)^*:\SQ\to \SQ$ defined on objects by $n^*=n$, and on morphisms by $\overline{\Th}^*=\overline{\Th^*}$, for $\Th$ an SWB datum, and by extending $\K$-linearly.
\end{theorem}

We refer to $\SQ$ as the square with bands (SWB) category. We refer to morphisms of the form $\overline{\Th}\in \SQ(n,m)$, for $\Th\in Sq(n,m)$ as homogeneous morphisms.  

\begin{remark} We can define two wide subcategories of $\SQ$, $\SQ_0\subset \SQ^+ \subset \SQ$, as the wide subcategories with the following $\K$-modules of morphisms:
    \begin{align*}
        \SQ^+(n,m)=\K \{ \overline{\Th}\in \SQ(n,m) \ | \ \Th \in Sq^+(n,m)  \},\\
        \SQ_0(n,m)=\K \{ \overline{\Th} \in \SQ(n,m) \ | \ \Th \in Sq_0(n,m) \}. 
    \end{align*}
    Diagrammatically, morphisms in $\SQ^+$ are $\K$-linear combinations of (classes of) SWB diagrams on orientable surfaces, and morphisms in $\SQ_0$ are $\K$-linear combinations of SWB diagrams in the square, \textit{i.e.} Temperley-Lieb diagrams. Since no orientable surface contains twisted curves we have that $\SQ^+=\SQ^+(\a,\gamma)$ (\textit{i.e.} $\SQ^+$ is only defined by these two parameters) and similarly, since the square only contains separating curves, $\SQ_0=\SQ_0(\a)$. Furthermore, one can identify $\SQ_0(\a)$ with the Temperley-Lieb category at loop parameter $\alpha$, see  \cite{TemperleyLieb, JonesIndex}.   

    In addition, define the following submodules of $\SQ(n,m)$:
    \begin{align*}
        \SQ_{\tor}(n,m)=\K \{ \overline{\Th}\in \SQ(n,m) \ | \ \Th \in Sq_{\tor}(n,m)  \}, \\
        \SQ_{\mob}(n,m)=\K \{ \overline{\Th} \in \SQ(n,m) \ | \ \Th \in Sq_{\mob}(n,m) \}. 
    \end{align*}
    It follows from Cor. \ref{hscor} that any homogeneous morphism, can be factored into homogeneous morphisms coming from $\SQ_0$, $\SQ_{\tor}$, and $\SQ_{\mob}$, with at most two factors from the latter.   
\end{remark}

\begin{prop} Let $\Th\in Sq(n,m)$, and consider the homogeneous morphism $\overline{\Th}\in \SQ(n,m)$. There exist unique non-negative integers $a,b,c \in \N$, and a SWB datum $\Th'\in Sq(n,m)$ without internal components, such that $\overline{\Th}=\a^a \, \b^b \, \gamma^c \ \overline{\Th'}$.
\end{prop}

\begin{theorem} There is a tensor product $(\_)\otimes(\_):\SQ\times \SQ \to \SQ$, defined on objects by $n_1\otimes n_2=n_1+n_2$ and uniquely defined on morphisms by the formulae:
\begin{equation*} \id_l\otimes \overline{\Th}=\overline{(\I_L)^l(\Th)}, \qquad \overline{\Th}\otimes \id_l=\overline{(\I_R)^l(\Th)},
\end{equation*}
for $\Th\in Sq(n,m)$, where $\I_L$ and $\I_R$ are as per \ref{Ilr}. In addition, $\SQ$ can be made into a rigid monoidal category with self-dual objects $n^\dagger=n$, and evaluation and coevaluation maps given by:
\begin{equation*}
    \mathit{ev}_n = \overline{\mathbb{V}_n} \in \SQ(2n,0), \qquad   \mathit{coev}_n = \overline{\mathbb{U}_n} \in \SQ(0,2n),\end{equation*}
Let $(\_)^\dagger: \SQ \to \SQ$ denote the associated contravariant functor (not to be confused with $(\_)^*$ from Theorem \ref{SWBcat}). Furthermore, we have the following identities 
$((\_)^\dagger)^\dagger=\text{id}:\SQ \to \SQ$ and $((\_)^\dagger)^*=((\_)^*)^\dagger:\SQ \to \SQ$.
\end{theorem}

\begin{remark} The restriction of the tensor product $\otimes$ to the subcategory $\SQ_0$ coincides with the usual ``horizontal juxtaposition" tensor product in the Temperley-Lieb category. The composition of contravariant functors $(\_)^\dagger$ and $(\_)^*$ results in a covariant endofunctor $R:=((\_)^{\dagger})^*:\SQ\to \SQ$. Since both contravariant functors are involutive and commute, it follows that $R^2=\text{id}$, and since $(\_)^\dagger$ is tensor product reversing, while $(\_)^*$ is tensor product preserving, it follows that $R$ is tensor product reversing. The restriction of $R$ to $\SQ_0$ is simply the functor, defined on TL-diagrams by horizontal reflection (\textit{i.e.} reflection across a vertical).   \end{remark}

\begin{remark} Since $(\_)^\dagger$ is involutive, $\SQ$ has a trivial pivotal structure $\text{id}:\text{id}\Rightarrow ((\_)^\dagger)^\dagger$.    
\end{remark}

\begin{prop}  $\SQ$ is monoidally generated by the homogeneous morphisms $\id_0$, $\id_1$, $\overline{\mathbb{U}}$, $\overline{\mathbb{V}}$, $\overline{\TT_{l,m}}$, and $\overline{\MM_n}$, for $l,m,n\in \N$. Furthermore, the morphisms $\overline{\TT_{l,m}}\in \SQ(l+m,l+m)$ are the components of a natural transformation $\overline{\TT}:(\_)\otimes(\_) \Rightarrow (\_)\otimes^{\text{op.}}(\_)$, and the morphisms $\overline{\MM}_n\in \SQ(n,n)$ are the components of a natural transformation $\overline{\MM}:\text{id} \Rightarrow R$.
\end{prop}

\begin{remark} Note that the morphisms $\overline{\TT_{l,m}}$ are not invertible in $\SQ$, so $\overline{\TT}$ does not make $\SQ$ into a braided monoidal category. These morphisms do, however, satisfy the Yang-Baxter equation, \textit{i.e.} for $l,m,n\in \N$ we have
\begin{equation*}
(\id_n\otimes \overline{\TT_{l,m}})\circ (\overline{\TT_{l,n}}\otimes \id_m) \circ (\id_l\otimes \overline{\TT_{m,n}})=
    (\overline{\TT_{m,n}}\otimes \id_{l})\circ (\id_{m}\otimes \overline{\TT_{l,n}}) \circ (\overline{\TT_{l,m}}\otimes \id_{n}).\end{equation*}
There are also mixed relations between the $\TT_{l,m}$ and $\MM_n$, for example, for $l,m\in \N$ we have
\begin{equation*}
    (\id_m\otimes \overline{\MM_l})\circ 
    \overline{\MM_{l+m}} \circ 
    (\id_l\otimes \overline{\MM_m})=(\id_m\otimes \overline{\MM_0})\circ (\overline{\TT_{l,m}}).
\end{equation*}

\end{remark}

\bibliographystyle{alpha}
\bibliography{bib.bib}

\end{document}